\pdfoutput=1
\documentclass[11pt]{nyjm}
\usepackage{lmodern}
\usepackage{hyperref,pb-diagram,amssymb,epic,eepic,verbatim,graphicx,graphics,braket}
\hypersetup{colorlinks}
\usepackage{color}
\definecolor{darkred}{rgb}{0.5,0,0}
\definecolor{darkgreen}{rgb}{0,0.5,0}
\definecolor{darkblue}{rgb}{0,0,0.5}
\hypersetup{ colorlinks,
linkcolor=darkblue,
filecolor=darkgreen,
urlcolor=darkred,
citecolor=darkblue }
\usepackage{pb-diagram,amssymb,epic,eepic,verbatim,graphicx}
\usepackage{graphics} 

\renewcommand\theenumi{\alph{enumi}}
\renewcommand\theenumii{\roman{enumii}}
\newtheorem{theorem}{Theorem}[subsection]
\newtheorem{assumption}{Assumption}[subsection]
\newtheorem{corollary}[theorem]{Corollary}

\newtheorem{proposition}[theorem]{Proposition}
\newtheorem{lemma}[theorem]{Lemma}
\newtheorem{lem}[theorem]{}
\theoremstyle{definition}
\newtheorem{definition}[theorem]{Definition}
\newtheorem{remark}[theorem]{Remark}
\newtheorem{example}[theorem]{Example}
\newcommand{\blem}{\begin{lem} \rm}
\newcommand{\elem}{\end{lem}}
\newcommand\cX{\mathcal{X}}
\newcommand\B{\mathcal{B}}
\newcommand\E{\mathcal{E}}
\newcommand\M{\mathcal{M}}

\newcommand\cH{\mathcal{H}}

\newcommand\cG{\mathcal{G}}

\renewcommand\M{\mathcal{M}}

\newcommand{\V}{\mathcal{V}}
\renewcommand{\L}{\mathcal{L}}
\renewcommand{\O}{\mathcal{O}}

\newcommand{\J}{\mathcal{J}}

\newcommand{\U}{\mathcal{U}}

\newcommand{\F}{\mathcal{F}}

\newcommand{\R}{\mathbb{R}}
\renewcommand{\H}{\mathbb{H}}

\newcommand{\RR}{\mathcal{R}}
\newcommand{\C}{\mathbb{C}}
\newcommand{\cC}{\mathcal{C}}

\newcommand{\Z}{\mathbb{Z}}

\newcommand{\PP}{\mathcal{P}}

\newcommand\lie[1]{\mathfrak{#1}}

\newcommand{\w}{\lie{w}}

\renewcommand{\sl}{\lie{sl}}
\newcommand{\on}{\operatorname}
\newcommand{\ainfty}{{$A_\infty$\ }}

\newcommand{\red}{{\on{red}}}
\newcommand{\unred}{{\on{unred}}}

\newcommand{\Fred}{\on{Fred}}

\newcommand{\Prin}{\on{Prin}}

\newcommand{\Or}{\on{Or}}
\newcommand{\dual}{\vee}

\newcommand{\myvert}{{\on{vert}}}
\newcommand{\Edge}{\on{Edge}}

\newcommand{\Lag}{\on{Lag}}

\newcommand{\loc}{{\on{loc}}}
\newcommand{\Ver}{\on{Vert}}

\newcommand{\End}{\on{End}}

\newcommand{\Aut}{ \on{Aut} }

\newcommand{\Hom}{ \on{Hom}}
\newcommand{\Ind}{ \on{Ind}}
\renewcommand{\ker}{ \on{ker}}
\newcommand{\coker}{ \on{coker}}
\newcommand{\im}{ \on{im}}

\newcommand{\Spin}{ \on{Spin}}

\newcommand{\diag}{  \on{diag}}

\newcommand{\ssm}{-}

\newcommand\dirac{/\kern-1.2ex\partial} 
\newcommand\qu{/\kern-.7ex/} 
\newcommand\lqu{\backslash \kern-.7ex \backslash} 

\newcommand\dr{r_+ \kern-.7ex - \kern-.7ex r_-}
 
\def\pd{\partial}

\newcommand{\labell}\label

\renewcommand{\d}{{\mbox{d}}}
\newcommand{\ol}{\overline}

\newcommand\eps{\epsilon}
\newcommand\Om{\Omega}

\newcommand{\ti}{\tilde}

\newcommand\pt{\on{pt}}
\newcommand\cE{\mathcal{E}}

\newcommand\cF{\mathcal{F}}
\newcommand\cP{\mathcal{P}}
\newcommand\cQ{\mathcal{Q}}

\newcommand\mE{\mathcal{E}}

\newcommand\Gr{\on{Gr}}
\newcommand\Map{\on{Map}}
\newcommand\rank{\on{rank}}
\newcommand\ev{\on{ev}}

\newcommand\Pic{\on{Pic}}
\newcommand\Vect{\on{Vect}}
\newcommand\ul{\underline}

\renewcommand\H{\mathcal{H}}

\renewcommand\bra[1]{ < \kern-.7ex {#1} \kern-.7ex >} 
\newcommand\bdefn{\begin{definition}}
\newcommand\edefn{\end{definition}}
\newcommand\bea{\begin{eqnarray*}}
\newcommand\eea{\end{eqnarray*}}
\newcommand\bcv{\left[ \begin{array}{r} }
\newcommand\ecv{\end{array} \right] }

\newcommand\bma{\left[ \begin{array} }
\newcommand\ema{\end{array} \right]}
\newcommand\ben{\begin{enumerate}}
\newcommand\een{\end{enumerate}}
\newcommand\beq{\begin{equation}}
\newcommand\eeq{\end{equation}}
\newcommand\bex{\begin{example}}
\newcommand\bsj{\left\{ \begin{array}{rrr} }
\newcommand\esj{\end{array} \right\}}

\newcommand\Id{\on{Id}}
\newcommand\cI{\mathcal{I}}

\newcommand\eex{\end{example}}

\newcommand\sx{*\kern-.5ex_X}

\def\mathunderaccent#1{\let\theaccent#1\mathpalette\putaccentunder}
\def\putaccentunder#1#2{\oalign{$#1#2$\crcr\hidewidth \vbox
to.2ex{\hbox{$#1\theaccent{}$}\vss}\hidewidth}}

\newcommand\GFuk{{\on{Fuk}^\#}}
\newcommand\Fuk{{\on{Fuk}}}

\newcommand{\bE}{\mathbb{E}}

\newcommand{\nS}{S}
\newcommand{\nE}{E}
\newcommand{\nF}{F}
\newcommand{\nxi}{\xi}
\newcommand{\dxi}{\ti{\xi}}
\newcommand{\nnS}{\hat{S}}
\newcommand{\nnE}{\hat{E}}
\newcommand{\nnF}{\hat{F}}

\newcommand{\rS}{S^{\rho}}
\newcommand{\rE}{E^{\rho}}
\newcommand{\rF}{F^{\rho}}
\newcommand{\rhS}{\hat{S}^{\rho}}
\newcommand{\rhE}{\hat{E}^{\rho}}
\newcommand{\rhF}{\hat{F}^{\rho}}
\newcommand{\dS}{\ti{S}}

\newcommand{\dE}{\ti{E}}
\newcommand{\dF}{\ti{F}}
\newcommand{\gluing}{\tau}
\newcommand{\disk}{\on{disk}}
\newcommand{\oS}{\hat{S}_{\on{d}}}
\newcommand{\oE}{\hat{E}_{\on{d}}}
\newcommand{\oF}{\hat{F}_{\on{d}}}
\newcommand{\iS}{\hat{S}_{\on{c}}}
\newcommand{\iE}{\hat{E}_{\on{c}}}

\newcommand\cU{\mathcal{U}}
\newcommand{\DD}{\mathcal{D}}
\newcommand\cV{\mathcal{V}}

\newcommand{\GG}{\mathcal{G}}
\renewcommand{\SS}{\mathcal{S}}
\renewcommand{\sharp}{{\#}}
\newcommand\sDel{\Delta}
\begin{document}

\setcounter{page}{1}
\setcounter{section}{0}

\title{Orientations for pseudoholomorphic quilts}

\author{Katrin Wehrheim} 

\address{Department of Mathematics,
UC Berkeley, CA 94720-3840.}
\email{katrin@math.berkeley.edu}

\author{Chris Woodward}
\address{Department of Mathematics, 
Rutgers University,
Piscataway, NJ 08854.}
\email{woodwardc@gmail.com, ctw@math.rutgers.edu}

\thanks{Partially supported by NSF grants CAREER 0844188 and DMS 0904358.
Any opinions, findings and conclusions or recommendations expressed in this material are those of the author(s) and do not necessarily reflect the views of the National Science Foundation.}

\begin{abstract}  
We construct coherent orientations on moduli spaces of
pseudoholomorphic quilts and determine the effect of various gluing
operations on the orientations.  We also investigate the behavior of
the orientations under composition of Lagrangian correspondences.
\end{abstract} 

\maketitle

\tableofcontents

\section{Introduction}  
\label{sec:coherent}

In previous work on quilted Floer cohomology \cite{ww:quilts},
\cite{ww:quiltfloer}, \cite{ww:isom}, we associated to a sequence of
Lagrangian correspondences between symplectic manifolds satisfying
certain conditions a {\em quilted Floer cohomology group}.  The
boundary operator in quilted Floer theory is defined by a signed count
of isolated {\em quilted pseudoholomorphic cylinders} consisting of
collections of pseudoholomorphic strips with Lagrangian seam
conditions, analogous to the way that Morse homology is defined by a
signed count of gradient trajectories.  In the case of Morse homology,
the signs are derived from orientations on the spaces of Morse
trajectories induced by choices of orientations on stable manifolds
for each critical point, as in \cite[Chapter 2]{schwarz}.
In this paper, we similarly construct coherent orientations on moduli spaces of
pseudoholomorphic quilts by auxiliary choices.  Much of the construction follows the work of Fukaya-Oh-Ohta-Ono in \cite{fooo}.  Since the preprint version of this paper, further work of Fukaya \cite{fuk:corresp} has appeared giving another approach to the construction of \ainfty functors. 

The main result, for a single quilted domain, is the following: Let
$\ul{S}$ be a {\em quilted surface}.  Such a surface is obtained from
a collection of {\em patches} $(S_p,j_p), \ p \in \PP$, complex
surfaces with strip-like ends, by gluing along boundary components.
Let $\ul{M}$ be a collection of {\em symplectic labels} for the
patches
\[ \ul{M} = (M_p, \ p \in \PP) \] 
assigning to each patch $S_p$ a symplectic manifold $M_p$, equipped
with compatible almost structures $J_p: TM_p \to TM_p$.  Let $\ul{L}$
be a collection of {\em Lagrangian seam and boundary conditions}: for each
seam or boundary component $\sigma \subset \ul{S}_{p_-} \cap
\ul{S}_{p_+}$ a Lagrangian seam or boundary condition
\[ \ul{L} = (L_\sigma \subset M_{p_-}^-
\times M_{p_+} ) \] 
where $M_{p_\pm}$ is a point if $\sigma$ represents a boundary
component.  Suppose that the components of $\ul{L}$ are
equipped with relative spin structures.  Let $\ul{x}$ denote a
collection of generalized intersection points for the quilted ends and
$\M(\ul{M},\ul{L},\ul{x})$ the space of pseudoholomorphic quilts
\begin{equation} \label{eq:Jhol} \ul{u} = (u_p: S_p \to M_p, \ \ J_p \d u_p = \d u_p j_p \ \ p \in
\PP) \end{equation} 
with domain $\ul{S}$, targets $\ul{M}$, Lagrangian seam and boundary
conditions $\ul{L}$, and limits $\ul{x}$; see Figure \ref{fig:bc quilt}
for a picture.  We suppose that almost
complex structures have been chosen so that the moduli space
$\M(\ul{M},\ul{L},\ul{x})$ is regular; that is, is cut out
transversally by the Cauchy-Riemann equation so that the tangent space
at $u \in \M(\ul{M},\ul{L},\ul{x})$ is the kernel of a surjective
Fredholm operator denoted $D_u$:
\[ T_u \M(\ul{M},\ul{L},\ul{x}) \cong \ker(D_u) .\]

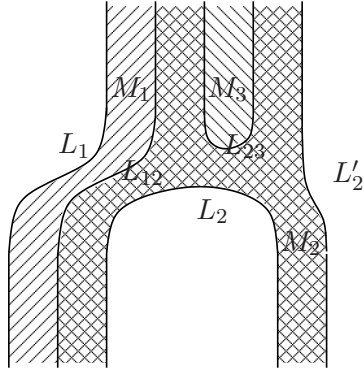
\begin{figure}[ht]
\begin{picture}(0,0)
\includegraphics{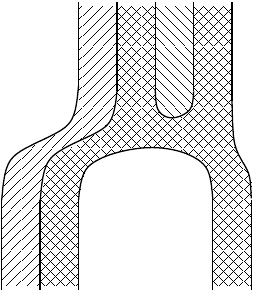}
\end{picture}
\setlength{\unitlength}{2693sp}
\begingroup\makeatletter\ifx\SetFigFont\undefined
\gdef\SetFigFont#1#2#3#4#5{
  \reset@font\fontsize{#1}{#2pt}
  \fontfamily{#3}\fontseries{#4}\fontshape{#5}
  \selectfont}
\fi\endgroup
\begin{picture}(2969,3419)(2819,-3458)
\put(5076,-2401){\makebox(0,0)[lb]{${M_2}$}}
\put(3511,-981){\makebox(0,0)[lb]{${M_1}$}}
\put(3026,-1501){\makebox(0,0)[lb]{${L_1}$}}
\put(3601,-1751){\makebox(0,0)[lb]{${L_{12}}$}}
\put(4300,-2106){\makebox(0,0)[lb]{${L_2}$}}
\put(4401,-981){\makebox(0,0)[lb]{${M_3}$}}
\put(5551,-1801){\makebox(0,0)[lb]{${L_2'}$}}
\put(4546,-1506){\makebox(0,0)[lb]{${L_{23}}$}}
\end{picture}
\caption{Lagrangian boundary conditions for a quilt}
\label{fig:bc quilt}
\end{figure}

Our goal is to orient the moduli space.  That is, we wish to provide
the top exterior power of the tangent space, isomorphic to the
determinant line of the Fredholm operator, with a system of non-zero
elements
\[ o_u \in \Lambda^{\max}( T_u \M(\ul{M},\ul{L},\ul{x})) \cong
\det(D_u) .\]
We do this for quilts whose patches admit a {\em coloring}, that is, a
function $\w: \PP \to \Z_2$ taking different values on the two patches
adjacent to any seam (Definition \ref{def:coloring}).  For example, any
quilt whose graph of patches is a tree, such as the quilts defining
the Fukaya categories and functors in Section \ref{sec:hfsec}, admits a
coloring, while the quilted cylinders for a cyclic sequence of
correspondences admit one if and only if the number of seams is even.
We then investigate the signs of various gluing operations.  The main
result is:

\begin{theorem} \label{thm:main}  Suppose that $\M(\ul{M}, \ul{L},\ul{x})$ is regular, 
that the patches of $\ul{S}$ admit a coloring, 
and $\ul{L}$ are relatively spin as above.   Then a choice of coloring, relative spin structures, orientations of the ends, and orderings of the patches and boundary components induces an orientation of $\M(\ul{M},\ul{L},\ul{x})$.  
The operations of gluing along strip-like ends, gluing at boundary nodes, and composition of seam conditions, for gluings compatible with the colorings, act on determinant lines by universal signs specified by Lemma \ref{lem:quiltedgluing} and Propositions \ref{prop:quiltcyl} and \ref{prop:collapse}.
\end{theorem}

In particular, under suitable monotonicity assumptions we obtain
versions of Lagrangian Floer cohomology defined over the integers,
the Fukaya category, and the quilted Fukaya category with the functors
associated to Lagrangian correspondences, described in Section
\ref{sec:hfsec}; the last two are constructed in Sections
\ref{sec:dfsec} and \ref{sec:fukops} respectively.  A family version is given in Theorem
\ref{thm:mainres} below.

The orientations are constructed as follows.  For each generator of
the cochain complex we fix an orientation on a certain determinant
line associated to a once-marked disk with a path of Lagrangian
subspaces along the boundary.  We then show that the orientations
constructed in this way have good gluing properties: the quilted Floer
boundary operator squares to zero, and the composition theorem of
\cite{ww:isom} holds over the integers.  The construction is a
generalization of the work of several authors already in the
literature.  For periodic Floer trajectories the construction is given
in Floer-Hofer \cite{fh:co}, while for pseudoholomorphic disks the
construction is outlined in Fukaya-Oh-Ohta-Ono \cite{fooo}, and
generalized in Ekholm-Etnyre-Sullivan \cite{ek:or}.  In the setting of
Fukaya categories orientations are constructed in Seidel's book
\cite{se:bo}.  Most of the proofs in this paper are slight
modifications of proofs in one of these sources, and so this part of the paper
should be considered largely expository.

For quilted surfaces we define the orientation by deforming the seam conditions to split form, so that the operator becomes a direct sum of operators on the individual patches, and orienting the direct sum using the orientations 
of the summands (Proposition \ref{prop:noends}); the coloring determines the patches whose orientations are defined using the conjugate complex structure, so that the seam conditions become boundary conditions for the individual patches.  The signs for gluing at seams and at boundary nodes in the quilted case are deduced from the split case in this way.  

\vskip .1in 
\noindent  {\em A note on authorship:}  The present paper is an updated and more detailed version of a paper
the authors have circulated since 2007. The authors have unreconciled
differences over the exposition in the paper, and explain their points
of view at
\url{https://math.berkeley.edu/~katrin/wwpapers/}
resp.
\url{http://christwoodwardmath.blogspot.com/}. The
publication in the current form is the result of a mediation.

 \vskip .1in 
\noindent {\em Artificial intelligence methodology: }  This revision  benefitted from several referee reports from large language models, most notably a report from Claude Opus 5.5 and a report from GPT-6.  Some of the text added to the previous arxiv version was produced by Claude as part of the refereeing discussion and edited for clarity.

\section{Orientations for Cauchy-Riemann operators}

In this section, we construct orientations for determinant lines of
Cauchy-Riemann operators.  Most of this material is standard:
Knudsen-Mumford \cite{km:det} study determinant lines of complexes
with a view towards orienting moduli spaces of curves; Quillen
\cite{qu:de} introduced determinant lines of Cauchy-Riemann operators;
Segal gave a construction of a determinant line bundle over the space
of Fredholm operators published eventually in \cite[Appendix
  D]{se:def}.  This construction is described in more detail in Huang
\cite[Appendix D]{hu:2d}, see also Freed \cite{fr:det}, \cite{freed:revisit}.  Determinant
lines for families of Cauchy-Riemann operators occurring in symplectic
geometry are studied in McDuff-Salamon \cite[Appendix A.2]{ms:jh},
Seidel \cite[Section 11]{se:bo}, and Solomon \cite{so:lb} sometimes
with different conventions.

\subsection{Determinant lines}

We begin with a review of the construction and properties of the
determinant line bundle over the space of Fredholm operators.  Let
$V,W$ be real Banach spaces.  Let $\Fred(V,W)$ be the space of
Fredholm operators $D: V \to W$, that is, operators with finite
dimensional kernel and cokernel
\[ \ker(D) = \{ v \in V \ | \ D(v) = 0 \}, \quad 
\coker(D) = W / \{ D(v) \ |  \ v \in V \} .\]
The image of a Fredholm operator is necessarily closed.

\begin{definition}  
\begin{enumerate}
\item 
 {\rm (Indices)} The {\em index} of a Fredholm operator $D: V \to W$ is the
  integer
\[\Ind(D) = \dim(\ker(D)) - \dim(\coker(D)).\]
\item {\rm (Determinant lines)} The {\em determinant line} of a
  Fredholm operator $D: V \to W$ is the one-dimensional vector space
\[\det(D) = \Lambda^{\max}(\coker(D)^\dual) \otimes
\Lambda^{\max}(\ker(D)) \]
where $\coker(D)^\dual := \Hom(\coker(D),\R)$ is the dual of $\coker(D)$
and $\Lambda^{\max}$ denotes the top exterior power. 
\end{enumerate} 
\end{definition} 

As in Freed \cite{freed:revisit}, the determinant line is best considered an object in the category of {\em $\Z_2$-graded lines}, that is, even or odd depending on whether the index was even or  odd.   Many of the signs in the paper are then Koszul signs. 

\begin{remark} Our convention follows Seidel \cite{se:bo}, which is adapted from Knudsen-Mumford \cite{km:det}.  The original definition 
of  Quillen \cite{qu:de} defines the determinant as 
$\Lambda^{\max}(\ker(D)^\dual) \otimes
\Lambda^{\max}(\coker(D)) $.   McDuff-Salamon \cite{ms:jh} uses our convention but with the factors reversed; however, the clarification posted on Salamon's website \cite{salamon:det} uses the opposite order.   Zinger \cite{zinger:det} defines the determinant line bundle as the 
tensor product of the top exterior power of the kernel with the dual of the top exterior power of the cokernel, which is essentially Quillen's
\cite{qu:de} definition.    
\end{remark}

\begin{remark}  \label{rem:sums}
\begin{enumerate}
\item {\rm (Behavior under direct sums)} Let $V_j,W_j$ be real Banach
  spaces for $j = 1,2$ and $D_j:V_j \to W_j$ Fredholm operators.
  Denote by
\[D_1 \oplus D_2: V_1 \oplus V_2 \to W_1 \oplus W_2\]
the direct sum of the operators $D_1$ and $D_2$.  Equality of indices
\[\Ind(D_1 \oplus D_2) = \Ind(D_1) + \Ind(D_2) \]
holds and there is a canonical isomorphism of determinant lines
\begin{equation}
\label{eq:caniso}
 \det(D_1 \oplus D_2) \to \det(D_1) \otimes \det(D_2) .\end{equation}
Explicitly, let
\[\{ v_{k,i}, i = 1,\ldots, a_k := \dim(\ker(D_k)) \} \subset \ker(D_k) \]
\[ \{ w_{k,i}^\dual, i =
1,\ldots, b_k:= \dim(\coker(D_k)) \} \subset \coker(D_k)^\dual\] 
be bases for $ k = 1,2$.  The isomorphism \eqref{eq:caniso} is defined by
\begin{multline} \label{eq:tens}
\left( \wedge_{i=1}^{b_2} w_{2,i}^\dual
\wedge \wedge_{i=1}^{b_1} w_{1,i}^\dual \right) \otimes \left( \wedge_{i=1}^{a_1} v_{1,i} \wedge
\wedge_{i=1}^{a_2} v_{2,i} \right)  \\
\mapsto (-1)^{\dim(\coker(D_2))\Ind(D_1)}  
 \left( \wedge_{i=1}^{b_1} w_{1,i}^\dual \otimes \wedge_{i=1}^{a_1} v_{1,i} \right)
\otimes \left( \wedge_{i=1}^{b_2} w_{2,i}^\dual \otimes \wedge_{i=1}^{a_2} v_{2,i} \right) .
\end{multline}
The isomorphism \eqref{eq:caniso} is associative and graded commutative
in the following sense: The composition 
\begin{equation} \label{eq:signs}
 \det(D_2) \otimes \det(D_1) \to \det(D_2 \oplus D_1) \to \det(D_1
\oplus D_2) \to \det(D_1) \otimes \det(D_2), \end{equation}
where the middle map is induced by exchange of summands, agrees with
the map 
\[  \det(D_2) \otimes \det(D_1) \to \det(D_1) \otimes \det(D_2) \]
induced by exchange of factors by a sign $(-1)^{\Ind(D_1) \Ind(D_2)}$.
\item {\rm (Determinant lines in finite dimensions)} Let $D:V\to W$ be
  a linear operator on finite dimensional spaces $V,W$.  There is a
  canonical isomorphism to the determinant of the trivial operator
  from $V$ to $W$,
\begin{equation} \label{eq:tD}
t_D: \det(D) \to \det(0)=\Lambda^{\max}(W^\dual) \otimes \Lambda^{\max}(V) .
\end{equation} 
To define the map \eqref{eq:tD} explicitly, choose bases
\[ \{ e_1,\ldots,e_n \} \subset V, \quad \{ f_1,\ldots,f_m  \} \subset  W  \] 
so that
\[D(e_j) = f_j, \ j=1,\ldots,k, \quad D(e_j) = 0, \ j = k+1,\ldots,n .\] 
Let $f_1^\dual,\ldots,f_m^\dual$ be the basis for $W^\dual$ dual to
$f_1,\ldots, f_m \in W$.
Define
\[
 t_D ( (f_m^\dual \wedge \ldots \wedge f_{k+1}^\dual) \otimes (e_{k+1} \wedge
\ldots  \wedge e_n)) := (f_m^\dual \wedge \ldots \wedge 
f_1^\dual) \otimes (e_1\wedge \ldots \wedge e_n) .
\]
Note that $t_D$ is independent of the choice of bases $e_i,f_j$.
\end{enumerate} 
\end{remark} 

\begin{remark}  {\rm (Determinant line bundles)} 
  For real Banach spaces $V,W$ let
\[\det(V,W) \to \Fred(V,W), \quad \det(V,W)_D := \det(D)\]
be the {\em determinant line bundle} whose fiber $\det(V,W)_D$ over
$D$ is the one-dimensional vector space $\det(D)$, as introduced by Quillen \cite{qu:de}.  Zinger \cite[Theorems 1, 2]{zinger:det} describes the topology in more detail and gives a classification of the topological structures on 
$\det(V,W)$ which satisfy desired conditions.  As a result,  
\label{rem:families}
the construction of determinant lines works in families: For a
topological space $X$ consider Fredholm morphisms $\ti{D}:\ti{V}\to
\ti{W}$ of Banach vector bundles $\ti{V}\to X$, $\ti{W}\to X$.  The
determinant line bundle of $\ti{D}$ is a topological line bundle over $X$
defined by 
\[ \det(\ti{D}) \to X, \quad \det(\ti{D})_x := \det(\ti{D}| \ti{V}_x
\to \ti{W}_x) \]
with fibers $\det(\ti{D})_x$ the determinant lines of the restriction
of $\ti{D}$ to fibers.  In particular, any homotopy of Fredholm
operators $\ti{D} = (D_t)_{ t \in [0,1]}$ induces a determinant line
bundle $\det(\ti{D})$ over $X = [0,1]$.   The bundle $\det(\ti{D})$
is trivial, and any trivialization induces an isomorphism of determinant lines $\det(D_0) \to \det(D_1)$, uniquely up to isotopy. We refer to this isomorphism throughout the text as an {\em isomorphism of
  determinant lines induced by a deformation of
  operators}. 
\end{remark}

\begin{remark} 
The determinant line bundle on Fredholm operators
induces the following version of the determinant line bundle on finite dimensional Grassmannians.  Let $V,W$ be finite dimensional.  Then 
$\Fred(V,W) = \Hom(V,W)$.   Let $\Hom_{\on{inj}}(V,W) \subset \Hom(V,W)$
denote the (possibly empty) subspace of injective operators. 
The Grassmannian $\Gr( \dim(V), W)$ of $\dim(V)$-dimensional subspaces
of $W$ is isomorphic to the quotient 
\[  \Hom_{\on{inj}}(V,W)/\Aut(V)  \to \Gr(\dim(V),W)\] 
with $\Aut(V)$ acting by pre-composition, by the  map 
$[D] \mapsto \on{Im}(D)$.   The tangent space
to $\Gr(\dim(V),W)$ is isomorphic to the bundle whose fiber 
at $[D] \in \Gr(\dim(V),W)$ is isomorphic to $\Hom(V,W/D(V))$.  Since $\ker(D) = 0$ for injective $D$, the bundle whose fiber is $\det(\ker(D)^\dual) \otimes \det(\coker(D)) = \Lambda^{\max}(W/D(V))$ is the determinant of the universal quotient bundle, that is, the Pl\"ucker line bundle $\mathcal{O}(1)$, while $\det(\coker(D)^\dual) \otimes \det(\ker(D))$ is its dual $\mathcal{O}(-1)$.   Note that in dualizing a tensor product it is natural to reverse the order of the factors.  

We make some remarks on the various historical conventions in this light.  Quillen's convention \cite{qu:de}, rather than Seidel's \cite{se:bo}, gives the ample line bundle over the Grassmannian.  On the other hand, on the space of  invertible operators the section $1^\dual \otimes 1$ extends to a section over the space of index zero operators only with Quillen's convention, as explained in Zinger \cite[3.11]{zinger:det}.  Segal \cite[Appendix B]{se:def} defines the determinant line of a Fredholm operator as the line as a quotient associated to the determinant, defined on determinant class operators between Banach spaces, and shows that the determinant line bundle descends to a semi-infinite version of the Grassmannian. 
\end{remark}

\subsection{Orientations for Fredholm operators}

By definition, an orientation for a Fredholm operator is an orientation
of its determinant line.  There are natural
constructions of orientations on duals and sums of Fredholm operators.

\begin{definition} {\rm (Orientations for Fredholm operators)} 
\begin{enumerate} 
\item 
Let $V$ be a finite dimensional real vector space, and
$\Lambda^{\max}(V)$ its top exterior power.  An {\em orientation} for
$V$ is a component of $\Lambda^{\max} V \setminus \{0 \}$, that is, a
non-vanishing element of $\Lambda^{\max} V$ up to homotopy.
Denote by 
\[ \Or(V) := ( \Lambda^{\max}(V) \setminus \{ 0 \}) / \R_{> 0 } \]
the space of orientations.
\item 
An {\em oriented vector space} is a pair $(V,o)$ of a vector space $V$
and an orientation $o \in \Or(V)$.  Given an oriented vector space
$(V,o)$ of dimension $n$, a basis $e_1,\ldots,e_n$ of $V$
is {\em oriented} if
\[ o = \R_{> 0} ( e_1 \wedge \ldots \wedge e_n) \in \Or(V) \]
defines the orientation $o$ on $V$.  
\item
Let $V$ and $W$ be finite dimensional vector spaces.  A linear
isomorphism $T: V \to W$ induces a map on orientations 
\[\Or(T): \Or(V)
\to \Or(W) .\]  
If $V,W$ are oriented then the map $T$ is orientation-preserving resp. 
reversing  if $\Or(T)$ is orientation-preserving resp. reversing. 
\item An {\em orientation} of a Fredholm operator $D: V \to W$ between
  real Banach spaces $V,W$ is an orientation of the one-dimensional
  vector space given by its determinant line $\det(D)$.
\end{enumerate} 
\end{definition} 

\begin{remark}  \label{rem:duals}
\begin{enumerate} 
\item {\rm (Orientations on duals)} An orientation for a finite
  dimensional vector space $V$ induces an orientation for the dual
  $V^\dual$.  Explicitly, let $e_1,\ldots,e_n, n = \dim(V)$ be an
  oriented basis for $V$ and $e_1^\dual,\ldots,e_n^\dual$ the dual
  basis for $V^\dual$.  Give $V^\dual$ the orientation defined by
\begin{equation} \label{eq:backwards}
o_{V^\dual} := [ e_n^\dual \wedge
\ldots \wedge e_1^\dual \in \Lambda^{\max}(V^\dual)]
.\end{equation} 
Note the reverse order.  Identify $V$ with $V^\dual$ by an inner
product $B: V \times V \to \R$:
\[L: V \to V^\dual, \quad v \mapsto B(v,\cdot) .\]
The orientation on $V$ relates to the pull-back orientation on
$V^\dual$ by
\[ L^* o_{V^\dual} = (-1)^{\dim(V)(\dim(V) - 1)/2} o_V .\]  
This convention is opposite to the convention of \cite{ek:or}.
\item {\rm (Orientations on direct sums)} Orientations on finite
  dimensional vector spaces $V,W$ induce an orientation on the direct
  sum $V \oplus W$ as follows.  Let 
\[ \{ e_1,\ldots,e_n \} \subset V, \quad  \{ f_1,\ldots,f_m \} \subset W \]
be oriented bases.  Define on the sum $V \oplus W$ the orientation
given by
\[e_1 \wedge \ldots \wedge e_n \wedge f_1 \wedge \ldots \wedge f_m \in
\Lambda^{\max}(V \oplus W)
.\]
The isomorphism $i : V \oplus W \to W \oplus V$ given by transposition
acts on orientations 
\[ o_{V \oplus W} = (-1)^{\dim(V)\dim(W)} i^* o_{W \oplus V} .\]
Note that $\Lambda^{\max}(V^\dual \oplus W^\dual)$ is naturally isomorphic 
to $\Lambda^{\max}(W^\dual) \otimes \Lambda^{\max}(V^\dual)$, so identifying
with $\Lambda^{\max}(V^\dual) \otimes \Lambda^{\max}(W^\dual)$ introduces
a Koszul sign $(-1)^{\dim(V) \dim(W)}$.
\item \label{item:orient} {\rm (Orientations on complex vector spaces)}
Let $V$ be a complex vector space.   If $v_1,\ldots, v_n$ is an oriented
basis over $\C$, then 
\[ \{  v_1, i v_1, \ldots, v_n, iv_n  \} \subset V \]
is a basis over $\R$ and the orientation given by  
\[ o(V) := [ v_1 \wedge i v_1 \wedge \ldots \wedge v_n \wedge iv_n ] \] 
is called the {\em complex orientation}
on $V$.   Let $W$ denote the real span of $v_1,\ldots, v_n$; note
that the isomorphism $W \oplus iW \to V$ induces an orientation on $V$
that differs from the complex orientation by a sign $(-1)^{\dim(W)(\dim(W) - 1)/2 }.$
\item {\rm (Orientation for the identity)} For finite-dimensional
  $V,W$, orientations on $V$ and $W$ induce an orientation $o_0$ on
  $\det(0)$.  By \eqref{eq:tD}, $o_0$ induces an orientation $o_D$ on
  $\det(D)$.  By convention \eqref{eq:backwards} $o_D$ is compatible with
  the canonical orientation on $\det({\rm Id})\cong\R$ for the
  identity operator $D = {\rm Id}$ if $V=W$.
\item {\rm (Orientation double cover)} For real Banach spaces $V,W$
  let 
\[\Fred^+(V,W) = \Set{ (D,o) \ | \begin{array}{ll} D: V \to W \ \text{Fredholm} \\ o \in
  \Or(D) := \det(D)^\times/\R_{>0}  \end{array} } \]
denote the space of Fredholm operators equipped with orientations of
their determinant bundles $\det(D)$.  Thus 
\[\Fred^+(V,W) \to \Fred(V,W), \quad (D,o) \mapsto D \]
is a double cover.  The pull-back of the determinant line bundle to
$\Fred^+(V,W)$ is automatically orientable.
\end{enumerate} 
\end{remark} 

\label{sec:oriented}
\label{sec:indices}

\begin{example} \label{ex:diffs} {\rm (Orientations induced by difference maps)}  
The following example of orientations for difference maps will be used
later.  We compare the difference maps
\[ D^{12}, D^{21}:  \ \R \oplus \R \to \R, \quad D^{12}(x_1,x_2) = x_1 - x_2 , \quad D^{21} (x_1,x_2) = - D^{12}(x_1,x_2) = x_2 - x_1 . \]
The kernel and cokernel of $D^{12}, D^{21}$ are 
\[ \ker(D^{12}) = \ker(D^{21}) = \{ (x,x) \ | \ x \in \R \}, \quad \coker(D^{12}) = \coker(D^{21}) = \{ 0 \} .\]
Choose standard bases for $\R \oplus \R, \R$:
\[ \{ e_1 = (1,0),e_2 = (0,1) \} \subset \R^2, \quad \{ f = 1 \} \subset \R \] 
The isomorphism \eqref{eq:tD} identifies
\[ \det(D^{12}) \to \det(0), \quad e_1 + e_2 \mapsto \tfrac{1}{2} f^\dual \otimes (e_1 - e_2) \wedge (e_1 +
e_2) \]
and so induces the standard orientation on the diagonal $\ker(D^{12})$;
\[ \det(D^{21}) \to \det(0), \quad e_1 + e_2 \mapsto -\tfrac{1}{2} f^\dual \otimes (e_1 - e_2) \wedge (e_1 + e_2) \]
and so one obtains the opposite orientation on $\ker(D^{21})$ from that
on $\ker(D^{12})$.
\end{example}

\subsection{Cauchy-Riemann operators}
\label{sec:cr} 

Our terminology for Cauchy-Riemann operators follows that of
McDuff-Salamon \cite{ms:jh}; in particular, the Cauchy-Riemann
operators arising in pseudoholomorphic curve theory are {\em real}
Cauchy-Riemann operators in the sense that the zero-th order term is
not complex-linear.  

\subsubsection{Cauchy-Riemann operators on surfaces with
  boundary} 

Let $S$ be a compact surface with boundary.

\begin{definition}
\begin{enumerate} 
\item {\rm (Bundles with boundary condition)} A {\em bundle with
  boundary condition} for $S$ is a complex vector bundle $E$, given as
  a real vector bundle $E \to S$ together with an operator 
  \[ J_E: E \to E, \quad J_E^2 = - \on{id}, \]  
  over $S$ with a maximally totally real subbundle $F\subset E|_{\partial S}$; that is,
\[ F \cap J_E( F) = \{ 0 \}, \ \ \ \rank_\R(F) = \rank_\R(E)/2 =
  \rank_\C(E) .\]
For each component $\partial S_b \subset \partial S$ we denote by
$F_b$ the restriction of $F$ to $\partial S_b$.  Given $E = (E,J_E)$
we denote by $E^- = (E,- J_E)$ over $S$ the bundle obtained by
reversing the complex structure.  Denote by $F^-$ the bundle $F$
considered as a totally real sub-bundle of $E^-$.
\item {\rm (Forms with boundary condition)} Let $\Omega^k(E)$ denote
  the space of $k$-forms with values in $E$ for integers $k \ge 0$.
  For $k = 0$ let
\[\Omega^0(E,F)= \{ \xi \in \Omega^0(E) \ | \ \xi |_{\partial S} \in \Omega^0(\partial S,F) 
\} \] 
denote the space of sections of $E$ with boundary values in $F$.
\item {\rm (Dolbeault forms)} Suppose that $S$ is equipped with a
  complex structure.  Let $\Omega^{k,l}(E)$ denote the forms of type
  $k,l$ for integers $k,l$ with values in $E$.  Thus
\[ \Omega^j(E) = \bigoplus_{k + l = j} \Omega^{k,l}(E) .\]
\end{enumerate} 
\end{definition}

\begin{definition} {\rm (Cauchy-Riemann operators)} 
\begin{enumerate} 
\item An operator $D_E: \Omega^0(E) \to \Omega^{0,1}(E)$ is a
  {\em Cauchy-Riemann operator} if it is complex linear and satisfies
  the Leibniz rule
\[ D_E( f \xi) = f D_E(\xi) + (\ol{\partial}f ) \otimes \xi, \quad 
\forall f\in C^\infty(S,\C), \ \xi\in\Om^0(E) .\]
We denote by $D_{E,F}: \Omega^0(E,F) \to \Omega^{0,1}(E)$ its
restriction to sections with boundary values in $F$.  The restriction
is only real linear, and satisfies the Leibniz rule for functions $f$
that are real-valued on $\partial S$, since otherwise $f \xi$ need not
have boundary values in $F$.
\item A {\em real Cauchy-Riemann operator}, according to the terminology introduced by McDuff and Salamon \cite{ms:jh}, is the sum of a
  Cauchy-Riemann operator with a zeroth order term taking values in
  $\End_\R(E)$.  
\item Consider trivial bundles with rank $n$ for some integer
  $n \ge 0$ given by 
\[
 E = S \times \C^n, \quad F = \partial S \times \R^n .\]
The {\em trivial Cauchy-Riemann operator} is the
  operator $D_{E,F}$ defined by
\[ D_{E,F}( f \otimes \xi) = (\ol{\partial} f) \otimes \xi .\]
\item 
{\rm (Adjoint Cauchy-Riemann operator)} Let $D_{E,F}$ denote a real
Cauchy-Riemann operator acting on sections of $E$ with boundary values
in $F$.  The cokernel of $D_{E,F}$ can be identified with the kernel
of the adjoint $D_{E,F}^*$.  The operator $D_{E,F}^*$ is a real
Cauchy-Riemann operator acting on sections of 
\[(E \otimes (T S))^* =
\Hom( E^- \otimes TS^-,\C). \] 
The sections are required to have boundary values in the annihilator subbundle
$(F \otimes T(\partial S))^{\on{ann}}$, the real sub-bundle of
$E^* \otimes (TS)^*$ whose evaluations on $F \otimes T(\partial S)$
vanish.
\end{enumerate}  
\end{definition} 

\begin{remark}  The set of all Cauchy-Riemann operators is an affine space modelled on
$\Omega^{0,1}(S,\End(E))$ 
 in the sense that if $D_{E,F}^0$ and $D_{E,F}^1$ are two such operators 
then 
\[ (D_{E,F}^0 - D_{E,F}^1) \sigma = \alpha \wedge \sigma, \forall
\sigma \in \Omega^0(E,F), \quad \text{ for some} \ \alpha \in
\Omega^{0,1}(S, \End(E)) .\]
The set of all real Cauchy-Riemann operators forms an affine space
modelled on $\Omega^{0,1}(S) \otimes_\R \End_\R(E) $.  In particular
both spaces are contractible.
\end{remark}  

The Riemann-Roch theorem generalizes to Cauchy-Riemann operators on
compact surfaces with boundary as follows; see for example
\cite[Appendix]{ms:jh}.

\begin{definition} 
{\rm (Euler characteristic and Maslov index)} For any compact surface
with boundary $S$ denote by
\[H^j(S) = \on{ker}(\d^j)/\on{im}(\d^{j-1}), \quad \d^j : \Omega^j(S) \to \Omega^{j+1}(S) \]
the $j$-th de Rham cohomology of $S$ for integers $j \ge 0$.  Let
\[\chi(S) = \dim H^0(S) - \dim
H^1(S) + \dim H^2(S) \] 
denote the Euler characteristic of $S$.  Let $I(E,F) \in \Z$ be the
Maslov index of the pair $(E,F)$, as in \cite[Appendix]{ms:jh}.  For
$S$ without boundary, the index $I(E,F)$ is twice the Chern number,
\[ I(E,\emptyset) = 2 \int_S c_1(E) .\]
On the other hand, for $E$ trivial, $I(E,F)$ is the sum of the winding
numbers of the boundary conditions around the boundary components,
considered as paths in the Grassmannian of totally real subspaces.
\end{definition} 

\begin{proposition}  \label{prop:rrthm} 
{\rm ( Riemann-Roch for surfaces with boundary
  \cite[Appendix]{ms:jh})} For any Cauchy-Riemann operator $D_{E,F}$ on
a surface with boundary $S$,
\begin{equation} \label{eq:rr}
 \Ind(D_{{E},{F}}) = \rank_\R(F) \chi({S}) + I(E,F) . \end{equation}
\end{proposition}  

\subsubsection{Cauchy-Riemann operators on surfaces with
strip-like ends} 

\begin{definition} \label{def:surfstrip} \ 
\begin{enumerate} 
\item A {\em surface with strip-like ends} consists of the following data:
\begin{enumerate}
\item 
a compact surface $\ol{S}$ with boundary 
\[\pd\ol{S}= C_1 \sqcup \ldots
\sqcup C_m \] 
and $d_n \ge 0$ distinct points 
\[z_{n,1},\ldots,z_{n,d_n}\in C_n\] 
in cyclic order on each boundary circle $C_n\cong S^1$.  (We will use
the indices on $C_n$ modulo $d_n$.  The index set for the marked
points is denoted
\[ \E=\E(S) :=\bigl\{ e=(n,l) \,\big|\, n\in\{1,\ldots,m\},
l\in\{1,\ldots,d_n\} \bigr\} ;
\]
We use the notation $e\pm 1:=(n,l\pm 1)$ for the cyclically adjacent
index to $e=(n,l)$.  Denote by
\[I_e :=I_{n,l}\subset C_n\] 
the component of $\partial S$ between $z_e=z_{n,l}$ and
$z_{e+1}=z_{n,l+1}$.  However, $\partial S$ may also have compact
components $I=C_n\cong S^1$.)
\item 
a complex structure $j_S$ on $S:=\ol{S}\setminus\{z_e \,|\, e\in \E
\}$;
\item
a set of {\em strip-like ends} for $S$, that is a set of embeddings
with disjoint images
\[ \eps_e : \R^\pm \times [0,\delta_e] \to S \]
for all $e\in\cE$ such that 
\begin{eqnarray*}
\eps_{e}(\R^\pm\times\{0,\delta_e\}) &\subset& \partial S \\ 
\lim_{s \to \pm \infty}(\eps_{e}(s,t)) &=& z_e \\
\eps_{e}^*j_S &=& j_0 \end{eqnarray*}
where $j_0$ is the canonical complex structure on the half-strip
$\R^\pm \times[0,\delta_e]$ of width 
$\delta_e>0$; (By a
  conformal change of coordinates, we can always assume the width to
  be $\delta_e=1$. The freedom of widths will only become relevant in
  the definition of quilted surfaces with strip-like ends.)
  Denote the set of incoming resp. outgoing ends
\[ \E_\pm := \E_\pm(S) := \{ \eps_{e}: \R^\pm \times [0,\delta_e] \to
S \} .\]
\item an ordering of the set of (compact) boundary components of
  $\ol{S}$ and orderings 
\[ \E_-=(e^-_1,\ldots, e^-_{N_-}), \quad 
  \E_+=(e^+_1,\ldots, e^+_{N_+})\] 
  of the sets of incoming and outgoing ends; (Here
  $e^\pm_i=(n^\pm_i,l^\pm_i)$ denotes the incoming or outgoing end at
  $z_{e^\pm_i}$).
\end{enumerate} 
\item Let $S$ be a surface with strip-like ends, and $E,F$ a pair of
  vector bundles as in \cite[Section 2]{ww:quilts}.  The bundle
  $E$ admits a trivialization with fiber $E_e$ over each strip-like 
  end $e$, and $F \subset E | \partial S$ is a totally real sub-bundle
  constant on the strip-like ends with fibers $F_e =
  (F_{e,0},F_{e,1})$.  A real Cauchy-Riemann operator $D_{E,F}$ for
  $(E,F)$ is {\em asymptotically constant} if the following condition
  is satisfied: On each strip-like end $e\in\cE(S)$ there exists a
  time-dependent operator
\[\cH_e: [0,1] \to \End_\R(E_e)\]
such that the operator $D_{E,F}$ on sections $(\eps_e)_*\xi,
\xi:\R^\pm\times[0,1]\to E_e$ has asymptotic limit given by the
following operator:
 \begin{equation} \label{eq:asymlimit}
\tfrac 12 \bigl( \d\xi + i_{E_e} \circ \d\xi \circ j \bigr) + \bigl(
(\cH_e\xi)\d s - (i_{E_e}\circ\cH_e\xi) \d t \bigr)
\end{equation}
where $i_{E_e}$ and $j$ denote the complex structures on $E_e$ and
$\R^\pm\times[0,1]$ respectively, and $\d$ is the trivial connection
on the trivial bundle $E_e$ over $\R^\pm\times[0,1]$.  More precisely,
the difference between $\eps_e^* \bigl( D_{E,F} (\eps_e)_*\xi \bigr)$
and \eqref{eq:asymlimit} is a zero-th order operator that approaches $0$
uniformly in all derivatives in $t$ as $|s| \to \infty$.
\item An asymptotically constant Cauchy-Riemann operator $D_{E,F}$ is
  {\em non-degenerate} if the operator
\[\partial_t + \cH_e: \Omega^0([0,1];E_e,F_e) \to
  \Omega^0([0,1];E_e) \]
has trivial kernel.  Any non-degenerate, asymptotically constant
operator $D_{E,F}$ is Fredholm on suitable Sobolev completions; see
for example Lockhart-McOwen \cite{loc:ell} for the case of surfaces
with cylindrical ends.
\end{enumerate} 
\end{definition} 

\begin{remark}  {\rm (Non-degeneracy of the trivial operator)}  Suppose that $E,F$ are trivial and for each end $e \in \mE(S)$ the subspaces
$F_{b_0},F_{b_1}$ for the components $b_0,b_1$ of $\partial S$ meeting
  $e$ are transversal, that is, $F_{b_0} \cap F_{b_1} = \{ 0 \}$.
  Then the trivial Cauchy-Riemann operator $D_{E,F}$ is
  non-degenerate:
\[ \ker(\partial_t) = F_{b_0} \cap F_{b_1} = \{ 0 \} .\]  
Furthermore, if in addition the surface is a strip then the kernel and
cokernel are trivial:
\[ S \cong \R \times [0,1]  \quad \implies \quad 
 \ker(D_{E,F}) = \{ 0 \},  \ \ \coker(D_{E,F}) = \{ 0 \} .\]
\end{remark}  

\subsubsection{Cauchy-Riemann operators on nodal surfaces}

\label{sec:nodalsec}

\begin{definition} {\rm (Cauchy-Riemann operators on nodal surfaces)} 
\label{def:crops}
\begin{enumerate}   
\item A {\em nodal surface $\nS$ (with boundary and strip-like ends)}
consists of
\begin{enumerate}
\item a surface with strip-like ends ${\rS}$ (here the superscript
  $\rho$ is used to indicate the surface ``with nodes resolved'') with
  boundary $\partial {\rS}$; 
\item a collection of {\em interior nodes}: pairs
\[ 
Z = \{ \{ z_1^-,z_1^+ \},\ldots, \{z_r^-,z_r^+ \} \} \]
of distinct interior points of ${\rS}$;
\item a collection of {\em boundary nodes}: ordered pairs 
\[ W = \{ ( w_1^-,w_1^+ ), \ldots, (w_s^-,w_s^+) \} \]
of distinct boundary points of ${\rS}$; and
\item an {\em ordering} of the set $\{ I_i \} \cup \{ w_j \} \cup \{
  e_k \}$ of boundary components $I_i$, boundary nodes $w_j$, and
  strip-like ends $e_k$.
\end{enumerate}
Note that $S^\rho$ is the normalization (resolution of singularities)
of $S$. 
\item A {\em complex vector bundle} $\nE \to\nS$ on a nodal surface
  with boundary consists of
\begin{enumerate}
\item a complex vector bundle ${\rE} \to {\rS}$;
\item isomorphisms $\rE_{z_i^+} \to \rE_{z_i^-}$ and $\rE_{w_i^+} \to
  \rE_{w_i^-}$ for each interior node $z_i^\pm$ and boundary node
  $w_i^\pm$; and 
\item a trivialization $\rE|_{\im\eps_{e}}\cong 
\nE_e\times(\R^\pm\times[0,1])$ for each strip-like end $e\in\cE(\rS)$.
\end{enumerate}
\item  
A {\em totally real boundary condition} ${F}$ for $E \to S$ is a
totally real subbundle $\rF \subset {\rE} |_{\partial S}$ such that
the following conditions hold:
\begin{enumerate}
\item  the identifications of the fibers at the boundary nodes 
induce isomorphisms $\rF_{w_i^+}\to \rF_{w_i^-}$;
\item $\rF$ is maximally totally real, that is 
$\rank_\R(\rF) = \rank_\C({\rE})$; and
\item in the trivialization over each strip-like end $e\in\cE(\rS)$,
  the subspaces 
\[\rF_{\eps_e(s,0)}=\nF_{e,0} \subset E_e, \quad
  \rF_{\eps_e(s,1)}=\nF_{e,1} \subset E_e\]
are constant along $s\in\R^\pm$.  (These subspaces form a transverse
pair 
\[ \nF_{e,0}\oplus\nF_{e,1}=E_e .) \] 
\end{enumerate}
\item Let $\nE \to \nS$ be a complex vector bundle on a nodal surface
  $\nS$ with totally real boundary condition $F$.  A {\em real
  Cauchy-Riemann operator} $D_{\nE,\nF}$ for $(\nS,\nE,\nF)$ is
  an operator
\[ D_{\nE,\nF}: \Omega^0(\nE,\nF) \to
\Omega^{0,1}(\nE,\nF), \ \ \ \sigma \mapsto D_{{\rE},\rF} \sigma \]
defined in terms of a real Cauchy-Riemann operator $D_{{\rE},{\rF}}$
on $\rS$ with values in $\rE$ and boundary conditions in $\rF$.  Here
we set $\Omega^{0,1}(\nE,\nF) := \Omega^{0,1}({\rE},\rF)$ and define
$\Omega^0(\nE,\nF) \subset \Omega^0({\rE},{\rF})$ as the kernel of the
surjective map
\begin{align} \label{eq:deltamap}
 \delta:\quad
 \Omega^0({\rE},\rF) \;&\longrightarrow\qquad\qquad\quad\;\;
\bigoplus_i \rE_{z_i^+} \oplus
\bigoplus_j \rF_{w_j^+} \\
\sigma \quad&\longmapsto\; \bigoplus_i ( \sigma(z_i^+) - \sigma(z_i^-) ) \oplus
\bigoplus_j ( \sigma(w_j^+) - \sigma(w_j^-) ) .
\end{align}
\item The family versions of the above definitions are as follows.  A
  {\em family of nodal surfaces} $S \to B$ is a smooth family $S^\rho
  \to B$ of complex surfaces (compact, possibly with boundary) over a
  smooth, open base $B$, together with nodes $Z,W\subset (S^\rho)^2$
  varying smoothly over $B$.  A {\em family of complex vector bundles}
  $E \to S$ is a complex vector bundle $ E^\rho \to S^\rho$, together
  with smoothly varying identifications of the fibers at the nodes and
  constant trivializations on the strip-like ends.  A {\em family of
    totally real boundary conditions} $F \to \partial S$ consists of a
  totally real boundary condition $F^\rho \to \partial S^\rho$ that is
  constant in the trivializations on the strip-like ends.  A family of
  real Cauchy-Riemann operators $D_{E,F}$ for the families $(S,E,F)\to
  B$ is a family of real Cauchy-Riemann operators $D_{b}$ for
  $(S_b,E_b,F_b)$, varying smoothly with $b\in B$.
\end{enumerate}
\end{definition} 

\begin{definition}  There are several variations on the linearized
Cauchy-Riemann operator.  Define an {\em unreduced linearized operator}
\[ D_{\nE,\nF}^{\unred}: \Omega^0({\rE},\rF) \to \bigoplus_i
\rE_{z_i^+} \oplus \bigoplus_j \rF_{w_j^+} \oplus \Omega^{0,1}({\rE},\rF) , \ \ \
\sigma\mapsto(\delta(\sigma),D_{{\rE},\rF}\sigma)\]
where $\delta$ is the operator of \eqref{eq:deltamap}.  The kernel and
cokernel are canonically isomorphic to those of $D_{\nE,\nF}$.  The
isomorphisms define an isomorphism of determinant lines
\begin{equation} \label{eq:unred}
 \det(D_{\nE,\nF}) \to \det(D_{\nE,\nF}^\unred) .\end{equation}
Define the {\em reduced linearized Cauchy-Riemann operator}
\begin{equation} \label{eq:reduced}
 D^\red_{\nE,\nF} : \ker(D_{{\rE},\rF}) \to \bigoplus_i \rE_{z_i^+}
\oplus \bigoplus_j \rF_{w_j^+} \oplus \coker(D_{{\rE},\rF}), \ \ \sigma
\mapsto (\delta(\sigma),0) .\end{equation}
\end{definition}

\begin{lemma} {\rm (Unreduced and reduced Cauchy-Riemann operators)}  
The determinant line $\det(D_{\nE,\nF})$ for a Cauchy-Riemann operator
$D_{E,F}$ over a nodal surface $\nS$ is related to the determinant
$\det(D_{\rE,\rF})$ for the corresponding operator over the smooth
surface $\rS$ with resolved nodes by an isomorphism
\begin{equation} \label{eq:nodal}
  \det(D^\red_{\nE,\nF}) 
\to \Lambda^{\max} \Bigl( \bigoplus_i
\rE_{z_i^+} \oplus \bigoplus_j \rF_{w_j^+}\Bigr)^\dual \otimes
\det(D_{{\rE},\rF}) .\end{equation} 
Hence orientations on $D_{{\rE},\rF}$ and the fibers $\rE_{z_i^+},
\rF_{w_j^+}$ induce an orientation on $D_{\nE,\nF}$.
\end{lemma}

\begin{proof}
The kernel and cokernel of $D^\red_{\nE,\nF}$ are canonically
isomorphic to those of $D_{\nE,\nF}^\unred$.  The isomorphisms define
an isomorphism of determinant lines
\begin{equation} \label{eq:unredred}
 \det(D_{\nE,\nF}^\unred) \to \det(D_{\nE,\nF}^\red) .\end{equation}
Since the domain and codomain of $D^{\red}_{\nE,\nF}$
are finite dimensional, \eqref{eq:tD} gives an isomorphism
\[ \det(D^{\red}_{\nE,\nF}) \to \Lambda^{\max}((N \oplus 
C)^\dual) \otimes \Lambda^{\max}(\ker(D_{\rE,\rF})) \] 
with 
\[ N = \bigoplus_i \rE_{z_i^+} \oplus \bigoplus_j \rF_{w_j^+}, \quad C := \coker(D_{\rE,\rF}) . \]
We identify $\Lambda^{\max}((N \oplus C)^\dual)$
with $\Lambda^{\max}(N^\dual) \otimes \Lambda^{\max}(C^\dual)$
by the isomorphism \eqref{eq:caniso} for the zero operators
$0 \to N$ and $0 \to C$.  With the reverse-order convention 
\eqref{eq:backwards} for duals, this differs from the identification with $\Lambda^{\max}(C^\dual) \otimes \Lambda^{\max}(N^\dual)$ by the sign 
$(-1)^{\dim(N) \dim(C)}.$  Since $\Lambda^{\max}(C^\dual) \otimes \Lambda^{\max}(\ker(D_{\rE,\rF})) = 
\det(D_{\rE,\rF})$, this gives the isomorphism of the statement of the Lemma.
\end{proof}

\begin{remark}
A similar
isomorphism holds when a surface $S$ and bundles $E,F$ are obtained
from another nodal surface $\nnS$ and bundles $\nnE,\nnF$ by resolving
some subset of the nodes of $\nnS$.  That is, $\nnS$ is obtained by
removing some subset of the sets of interior and boundary nodes $Z,W$
so that some nodal points of $\nnS$ are replaced by pairs of points in
$S$, and $\nnE,\nnF$ are the bundles obtained by pullback under $\nnS
\to S$.

The construction of orientations depends on orderings of the components
and boundary nodes as in Definition \ref{def:crops}.  In the case that the ordering of the boundary nodes and components is such that the boundary nodes are ordered first ($w_i$ appears before
$I_j$, for each $i,j$) we take the orientation from the previous
paragraph to be the orientation of $D_{\nE,\nF}$.  In general, an
orientation of $D_{\nE,\nF}$ is defined by the orientation from the
previous paragraph times the sign arising from permuting the
determinant lines $\Lambda^{\max}(F_{w_i})$ of the boundary nodes so
that they appear before the determinant lines for the boundary
components.
\end{remark} 

\begin{example}  \label{ex:diffs2} {\rm (Orientation for the trivial bundle over a nodal disk)}  
  Continuing Example \ref{ex:diffs}, suppose that $S$ is a nodal surface
  consisting of two disks joined with a single boundary node
  $(w_-,w_+)$. Also suppose that the ordering of the disks induces
  the ordering of $(w_-,w_+)$.  Equip $S$ with the trivial bundles
  $E,F$.  Then the reduced operator is
\[
 D^\red_{\nE,\nF} :(x_1,x_2) \mapsto x_2 - x_1. \] 
%
Thus the reduced operator has kernel equal to the diagonal 
\[ \on{ker}( D^\red_{\nE,\nF}) = \{ (x,x)  \ | \ x \in F \}  \subset \nF
\oplus \nF .\]  
By Example \ref{ex:diffs}, the determinant line $ \det(D^\red_{\nE,\nF}) $
inherits the orientation of $\det(\nF)$ times $(-1)^{\rank(\nF)}$.  If
the ordering of the boundary components and boundary node is (first
component, boundary node, second component), then the orientation
induced on $D_{\nE,\nF}$ is the standard one.
\end{example} 

\begin{remark}  The class of Cauchy-Riemann operators is 
closed under the following operations:
\begin{enumerate} 
\item {\rm (Conjugates)} Let $(E,F)$ be a bundle with boundary
  condition over $S$. Let $E^-$ be the complex conjugate of $E$, and
  $F^-$ the subspace $F$ considered as a totally real subspace of $E^-$.
  Let $S^-$ denote the surface $S$ with complex structure $\ol{J_S} = -
  J_S$.  Given a Cauchy-Riemann operator $D_{E,F}$ the first order part
  of $D_{E,F}$ is complex linear with respect to the dual complex
  structures $-J, -j$ and defines a Cauchy-Riemann operator
  $D_{E^-,F^-}$ on the dual $(E^-,F^-)$.
\item {\rm (Direct Sums)} Let $(E_k,F_k), k = 0,1$ be bundles with
  real boundary conditions over a surface $S$, and 
\[(E,F) = (E_0,F_0) \oplus (E_1,F_1) .\]  
Suppose that $D_{E_k,F_k}, k = 0,1$ are Cauchy-Riemann operators for the
components.  The direct sum
\[D_{E,F} = D_{E_0,F_0} \oplus D_{E_1,F_1} \] 
is a Cauchy-Riemann operator for the direct sum.
\item {\rm (Disjoint Unions)} Let $(E_k,F_k)$ denote bundles with
  totally real boundary condition over surfaces $S_k$ for $k = 0,1$.
  Then 
\[(E,F) = (E_0,F_0) \sqcup (E_1,F_1)\] 
is a bundle with totally real boundary condition over $S = S_0 \sqcup
S_1$. Then the space of forms $\Omega^0(E,F)$ is naturally isomorphic
to the direct sum of the $\Omega^0(E_k,F_k)$.  If $D_{E_k,F_k}, k =
0,1$ are Cauchy-Riemann operators for the components then the direct
sum $D_{E,F} = D_{E_0,F_0} \oplus D_{E_1,F_1} $ is a Cauchy-Riemann
operator for the disjoint union.
\end{enumerate} 
\end{remark} 
Now we turn to quilted surfaces. 

\begin{definition} 
{\rm (Quilted surfaces)}
A {\em quilted surface} $\ul{S}$ with strip-like ends consists of the
following data:
\ben 
\item A collection of {\em patches} $(S_p)_{p \in \PP}$ indexed by a
  set $\cP$, so that each patch $S_p$ is a surface with strip-like
  ends.  Each $S_p$ carries a complex structure $j_p$ and has
  strip-like ends $(\eps_{p,e})_{e\in \E(S_p)}$ of widths
  $\delta_{p,e}>0$.  Each end has limit equal to a marked point
\[\lim_{s \to   \pm \infty} \eps_{p,e}(s,t) = :z_{p,e} \in \partial\overline{S}_p .\]
Denote by $I_{p,e}\subset\partial S_p$ the noncompact boundary
component between $z_{p,e-1}$ and $z_{p,e}$.
\item 
A collection of {\em seams} $\mathcal{S}$.  Each seam $\sigma \in \mathcal{S}$ is a
pairwise disjoint $2$-element subset of the set of patches and
boundary components:
\[ \sigma \subset \bigcup_{p \in \cP} \{ p \} \times \pi_0(\partial S_p) . \]
We write
\[ \sigma = \{ (p_-(\sigma),I_{\sigma,-}),(p_+(\sigma),I_{\sigma,+})
\} \]
recording the patches and components of the boundary that are
identified.  For each $\sigma \in \mathcal{S}$, a diffeomorphism of boundary
components 
\[\varphi_\sigma: \partial S_{p_-(\sigma)} \ni I_{\sigma,-} \overset{\sim}{\to}
I_{\sigma,+} \subset \partial S_{p_+(\sigma)}  \]
is given and supposed to satisfy the conditions:
\begin{enumerate}
\item {\it real analytic:} Every point $z\in I_{\sigma,-}$ has an open
  neighborhood $\U\subset S_{p_-(\sigma)}$ on one side of the seam
  such that $\varphi_\sigma|_{\U\cap I_{\sigma,-}}$ extends to an
  antiholomorphic embedding on the other side:
\[\psi_z:\U\to S_{p_+(\sigma)}, \quad \psi_z^* j_{p_+(\sigma)} = -
  j_{p_-(\sigma)} .\]
In particular, this condition forces $\varphi_\sigma$ to reverse the
orientation on the boundary components.  This condition is presumably technical, and could be removed at the cost of working with linearized operators whose principal symbol is not the same as that of a standard Cauchy-Riemann operator.
\item {\it compatible with strip-like ends:} Let $I_{\sigma,-}$ (and
  hence $I_{\sigma,+}$) be noncompact, i.e.\ lie between marked points,
  $I_{\sigma,-}=I_{p_\sigma,e_\sigma}$ and
  $I_{\sigma,+}=I_{p_\sigma',e_\sigma'}$, where $p_\sigma = p_-(\sigma)$ and $p_\sigma' = p_+(\sigma)$.  We require that
  $\varphi_\sigma$ matches up the end $e_\sigma$ with $e_\sigma'-1$
  and the end $e_\sigma-1$ with $e_\sigma'$.  That is,
  $\eps_{p_\sigma',e_\sigma'}^{-1}\circ\varphi_\sigma\circ\eps_{p_\sigma,e_\sigma-1}$
  maps $(s,\delta_{p_\sigma,e_\sigma-1}) \mapsto (s,0)$ if both ends
  are incoming, or it maps $(s,0)\mapsto
  (s,\delta_{p_\sigma',e_\sigma'})$ if both ends are outgoing.  We
  disallow matching of an incoming with an outgoing end, and the
  condition on the other pair of ends is analogous.
\end{enumerate}
\end{enumerate}
Given a quilted surface with strip-like ends $\ul{S}$ as above:
\begin{enumerate} 
\item The {\em true boundary components} $I_b\subset\partial S_{p_b},
  b \in \B$ are those that are not identified with another boundary
  component of $\ul{S}$.  Let $\B$ denote the set of true boundary
  components, and for each $b \in \B$ let $p_b$ denote the patch and
  $I_b$ the component.
\item The {\em quilted ends} 
\[\ul{e}\in \E(\ul{S})=\E_-(\ul{S})\sqcup
  \E_+(\ul{S})\] 
consist of a maximal sequence
\[\ul{e}=(p_i,e_i)_{i=1,\ldots,n_{\ul{e}}}\] 
of ends of patches.  The boundaries of each end of each patch are
identified
\[\eps_{p_i,e_i}(\cdot,\delta_{p_i,e_i}) \cong
\eps_{p_{i+1},e_{i+1}}(\cdot,0)\]
via some seam $\phi_{\sigma_i}$.  The end sequence could be cyclic,
i.e.\ with an additional identification
$\eps_{p_n,e_n}(\cdot,\delta_{p_n,e_n}) \cong
\eps_{p_{1},e_{1}}(\cdot,0)$
via some seam $\phi_{\sigma_n}$.  Otherwise, the end sequence is
noncyclic, i.e.\ $\eps_{p_1,e_1}(\cdot,0)\in I_{b_0}$ and
$\eps_{p_n,e_n}(\cdot,\delta_{p_n,e_n})\in I_{b_n}$ take values in
some true boundary components $b_0,b_n\in\B$.
\item The ends $\eps_{p_i,e_i}$ of patches in a quilted end $\ul{e}$
  are either all incoming, $e_i\in \E_-(S_{p_i})$, in which case we
  call the quilted end {\em incoming}, $\ul{e}\in\cE_-(\ul{S})$, or
  they are all outgoing, $e_i\in \E_+(S_{p_i})$, in which case we call
  the quilted end {\em outgoing}, $\ul{e}\in\cE_+(\ul{S})$.
\end{enumerate}
We assume, as part of the definition, that orderings of the patches
and of the boundary components of each $\ol{S}_p$, and orderings
$\E_\pm(\ul{S})=(\ul{e}^\pm_1,\ldots,\ul{e}^\pm_{N_\pm(\ul{S})})$ of
the quilted ends are given.
\end{definition} 

\begin{definition}  {\rm (Cauchy-Riemann operators for quilted surfaces with strip-like ends)} 
\ben
\item {\rm (Boundary and seam conditions)} A collection of {\em
  bundles with totally real boundary and seam conditions} is a pair
  $(\ul{E},\ul{F}) \to \ul{S}$ consisting of a family of complex
  vector bundles over the components $\ul{E}$ together with totally
  real subbundles $\ul{F}$ over the boundary components and seams.
  That is, for each seam $\sigma$, the resulting component
  $F_\sigma$ is a totally real subspace of the restriction of
  components of $\ul{E}$:
\[ 
F_\sigma \subset 
 E^-_{p_+(\sigma)} | (\partial
  S_{p_+(\sigma)})_\sigma \times E_{p_-(\sigma)} | (\partial
  S_{p_-(\sigma)})_\sigma \] 
of the bundles $E_{p_\pm(\sigma)}$ on the patches $S_{p_\pm(\sigma)}$
adjacent to $\sigma$.
\item {\rm (Quilted Cauchy-Riemann operators)} A {\em quilted
  Cauchy-Riemann operator} for $(\ul{E},\ul{F})$ is a collection
of Cauchy-Riemann operators 
\[D_{\ul{E},\ul{F}} =  ( D_p, p \in \PP ) \]
on the patches $S_p, p \in \PP$, acting on the space of sections with
the given boundary and seam conditions.  \een
\end{definition} 

\subsection{Gluing of Cauchy-Riemann operators} 
\label{sec:crglue}

Nodal surfaces with strip-like ends can be glued along the ends, or at
interior or boundary nodes.  We explain the behavior of determinant lines under gluing of strip-like ends.

\begin{definition} {\rm (Gluable ends)}  Let $\nS$ be a surface with strip-like ends.
Let $\nE \to \nS$ be a complex vector bundle and $\nF \to \partial
\nS$ a totally real boundary condition.  Let $D_{E,F}$ be a real
Cauchy-Riemann operator.  Let $e_+\in\cE_+(\nS)$ and
$e_-\in\cE_-(\nS)$ be an outgoing resp.\ incoming end.  Suppose a
complex isomorphism is given that maps the totally real boundary
conditions on the ends:
\begin{equation} \label{eq:fiberid} E_{e_+} \to E_{e_-}, \quad F_{e_+,k} \mapsto F_{e_-,k} , \quad k \in \{ 0
, 1 \} \end{equation}
We say that the ends $e_\pm$ are {\em gluable} if the asymptotic
limits \eqref{eq:asymlimit} of $D_{E,F}$ on the ends $e_\pm$ are equal,
after the identification of fibers \eqref{eq:fiberid}.
\end{definition}  

\begin{definition} \label{def:gsurf} {\rm (Glued surface and Cauchy-Riemann operator)}  
Let $\nS$ be a surface with gluable ends $e_\pm$ equipped with bundles
$\nE,\nF$ and a gluable Cauchy-Riemann operator $D_{\nE,\nF}$.
\ben
\item Let $\dS=\#^{e_-}_{e_+}(S)$ be the {\em glued surface} formed by
  gluing the ends of $\nS$.  That is, consider a pair of ends
\[\eps_{e_+}(\R^+\times[0,1])\cup
  \eps_{e_-}(\R^-\times[0,1])\subset\nS.\] 
We construct a surface  $\dS$ with two fewer ends by 
replacing the ends $\eps_{e_\pm}$ by a strip $[-\tau,\tau]\times[0,1]$ depending on a gluing parameter $\tau>0$, where $\{\pm\tau\}\times[0,1]$ is
  identified with $\eps_{e_\mp}(\{0\}\times[0,1])$; the result is a surface with strip-like ends
 after a choice of a new ordering on the
  boundary components and strip-like ends.  
\item 
  Let $\dE, \dF$ be the complex vector bundle and totally real
  boundary condition over $\dS$ that arise from gluing $\nE,\nF$ via the isomorphism $E_{e_+} \cong E_{e_-}$
  on the middle strip.  Let 
  \[ \nS_\tau = \nS \backslash \bigl( \eps_{e_+}(( 2\tau,\infty) \times [0,1]) \cup \eps_{e_-}((-\infty,-2\tau) \times [0,1]) \bigr)  \] 
and let $\pi: \nS_\tau \to \dS$ be the gluing map, 
equal to the identity away from the ends $e_\pm$
and given on the ends by 
\[ \eps_{e_+}(s,t) \mapsto (s - \tau,t), s \in [0,2 \tau], 
\qquad
\eps_{e_-}(s,t) \mapsto (s+ \tau,t), \ s \in [-2 \tau,0], \] 
so that $\pi$ is two-to-one over the strip $[-\tau,\tau]
\times [0,1]$ and one-to-one elsewhere.  Let $\rho_\pm$
be cutoff functions on $\eps_{e_\pm}(\pm [0,2 \tau] \times [0,1])$, vanishing near 
$\eps_{e_\pm}(\{ \pm 2 \tau \} \times [0,1])$ and equal to $1$ near $\eps_{e_\pm}( \{ 0 \} \times [0,1])$, with 
$\rho_+(z_+) + \rho_-(z_-) = 1$ whenever $\pi(z_+) = \pi(z_-)$.  Given a section $\ti{\sigma}$ of $\dE$
define a section $\sigma$ of $\nE$ by $\sigma = \pi^* \ti{\sigma}$ away from the ends $e_\pm$, $\sigma = \rho_\pm\pi^* \ti{\sigma}$ on $\eps_{e_\pm}(\pm [0, 2 \tau] 
\times [0,1])$ and $\sigma = 0 $ on the rest of the ends
$e_\pm$.  Given $D_{E,F}$ define a {\em glued real Cauchy-Riemann operator} $D_{\dE,\dF}$ for $(\dS,\dE,\dF)$ by 
adding together the forms on the two ends:
\begin{equation} \label{eq:adding}
D_{\dE,\dF} \ti{\sigma} = \pi_* D_{\nE,\nF} \sigma , \end{equation}
where $\pi_*$ is summation over the fibers, 
\[ (\pi_* \eta)(z) = \eta(z_-) + \eta(z_+) \ \text{if} \ 
\pi^{-1}(z)=  \{ z_-, z_+ \} , \quad
(\pi_* \eta)(z) = \eta(\pi^{-1}(z)) \
\text{otherwise}. \]
\een
\end{definition} 

\begin{proposition}  {\rm (Identification of indices and determinant lines under
gluing strip-like ends)} \label{prop:glueprop} Suppose that $D_{\dE,\dF}$
  is obtained from $D_{\nE,\nF}$ by gluing strip-like ends.  Then
  there is an equality of indices $ \Ind(D_{\nE,\nF}) = \Ind (
  D_{\dE,\dF}) $ and a canonical isomorphism of determinant lines
\begin{equation} \label{eq:linearglue}
 \det(D_{\nE,\nF}) \to \det( D_{\dE,\dF}) .
\end{equation}
\end{proposition} 

\begin{proof}    For simplicity, we assume that the surface is unquilted. 
 For sufficiently large $\tau$ there exist isomorphisms 
\[ \ker( D_{\nE,\nF}) \overset{\sim}{\to} \ker( D_{\dE,\dF}), \quad
 \coker( D_{\nE,\nF}) \overset{\sim}{\to} \coker( D_{\dE,\dF}) \]
defined as follows. Given a section $\nxi$ in the kernel of
$D_{\nE,\nF}$, one may use cutoff functions on $[-\tau,\tau]$ to glue
it together to a section $\dxi =\sharp_\tau \nxi$ of $\dE \to \dS$
with boundary conditions in $\dF$.  Explicitly
\[ \dxi = \nxi \ \text{on} \ \dS \backslash \eps_\pm( \pm ( 0,\infty)
\times [0,1]), \quad \dxi = \rho_\pm \nxi \ \text{on} \ {\eps_\pm(
  \pm ( 0,\infty) \times [0,1])} .\]
Then $\dxi$ is an approximate zero of $D_{\dE,\dF}$.  For $\tau$ sufficiently large, the isomorphism
is defined by gluing followed by orthogonal projection onto the kernel of $D_{\dE,\dF}$;  see \cite[Section
  5.3]{mau:gluing} for details of the analysis.  The construction for
the cokernels follows by identifying the cokernels of $D_{\nE,\nF}$
and $D_{\dE,\dF}$ with the kernels of their adjoints.  Gluing of
Cauchy-Riemann operators on quilted surfaces along quilted ends is
similar.
\end{proof} 

\subsection{The exact sequence for a deformation of a node}

We describe the behavior of determinant lines under deformation
of nodes.  The result of this section is analogous to
the long exact sequence in homology of sheaves:

\begin{definition}\label{def:gluing}  {\rm (Deformation of a node)}  Consider an interior or boundary node of $S$ represented by a pair $z^\pm \in
S^\rho$, and a gluing parameter  $\gluing \in \R_{> 0} + [0,1]i$, in the interior case or $\gluing \in \R_{> 0}$, in the boundary case. 
\ben
\item {\rm (Deformed surface)} Let $\dS$ be the (possibly still nodal)
  {\em deformed surface} with strip-like ends obtained by deforming
  the node.  Thus $\dS$ is the surface obtained by gluing punctured disks
  around $z^\pm$, identifying the point with coordinate $z$ in a holomorphic coordinate centered at $z^+$ with the point with coordinate $\exp(-2\pi \gluing)/z$ in a holomorphic coordinate centered at $z^-$.
  Denote by
\[s + it = \ln(z)/\pi + \gluing\] 
the coordinates on the cylindrical neck $[-|\gluing|,|\gluing|] \times
S^1$.  In the case of a boundary node, we require that the gluing
parameter $\gluing$ is real and glue together half-disks by $z \mapsto
\exp(-2\pi \gluing)/z$ and identify the neck with $[-\gluing,\gluing]
\times [0,1]$ with coordinates $s + it$.  See Figure \ref{fig:deformed},
in which the glued disks/neck regions are shaded.
\begin{figure}[ht]
\includegraphics[width=5in]{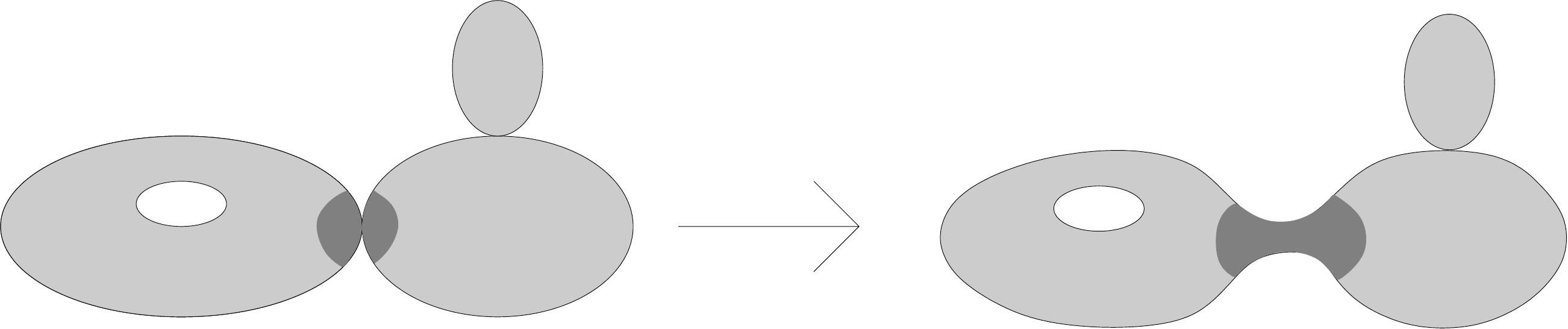} 
\caption{Deformation of a boundary node} 
\label{fig:deformed}
\end{figure} 

\item {\rm (Deformed vector bundles and Cauchy-Riemann operators)} Let
  $\dE,\dF$ denote the vector bundles over $\dS,\partial \dS$ obtained
  by gluing in the trivial bundles
\[(E_z,F_z)=(E^\gluing_{z^-},F^\gluing_{z^-})=(E^\gluing_{z^+},F^\gluing_{z^+})\]
  in the fixed trivialization over the (half)disks around $z^\pm$.
  Using cutoff functions, one constructs from $D_{E,F}$ a family of
  real Cauchy-Riemann operators 
\[D_{\dE,\dF}: \Omega^0(\dE,\dF) \to \Omega^{0,1}(\dE) \] 
for $(\dS,\dE,\dF)$ similar to the construction of \eqref{eq:adding}.
Each operator $D_{\dE,\dF}$ in the family is equal to $D_{E,F}$ away
from the gluing region and approaches the trivial operator on the neck
in the limit $\gluing \to \infty$.  \een
\end{definition} 

Note that the conformal structure of $\dS$ depends on the value of the
gluing parameter $\gluing$, as well as the choices of local
coordinates $\R^\pm\times S^1$ or $\R^\pm\times[0,1]$ on punctured
neighborhoods of $z^\pm$.  In addition, to obtain a surface with
strip-like ends in our sense one has to choose a new ordering on the
nodes and possibly the boundary components of $\dS$.

The following gluing result is the basic result used in the
identification of determinant lines of the deformed Cauchy-Riemann
operator with the determinant line of the original.  The result is a
slight modification of \cite[Lemma 3.1]{ek:or}.  We suppose that the
node is on the boundary; the interior case is similar.  We also assume
for simplicity that $\dS$ is smooth.

\begin{theorem} {\rm (Long exact sequence in homology for surfaces with boundary
and strip-like ends)} \label{thm:fourtermthm} Let $\dS,\dE,\dF$ be a
  deformation of a boundary node from $S,E,F$ obtained from the resolved
  surface and bundles $\rS,\rE,\rF$ by identifying small balls around
  the node.  For sufficiently large values of the gluing parameter
  $\gluing$ there is an exact sequence
\begin{equation} \label{eq:fourterm} 
 0 \to \ker(D_{\dE,\dF}) \stackrel{\iota_\tau}{\to} \ker(D_{{\rE},\rF})
\stackrel{D_{\dE,\dF}^\red}{\to} F_{z} \oplus \coker(D_{{\rE},\rF})
\to \coker(D_{\dE,\dF}) \to 0
\end{equation} 
such that in the limit $\gluing \to \infty$, the middle map
$D_{\dE,\dF}^\red$ converges to $D_{\nE,\nF}^\red$ from
\eqref{eq:reduced}  as operators on
finite-dimensional vector spaces (after rescaling by a positive
constant) locally uniformly in families.

In the case of an interior node, there is a similar 
exact sequence with $\rE_{z_\pm}$ instead of 
$\rF_{z_\pm}$.   The middle term is independent of the gluing parameter $\tau$, 
the map $\iota_\tau$ depends continuously on $\tau$,  
and on parameters when the data vary in a family; and 
the convergence is uniform in phase $t$ in the interior case.
\end{theorem}

\begin{remark} {\rm (Relation to the long exact sequence for a normalization)} 
The sequence \eqref{eq:fourterm} is a real version of a standard long
exact sequence in algebraic geometry.  Namely, suppose $\bE$ is a
vector bundle on a nodal curve $S$ with a node at $z$, and $\pi:
S^\rho \to S$ is the normalization of $S$ at $z$.  There is a short
exact sequence of sheaves given by
\[ 0 \to \bE \to \pi_* \pi^* \bE \to \iota_{z,*} \iota_{z}^* \bE \to 0 \]
where abusing notation $\iota_{z,*} \iota_z^* \bE$ denotes the
skyscraper sheaf with fiber $\bE_z$.  The short exact sequence induces
a long exact sequence of cohomology groups, described in terms of
linearized Cauchy-Riemann operators as
\[ 0 \to \ker(D_{\bE}) \to \ker(D_{\bE^\rho}) \to \bE_{z} \to
\coker(D_{\bE}) \to \coker(D_{\bE^\rho}) \to 0 .\]
See for example \cite[(1)]{be:gw}.
\end{remark} 

\begin{proof}[Proof of Theorem \ref{thm:fourtermthm}] 
We deal first with the case of boundary node gluing.  First, we construct suitable Sobolev spaces on the glued surface
$\dS_{}$ that depend on the gluing parameter $\gluing$.  We will
require a nested sequence of cutoff functions on the neck region of
$\ti{S}$ for which we introduce the following notation.  For each
integer $n \in [2,6]$ let $\beta_n$ be a family of smooth functions on
$\dS_{}$ depending on $\tau$ with the following properties:
\begin{enumerate} 
\item $\beta_n$ is supported on the part of the neck parametrized by
  $[-n\gluing/7,n\gluing/7] \times [0,1]$;
\item $\beta_n$ is equal to $1$ on $[-(n-1)\gluing/7,(n-1)\gluing/7]
  \times [0,1]$, and takes values in $[0,1]$ elsewhere;

\item $\beta_n$ has first derivative bounded by $C/\gluing$ for some
  constant $C> 0$.
\end{enumerate} 
For $\delta \in (-1,0) $ consider the weight function
\begin{equation} \label{eq:zeta} \zeta_\gluing \in C^\infty(\dS_{}), \ \ \ \zeta_\gluing = (1-
 \beta_6) + \beta_6
  \left( e^{\pi \delta (s + \gluing)}+ e^{- \pi \delta (s - \gluing)}  \right) .\end{equation}
The second term is well-defined since $\beta_6$ is supported on the
neck.  Let $W^{1,2}_{\delta}(\dE_{},\dF_{})$ be the Sobolev space with
weight function $\zeta_\gluing$, or more precisely, the space of $W^{1,2}_\loc$ functions with finite weighted norm
\[ \Vert \xi \Vert_{W^{1,2}_{\delta}(\dE_{},\dF_{})} = \Vert
\zeta_\gluing \xi \Vert_{W^{1,2}(\dE_{},\dF_{})} .\]
The Cauchy-Riemann operator $D_{\dE_{},\dF_{}}$ is Fredholm on this
Sobolev space by standard results.

With these Sobolev spaces defined we study the kernel and the cokernel
of the linearized Cauchy-Riemann operator on the glued surface.  We
have an exact sequence of Banach spaces defined by the linearized operator
\begin{equation} \label{eq:exact} 0 \to \ker(D_{\dE_{},\dF_{}}) \to
\Omega^0(\dE_{},\dF_{}) \to \Omega^{0,1}(\dE_{}) \to
\coker(D_{\dE_{},\dF_{}}) \to 0 \end{equation}
where the middle terms are the $W^{1,2}_\delta$ resp.
$W^{0,2}_\delta$ spaces defined above.  We wish to show that
\eqref{eq:exact} is equivalent, modulo stabilization, to an exact
sequence of finite-dimensional spaces.  More precisely, we show that
the middle map in \eqref{eq:exact} is uniformly bounded from below on a
space isomorphic to the complement of the kernel of the Cauchy-Riemann
operator for the normalization.  

First we identify the kernel of the operator on the normalization with
a subspace of sections on the deformed bundle.  For any $\xi \in
\Omega^0(E,F)$ we denote by $(1- \beta_2)\xi \in
\Omega^0(\dE_{},\dF_{})$ the section obtained by multiplying by the
cutoff function $(1- \beta_2)$ and using the identification of
$\ti{E}_{}$ and $E$ away from the neck.  The map
\begin{equation}  \label{eq:thismap} \ker(D_{\rE,\rF}) \to (1- \beta_2)\ker(D_{\rE,\rF}), \ \ \xi
\mapsto (1 - \beta_2)\xi \end{equation}
is an isomorphism, since by analytic continuation no element is
supported on the neck.  We identify $\ker(D_{\rE,\rF})$ with its image
under the map \eqref{eq:thismap}.

The first map in the four-term sequence \eqref{eq:fourterm} is now
defined by orthogonal projection.  Namely, let $V_\gluing$ denote the
$W^{0,2}_\delta$-orthogonal complement of $\ker(D_{\rE,\rF})$.  Define
the first map in \eqref{eq:fourterm} to be the composition
\[ \ker(D_{\dE,\dF}) \to \ker(D_{{\rE},\rF}) \subset
\Omega^0(\dE_{},\dF_{})\]
of inclusion and projection along $V_\gluing$.

To find the second map in the four-term sequence \eqref{eq:fourterm}, we
find an approximate description of the image of $V_\gluing$ under the
linearized operator $D_{\dE_{},\dF_{}}$ for the glued surface.  

\vskip .05in

\noindent {\em Claim: For $\delta \in (-1,0)$, the restriction of
  $D_{\dE_{},\dF_{}}$ to $V_\gluing$ satisfies the following lower bound:
  There exist constants $C$ and $\gluing_0$ such that for
  $\gluing > \gluing_0$,
\begin{equation} \label{eq:claim}
C \Vert \xi \Vert_{W^{1,2}_\delta} \leq \Vert D_{\dE_{},\dF_{}} \xi
\Vert_{L^2_\delta}, \ \forall \xi \in V_\gluing .\end{equation}
}

Suppose otherwise.  There exists a sequence $ \gluing(\alpha) \to
\infty, \xi(\alpha) \in V_{\gluing(\alpha)}$ with
\begin{equation} \label{eq:collect}
\Vert \xi(\alpha) \Vert_{W^{1,2}_\delta(\dE)} =1, \ \ \ \lim_{\alpha \to
 \infty} \Vert D_{\dE_{},\dF_{}} \xi(\alpha)
 \Vert_{L^2_\delta(\dE)} = 0 .\end{equation}
Denote by $S^\circ$ resp. $E^\circ,F^\circ$ the surface with
strip-like ends obtained by removing the node $z$, resp. the fibers
$E_z,F_z$.  Let 
\[S^{\nu} \cong [ -\gluing,\gluing ] \times [0,1] ,\quad E^\nu = E|
S^\nu, \quad F^\nu = F | \partial S^\nu \]
be the neck and bundles restricted to the neck.  We split
$\xi(\alpha)$ into sections supported away from and on the neck, and
apply elliptic estimates for $S^\circ, S^\nu$ to obtain a
contradiction.  First, note that the kernel of $D_{\rE,\rF}$ may be
identified with the kernel of $D_{E^\circ,F^\circ}$ for any Sobolev
weight $\delta \in (-1,0)$.  Indeed, we may identify $S$ locally with
the half-space $\H$.  We assume that our Sobolev spaces on $S$ use a
measure that is locally the pull-back of the standard measure on $\H$.
The conformal transformation $ (s,t) \mapsto \exp(-\pi ( s + it))$ maps
the infinite strip $\R \times [0,1]$ to $\H$.  The pull-back of the
canonical measure on $\H$ is $\pi^2 e^{-2 \pi s} \d s \d t$. 
With our
conventions, this pullback is the measure with Sobolev weight $\delta
= -1$.  Thus pullback gives an identification of the kernel
\[W^{1,2}_{-1}(E^\circ,F^\circ)
\supset 
\ker(D_{E^\circ,F^\circ}) \to \ker(D_{\rE,\rF}) \subset W^{1,2}(\rE,\rF) .\]
Elliptic regularity gives an identification with the kernel of
$D_{\rE,\rF}$ on $W^{k,2}(\rE,\rF)$ for any $k \ge 1$.  The operator
$D_{E^\circ,F^\circ}$ is Fredholm for weights $\delta$ not in the
spectrum $\pi \Z$ of the limiting operator on the strip-like ends, see
e.g. \cite{loc:ell}.  The kernel $\ker(D_{E^\circ,F^\circ})$ is
unchanged by any non-negative perturbation of Sobolev weight not
passing through the spectrum of the limiting operator:
\[ ( [ \delta_1,\delta_2 ] \cap \on{Spec}(\partial_t + \cH_e) =
\emptyset , \ \ \forall e \in \cE) \implies
\left(\ker(D_{E^\circ,F^\circ})_{W^{1,2}_{\delta_1}} =
\ker(D_{E^\circ,F^\circ})_{W^{1,2}_{\delta_2}} \right) .\]
Hence 
\[ \ker(D_{E^\circ,F^\circ})_{W^{1,2}_{-1}} =
\ker(D_{E^\circ,F^\circ})_{W^{1,2}_{\delta}}, \quad \forall \delta \in
(-1,0) .\]
Let $\beta_2$ denote the cutoff function introduced at the beginning
of the proof.  Since the operator $D_{\ti{E},\ti{F}}$ is equal to
$D_{E^\circ,F^\circ}$ away from the neck and approaches
$D_{E^\nu,F^\nu}$ on the neck, we have for some constant $C > 0$
independent of $\alpha$,
\begin{eqnarray*}
\Vert \xi(\alpha) \Vert_{{W^{1,2}_\delta(\dE)}} &\le& C \Vert (1 - \beta_2) \xi(\alpha)
\Vert_{W^{1,2}_\delta(E^\circ)} + C \Vert \beta_2 \xi(\alpha) \Vert_{W^{1,2}_\delta(E^\nu)} \\ &\le &
C \Vert D_{E^\circ,F^\circ} (1- \beta_2) \xi(\alpha) \Vert_{L^2_\delta(E^\circ)} \\ && +
C \Vert \on{proj}_{\ker(D_{E^\circ,F^\circ})} (1 - \beta_2)
\xi(\alpha) \Vert_{W^{1,2}_\delta(E^\circ)}\\ && + C \Vert D_{E^\nu,F^\nu} \beta_2
\xi(\alpha) \Vert_{L^2_\delta(E^\nu)} \\ &\to& 0.
\end{eqnarray*} 
We show that  the term $\Vert \on{proj}_{\ker(D_{E^\circ,F^\circ})} 
(1 - \beta_2) \xi(\alpha) \Vert$ tends to zero. 
For $\kappa \in \ker(D_{E^\circ,F^\circ})$, the section $(1- \beta_2)\kappa$ lies in $\ker(D_{\rE,\rF})$ up to the
identification \eqref{eq:thismap}.  Since 
$\xi(\alpha)  \in V_{\gluing(\alpha)}$, 
\[ \langle \xi(\alpha), (1 - \beta_2) \kappa \rangle_{\zeta_{\gluing(\alpha)}}  = 0 . \] 
The weight function $\zeta_\gluing = 1$ away from the neck and agrees with the weight of 
 $S^\circ$ up to a factor of at most $2$ on 
 $\on{supp}( 1- \beta_2)$.  Thus the pairing 
$\langle (1 - \beta_2) \xi(\alpha), \kappa \rangle_{E^\circ}$ differs from $\langle \xi(\alpha), (1 - \beta_2) \kappa \rangle_{\zeta_{\gluing(\alpha)}}$ by at most 
$C e^{-c \gluing(\alpha)} \Vert \xi(\alpha) \Vert 
\Vert \kappa \Vert.$  Indeed, the integrands are supported where $\beta_2< 1$ and the weights differ
there by a universally bounded factor at the outer end of the neck.   This contradicts \eqref{eq:collect}.    The first inequality follows from
comparability of the norms on $E^\circ$, $E^\nu$, and $E$, and the second
inequality combines the elliptic estimates for $(E^\circ,F^\circ)$ and
$(E^\nu,F^\nu)$.  The last limit uses the bound on the derivative of
$\beta_2$ and the fact that $D_{\dE,\dF} \xi(\alpha) \to 0$.  This
proves the claim.
\vskip .05in

We continue with the construction of the second map in the four-term
sequence.  Identify $\coker(D_{\rE,\rF})$ with the
$L^2_\delta$-perpendicular of $\on{im}(D_{\rE,\rF})$.  Also
identify $E$ and $E^\circ$ away from the neck.  Let $\beta_4$ denote
the cutoff function introduced above, and define an injection for
$\gluing$ sufficiently large
\[ \coker(D_{\rE,\rF}) \to \Omega^{0,1}(\dE_{},\dF_{}), 
\ \ \ \xi \mapsto (1-\beta_4)\xi ;\]
let $(1-\beta_4)\coker(D_{\rE,\rF}) $ denote its image.  Let $\ul{F_z}$
denote the subspace of $\Omega^{0,1}(\dE_{},\dF_{})_{L^2_\delta}$ consisting
of one-forms equal on the neck to $ f (\d s - i \d t)$ for some $ f
\in F_z$.  Multiplying by $\beta_4$ gives a finite-dimensional
subspace of $\Omega^{0,1}(\dE_{},\dF_{})_{L^2_\delta}$, isomorphic to
$F_z$ by evaluation at a point $z_{\on{mid}}$ at the mid-point of the
neck:
\[ \ul{F_z} \cong F_z, \quad \xi \mapsto \xi(z_{\on{mid}}) .\]
For $\gluing$ sufficiently large, the sum $ (1 - \beta_4)
\coker(D_{\rE,\rF}) + \beta_4 \ul{F_z}$ is direct, since the
intersection of the two subspaces is trivial.  Let
\[U_\gluing := ( (1 - \beta_4) \coker(D_{\rE,\rF}) + \beta_4 \ul{F_z}
)^\perp \subset \Omega^{0,1}(\dE_{},\dF_{})\]
denote the $L^2_\delta$-perpendicular.  Let
\[\pi_\gluing: \Omega^{0,1}(\dE_{},\dF_{}) \to U_\gluing\]
denote the projection.  

\vskip .05in

\noindent {\em Claim: The operator
$\pi_\gluing \circ
D_{\dE_{},\dF_{}}: V_\gluing \to U_\gluing$
is an isomorphism with uniformly bounded right inverse, for $\gluing$
sufficiently large.}

\vskip .05in

We first show that $U_\gluing^\perp = (1 - \beta_4) \coker(D_{\rE,\rF}) + \beta_4 \ul{F_z}$ is almost orthogonal to the image of $D_\gluing$.  For $\eta \in 
U_\gluing^\perp$, the image $D^*_\gluing \eta$
consists only of terms involving the derivative of $\beta_4$, the logarithmic derivative of the weight $\zeta_\gluing$, and the difference between 
$D_{\dE,\dF}$ and the trivial operator on the neck. Therefore, there
is a sequence $\eps_\gluing \to 0$ with 
\begin{equation}
| \langle D_{\dE,\dF} \xi, \eta \rangle_{L^2_\delta} | = 
| \langle \xi, D^*_\gluing \eta \rangle_{L^2_\delta} | 
\le \eps_\gluing \Vert \xi \Vert_{W^{1,2}_\delta}
\Vert \eta \Vert_{L^2_\delta}, \quad 
\xi \in \Omega^0(\dE,\dF), \eta \in U_\gluing^\perp. \end{equation}

Suppose, by way of contradiction, that there is no uniform lower bound.  Then there is a sequence
\[\gluing(\alpha) \to \infty, \quad \xi(\alpha) \in
V_{\gluing(\alpha)},  \quad \zeta(\alpha) \in U_{\gluing(\alpha)} 
\] 
with
\begin{equation} 
\label{eq:list} 
\Vert \xi(\alpha) \Vert_{V_{\gluing(\alpha)}} 
= \Vert \zeta(\alpha) \Vert_{U_{\gluing(\alpha)}}
=1, \quad (
D_{\dE_{\gluing(\alpha)},\dF_{\gluing(\alpha)}} \xi(\alpha),
\zeta(\alpha) ) \to 0.
\end{equation}
The pairing of $ D_{\dE_{\gluing(\alpha)},\dF_{\gluing(\alpha)}}
\xi(\alpha) $ with any sequence of elements 
\[ ( 1- \beta_4)
\zeta(\alpha) + \beta_4 \zeta(\alpha)' \in  (1 - \beta_4)
\coker(D_{\rE,\rF}) + \beta_4 \ul{F_z} \] 
of norm one approaches zero since the cut-off functions are slowly
varying.  This convergence implies $\Vert
D_{\dE_{\gluing(\alpha)},\dF_{\gluing(\alpha)}} \xi(\alpha) \Vert \to
0$ which contradicts \eqref{eq:claim}.  The claim follows.

The discussion above shows that we have splittings
\begin{equation} \label{eq:split1}  \Omega^0(\dE_{},\dF_{}) \cong 
V_\gluing \oplus \ker(D_{\rE,\rF}) , \quad 
 \Omega^{0,1}(\dE_{}) \cong U_\gluing \oplus F_{z} \oplus
 \coker(D_{{\rE},\rF}) .\end{equation}
By \eqref{eq:exact} and \eqref{eq:split1} for $\gluing$ sufficiently large
there is an exact sequence
\begin{equation} \label{eq:infdim} 
 0 \to \ker(D_{\dE_{},\dF_{}}) \to V_\gluing \oplus
 \ker(D_{{\rE},\rF}) \to U_\gluing \oplus F_{z} \oplus
 \coker(D_{{\rE},\rF}) \to \coker(D_{\dE_{},\dF_{}}) \to 0
 .\end{equation}
We obtain from this sequence of Banach spaces an exact sequence of
finite-dimensional spaces as follows.  By the Riemann-Roch theorem for
surfaces with boundary \eqref{eq:rr}, the index of the deformed
Cauchy-Riemann operator is
\begin{equation} \label{eq:drop}
 \Ind(D_{\dE_{},\dF_{}}) = \Ind(D_{{\rE},\rF}) -
 \dim(F_z) \end{equation}
since the Euler characteristic of the glued surface is one less than
the resolved surface.  The identity \eqref{eq:drop} implies that the
restriction of $\pi_\gluing \circ D_{\dE_{},\dF_{}}$ to $V_\gluing$ is
an isomorphism onto $U_\gluing$.  Let $\ti{D}_{ij}, i,j = 1,2$ denote
the components of $D_{\dE_{},\dF_{}}$ with respect to the splittings
\eqref{eq:split1}.  The kernel of $D_{\dE_{},\dF_{}}$ consists of pairs
$(\xi_1,\xi_2)$ such that
\[ \xi_1 = - \ti{D}_{11}^{-1} \ti{D}_{12} \xi_2, \ \ \ (- \ti{D}_{21}
\ti{D}_{11}^{-1} \ti{D}_{12} + \ti{D}_{22}) \xi_2 = 0 .\]
Define
\[ D_{\dE_{},\dF_{}}^{\on{red}} := - \ti{D}_{21} \ti{D}_{11}^{-1}
\ti{D}_{12} + \ti{D}_{22} .\] 
We have an identification
\[ \ker(D_{\dE_{},\dF_{}}) \to \ker(D_{\dE_{},\dF_{}}^{\on{red}}),
\ \ \ \xi_2 \mapsto ( - \ti{D}_{11}^{-1} \ti{D}_{12} \xi_2,\xi_2) .\]
The image of $D_{\dE_{},\dF_{}}$ consists of pairs $(\eta_1,\eta_2)$
such that $ \eta_2 - D_{21} D_{11}^{-1} \eta_1$ lies in the image of
$D_{\dE_{},\dF_{}}^{\on{red}}$.  The inclusion of $F_{z} \oplus
\coker(D_{{\rE},\rF})$ into $ U_\gluing \oplus F_{z} \oplus
\coker(D_{{\rE},\rF})$ induces an identification of cokernels of
$D_{\dE_{},\dF_{}}$ and $D_{\dE_{},\dF_{}}^{\on{red}}$.  Applying this
identification to \eqref{eq:infdim} gives the desired exact sequence.

To compute the limit of the middle operator in the limit of large
gluing parameter, note that the component of
$D_{\dE_{},\dF_{}}^{\on{red}}$ in $F_z$ is given asymptotically by
projecting $ D_{\dE_{},\dF_{}} ((1 - \beta_2)\xi)$ onto $\beta_4
\ul{F}_z$.  We have
\[ D_{\dE_{},\dF_{}} ((1 - \beta_2)\xi) \to - \tfrac{1}{2} (\partial_s \beta_2)\xi
(\d s - i \d t) .\]
With the coordinates of Definition \ref{def:gluing}, the end $s = \gluing$ of the neck lies on the side of $z_+$ and the end $s = -\gluing$ on the side of $z_-$.  Hence pairing with $f \in \ul{F}_z$ gives a positive multiple of the difference of evaluation maps
$\xi(z_+) - \xi(z_-)$ paired with $f$.  It follows that, after rescaling by a positive constant, the limit is
\[ \lim_{\gluing \to \infty} D_{\dE_{},\dF_{}}^{\red} \xi = (\xi(z_+)
- \xi(z_-),0) =  D_{\nE,\nF}^{\red} \xi\]
as claimed in the statement of the Theorem.  

For an interior node, the proof is the same with the following changes.  The neck is
parametrized by 
\[ [-|\gluing|,|\gluing|] \times (\R/2\Z), \quad s+ it = \ln(z) / \pi + \gluing , \]
and the cutoff functions
$\beta_n$ are supported on $[-n |\gluing|/7 , n |\gluing|/7] 
\times (\R/2\Z)$.  The middle term of 
\eqref{eq:fourterm} is $E_z$ in place of $F_z$, and the 
limit of the reduced operator is 
\[ \xi \mapsto (\xi(z_+) - \xi(z_-),0) = D^\red_{\nE,\nF} \xi \] 
with values in $E_z$, using the given identification
$E_{z_+} \cong E_{z_-}$.  The kernel of $D_{\rE,\rF}$ is 
identified with that of $D_{E^\circ,F^\circ}$ using the map 
$(s,t) \mapsto \exp(- \pi(s+ it))$ from the cylinder
$\R \times (\R/2\Z)$ to $\C^\times$ (in place of the half-plane); the pulled-back measure is again the measure
with weight $\delta = -1$.  The asymptotic operator 
on $\R/2\Z$ has spectrum $\pi \Z$ with kernel $E_z$, 
so $D_{E^\circ,F^\circ}$ is Fredholm 
for $\delta \in (-1,0)$ exactly as in the boundary 
case.  Changing the phase $\theta = \arg \gluing$ translates the second variable of the neck coordinates.  Thus the zeroth order term $a(s,t)$ of the operator on the neck is replaced by $a(s,t - \theta)$.  The constants 
in the estimates depend only on $C^k$ norms of $a$.  It follows that these constants
are uniform in $\theta  \in \R/2\Z$, and the weights, cutoff functions and the limit $D^{\red}_{\nE,\nF}$ do not depend on $\theta$.  This proves the interior case, uniformly in the phase. 
\end{proof}  

\begin{corollary}  \label{cor:deformcor} \label{cor:detiso2}
{\rm (Isomorphism of determinant lines induced by deformations of
  nodes)} Let $D_{\dE, \dF}$ be the operator obtained from a
Cauchy-Riemann operator $D_{\nE,\nF}$ by deforming a node.  There is a
canonical up to deformation {\em gluing isomorphism}
$\det(D_{\nE,\nF}) \to \det(D_{\dE,\dF}).$
\end{corollary} 

\begin{proof} 
The existence of the exact sequence \eqref{eq:fourterm} is equivalent to the existence of
isomorphisms
\begin{equation} \label{eq:isos}
 \ker(D_{\dE,\dF}^\red) \to \ker(D_{\dE,\dF}), \ \ \
\coker(D_{\dE,\dF}^\red) \to \coker(D_{\dE,\dF})
.\end{equation}
These induce an isomorphism of determinant lines
\begin{equation} \label{eq:Dred}  
\det(D_{\dE,\dF}) \to \det(D_{\dE,\dF}^\red).
\end{equation}
The homotopy of Theorem \ref{thm:fourtermthm} induces an isomorphism of
determinant lines $\det(D^\red_{\nE,\nF}) \to \det(D_{\dE,\dF}^\red)
.$ Combining this with \eqref{eq:Dred}, \eqref{eq:unredred}, and
\eqref{eq:unred} proves the corollary.
\end{proof} 

\subsection{Associativity of gluings and deformations}

\label{sec:iterated} 
Next, we show that the gluing maps of Proposition \ref{prop:glueprop} and
Corollary \ref{cor:deformcor} satisfy an {\em associativity} property:

\begin{proposition} \label{prop:assoc} {\rm (Associativity of gluing)}   Let 
$\nS$ be a nodal surface with strip-like ends and $\dS$ the surface
  obtained by one of the following:
\begin{enumerate} 
\item deforming two nodes $\ul{w}_0,\ul{w}_1$, or
\item deforming one node $\ul{w}$ and gluing two strip-like ends
  $e_-,e_+$, or
\item gluing two pairs of strip-like ends $e_{0,\pm}, e_{1,\pm}$.
\end{enumerate} 
Suppose that $D_{\dE,\dF}$ is obtained from $D_{\nE,\nF}$ by deforming
the nodes.  Then the resulting gluing isomorphisms $\det(D_{\nE,\nF}) \to \det(D_{\dE,\dF})$ are independent, up to homotopy, of the order of
deformation/gluing.   
\end{proposition} 

\begin{proof}  We consider only the case of two boundary nodes $z,z'$; the cases of
interior nodes, strip-like ends, and mixed cases are similar but
easier.  We claim that if $\delta$ denotes the deformation of $z$ and
$\delta'$ the deformation of $z'$ then the diagram
\begin{equation} \label{eq:square}
 \begin{diagram} 
\node{\det(D_{E,F})} \arrow{s} \arrow{e} \node{\det(D_{E,F^{\delta}})} 
\arrow{s} \\
\node{\det(D_{E,F^{\delta'}})} \arrow{e} \node{\det(D_{E,F^{\delta,\delta'}})} 
\end{diagram} \end{equation}
commutes.  The proof is a minor modification of e.g.  \cite[Lemma
  3.5]{ek:or}.  Simultaneous deformation of the two nodes leads to an
exact sequence
\begin{equation} \label{eq:fourterm2} 
 0 \to \ker(D_{E,F^{\delta,\delta'}}) \to \ker(D_{E,F^{\rho,\rho'}})
\to F_{z} \oplus F_{z'} \oplus
\coker(D_{E,F^{\rho,\rho'}}) \to \coker(D_{E,F^{\delta,\delta'}}) \to
0.
\end{equation} 
Together with the identification of $D_{E,F}$ with the reduced operator
in \eqref{eq:reduced}, this induces an isomorphism
\begin{equation} \label{eq:composedet}
 \det(D_{E,F}) \to \det(D_{E,F^{\delta,\delta'}})
.\end{equation}
We claim that this isomorphism is equal to the isomorphism given by
going either way around the square \eqref{eq:square}.  To prove the claim,
consider the diagram
\dgARROWLENGTH=.1cm
\[ \begin{diagram} 
 \node{\ker(D_{E,F^{\delta,\rho'}})}
\arrow{e,t}{\on{Id}}
\node{\ker(D_{E,F^{\delta,\rho'}})}
\arrow{e,t}{D_{E^\delta,F^\delta}^{{\red}}}
 \arrow{s} 
\node{F_{z} \oplus
\coker(D_{E,F^{\delta,\rho'}})} 
\arrow[1]{e,t}{\on{Id}}
\node{F_{z} \oplus
\coker(D_{E,F^{\delta,\rho'}})} 
\arrow{s}  \\
 \node{\ker(D_{E,F^{\delta,\delta'}})}
\arrow{e} \arrow{n} 
\node{\ker(D_{E,F^{\rho,\rho'}})}
\arrow{e,t}{D_{E,F}^{\red}} 
\node{F_{z} \oplus F_{z'} \oplus
\coker(D_{E,F^{\rho,\rho'}})} 
\arrow[1]{e} \arrow{n}
\node{\coker(D_{E,F^{\delta,\delta'}})} 
\end{diagram} .\]
\dgARROWLENGTH=1cm For fixed gluing parameters $\gluing,\gluing'$ the
diagram commutes up to a small error term which is irrelevant for the
purposes of orientations.
 By approximate commutativity
of the diagram, the composition of the top and right maps in
\eqref{eq:square} is equal up to homotopy to \eqref{eq:composedet}.  A
similar argument shows the same for the composition of the two maps on
the other side of \eqref{eq:square}.  This completes the proof.
\end{proof}

\subsection{Determinant lines for families of Cauchy-Riemann operators}

The existence of the gluing isomorphisms of determinant lines can be
phrased as the existence of a topological determinant line bundle over the parameter space for a family of Cauchy-Riemann operators with degenerations, following the discussion
in \cite{Ainfty}.

\begin{definition} {\rm (Decomposed spaces)}  
 Let $\GG$ be a partially ordered set with partial order $\leq$.  Let
 ${B}$ be a Hausdorff paracompact space.  A {\em $\GG$-decomposition}
 of ${B}$ is a locally finite collection of disjoint locally closed
 subspaces ${B}_\Gamma, \Gamma \in \GG$ each equipped with a smooth
 manifold structure of constant dimension $\dim({B}_\Gamma)$, such
 that
\[ {B} = \bigcup_{\Gamma \in \GG} {B}_\Gamma \]
and 
\[ {B}_\Gamma \cap \ol{{B}_{\Gamma'}} \neq \emptyset \iff {B}_{\Gamma}
\subset \ol{{B}_{\Gamma'}} \iff \Gamma \leq \Gamma' .\]
The {\em dimension} of a $\GG$-decomposed space ${B}$ is
\[ \dim {B} = \sup_{\Gamma \in \GG} \dim( {B}_\Gamma) .\]
The {\em stratified boundary} $\partial_s {B}$ resp. {\em stratified
  interior} $\on{int}_s {B}$ of a $\GG$-decomposed space ${B}$ is the
union of pieces 
\[
 \partial_s B = \bigcup_{\dim({B}_\Gamma) < \dim({B})}
{B}_\Gamma, \quad
 \on{int}_s B = \bigcup_{\dim({B}_\Gamma) = \dim({B})}
{B}_\Gamma .\]
An {\em isomorphism} of $\GG$-decomposed spaces ${B}_0,{B}_1$ is a
homeomorphism ${B}_0 \to {B}_1$ that restricts to a diffeomorphism on
each piece $B_{0,\Gamma}$.
\end{definition}  

\begin{example} {\rm (Cone construction)}  
Let ${B}$ be a $\GG$-decomposed space.  The {\em cone} on ${B}$
\[ C{B} := \left( {B} \times [0,\infty) \right) / \left( (r,0) \sim (r',0), r,r' \in {B} 
\right)  \]
has a natural $\GG^+$-decomposition, where $\GG^+ = \GG \cup \{ \star \}$
is obtained by adjoining a new minimal element $\star < \Gamma$ for all
$\Gamma \in \GG$: the vertex $\star_B$, the image of ${B} \times \{ 0 \}$,
is a separate piece of dimension zero, and the remaining pieces are the
open cones
\[ (C{B})_\star = \{ \star_B \}, \quad (C{B})_\Gamma = {B}_\Gamma \times (0,\infty), \quad \dim(C{B}) = \dim({B}) + 1 .\]
The vertex lies in the closure of every piece, as required by the
partial order on $\GG^+$.  More generally, if ${B}$ is a $\GG$-decomposed space equipped with a
locally trivial map $\pi$ to a manifold $A$, the {\em cone bundle} on
${B}$ is the union of cones on the fibers, that is,
\[ C_A {B} := \left( {B} \times [0,\infty) \right) / \left( (r,0) \sim (r',0), \pi(r) = \pi(r')
  \in A \right),
 \]
is again a $\GG^+$-decomposed space, with the set of vertices, a copy of
$A$, as the piece labelled $\star$, the pieces
${B}_\Gamma \times (0,\infty)$ labelled $\Gamma$, and dimension $\dim(C_A {B}) =
\dim({B}) + 1$.  
\end{example} 

\begin{definition} \label{def:stratified}
\begin{enumerate}
\item {\rm (Stratified spaces)} A decomposition ${B} = \cup_{\Gamma
  \in \GG} {B}_\Gamma$ of a space ${B}$ is a {\em stratification} if
  the pieces ${B}_\Gamma$ fit together in a nice way: Given a point
  $r$ in a piece ${B}_\Gamma$ there exists an open neighborhood $U$ of
  $r$ in ${B}$, an open ball $V$ around $r$ in ${B}_\Gamma$, a
  stratified space $L$ (the {\em link} of the stratum) and an
  isomorphism of decomposed spaces $\phi: V \times CL \to U $ that
  preserves the decompositions, where $CL$ carries the decomposition of
  the cone construction above with the vertex as a separate piece.  That is, $\phi$ restricts to a
  diffeomorphism $\phi_{\Gamma'}$ from each piece $(V \times
  CL)_{\Gamma'}$ of $V \times CL$ to a piece $U \cap {B}_{\Gamma'}$; in
  particular, $\phi$ maps $V \times \{ \star_L \}$ onto $V \subset {B}_\Gamma$.
  A {\em stratified space} is a space equipped with a stratification.
\item {\rm (Families of quilted surfaces)} Let ${B} = \cup_{\Gamma \in
  \GG} {B}_\Gamma $ be a stratified space.  A {\em family of quilted
  surfaces with strip-like  ends} over ${B}$ is a stratified space $\SS
  = \cup_{\Gamma \in \GG} \SS_\Gamma $ equipped with a
  stratification-preserving map to ${B}$ such that each $\SS_\Gamma
  \to {B}_\Gamma$ is a smooth family of quilted surfaces with fixed
  type, and furthermore, local neighborhoods of $\SS_\Gamma$ in $\SS$
  are given by the gluing construction of Definition \ref{def:gluing}:
There exists 
\begin{enumerate} 
\item a neighborhood $U_\Gamma$ of ${B}_\Gamma$ in ${B}$, 
\item a projection $\pi_\Gamma: U_\Gamma \to {B}_\Gamma$, and
\item a map $\delta_\Gamma: U_\Gamma \to (\R_{\ge 0})^m \times \C^n$
\end{enumerate} 
  such that if $r \in U_\Gamma$ then $\ul{S}_r$ is obtained from
  gluing $\ul{S}_{\pi_\Gamma(r)} $ with gluing parameters
  $\delta_\Gamma(r)$.
\item {\rm (Families of bundles)} A {\em family of complex bundles
  with totally real boundary and seam conditions} is a collection $
  (\ul{E},\ul{F}) = (\ul{E}_b,\ul{F}_b )_{b \in B}$ of complex bundles
  with totally real boundary and seam conditions, such that for each
  $b \in B$, the nearby bundles are given by the gluing construction
  of Definition \ref{def:gluing}.
\end{enumerate} 
\end{definition} 

In other words, for a family of quilted surfaces with strip-like ends,
degeneration as one moves to a boundary stratum is given by
neck-stretching.  

\begin{proposition}  \label{prop:doublecover}
{\rm (Orientation double cover of a family with nodal degeneration)}
Let $\ul{S}_b, \ul{E}_b, \ul{F}_b, b \in B$ be a family of complex
vector bundles with totally real boundary conditions on quilted
surfaces with strip-like ends over a stratified space $B$. Then 
$\det(D) = \bigcup_b \det(D_b)$ carries a unique topology such that 
\begin{enumerate}
\item over each stratum $B_\Gamma$, $\det(D)$ is the determinant line bundle
of the smooth  family $D|_{B_\Gamma}$ from Remark \ref{rem:families}, and 
\item for every $r \in B_\Gamma$ and every neighborhood $U \cong V \times CL$ of $r$ as in Definition \ref{def:stratified}, the gluing isomorphisms of Corollary 
\ref{cor:deformcor} give a homeomorphism
$\det(D) |_U \cong \pi_\Gamma^* \bigl(\det(D)|_V \bigr)$ that restricts
to the identity over $V$.
\end{enumerate}
With this topology $\det(D) \to B$ is a topological line bundle and
$\Or(D)=\det(D)^\times/\R_{>0}\to B$ is a double cover.
\end{proposition}

\begin{proof}
We introduce local identifications as follows.  Let $U \cong V \times CL$ be as above.  For $(v,\ell,t) \in V \times  L \times (0,\infty)$, the surface $\ul{S}_{(v,\ell,t)}$ is obtained from $\ul{S}_v$ by 
gluing with parameters
$\delta_\Gamma(v,\ell,t)$,
and $D_{(v,\ell,t)}$ is deformable to the glued operator of $D_v$.  For 
$t$ large, Corollary \ref{cor:deformcor} gives an isomorphism 
\[ 
\Phi_U : \det(D_v) \to \det(D_{(v,\ell,t)}) . \] 
We set $\Phi_U = \mathrm{id}$ over $V$. The different choices in Corollary 
\ref{cor:deformcor} change $\Phi_U$ by an isomorphism homotopic to the identity, so $\Phi_U$ is well defined up to homotopy, and the induced map on orientations is independent of choices.
For fixed $\ell,t$, the  map $\Phi_U$ depends continuously on $v \in V$, since the exact sequence  \eqref{eq:fourterm} and the reduced operator vary continuously with the data as in Remark \ref{rem:families}. 
For an interior node, Theorem 
\ref{thm:fourtermthm} in the interior case 
shows that the linear operator 
$D_\gluing$ fits into an exact sequence
whose middle terms 
$\ker (D_{\rE,\rF})$, $E_z$, and 
$\coker(D_{\rE,\rF})$ do not depend on 
the phase of $\gluing$. Its middle map converges uniformly in the phase $\theta$ of $\gluing$ to the $\theta$-independent reduced operator $D^{\red}_{\nE,\nF}$
of \eqref{eq:reduced}.  Hence the isomorphisms $ \det(D_\gluing) \to 
\det(D^{\red}_{\nE,\nF})$ form a continuous family over $\theta$.
Thus $\Phi_U$ does not depend on 
$\theta$ and the monodromy of $\det(D)$
around $\theta$ is trivial. 

Finally, we check that the trivializations agree on overlaps.
Let $U,U'$ be neighborhoods of points 
of strata $B_\Gamma, B_{\Gamma'}$ with 
$\Gamma' < \Gamma$ that intersect non-trivially. On $U \cap U'$
the difference  map 
$\Phi_{U'}^{-1} \Phi_U$ is the gluing isomorphism for the nodes (or ends) 
smoothed in $U$ but not in $U'$. 
By Proposition \ref{prop:assoc}
the isomorphism obtained by smoothing in either order is the same up to homotopy, so the two identifications agree on $U \cap U'$. The topology on 
the determinant line is defined by 
declaring the maps $\Phi_U$ to be homeomorphisms.  Since each fiber of the orientation bundle has two components, the projection is a double cover. 
\end{proof}

\section{Relative spin structures}
\label{sec:Gstructures}

The construction of orientations for pseudoholomorphic maps with
Lagrangian boundary conditions depends on the existence of a structure
on the Lagrangians called a {\em relative spin structure} as
introduced by Fukaya-Oh-Ohta-Ono \cite{fooo}.  In this section, we
give a description of these structures in somewhat greater generality.  In
the case of spin structures the discussion is equivalent to the one introduced in
\cite{fooo}, but avoids triangulations.  A more general notion of
relative pin structures that does not require orientability of the
Lagrangians is developed in Solomon \cite{so:lb}.  A related notion of a 
{\em coupled spin structure} is used in Abouzaid-Manolescu \cite{am:orient}, and suffices to define Floer cohomology of a pair. 

\subsection{Relative spin structures via 
\v{C}ech cohomology}

\begin{definition}  {\rm (First non-abelian cohomology)} Let $G$ be a Lie group and
  $M$ a smooth manifold. Let 
\[ \mathcal{U} = \{ U_i, i \in I \} \] 
be an  open cover of $M$.  For integers $j \ge 0$ let
\[C^j(\cU,G) = ( g_{i_0,\ldots,i_j}: U_{i_0} \cap \ldots \cap U_{i_j}
\to G )_{i_0,\ldots,i_j} \]
be the space of cochains of degree $j$ and $\partial$ the coboundary
operator defined by
\[ \partial: \ C^j(\cU,G) 
\to C^{j+1}(\cU,G),\ (\partial
  g)_{i_0,\ldots,i_{j+1}} = \prod_{k=0}^{j+1}
  g_{i_0,\ldots,\widehat{i_k},\ldots,i_{j+1}}^{(-1)^k}.\]
The groups $ C^j(\cU,G) $ form a complex in the following sense.
Consider the space of {\em one-cycles}
\[ Z^1(\cU,G)
 := \ker(\partial|_{C^1}: C^1(\cU,G) \to C^{2}(\cU,G))
.\]
Then $C^0(\cU,G)$ acts on the left on $Z^1(\cU,G)$ by the formula
\[ (h,g)\mapsto hg, \quad (hg)_{i_0,i_1} := h_{i_0} g_{i_0,i_1}
h_{i_1}^{-1} .\]
The zeroth and first non-abelian cohomology groups are 
\[ H^0(\cU,G) := Z^0(\cU,G), \quad H^1(\cU,G) := C^0(\cU,G) \backslash
Z^1(\cU,G). \]
Any refinement $\cV \to \cU$ induces maps $H^j(\cU,G) \to H^j(\cV,G)$
for $ j =0 ,1$.  Denote by
\[ H^k(M,G) = \varinjlim_{\cU} H^k(\cU,G), \quad  k = 0,1 \]
the limit over refinements.  For $G$ abelian, all cohomology groups
\[ H^j(M,G), j = 0,1,2, \ldots \] 
are well-defined in a similar way.  
\end{definition}

One has the usual exact sequence as in Hatcher \cite[Section 3E]{ha:at}.

\begin{proposition}
    {\rm (Long exact sequence)} If $A \subset G$ is an abelian
  subgroup then there is a long exact sequence of pointed sets
\begin{equation} \label{eq:long}
 \ldots H^0(M,G/A) \to H^1(M,A) \to H^1(M,G) \to 
H^1(M,G/A) \to
 H^2(M,A).
\end{equation}
That is, $H^1(M,A)$ acts transitively on the kernel of $H^1(M,G) \to
H^1(M,G/A)$, and the set-theoretic kernel of the connecting
homomorphism $H^1(M,G/A) \to H^2(M,A)$ is equal to the image of
$H^1(M,G)$.
\end{proposition}

\begin{definition}
    \begin{enumerate}
\item {\rm (Characteristic class)} The image $c(z)$ of a class $z$ in
  $H^1(M,G/A)$ under the connecting homomorphism $c: H^1(M,G/A) \to
  H^2(M,A)$ in \eqref{eq:long} is called the {\em characteristic class}
  of $z$.
\item {\rm (Chern class)} As an example of a characteristic class, consider
  the exact sequence
\[1 \to \Z \to \R \to S^1 \to 1 \]
where $1$ denotes the trivial group.  In this case, there exists an
isomorphism with the Picard group
\[H^1(M,S^1) \to \Pic(M)\]
of isomorphism classes of line bundles.  The characteristic class map
$H^1(M,S^1) \to H^2(M,\Z)$ is equivalent to the first Chern class
\[c_1: \Pic(M) \to H^2(M,\Z) .\]
\end{enumerate}
\end{definition}

\begin{definition} {\rm ($G$-structures)} Let $\cU$
  be an open cover of $M$ as above, $A \subset G$ an abelian subgroup and
  $\phi: G \to G/A$ the projection, the last map in the exact sequence
\[ 1 \to A \to G \to G/A \to 1 .\]
Let $z \in Z^1(\cU,G/A)$ be a cocycle.  A {\em $G$-structure} on $z$
is a cocycle $a \in C^1(\cU,G)$ with $\phi_*(a) = z$; an isomorphism
from $a$ to $a'$ is an element $b \in C^0(\cU,A)$.  Let $H^1(\cU,G,z)$
denote the set of isomorphism classes of \v{C}ech $G$-structures and
\[H^1(M,G,z) = \varinjlim H^1(\cU,G,z) \]
the direct limit over open covers. 
\end{definition}

\begin{remark}  The obstruction to
admitting a $G$-structure is the characteristic class in \eqref{eq:long}.
\end{remark}

\begin{definition}  {\rm (Relative cohomology for a map)} Let $G$ be an abelian Lie
  group and $f: M \to N$ a smooth map of manifolds. Suppose the open
  cover $\U$ on $M$ is a refinement of the pull-back $f^* \V$ of the
  open cover $\V$ on $N$.  Let $\psi: \U \to f^* \V$ be a morphism of
  open covers; that is, for each $U \in \U$ an element $\psi(U) \in
  \V$ such that $f(U) \subset \psi(U)$.  Pull-back gives a morphism of
  cochain groups
\[ \psi^* : C^j(\V,G) \to C^j(\U,G) .\]
For non-negative integers $j$ define
\[ C^j(\psi,G) := C^j(\U,G) \times C^{j+1}(\V,G), \ \ \partial(a,b) =
( (\partial a) \cdot ( \psi^*b)^{(-1)^j}, \partial b). \]
The space $C^{j-1}(\psi,G)$ acts on the space of cocycles
$Z^j(\psi,G)$.  Let
\[H^1(\psi,G) := C^0(\psi,G) \backslash Z^1(\psi,G) .\]
Let
\[H^1(f,G) = \varinjlim H^1(\psi,G) \]
denote the limit over morphisms of open covers $\psi$. \end{definition}

\begin{remark} Standard
arguments (see e.g. \cite[Theorem 2.4.1]{shah:rg}) show that $H^1(f,G)
= H^1(\psi,G)$ where $\psi$ is any morphism of good covers.\end{remark}

\begin{definition}
    {\rm (Relative cohomology for a map of manifolds and a
  homomorphism of groups)} Let $f: M \to N$ be a smooth map of smooth
  manifolds, and $\cU, \cV$ open covers. A {\em morphism of open
    covers} $\psi: \cU \to \cV$ is an assignment of an element
  $\psi(U) \in \cV$ for every $U \in \cU$ such that $f(U) \subset
  \psi(U)$.  Let $A \subset G$ be a closed central subgroup of a Lie
  group $G$ and $\phi: G \to G/A$ the projection.  Let
\[z \in C^1(\cU,G/A), \quad \partial z = 0 \in C^2(\cU,G/A) \] 
be a cocycle.   A {\em \v{C}ech relative $G$-structure}
on $z$ is a pair
\[ (a,b) \in C^1(\cU,G) \times C^2(\cV,A) \quad 
\text{with}  \quad \phi(a) = z, \quad \partial a = \psi^* b, \quad \partial b = 1 . \] 
Here $\partial a \in C^2(\cU,G)$ takes values in $A$ because
$\phi(\partial a) = \partial z = 1$.  An isomorphism 
from $(a,b)$ to $(a',b')$ is a pair $(h,k) \in C^0(\cU,A) \times C^1(\cV,A)$ such that 
\[ a'_{ij} = (h_i a_{ij} h_j^{-1} (\psi^* k)), 
\quad (b'_{ijl} = (b \partial k)_{ijl}) . \] 
Denote by $H^1(\psi,G,z)$ the
space of isomorphism classes of relative $G$-structures and
\[H^1(f,G,z) = \varinjlim H^1(\psi,G,z) \]
the direct limit over morphisms of open covers.
\end{definition} 

The above constructions in \v{C}ech cohomology can be connected to
principal bundles as follows.

\begin{definition}
\begin{enumerate} 
\item {\rm (Principal bundles)} A principal $G$-bundle over a smooth
  manifold $M$ consists of a smooth right $G$-manifold $P$ together
  with a projection $\pi: P \to M$ such that $G$ acts freely
  transitively on the fibers of $\pi$ and is locally trivial in the
  following sense: For any $m \in M$, there exists an open
  neighborhood $U$ of $m$ and a $G$-equivariant diffeomorphism 
\[\tau: \pi^{-1}(U) \to U \times G, \quad \pi_1 \circ \tau = \pi\]
where $\pi_1: U \times G \to U$ is projection on the first factor.  An
{\em isomorphism} of $G$-bundles $P_1,P_2$ is a $G$-equivariant
diffeomorphism from $P_1$ to $P_2$ that induces the identity on $M$.
Let 
\[\Prin(M,G) = \{ P \to M \}/ \sim \] 
denote the set of isomorphism classes of $G$-bundles over $M$.  Then
$\Prin(M,G)$ is canonically in bijection with $H^1(M,G)$ via the map 
given by gluing: 
\[ [ \psi_{ij} \in C^1(\cU,G) ] \mapsto \sqcup_{U_i \in \cU} (U_i
\times G)/ (u_i,p)\sim (u_j, \psi_{ij}(u)p ) .\]
\item {\rm (Relative $G$-structures)} A {\em relative $G$-structure}
  on a $G/A$-bundle $Q \to M$ trivial on an open cover $\cU$ relative
  to a map $f: M \to N$ is given by a morphism of open covers $\psi:
  \cU \to \cV$ and a relative $G$-structure $(a_1,a_2) \in
  C^1(\psi,G)$ on a cocycle $z \in C^1(\cU,G/A)$ representing $Q$.
  The class
\[b(Q) = [a_2] \in H^2(N,A)\]
is the {\em background class} of the relative $G$-structure on $Q$.
An {\em isomorphism} of relative $G$-structures $(a_1,a_2),
(a_1',a_2')$ is a zero-cochain 
\[w \in C^0(\cU,A), \quad w (a_1,a_2) = (a_1',a_2') .\]
\item {\rm (Relative $\Spin$-structures)} Recall that for $r = 2$
  resp. $r > 2$ the special orthogonal group $SO(r)$ has fundamental
  group $\Z$ resp. $\Z_2$.  The spin group $\Spin(r)$ is the canonical
  double cover of $SO(r)$, isomorphic to its universal cover for $r > 2$:
\[ 1 \to  \Z_2 \to \Spin(r) \to SO(r) \to 1 .\]   
A {\em relative spin structure} on an $SO(r)$-bundle $Q \to M$
relative to a morphism of open covers $\psi$ is a relative
$\Spin(r)$-structure on a cocycle representing $Q$.  The background class in this case is an element $b \in H^2(N,\Z_2)$.  For $f: M \to N$
a smooth map, let
\[H^1(f,\Spin(r),E) = \{ (a_1,a_2) \in C^1(f, \Spin(r)) \ | \ (a_1,a_2) \ \text{represents} \ E \}/ \sim \]
denote the set of isomorphism classes of relative spin structures on
$E$.
\end{enumerate}
\end{definition}  

\begin{remark} {\rm (Relative spin structures as relative trivializations of the second Stiefel-Whitney class)}  
  In concrete terms, a relative spin structure is a lift of the
  transition maps $\psi_{ij} : U_i \cap U_j \to SO(r)$ of the bundle
  $Q$ to a collection
\[\hat{\psi}_{ij}: U_i \cap U_j \to \Spin(r) \] 
of transition maps with values in $\Spin(r)$ that satisfy the cocycle
condition up to the boundary of a $\Z_2$-valued cochain pulled back
from a two-cochain on $N$ 
\[ (\delta_{ijk})_{i,j,k \in I} \in C^2(N,\Z_2) . \]
That is, 
\[ \hat{\psi}_{jk} \hat{\psi}_{ik}^{-1} \hat{\psi}_{ij}
 = f^* \delta_{ijk}, \quad \forall i,j,k \in I .\]
Since 
\[ [ \hat{\psi}_{jk} \hat{\psi}_{ik}^{-1} \hat{\psi}_{ij} ] = w_2(E) \in H^2(M,\Z_2) \]
is the second Stiefel-Whitney class of $E$, this description shows
that there is a natural bijection from $H^1(f,\Spin(r),E)$ to the set
of isomorphism classes of trivializations of the image of the second
Stiefel-Whitney class $w_2(E)$ in $C^2(f,\Z_2)$.
\end{remark} 

\subsection{Relative spin structures in terms of classifying spaces}

There is a definition of relative spin structures, and 
$G$-structures more generally, in terms of classifying spaces as follows.
 Let $EG \to BG$ be a universal
  $G$-bundle for $G$. Let  $[M,BG]$ denote the set of homotopy classes of
  continuous maps to $BG$, canonically in bijection with the set $\Prin(M,G)$ of isomorphism classes of $G$-bundles.
  
\begin{definition} 
\begin{enumerate} 
\item {\rm (Homotopy $G$-structures)}
  Let $A$ be a subgroup of the center of $G$.  Let $\phi$ be a map from
  $M$ to $B(G/A)$. A {\em homotopy $G$-structure} on $\phi$ is a lift
  to $BG$.
\item {\rm (Homotopy relative $G$-structures)} Let $A \subset G$ be a
  central subgroup and $f: M \to N$ a smooth map.  Since $A$ is
  abelian, we have a natural $A$-bundle defined by
\[EA \times_A EA \to BA \times BA .\]  
The classifying map  of this bundle 
\[ m: BA \times BA \to BA, \quad m^*EA \cong (EA \times_A EA) \]
gives $BA$ the structure of an $H$-space; see for example 
Segal \cite[Section 1]{se:cat}.  Let $B^2 A$ denote the
classifying space of $BA$.  Consider a $G/A$-bundle $Q \to M$ with a
classifying map $\phi: M \to B(G/A)$.  A {\em homotopy relative
  $G$-structure} on $\phi$ is a homotopy class of a pair
\[\beta: N \to B^2A, \quad \alpha: M \to f^* \beta^* E(BA)
\times_{BA} BG \]
where $\alpha$ is a section of the $BG$-bundle associated to the
$BA$-bundle pulled back from $\beta$, such that the associated section
of the trivial $B(G/A)$-bundle is the given classifying map $\phi$ for
$Q$.
\end{enumerate} 
\end{definition} 

The space of homotopy relative $G$-structures is in one-to-one
correspondence with \v{C}ech relative versions; this seems to be a
special case of \cite[Theorem 1]{baez:class}.  See also Shahbazi
\cite{shah:rg} for a discussion of relative gerbes in the abelian
case.  

\subsection{Relation between relative and absolute spin structures}

The following proposition connects the definitions above with
that of Fukaya et al \cite{fooo}:

\begin{proposition} \label{prop:realiz}
\begin{enumerate} 
\item Suppose that $Q \to M$ is a $G/A$-bundle and $R \to N$ a
  $G/A$-bundle with characteristic class $ c(Q) = f^* c(R) .$ There is
  a one-to-one correspondence between equivalence classes of relative
  $G$-structures on $Q$ and equivalence classes of $G \times_A
  G$-structures on $Q \oplus f^*R$; here $G \times_A G $ is the quotient of $G \times G$ by the equivalence relation $(g_1a,g_2) \sim (g_1,ag_2), g_1,g_2 \in G, a \in A.$
\item Suppose that $Q \to M$ and $R \to N$ are Euclidean vector
  bundles over manifolds $M,N$ and $f: M \to N$ is a smooth map with
  $w_2(Q) = f^* w_2(R) $.  There is a one-to-one correspondence
  between isomorphism classes of relative spin structures on $Q$ and
  isomorphism classes of spin structures on $Q \oplus f^*R$.
\end{enumerate} 
\end{proposition}

\begin{proof}  For the first statement, let 
\[ (a,b) \in C^1(M,G) \times C^2(N,A) \] 
be a relative $G$-structure on $Q$.  Let $c \in C^1(N,G/A)$ be a cocycle
representing $R$, mapping to $b \in C^2(N,A)$ under the coboundary
map.  By definition, $c$ has a lift 
\[d \in C^1(N,G), \quad \partial d = b + \partial e \ \text{for some}
\ e \in C^1(N,A) .\]
The cochain
\[(a,f^* (d - e)) \in C^1(M,G) \times C^1(M,G) \cong C^1(M,G \times G)\]
has boundary 
\[ \partial (a, f^*(d - e)) = (-f^*b, f^* b) \in C^2(M,A \times A) .\]
The image of $(a, f^* (d - e))$ in $C^1(M,G/A \times G/A)$ represents $Q
\oplus f^*R$.  Conversely, suppose $Q \oplus f^* R$ is equipped with a
$G \times_A G$-structure.  Any lift of the transition maps to $G \times G$ is of the form
\[(a,e) \in C^1(M,G \times G), \quad \partial a = - \partial
e .\]  
and $\partial(ae^{-1})$ takes values in $A$. Thus the pair $(a, \partial e)$ defines a relative $G$-structure on $Q$ with
$b = [\partial(ae^{-1})]$.  This proves the first part of the statement of the
Proposition.  The second statement is the special case $G = \Spin(r),
A =\Z_2$.
\end{proof}

\subsection{Duals, direct sums, and other constructions}

The usual operations of duals, direct sums, and tensor products extend
to the relative spin case: In addition, there is also a canonical
relative spin structure on the ``double'' of any oriented vector
bundle:

\begin{proposition}  \label{prop:double} 
{\rm (Relative spin structures on direct sums and tensor products)}
\begin{enumerate} 
\item  \label{fi}
Let $f: M \to N$ be a smooth map and $E_1,E_2 \to M$ oriented
Euclidean vector bundles equipped with relative spin structures for
the map $f$.  There are canonical relative spin structures on $E_1
\oplus E_2$ and $E_1 \otimes E_2$ for the map $f$.
\item \label{se}
Let $E \to M$ be an oriented vector bundle.  The direct sum $E \oplus
E$ has a canonical orientation and spin structure.
\item \label{th}
Let $f: M \to N$ be a smooth map, $E \to M$ an oriented vector bundle
and $F \to N$ an oriented vector bundle such that $f^*F \cong G \oplus
G$ for some vector bundle $G \to M$.  There is a bijection
between relative spin structures on $E$ for the map $f$ with
background class $b \in H^2(N,\Z_2)$ and relative spin structures on $E$ for the map
$f$ with background class $b + w_2(F)$.
\end{enumerate} 
\end{proposition}  

\begin{proof}  
\eqref{fi} Let $r_1,r_2$ denote the ranks of $E_1,E_2$.  The claim on
the tensor product and direct sum follows from the existence of the
group homomorphisms:
\[ \Spin(r_1) \times \Spin(r_2) \to \Spin(r_1 + r_2), \quad \Spin(r_1)
  \times \Spin(r_2) \to \Spin(r_1 r_2) .\]
\eqref{se} The claim on the self-sum follows from the fact that, if
$r \ge 1$ denotes the rank of the bundle $E$, then the diagonal
homomorphism 
\[ SO(r) \to SO(2r), A \mapsto \diag(A,A) \] 
induces the trivial map on $\pi_0$ and $\pi_1$ and so lifts to $\Spin(2r) .$ \eqref{th} The third item follows by combining the first two: By \eqref{se}, $f^*F \cong G \oplus G$ has
a canonical spin structure.  By \eqref{fi}, relative spin structures
on $E$ with background class $b \in H^2(N,\Z_2)$ are in one-to-one correspondence with
relative spin structures on $E$ with background class $b + w_2(F)$.
\end{proof} 

The relevance of relative spin structures in Floer theory is provided
by the following proposition.  In particular, Proposition \ref{prop:double} implies that relative spin structures for boundaries of surfaces give stable trivializations.

\begin{proposition} \label{prop:spintriv}
{\rm (Stable trivializations via relative spin structures)} Suppose
that $S$ is a compact, oriented surface with boundary $\partial S$,
and $Q \to \partial S$ is an $SO(r)$-bundle.  Any relative spin
structure on $Q$ for the inclusion $\partial S \to S$ induces a 
homotopy class of stable trivializations of $Q$, well-defined up to 
the action of $H^1(S,\Z_2)$ via the restriction map to $H^1(\partial S,\Z_2)$.
\end{proposition}

\begin{proof}   
We first show that any relative spin structure induces a stable
trivialization.  Let $f: \partial S \to S$ be the inclusion of the
boundary.  Since $S$ is two-dimensional, any cohomology class $w \in
H^2(S,\Z_2)$ is the second Stiefel-Whitney class of some oriented
bundle $R \to S$ with $w = w_2(R)$.  
Indeed, the third Postnikov truncation of $BSO$ is the
Eilenberg-MacLane space $K(\Z_2,2)$.  From Proposition \ref{prop:realiz}
(or the homotopy definition) we obtain a bundle $R \to S$ together
with a spin structure on $Q \oplus f^* R$.  We may assume that
$\partial S$ is non-empty, since otherwise, the statement is vacuous.
Thus $S$ is homotopy equivalent to a bouquet of circles
and the bundle $R \to S$, being oriented, is trivial.  The
relative spin structure gives a stable trivialization of $Q$.  If $S$
is a disk, then the trivialization of $R$ (and therefore also the
stable trivialization of $Q$) is unique up to homotopy.  In general,
two stable trivializations differ by a map $S \to SO(r)$ for some $r$
sufficiently large, which is to say, by the action of $H^1(S;\Z_2)$.
\end{proof}

The following lemma will be used later to show that quilted Floer
cohomology is unaffected, in a certain sense, by ``insertion of a
diagonal seam'', see Proposition \ref{prop:delta} below.  We consider the
following situation.  Let $(S,\partial S)$ be a surface with boundary
and let $f: \partial S \to S$ denote the inclusion of the boundary as
above, and $Q \to \partial S, R \to S$ vector bundles. Suppose that
$Q \oplus f^* R$ is equipped with a spin structure, giving rise to a
relative spin structure $\sigma_1$ on $Q$ with background class
$w_2(R) \in H^2(S,\Z_2)$ and a stable trivialization of $Q$.  We consider the effect
of ``shifting'' the relative spin structure on $Q$ by adding another
bundle whose restriction to the boundary is the complexification of a
real bundle.

\begin{lemma} \label{lem:togglesame} Let $Q,R,\sigma_1$ be as above.
  Suppose that $U \to S,V \to \partial S$ are bundles equipped with an
  isomorphism
\[U|_{ \partial S} \to V \oplus V  .\] 
Let $\sigma_2$ be the induced relative spin structure on $Q$ with
background class $w_2(R) + w_2(U)$ as in Proposition
\ref{prop:double}.  Then the stable trivializations of $Q$ induced by
$\sigma_1,\sigma_2$ are related by the sign
$(-1)^{I(U,V)/2}$ where $I(U,V)$ is the total Maslov index of $V$ computed
in a trivialization of $U$ over $S$. 
\end{lemma} 

\begin{proof} 
The pullback $ f^*U = V \oplus V$ is stably trivialized on the
boundary $\partial S$ using the triviality of the diagonal map
$\pi_1(SO(n)) \to \pi_1(SO(2n))$.  On the other hand, the
identification of stable trivializations of $Q \oplus f^* R$ and $Q
\oplus f^*R \oplus f^* U$ uses the trivialization of $U$ over the
disk.  The two trivializations differ exactly if the Maslov index of
the pair $(U,V)$ is equal to $2$ mod $4$.
\end{proof}

\section{Orientations for families of Cauchy-Riemann operators}
\label{sec:families}

In this section, we define orientations for families of quilted Cauchy-Riemann
operators from an orientation and relative spin structure on the
totally real boundary condition, and investigate their behavior under
gluing; here the definition of families allows bubbling.
These results are generalizations of results from Fukaya et
al \cite{fooo}, Seidel \cite{se:bo}, and Solomon \cite{so:lb}.  The
constructions of this section are for families of quilted maps over
smooth bases.  That is, while the determinant line bundle exists for a
family of quilts with varying type over a stratified base, our purpose
here is to construct trivializations for families of a fixed type.  We
then in Section \ref{sec:glueor} compute the gluing signs relating these trivializations for
different types.

\subsection{The pinching-off-disks construction} 

First, we construct orientations on the determinant lines arising from
Cauchy-Riemann operators on surfaces without strip-like ends with
relative spin structures on the boundary conditions. The strategy is to reduce to the case without boundary, where the linearized operator 
can be deformed to one that is complex-linear, by pinching off a disk for each boundary component. 

\begin{proposition}  \label{prop:nostrips} {\rm (Orientations via relative spin structures)} 
Suppose that $S \to B$ is a family of nodal surfaces without
strip-like ends, $(E,F) \to B$ is a family of complex vector bundles
$E\to S$ with oriented totally real boundary conditions $F\subset
E|_{\pd S}$, and $D_{E,F}$ is a family of real Cauchy-Riemann
operators for $(S,E,F)$.  A relative spin structure for the bundle $F
\to \partial S$, if it exists, defines an orientation
\[o_{E,F}: B \to \Or(D_{E,F}) = \det(D_{E,F})^\times/\R_{>0}\]
for the determinant line bundle $\det(D_{E,F}) \to B$.
\end{proposition}    

Here $B$ is a smooth open base, so $\pd S=\bigcup_{b\in B} \pd S_b$ is a bundle
over $B$ whose fibers are the boundaries of the fibers of $S$. 

\begin{proof}[Proof of Proposition \ref{prop:nostrips}.] 
\noindent {\em Step 1: Orientations for families of smooth, closed
  surfaces:} Suppose that $S\to B$ and $(E,F)$ are as in the statement
of the Proposition and $S$ has empty boundary.  Consider a family 
\[D_E = (D_{E,b})_{b \in B}\]
of real Cauchy-Riemann operators acting on sections of a family of
complex vector bundles $E = (E_b \to S)_{b \in B}$.  Since the space
of real Cauchy-Riemann operators is an affine space containing the
complex linear Cauchy-Riemann operators, there exists a homotopy from
\[ (D_{E,b})_{ b\in B} \sim (D'_{E,b})_{b \in B} \]
where $D'_{E,b}$ is a family of complex linear operators.  The complex
structure on the kernels and cokernels of $D_E'$ induce orientations
for $D_E'$.  These orientations pull back under the isomorphism of determinant
lines
$ \det(D_E) \to \det(D_E') $
to orientations of $D_E$.  Any two homotopies are related by a
homotopy of homotopies, since the spaces of real Cauchy-Riemann
operators and complex-linear Cauchy-Riemann operators are
contractible.  So the orientation on $\det(D_E)$ is independent of the
choice of $D_E'$ and homotopy to $D_E'$.

\vskip .1in
\noindent {\em Step 2: Orientations for smooth, compact surfaces with
  boundary:} Suppose that the base of the family $B=\{\pt\}$ is a
point, $S$ is a smooth, compact surface with boundary, and $E,F$ are
as above.  The relative spin structure on $F$ gives an equivalence class
of stable trivializations of $F \to \partial S$ by Proposition
\ref{prop:spintriv}.  We first fix a stable trivialization of $F$ and
construct an orientation for $D_{E,F}$; later we will show that the
orientation depends only on the equivalence class of stable
trivializations.  The real Cauchy-Riemann operator $D_{{E},{F}}$ acts
on sections of $E \to S$ with totally real boundary conditions ${F}
\subset {E} |_{\partial {S}}$.  We may assume, after adding a trivial
bundle, that $F$ is trivialized.  The trivialization ${F}\cong
\R^n\times\pd S$ induces a trivialization
\[{E} |_{\partial {S}}=F\oplus i F\cong \C^n\times\pd S \] 
which extends to a neighborhood $U\subset S$ of $\partial {S}$.
The resulting trivialization is well-defined up to stabilization and
the action of $H_1(S;\Z_2)$ on the space of equivalence classes of trivializations on the boundary.

Deform the surface to a nodal surface by pinching off a disk for each
boundary component, as follows. Choose a tubular neighborhood of the
boundary
\[U=\sqcup_i U_i\subset S\]
equal to a disjoint union of annuli
\[U_i\cong [-1,1]\times S^1, \quad \pd
S\cap U_i \cong\{1\}\times S^1 .\]  
Replacing $U_i$ with complex annuli of increasing radius produces a
family of surfaces.  The limit is the nodal surface
\[\nnS = S / \left( U_i \mapsto (D_i^- \sqcup D_i^+)/(z_i^- \sim
z_i^+), i = 1,\ldots, n \right) \]
obtained by replacing $U_i$ with two disks $D^-_i\sqcup D_i^+$ glued
at an interior node 
\[\{z_i^-,z_i^+\}, \ z_i^-=0\in D_i^-, \ z_i^+=0\in D_i^+ .\]
Here $D_i^+$ is the unit disk with standard complex structure $j_{\rm
  std}$ and boundary $\pd D_i^+$ identified with $\{1\}\times
S^1\subset\pd U_i$, whereas $D_i^-$ is the unit disk with complex
structure $-j_{\rm std}$ and boundary $\pd D_i^-$ identified with
$\{-1\}\times S^1\subset\pd U_i$.  So the nodal surface $\nnS$ has
resolution $\nnS^\rho=\iS\sqcup\oS$, consisting of a closed surface
$\iS$ and a union of disks $\oS$:
\[ \iS=(S\setminus U) \sqcup \bigsqcup_i D_i^-, \quad  \oS=\bigsqcup_i D_i^+ .\] 
The identifications needed to produce $\nnS$ from $\nnS^{\rho}$ are a
collection of interior nodes 
\[Z=\{\{z_i^-,z_i^+\}\}, \quad z_i^-\in\iS, \quad z_i^+\in\oS ,\]
see Figure \ref{fig:diskpinch}.
\begin{figure}
\includegraphics[height=2in]{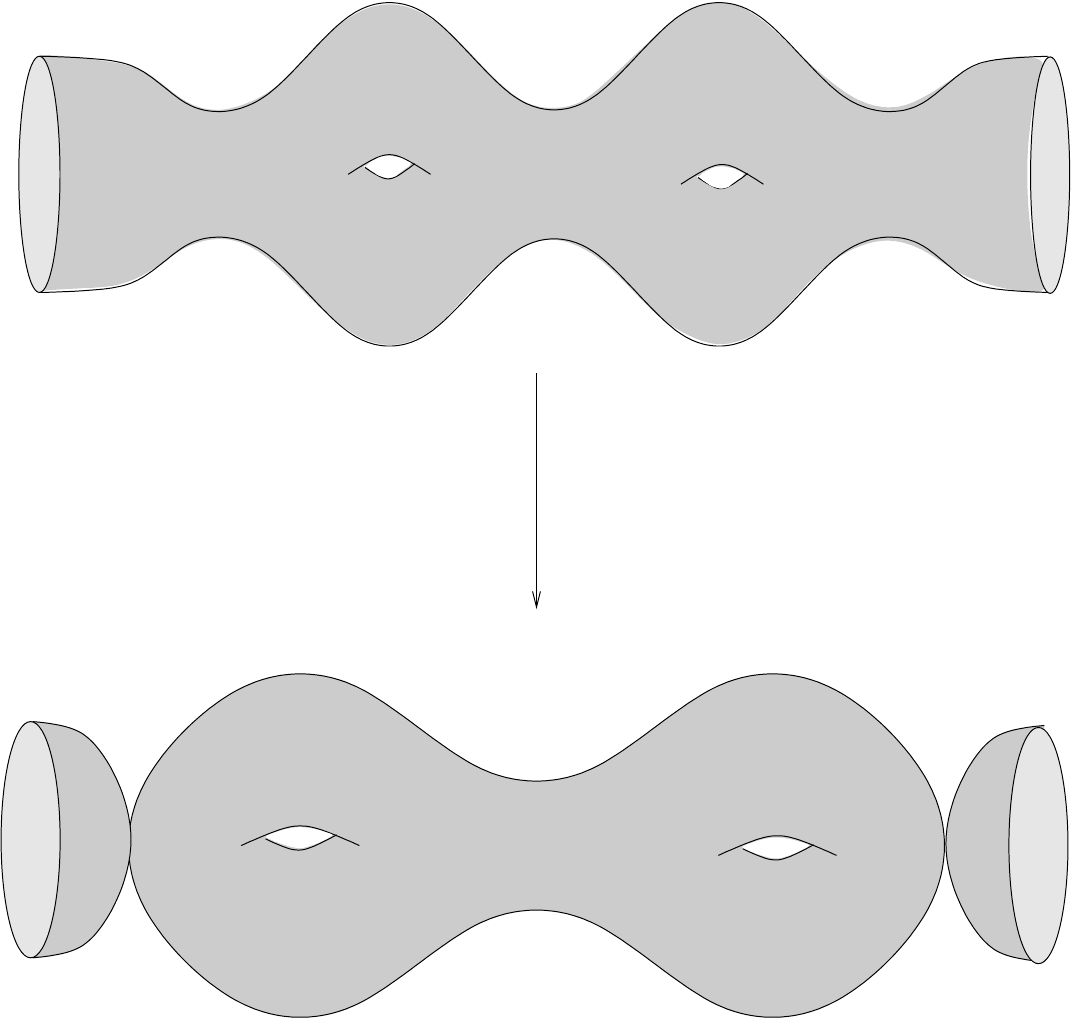} 
\caption{Pinching off a set of disks}
\label{fig:diskpinch}
\end{figure} 

The pinching construction extends to define a complex vector bundle
and totally real boundary condition as follows: Let $\iE\to\iS$ be the
complex vector bundle defined by gluing together $E|_{S\setminus U}$
(which is trivialized $\cong \C^n\times\sqcup_i \pd D_i^-$ on the
boundary) with the trivial bundle on $\sqcup_i D_i^-$.  Consider the
trivial bundles
\[\oE= \C^n\times
\sqcup_i D_i^+, \quad
\oE|_{\pd\oS} \supset 
\rhF = \R^n\times \sqcup_i \pd D_i^+\cong F .\]  
Then $\nnE\to\nnS$ is obtained from $\rhE:=\iE\sqcup\oE\to\rhS$ and
identification at the nodes $Z$. Similarly, the boundary condition
$\nnF=\oF$ is induced from $\rhF$.  Conversely, $(S,E,F)$ is obtained
from $(\nnS,\nnE,\nnF)$ by gluing at the interior nodes.

We use the canonical identification of determinant lines produced by
the homotopy above to produce an orientation on the determinant line
of the original surface.  By Corollary \ref{cor:detiso2} the pinching of bundles
induces an isomorphism of determinant lines
\[\det(D_{\nnE,\nnF}) \to \det(D_{E,F})  .\]
Combining this isomorphism with \eqref{eq:unred}, \eqref{eq:unredred}, and \eqref{eq:nodal} we
obtain an isomorphism
\begin{equation}\label{eq:doit}
\det(D_{E,F}) \to \Lambda^{\max} \Bigl( \bigoplus_i \rhE_{z_i^+} \Bigr)^\dual 
\otimes \det(D_{{\rhE},\rhF}) .
\end{equation}
Here the first factor is oriented by the complex structure on
$\rhE_{z_i^+}$.
The second factor decomposes into
\[\det(D_{{\rhE},\rhF}) = \det(D_{\iE} \oplus D_{\oE,\oF}) 
\cong \det(D_{\iE}) \otimes \det(D_{\oE,\oF}) .\]
The operator $D_{\iE}$ has an orientation given by Step 1,
since $\iS$ is smooth and closed.  On the other hand, by construction
the operator $D_{\oE,\oF}$ is the direct sum of real Cauchy-Riemann
operators on the disk.  After a homotopy, the operators on the disks
are the standard Cauchy-Riemann operators which are surjective.  Their
kernel is isomorphic to a sum of fibers via evaluation at points
$\ul{s} = (s_i\in\pd D^+_i\subset\pd S)$ on the boundary:
\[ \ker D_{\oE,\oF} \to \oplus_i F_{s_i} , \quad \ul{\xi} \mapsto
\ul{\xi}(\ul{s}) .\]
The orientation of the boundary condition $F$ (induced by the
trivialization) thus defines an orientation on $D_{\oE,\oF}$.  The
orientation on $D_{E,F}$ is induced from the isomorphism \eqref{eq:doit}.

We claim that the orientation is independent of the auxiliary choices:
the trivializations of $F$ and $E$ and the extension of the induced trivialization
of $E$ to the neighborhood $U$, and the choice of coordinates on $U$. 
First of all, consider the dependence on the stable trivialization of $F$ on the boundary.    The change in equivalence class of trivializations induces a change
of sign of the resulting determinant line by $-1$, by \cite[Proposition 8.1.16]{fooo}, independent of the rank of the bundle.   Since
the image of $H_1(S,\Z_2)$ in $H_1(\partial S;\Z_2) \cong \Z_2^{\pi_0(\partial S)}$ consists of elements with even sum, the overall orientation induced by the action of $H_1(S;\Z_2)$ is unchanged.  

Any two choices of extensions and coordinates on $U$ are homotopic.
Any two trivializations of ${F}\to\pd S$ differ by a map
\[ \partial {S} \to SO(\rank({F})) .\]  
Hence there are 
\begin{itemize}
    \item two trivializations up to homotopy for each boundary
component if $\rank({F}) > 2$, 
\item infinitely many trivializations if $\rank({F}) = 2$,
and 
\item a unique trivialization if $\rank({F}) =1$. 
\end{itemize}
In particular, there are two
stable homotopy classes of stable trivializations of $F$, for any
rank.  Consider two choices of extensions and coordinates, and an 
isotopy of trivializations
\[ \tau_\delta: F \to \partial S \times \R^k, \quad \delta \in [0,1] \]
of $F$.  The map $\tau_\delta$ gives continuous families of
nodal surfaces and bundles $\nnS_\delta,\nnE_\delta,\nnF_\delta$,
Cauchy-Riemann operators $D_{\nnE_\delta,\nnF_\delta}$, and
isomorphisms
\[\det(D_{E,F}) \to
\det(D_{\nnE_\delta,\nnF_\delta}), \quad \delta\in[0,1] .\]  
The construction fixes an orientation for each
$\det(D_{\nnE_\delta,\nnF_\delta})$ from the orientations for the
nodal fibers $(\hat{E}^\rho_\delta)_{z_i^+(\delta)}$, the operators
$D_{(\iE)_\delta}$ on complex bundles over closed surfaces, and the
operators $D_{(\oE)_\delta,(\oF)_\delta}$ on trivial bundles over
disks.  Each of these orientations is continuous in families.  Hence
the orientations on $\det(D_{\nnE_\delta,\nnF_\delta})$ vary
continuously in $\delta$.  It follows that the map
$\det(D_{\nnE_0,\nnF_0})\to\det(D_{\nnE_1,\nnF_1})$ induced by the
homotopy of operators $(D_{\nnE_\delta,\nnF_\delta})_{\delta \in[0,1]}$
preserves the given orientations.  The composition of this map with
$\det(D_{E,F}) \to \det(D_{\nnE_0,\nnF_0})$ is homotopic to
$\det(D_{E,F}) \to \det(D_{\nnE_1,\nnF_1})$.  Hence the two
isomorphisms induce the same orientation on $\det(D_{E,F})$.

We show that trivializations of the boundary condition that
are homotopic after stabilization also define the same orientation on
the determinant line of the Cauchy-Riemann operator.  For $\rank({F})
> 2$ there is nothing to show, since the trivializations are homotopic
if and only if they are stably homotopic.  Let $F_{{\tau}}$ be the trivial
$\R^k$-bundle over $\pd S$, and $E_{{\tau}}$ the trivial $\C^k$-bundle
over $S$.  Consider two trivializations
\[ \tau_i: F \to \R^k \times  \partial S, \quad  i \in \{ 0 ,1 \} \]
of $F$ such that the induced trivializations of 
\[ \widehat{\tau}_i : F_{{\tau}}\oplus F \to \R^{2k } \times \partial
S \]
are homotopic.  Being homotopic, the trivializations
$\widehat{\tau}_i$ define the same orientation 
$o_{{E_{{\tau}}\oplus
    E}, F_{{\tau}} \oplus F}$ for
\[D_{{E_{{\tau}}\oplus E}, F_{{\tau}} \oplus F}
:=
D_{E_{{\tau}},F_{{\tau}}} \oplus D_{{E},{F}} ,\] 
where $D_{E_{{\tau}},F_{{\tau}}}$ is the standard Cauchy-Riemann
operator.  On the other hand, applying the direct sum isomorphism
(\ref{eq:caniso}) provides an orientation $o_{E_{\tau},F_\tau}$ of
\[\det(D_{E_{{\tau}},F_{{\tau}}}) \otimes
\det(D_{{E},{F}})\cong\det(D_{{E_{{\tau}}\oplus E}, F_{{\tau}} \oplus
  F}) .\]  
The orientation $o_{E,F,i}$ induced by $\tau_i$ for $i \in \{ 0, 1\}$
is related to $o_{{E_{{\tau}}\oplus E}, F_{{\tau}} \oplus F}$ by a
universal sign that only depends on the combinatorics of the surface
and the ranks of the bundles.  Hence $o_{E,F,0} = o_{E,F,1}$ as
claimed.

\vskip .1in {\noindent \em Step 3: Orientations for families of
  smooth, compact surfaces with boundary:} We now consider the case of
families.  Let $D_{E,F} = (D_{E,F}^b)_{b \in B}$ be a family of
Cauchy-Riemann operators.  It suffices to show that the orientations
constructed above vary continuously in $B$.  For this it suffices to
consider a family $S \to B$ of smooth surfaces with $B$ contractible.  A
trivialization of $F\to \pd S$ induces a trivialization of ${E}$ near
the boundary $\partial {S}$:
\[ \tau:  E |_{\partial S} \cong F \oplus iF \to \R^{2k} \times \partial S .\]   
Choose a deformation of the conformal structure on $S \to B$ as in Step 2
to produce a family of nodal surfaces $\nnS\to B$.  The family $\nnS$
consists of a family of disks $\oS\to B$, a family of closed surfaces
$\iS\to B$ (obtained by gluing a disk bundle to $\pd S$), and
identifications of $\oS$ and $\iS$ at families of interior nodes.
This deformation provides an isomorphism of determinant line bundles
\[
\det(D_{E,F}) \to \det(D_{\nnE,\nnF}) \to 
\Lambda^{\max} \Bigl( \bigoplus_i \rhE_{z_i^+} \Bigr)^\dual 
\otimes \det(D_{{\iE}}) \otimes \det(D_{{\oE},{\oF}}).
\] 
This isomorphism defines the orientation on $\det(D_{E,F})$ by
pullback from the right hand side.  To see that these orientations
vary continuously, note that the orientation on the first factor is
induced from the complex structure on $\rhE_{z_i^+}\to B$.  On the
second factor the orientation is given by the construction
for families of closed surfaces in Step 1.  The third factor is isomorphic
(using a homotopy to the standard Cauchy-Riemann operator on disks) to
$\Lambda^{\max}(\oplus_i F_{z_i})$ for a smooth family of boundary
points $z_i\subset\pd S$ in each connected component. These fibers of
$F\to\pd S$ are oriented by assumption, inducing a continuous
orientation on $\det(D_{{\oE},{\oF}})$ and hence on
$\det(D_{{E},{F}})$.

\vskip .1in
\noindent {\em Step 4: General definition of orientations:} Finally,
we consider a general family of nodal (but compact) surfaces.  Let
$\nS\to B$ be such a family equipped with families of complex vector
bundles $\nE\to\nS$ and totally real boundary conditions $\nF$, and a
family of real Cauchy-Riemann operators $D_{\nE,\nF}$.  By assumption
the family of operators $D_{\nE,\nF}$ is produced from identifications
of families of nodes from a family of real Cauchy-Riemann operators
$D_{\rE,\rF}$ for families of bundles $\rE\to\rS$ and $\rF\to\pd\rS$
over the family of smooth resolutions $\rS\to B$.  We fix a
trivialization of $\nF$, that is a trivialization of $\rF\to\pd\rS$
that is compatible with the identifications at nodes.  From
\eqref{eq:unred}, \eqref{eq:unredred}, and \eqref{eq:nodal} we have a bundle
isomorphism
\begin{equation} \label{eq:bundleiso}
\det(D_{\nE,\nF}) \to \Lambda^{\max} \Bigl( \bigoplus_i
\rE_{z_i^+} \oplus \bigoplus_j \rF_{w_j^+}\Bigr)^\dual \otimes
\det(D_{{\rE},\rF}) \end{equation}
Here an orientation on $D_{\rE,\rF}$ is defined by Step 2,
the complex fibers of $E$ are naturally oriented, and the fibers of
$F$ are oriented by assumption.  Hence this isomorphism defines
orientations on $D_{\nE,\nF}$.
\end{proof} 

\begin{remark} {\rm (Orientation of the trivial operator)} 
 Suppose that $S$ is a disk, $(E,F)$ are trivial and $D_{E,F}$ is a
 trivial Cauchy-Riemann operator.  In this case the constructed
 orientation on $\det(D_{E,F})$ is isomorphic to the given orientation
 on $\Lambda^{\max}(F)$, via the identification $\ker(D_{E,F}) \to F_z$ for any
 point $z \in \partial S$.  Indeed, in this case the Maslov index is
 already zero.
\end{remark} 

We now investigate the behavior of orientations for surfaces without strip-like ends with respect to basic
operations:

\vskip .1in 

\begin{remark} \label{rem:crem}
\begin{enumerate} 
\item {\rm (Conjugates)} \label{item:cs} Let $(E,F)$ be a bundle with totally real
  boundary condition, and suppose that $F$ is equipped with a relative
  spin structure.  Let $E^-$ be the complex conjugate of $E$, and
  $F^-$ the subspace $F$ considered as a totally real subspace of
  $E^-$.  Let $S^-$ denote the surface $S$ with complex structure
  $\ol{j} = - j$.  Given a Cauchy-Riemann operator $D_{E,F}$ let
  $D_{E^-,F^-}$ denote the same operator on the conjugate spaces
  (minus complex structures), as in Section \ref{sec:cr}.  The kernel and
  cokernel of $D_{E,F}$ are canonically identified with those of
  $D_{E^-,F^-}$ as real vector spaces.  Hence $\det(D_{E,F})$ is
  canonically identified with $\det(D_{E^-,F^-})$.  However,
  the complex structures on the kernel and cokernel of $D_{E,F}$ are
  reversed.  It follows that the orientations are related by
\begin{equation} \label{eq:dualo}
 o_{E^-,F^-} = (-1)^{(\Ind(D_{E,F}) - \on{tot-rank}(F))/2} o_{E,F} \end{equation}
where $\on{tot-rank}(F)$ is the sum of the ranks of the restrictions of $F$ to the boundary components. 
\item  \label{item:sums} {\rm (Direct Sums)} Let $(E_j,F_j), j = 0,1$ be bundles with
  real boundary conditions over a compact surface with boundary $S$,
  and $(E,F) = (E_0,F_0) \oplus (E_1,F_1)$.  The isomorphisms
\[ \ker D_{E_0,F_0} \oplus \ker D_{E_1,F_1} \to \ker D_{E,F}, \quad
\coker D_{E_0,F_0} \oplus \coker D_{E_1,F_1} \to \coker D_{E,F} \]
induce an isomorphism
\begin{equation} \label{eq:Diso}
p: \det(D_{E_0,F_0}) \otimes \det(D_{E_1,F_1}) \to \det(D_{E,F}) .\end{equation}
By definition, the isomorphism \eqref{eq:caniso} is continuous in
families. For each $ j = 0,1$ the orientation $o_{E_j,F_j}$ is defined
via an isomorphism
\[ \det(D_{E_j,F_j}) \to \det(D_{E_j'}) \otimes
\Lambda^{\max}(E'_{j,z})^{-1} \otimes \Lambda^{\max}(F_{j,w}) \]
where $z$ is the node in the deformed surface and $w$ a point in the
boundary of the deformed surface.  Since the operators $D_{E_j'}$ have
complex linear kernel and cokernel, their indices are even
dimensional.  Similarly, the fiber at the node has even dimension.  It
follows that the $\Lambda^{\max}(F_{j,w})$ factor commutes with the other
factors, in the sense that the sign in Remark \ref{rem:sums} is positive, and 
\[o_{E,F} = p (o_{E_0,F_0} \otimes o_{E_1,F_1}) \] 
is the image of $ o_{E_0,F_0} \otimes o_{E_1,F_1}$ under \eqref{eq:Diso}.
\item \label{item:unions} {\rm (Disjoint Unions)} Let $(E_j,F_j)$ denote bundles with
  totally real boundary condition over surfaces $S_j$ for $j = 0,1$.
  Then 
\[(E,F) = (E_0,F_0) \sqcup (E_1,F_1)\] 
is a bundle with totally real boundary condition over $S = S_0 \cup
S_1$.  We take $o_{E,F}$ to be the image of $o_{E_0,F_0} \otimes
o_{E_1,F_1}$ under the canonical isomorphism, by a special case of 
the direct sum construction in \eqref{item:sums}.
\end{enumerate}
\end{remark} 

\subsection{The pinching-off-caps construction}

  Next, we construct orientations for Cauchy-Riemann operators 
  with strip-like ends, by pinching off a cup or cap (once-punctured disk) for each strip-like end.   Given orientations on the Cauchy-Riemann operators on the cups and caps, we obtain an orientation 
  on the Cauchy-Riemann operator on the surface with strip-like ends by deformation. These auxiliary choices are similar to the choice of 
  orientation on the stable and unstable manifolds in the Morse case.

\begin{definition} \label{def:endorient}
\label{def:enddatum}
\begin{enumerate} 
\item {\rm (End Datum)} An {\em end datum} is a tuple
  $(E,F_-,F_+,\cH)$ consisting of
\begin{enumerate}
\item  a finite-dimensional complex vector
  space $E$, 
\item a pair $(F_-,F_+)$ of transverse, oriented, totally real
  subspaces of half-dimension, equipped with spin structures, and
\item $\cH$ a normal form for a Cauchy-Riemann operator on the strip
  as in \eqref{eq:asymlimit}.
\end{enumerate}
\item {\rm (Orientation for an end datum)} An {\em {orientation}} for
  an end datum $(E,F_-,F_+,\cH)$ consists of
\begin{enumerate} 
\item a smooth path 
\[\Gamma:\R \to \on{Real}(E), 
\quad \Gamma(\pm\infty)=F_\pm \]
of totally real subspaces connecting $F_\pm$.  We view $\Gamma$ as a
totally real boundary condition $\Gamma\subset E\times\pd S_1$ for the
trivial bundle $E\times S_1$;
\item a real Cauchy-Riemann operator 
\[D_\Gamma: \Omega^0(E,\Gamma) \to \Omega^{0,1}(E) \] 
on $S_1$ for sections with values in the trivial bundle $E$ and
boundary values in $\Gamma$, with asymptotic limit $\lim_{s \to
  \infty} \eps^* D_\Gamma$ given by $\cH$ in the sense of
\eqref{eq:asymlimit};
\item an orientation for $D_\Gamma$;
\item  a spin structure on $\Gamma$, extending the given 
spin structures on $F_-,F_+$.
\end{enumerate}
\end{enumerate} 
\end{definition}  

\begin{remark} 
\label{rem:enddatumsums}
\begin{enumerate}
\item 
{\rm (Conjugates)} \label{item:conjend} Let $(E,F_-,F_+,\cH)$ be an end datum equipped with
an orientation.  The {\em dual} end datum $(E^-,F_+^-,F_-^-,\cH^-)$
has a canonical orientation given by the same path on $\partial S_1$
with the complex structure (hence direction on the boundary) of $S_1$
reversed, the same Cauchy-Riemann operator $D_\Gamma$, (now complex
linear with respect to the reversed complex structures on the domain
and codomain), the given orientation on $D_\Gamma$, and the given spin
structure on $\Gamma$.

\item 
{\rm (Direct Sums)} \label{item:sumend} Let $(E_j,F_{j,0},F_{j,1},\cH_j)$ be end data
equipped with orientations for $j \in \{ 0,1 \}$.  The direct sum
\[E = E_0 \oplus E_1,  \quad F_k = F_{0,k} \oplus F_{1,k}, \quad \cH = \cH_0
  \oplus \cH_1\] 
has an orientation given by the direct sum of the paths and
Cauchy-Riemann operators, the orientation given by the direct sum
isomorphism \eqref{eq:caniso}, and the direct sum spin structure.
\end{enumerate}
\end{remark} 

\begin{definition} {\rm (End orientations)} 
Let $S$ be a surface with strip-like ends and $\ol{S}$ the surface
obtained by adding the points at infinity:
\[ \ol{S} = S \cup \{ z_e, e \in \mE(S) \} .\]
Let $E \to S$ be a complex vector bundle and $F\subset E|_{\pd S}$
totally real boundary conditions.  For each end $e\in\cE(S)$, the
corresponding point at infinity is denoted $z_e \in \ol{S}$, and the two real
boundary conditions meeting it are denoted $F_{e}, F_{e-1}$.  By assumption,
these are constant transverse subspaces, so that $F_{e,0}\oplus
F_{e,1}= E_e$ over the strip-like end.  We suppose we have chosen
relative spin structures on the fibers $F_{e,0},F_{e,1}$ at infinity,
and there are asymptotic limits $\cH = (\cH_e, e \in \cE(S))$.
Let $S$ be a surface with strip-like ends equipped with a bundle
with totally real boundary condition $(E,F)$ and a relative spin structure
on $F$.  An {\em {orientation}} for this datum is an orientation
for each end
\[(\Gamma_e,D_e,o_e, \Spin(\Gamma_e)), \quad e\in\cE(S) .\]
\end{definition}

\begin{definition} \label{def:compat} {\rm (Compatibility of end orientations)}   Let $S$ be a surface with strip-like ends
$\mE = \mE_- \cup \mE_+$ and $o_e, e\in \mE(S)$ a set of orientations
  on the ends $D_{e^\pm}$.  Let $D_{e^\pm}$ denote Cauchy-Riemann
  operators on the caps $S_{e^\pm}$ added to the outgoing and incoming
  ends in \eqref{eq:glueEF}.  Gluing together the caps $S_{e^\pm}$
  produces a surface
\[\ti{S}_e = S_{e^-} \# S_{e^+} \]
diffeomorphic to the disk with a boundary value problem $(E_e,F_e)$ with zero Maslov index, canonically homotopic to the trivial boundary value problem $\ti{S}_e \times (\C^n, \R^n)$.  By the previous
constructions, the Cauchy-Riemann operator $\tilde{D}_e$ on $\ti{S}_e$
obtained by gluing from $D_e$ is equipped with an orientation via the orientation on $F_e$.  We say
that the orientations $o_e, e \in \mE_\pm$ are {\em compatible} if the
gluing isomorphism
\begin{equation} \label{eq:glueDe}
\det(D_{e^+}) \otimes \det(D_{e^-}) \to \det(\tilde{D}_e)
\end{equation}
is orientation-preserving.\footnote{The analogous choice in Morse theory would be choices of orientations on the stable and unstable manifolds such that the induced orientation at the critical points is the given one on the manifold.}
\end{definition} 

Compatibly chosen orientations are compatible with the basic
operations on Cauchy-Riemann operators: 

\begin{remark} \label{rem:compatops}
\begin{enumerate} 
\item {\rm (Conjugates)} \label{item:conjcompat} Suppose that $(E,F)$ is a pair of bundles for a
  surface with strip-like  ends $S$, and $E^-,F^-$ denote the conjugate
  bundles over the conjugate surface $S^-$.  Suppose that a set of
  orientations $o_e, e \in \mE(S)$ for the operators $D_{e^\pm}$ have
  been chosen compatibly.  Then the isomorphisms $\det(D_{e^\pm}) \to
  \det(D^-_{e^\pm})$ induce a collection of orientations 
\[o^-_e \in \det(D^-_{e^\pm}) \] 
for $E^-,F^-$.  Since complex conjugation induces an isomorphism of
orientations for disks with Maslov index zero, the orientations
$o^-_e$ are also compatible.
\item {\rm (Direct Sums)} \label{item:sumcompat}  Suppose that $E_j,F_j$ are bundles
over a surface with strip-like ends $S$, and let $o_{e,j}$
be compatible orientations for the operators $D_{e^\pm,j}$
for $E_j,F_j$ at the end $e \in \mE(S)$.  Let $E,F$ denote the direct sum bundles and $D_{e^\pm}$ the direct sum Cauchy-Riemann operators, and give $D_{e^-}$ the orientation induced by the isomorphism \eqref{eq:caniso}
\[ \det(D_{e^-,0}) \otimes \det(D_{e^-,1}) \to \det(D_{e^-}).\]
Under the isomorphisms \eqref{eq:caniso}, the gluing isomorphism
$\det(D_{e^+}) \otimes \det(D_{e^-}) \to \det(\ti{D}_e)$ corresponds
to the exchange of the factors $\det(D_{e^+,1})$ and
$\det(D_{e^-,0})$ followed by the tensor product of the gluing
isomorphisms for $j = 0,1$, and $\ti{D}_e$ carries the direct sum
orientation by Remark \ref{rem:crem} \eqref{item:sums}.  By
\eqref{eq:signs}, the orientations on $D_{e^\pm}$ are therefore
compatible if $D_{e^+}$ is given the orientation induced by
\eqref{eq:caniso} multiplied by
$(-1)^{\Ind(D_{e^+,1}) \Ind(D_{e^-,0})}$.  For $k$ summands the
corresponding sign is given by 
$\prod_{i<j} (-1)^{\Ind(D_{e^-,i}) \Ind(D_{e^+,j})}$.
\end{enumerate}
\end{remark} 

Orientations for the ends and a relative spin structure suffice to
induce orientations on the determinant line of a family of
Cauchy-Riemann operators:

\begin{proposition}  \label{prop:disktriv} 
{\rm (Orientations for determinant line bundles via relative spin
  structures and end orientations)} Let $S \to B$ be a family of nodal
surfaces with strip-like ends, and $E \to B$ a family of complex
vector bundles with totally real boundary conditions $F \to B$.  A
choice of a relative spin structure on $F$, if it exists, and
{orientation}s for the ends of $(S,E,F, \cH)$ induce an orientation of
the determinant line bundle $\det(D_{E,F}) \to B$.
\end{proposition} 

\begin{proof}  First, consider the case that the boundary of 
$S$ is connected.  A deformation of $S$ is obtained by ``bubbling
  off'' disks with strip-like ends for each strip-like end, see Figure
  \ref{fig:closed} below, using the path of totally real subspaces
  specified by the end datum.  Given the orientations for the ends and
  the behavior of determinant lines under deformation, the problem reduces
  the construction of orientations to the case without strip-like
  ends.  Namely, on each strip-like end $e$ consider the deformation
  $F_{e,\pm,\delta}$ of the boundary conditions $F_{e,\pm}$ on a
  neighborhood of infinity to the boundary condition formed by
  concatenating the restriction of $\Gamma$ with $\Gamma^{-1}$:
\[ F_{e,\pm,0} = F_{e,\pm}, \quad F_{e,\pm,1} = \Gamma \# \Gamma^{-1} .\]  
The sub-bundle $\Gamma \# \Gamma^{-1}$ has a canonical deformation to
the boundary condition with constant value $\Gamma(\infty)$.  The
resulting boundary value problem is obtained by deformation of the
nodes of a nodal surface $\nS$ with vector bundles
\[ (\nE,\nF) = \coprod_e (E_e,F_{e,0} \sqcup F_{e,1}) \# (E^c,F^c) \]
obtained by gluing together the problems $(E_e,F_{e,0} \sqcup
F_{e,1})$ and a problem $(E^c,F^c)$ on a closed (possibly nodal)
surface with bundles obtained by gluing $(D_e,E_e,F_{e,0},F_{e,1})$
onto $(S,E,F)$ for each end $e$.  The procedure is illustrated in
Figure \ref{fig:closed}.  The nodal surface $\hat{S}$ has a canonical
order of patches given by taking the ordering of the additional
patches and the boundary nodes so that the original component $\ol{S}$
is ordered first, and the nodes $\ul{w}_e, e \in \cE(S)$ are ordered
in the ordering of the strip-like ends $e \in \cE(S)$.  Let
$(\hat{E},\hat{F})$ denote the vector bundles on the nodal surface.
\begin{figure}[ht]
\includegraphics[width=4in,height=1in]{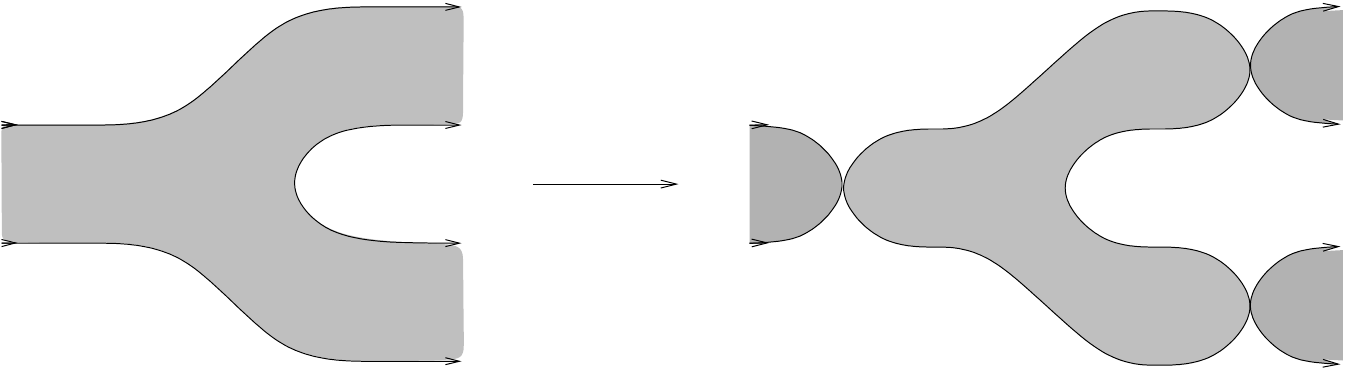}
\caption{Bubbling off the strip-like ends}
\label{fig:closed}
\end{figure}
The equation \eqref{eq:linearglue} gives an isomorphism of determinant
lines
\begin{equation} \label{eq:glueEF}
 \det(D) \to \det(\hat{D}) .\end{equation}
The equation \eqref{eq:nodal} gives an identification
\begin{multline} \label{eq:order} \det(\hat{D}) \to 
\bigotimes_{e \in \mE_-} 
\left(
 \det(D_e^-) \otimes 
\Lambda^{\max}(\Gamma_e(0)^\dual) 
 \right) \otimes \det(D^c) \\
 \otimes 
\left( \bigotimes_{e \in \mE_+} 
\Lambda^{\max}(\Gamma_e(0)^\dual) \right) 
\otimes 
\left( \bigotimes_{e \in \mE_+} 
 \det(D_e^+) \right) 
\end{multline}
where $\Gamma_e(0)$ is the fiber given by evaluating the 
path $\Gamma_e$ at $0$.    The change in order of determinant lines from \eqref{eq:nodal} induces a Koszul sign described in \eqref{eq:signs}.
This choice of order means that when an outgoing end of one surface is
glued to an incoming end of another, we obtain the determinant lines
of the outgoing cap and of the incoming cap, in this order, and a
determinant line of a dualized fiber
\[ \det(D_e^+) \otimes \det(D_e^-) \otimes
\Lambda^{\max}(\Gamma_e(0)^\dual) .
\] 
Each of the factors is canonically trivial by the construction of end
orientations, leaving a product of determinant lines of dualized
fibers $\Lambda^{\max}(\Gamma_e(0)^\dual)$ and the determinant lines
$\det(D^c)$ for the surfaces without strip-like ends.  
The
relative spin structure on $F$ and the bundles $\Gamma_e$ define a
relative spin structure on $F^c$.  This construction gives an
orientation on the determinant line by Proposition \ref{prop:nostrips}.

In the case that $S$ is disconnected, we define the orientation on $S$
as a product of expressions in the right hand side of \eqref{eq:order}
for each component $S_i$.  \end{proof}

We investigate the behavior of the constructed orientations under
various operations.  The behavior of orientations under duals and
direct sum is the same as that in the case of no strip-like ends.

\begin{remark} \label{rem:permuteremark}
\begin{enumerate}  
\item {\rm (Conjugates)} \label{item:conjlin} 
Suppose that $S$ is equipped with a bundle
  with totally real boundary conditions $(E,F)$, a spin structure
  $\cH$ on $F$, and orientations for the ends.  Let $S^-, E^-, F^-$
  denote the complex-conjugate surface and bundles, and suppose 
  the bundles have been equipped with the orientations given by the dual
  construction in Remark \ref{rem:duals}.  Let 
\[D_{E^-,F^-}: \Omega^0(E^-,F^-) \to \Omega^{0,1}(E^-) \]
be the dual Cauchy-Riemann operator as in \ref{rem:crem} \eqref{item:cs}.  The
determinant lines $D_{E,F}$ and $D_{E^-,F^-}$ are oriented by \eqref{eq:order}
using gluing on the orientations for the ends and bubbling off the
boundary components on disks with Maslov index zero.  Complex
conjugation acts on the resulting products \eqref{eq:order} of
determinant lines 
\[ \bigotimes_{e \in \mE_-} 
\left(
 \det(D_e^-) \otimes 
\Lambda^{\max}(\Gamma_e(0)^\dual)  
 \right) \otimes \det(D^c)  
 \otimes  
\left( \bigotimes_{e \in \mE_+} 
\Lambda^{\max}(\Gamma_e(0)^\dual) \right)  
\otimes  
\left( \bigotimes_{e \in \mE_+} 
 \det(D_e^+) \right)  
\] 
preserving the orientation on the determinant lines
of the disks and the orientations on the ends.  However it reverses
the complex structure on the bundles $(E^c,F^c)$ on the closed
surfaces and interior nodes.  Thus the total sign change is
\[  \det(D_{E^-,F^-}) \cong \det(D_{E,F}) 
^{(-1)^{ (\Ind(D^c) - \# \pi_0(\partial S) \rank_\C(E))/2}} \]
in agreement with \eqref{eq:dualo}.
\item {\rm (Re-ordering components or ends)}
\label{item:permute} 
\begin{enumerate}
\item Suppose that $S'$ is a nodal surface obtained by re-ordering a
  boundary node:
\[ S' = S/ ((w_+,w_-) \mapsto (w_-,w_+)) .\]  
Let $D'_{E,F}$ be the Cauchy-Riemann operator obtained from $D_{E,F}$.
The isomorphism $\det(D_{E,F}) \to \det(D'_{E,F})$ of determinant
lines induced by the isomorphism of kernel and cokernel acts on
orientations by $(-1)^{\rank(F)}$.
\item {\rm (Transposing patches)}  \label{item:transpose}  Suppose that $S'$ is a nodal surface obtained by transposing two
  patches $S_i,S_j$ of $S$.  The isomorphism $\det(D_{E,F}) \to
  \det(D'_{E,F})$ of determinant lines induced by the isomorphism of
  kernel and cokernel acts on orientations by $(-1)^{\Ind(D_{E_i,F_i})
  \Ind(D_{E_j,F_j})}$.
\item Suppose that $S'$ is a nodal surface obtained by re-ordering the
  boundary components (resp. boundary nodes) by a permutation
  $\sigma$ of the set of nodes:
\[ \sigma: \{ w_1,\ldots, w_m \} \to \{ w_1,\ldots, w_m \}  .\]  
The isomorphism $\det(D_{E,F}) \to \det(D'_{E,F})$ of determinant
lines induced by the isomorphism of kernel and cokernel acts on
orientations by $\det(\sigma)^{\rank(F)}$.
\item Suppose that $S'$ is obtained by transposing two strip-like ends 
$e,e'$ that are consecutive in $\E_-(S)$ or in $\E_+(S)$, and let 
$n_e = \rank_\C(E_e)$.  The isomorphism $\det(D_{E,F}) \to 
\det(D'_{E,F})$ acts on orientations by 
\[ (-1)^{(\Ind(D_e^-) - n_e)(\Ind(D_{e'}^-) - n_{e'})} \ \text{resp.} \  (-1)^{n_e n_{e'} + \Ind(D_e^+) \Ind(D_{e'}^+)} \]
 for incoming resp.  outgoing ends.
\end{enumerate}
These claims follow from the behavior of determinant lines under
permutations \eqref{eq:signs}, the behavior of the isomorphism with the
trivial determinant \eqref{eq:tD}, and the definition of the orientation
on nodal surfaces \eqref{eq:nodal}.
\end{enumerate} 
\end{remark} 

\subsection{The gluing signs} 
\label{sec:glueor}

We have already introduced in Section \ref{sec:cr} three types of gluing
for Cauchy-Riemann operators: gluing along strip-like ends, deforming
an interior node, and deforming a boundary node.  Let $\nS$ be a nodal
surface with strip-like ends, and $\dS$ a nodal surface obtained by
either deforming away a boundary node, deforming away an interior
node, or gluing two strip-like ends. Let $\nE$ be a complex vector
bundle with totally real boundary condition $\nF$, and $\dE,\dF$ the
vector bundles on $\dS$ obtained by gluing.  Similarly, let $D_{E,F}$
be a real Cauchy-Riemann operator with non-degenerate limits that are
equal along the strip-like ends $e_\pm$, and $D_{\ti{E},\ti{F}}$ an
operator obtained by gluing the ends $e_+,e_-$.

\begin{assumption} \label{ass:adj}
In the following, in the case of gluing boundary nodes or strip-like
ends we assume that the boundary components joined by the gluing are
adjacent in ordering.  We give the boundary components of  the glued
surface the ordering obtained by inserting the new boundary component(s) in place of the old in the ordered sequence.
\end{assumption}

\begin{theorem} \label{thm:gluingsigns}  {\rm (Behavior of orientations
under gluing)} Suppose that a surface with strip-like ends $\ol{S}$ is
  obtained from $S$ by gluing ends or deforming nodes, with orderings as in Assumption \ref{ass:adj}, and similarly the bundles $(\ol{E}, \ol{F})$ is obtained from gluing $(E,F)$ and 
  the Cauchy-Riemann operator $D_{\ol{E},\ol{F}}$ is obtained from $ D_{E,F}$.  The isomorphism of determinant lines
\[\cG_{E,F}: \det(D_{\ol{E},\ol{F}}) \to \det(D_{E,F})\] 
from \eqref{eq:linearglue} has the following signs 
\[ \Or(\cG_{E,F}): \Or(\det(D_{\ol{E},\ol{F}})) \to \Or(
\det(D_{E,F})) \]
in the respective cases below with respect to the constructed
orientations $o_{\ol{E},\ol{F}}, o_{E,F}$.
\begin{enumerate} 
\item \label{case:int}
{\rm (Deforming an interior node)}  Deforming an interior node is orientation-preserving:
\[(\Or(\cG_{E,F})) (o_{\ol{E},\ol{F}}) = ( o_{E,F}) .\]
\item \label{case:bounddiff} {\rm (Deforming a boundary node for a nodal surface with a
  single node $(w_-,w_+)$ joining two distinct boundary components
  adjacent in ordering)} 
The map on determinant lines for deforming a boundary node joining distinct components satisfies
\[(\Or(\cG_{E,F})) (o_{\ol{E},\ol{F}}) = (
  o_{E,F}) (\pm 1)^{\rank(F)},\] 
with positive sign in $(\pm 1)^{\rank(F)}$ if and only if $w_-$
lies on the earlier of the two components and the node is ordered between them;
\item \label{case:boundsame} {\rm (Deforming a boundary node joining a single boundary component to itself)}
The map on determinant lines for deforming a boundary node joining  a component to itself satisfies\footnote{See Sheridan \cite{sheridan:big} for a related discussion.}
\[(\Or(\cG_{E,F})) (o_{\ol{E},\ol{F}}) = ( o_{E,F}) (\pm
  1)^{\rank(F)}  (-1)^{\dim(F) (\dim(F) -1 )/2},\] 
with positive sign if and only if the ordering of the boundary components of the
glued surface has the boundary component containing the segment
from $w_-$ to $w_+$ ordered first;
\item \label{case:gluingends} {\rm (Gluing of strip-like ends of distinct components $S_-,S_+$
  such that $S_\pm$ have connected boundary, the end $e_+$ is the last
  outgoing end of $S_-$ and the end $e_-$ is the first incoming end of
  $S_+$, and the ordering of the ends on the glued surface is induced
  by the ordering of ends on $S_-,S_+$.)}  
\[(\Or(\cG_{E,F}))
  (o_{\ol{E},\ol{F}}) = ( o_{E,F}) (\pm 1)^{\rank(F)} \dagger \diamondsuit \] 
with positive sign in $(\pm 1)$ if and only if the ordering of
$e_-,e_+$ is $(e_-,e_+)$, with
\[ \dagger := (-1)^{(\sum_{e \in \mE_-(S_+) \ssm \{ e_- \}}
  \rank(F) - \Ind(D_e)) (\sum_{f \in \mE_+(S_-) \ssm \{ e_+ \}}
  \rank(F) - \Ind(D_f))} \]
times 
\[ \diamondsuit := (-1)^{ (\rank(F) + \sum_{f \in \mE_+(S_+)} \rank(F) ) (\sum_{f
    \in \mE_+(S_-) \ssm \{ e_+ \}} \Ind(D_f))} ,\]
where the ends of the glued surface are ordered with the incoming
ends of $S_-$ before the remaining incoming ends of $S_+$, and the
remaining outgoing ends of $S_-$ before the outgoing ends of $S_+$.
In particular, if $S_-$ has a single outgoing end and the ordering is
$(e_-,e_+)$, the gluing sign is positive.
\item \label{case:selfgluing} {\rm (Gluing two strip-like ends $e_-,e_+$ of the same connected surface $S$.)}  Let $\kappa \in \{ \pm 1 \}$ be the Koszul sign
of \eqref{eq:signs} of the permutation of factors in \eqref{eq:order} that places
$\Lambda^{\max}(\Gamma_{e_+}(0)^\dual) \otimes \det(D^+_{e_+})$
immediately before $\det(D^-_{e_-}) \otimes \Lambda^{\max}(\Gamma_{e_-}(0)^\dual)$
and leaves the other factors in order.  Then 
\[ (\Or(\cG_{E,F}))(o_{\ol{E}, \ol{F}}) = o_{E,F} \kappa \eps_1 \eps_2 \] 
where $\eps_1 = (\pm 1)^{\rank(F)}$ is the sign of case
\eqref{case:bounddiff} for the node joining the disk 
$\ti{S}_e = S_{e_-} \# S_{e_+} $
 to $\ol{S}$, and $\eps_2$ is the sign for the second node, 
 given by \eqref{case:boundsame} if the points in the node lie on the same boundary component and by \eqref{case:bounddiff} otherwise.  In particular $\eps_1 = 1$ if the boundary of $\ti{S}_e$ and its node are ordered first, as in the proof of \eqref{case:gluingends}.
 \end{enumerate}
\end{theorem} 

\begin{proof}
\noindent {\em Case \eqref{case:int}, Deformation of an interior node:} In the case of an interior
node we deduce preservation of orientations from complex linearity.
Let $\rS$ denote the resolution of $\nS$, and $\rE,\rF$ 
the vector bundles obtained
from $E,F$ by pullback.  First, we assume that $\rS$ has empty
boundary.  Consider a deformation 
\[ D_{\rE,\rF,\delta}: \Omega^0(E,F) \to \Omega^{0,1}(E,F), \quad
D_{\rE,\rF,0} = D_{\rE,\rF}, \quad D_{\rE,\rF,1} \ \text{complex
  linear} \]
of $D_{\rE,\rF}$ to a complex-linear operator.  The gluing isomorphism
\ref{cor:detiso2} induces an identification of determinant lines for each
bundle in the homotopy.  Since the identification of determinant lines
for the complex-linear operators is complex linear, the identification
of determinant lines is orientation-preserving, for each bundle in the
homotopy.  

Next, suppose the boundary is non-empty.  A deformation of $(\rE,\rF)$
to the connect sum of a problem on a closed surface $\iS$, glued to a
trivial problem on a union of disks $\oS$, induces a 
deformation for the glued problem $(\dE,\dF)$.  Compatibility of
orientations for gluing closed surfaces implies that the gluing map is
orientation-preserving.

\vskip .1in

\noindent {\em Cases \eqref{case:bounddiff}, \eqref{case:boundsame}, Deformation of a boundary node}: We reduce to the case of deforming a node connecting disks of index zero by the following argument.  Suppose that
$(w^+,w^-) $ is a boundary node of $\nS$, and $\dS,\dE,\dF$ a surface
and bundles obtained by deforming the node.  Consider the diagram of
indices shown in Figure \ref{fig:sixterm}; for self-gluing of a disk, see
also Figure \ref{fig:selfglue}.  In the diagram, the maps $\phi_1,\phi_2$ are the gluing maps for the determinant lines for $\iS \cup \oS$ to those of
$S$ and $\iS \cup \oS$ to $\nS$.  These maps are orientation-preserving by
definition.  The surface $S_{\times}$ at bottom left is obtained 
from the surface at upper right by deforming  the boundary node.   The gluing map $\phi_3$ is orientation
preserving by definition.  The map $\delta$ represents gluing of a
collection of disks equipped with trivialized boundary condition,
while $\beta,\psi_2$ represent deformation of an interior node.  The
induced isomorphism of determinant lines is orientation
preserving by \eqref{case:int}.  Both the lower square and the
upper left triangle in the diagram commute by associativity of gluing
Proposition \ref{prop:assoc}. Therefore, the map $\psi_1$ representing
gluing of determinant lines from $\dS$ to $\nS$, induces the same sign
on orientations as the diagonal map $\delta$.

\begin{figure}[ht]
\begin{picture}(0,0)%
\includegraphics{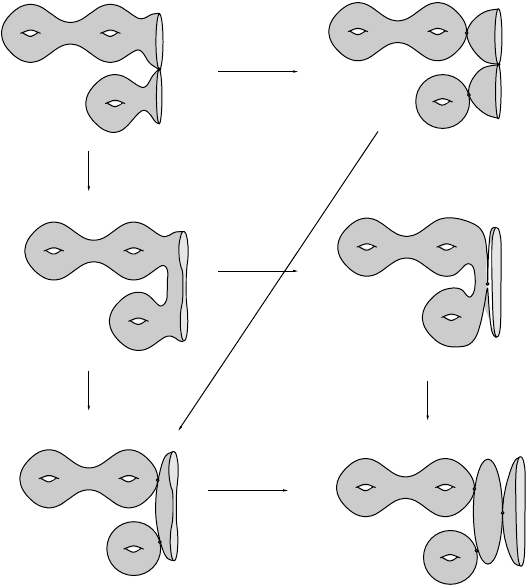}%
\end{picture}%
\setlength{\unitlength}{4144sp}%
\begingroup\makeatletter\ifx\SetFigFont\undefined%
\gdef\SetFigFont#1#2#3#4#5{%
  \reset@font\fontsize{#1}{#2pt}%
  \fontfamily{#3}\fontseries{#4}\fontshape{#5}%
  \selectfont}%
\fi\endgroup%
\begin{picture}(4014,4452)(-1753,-3158)
\put(-1374, 35){\makebox(0,0)[lb]{{{{$\psi_1$}%
}}}}
\put(-1366,-1674){\makebox(0,0)[lb]{{{{$\psi_2$}%
}}}}
\put( 70,-1378){\makebox(0,0)[lb]{{{{$\delta$}%
}}}}
\put( 78,-701){\makebox(0,0)[lb]{{{{$\phi_2$}%
}}}}
\put(131,826){\makebox(0,0)[lb]{{{{$\phi_1$}%
}}}}
\put(1529,-1758){\makebox(0,0)[lb]{{{{$\beta$}%
}}}}
\put( 55,-2388){\makebox(0,0)[lb]{{{{$\phi_3$}%
}}}}
\put( -1366,-3088){\makebox(0,0)[lb]{{{{$S_{\times}$}%
}}}}

\put( -1366,688){\makebox(0,0)[lb]{{{{$S$}%
}}}}

\put( -1366,-1088){\makebox(0,0)[lb]{{{{$\ti{S}$}%
}}}}
\end{picture}%

\caption{Deforming a boundary node}
\label{fig:sixterm}
\end{figure} 

We determine the sign for deforming  a node connecting disks of index zero.
Suppose that $\nS$ is obtained from a pair $\rS$ of disks by joining
them at a boundary node $w^\pm$.  After deformation we may assume that
the Cauchy-Riemann operator $D_{\rE,\rF}$ is the trivial operator.
Then the kernel is isomorphic to $F_{w^\pm} \oplus F_{w^\mp}$ (via the
two evaluation maps on the boundary) and the cokernel vanishes. The
reduced operator of \eqref{eq:reduced} is
\begin{equation} \label{eq:explicit}
 D_{\nE,\nF}^\red : F_{w^\pm} \oplus F_{w^\mp} \to F_w, \ \
 (f_\pm,f_\mp) \mapsto f_+ - f_- .\end{equation}
The ordering of the factors is determined by the ordering of the
boundary components of $S$.  By \eqref{eq:tD} and \eqref{eq:Dred} the
induced map $ \det(D_{\nE,\nF}^\red) \to \det(D_{\ti{E},\ti{F}}) $
changes the defined orientations by a sign 
\[ \eps(\nS,\dS,\rank(F)) = (\pm 1)^{\rank(F)}\] 
depending on whether the ordering of the pair $w^\pm$ agrees with the
ordering of the boundary components of $\rS$.  See Example
\ref{ex:diffs2}.

Next, we consider the case of a single disk joined to itself by a
boundary node and a trivial Cauchy-Riemann operator. Thus
the boundary component is split into two, as in Figure \ref{fig:selfglue}.
On the normalization $\rS$ of the nodal disk $S$ the kernel
$\ker(D_{{\rE},\rF})$ is isomorphic to the boundary fiber $\nF_{w^\pm}$ via evaluation at
a boundary point and trivial cokernel.  The reduced operator of
\eqref{eq:reduced} is
\begin{equation} \label{eq:explicit2}
 D_{\nE,\nF}^\red : \nF_{w^\pm} \to \nF_{w^\pm}, \ \ f \mapsto 0.
 \end{equation}
The kernel is isomorphic to $\nF_{w^\pm}$ and the cokernel is
isomorphic to $\nF_{w^\pm}$.  The deformed surface $\dS$ is an
annulus, equipped with trivial bundles $\dE,\dF$.  The orientation for
$D_{\dE,\dF}$ is induced from pinching off a pair of disks so that
$\nS_1$ is obtained by joining two disks and a sphere at interior
points.  A choice of ordering of boundary components on $\dS$ induces
an ordering of the nodes $z_-,z_+$ of $\nS_1$.  On the normalization
$\rS_1$ the reduced operator can be identified with
\begin{equation} \label{eq:explicit3} 
 D_{\nE_1,\nF_1}^{\red}
: \nF_{z^-} \oplus E_{z^\pm} \oplus \nF_{z^+}
 \to E_{z^-} \oplus E_{z^+}, \ \ (f_-,e,f_+) \mapsto (f_- - e, f_+ - e
 ) .\end{equation} 
The kernel is isomorphic to $\nF_{z^\pm}$, via evaluation at any
boundary point.  On the other hand, the cokernel is isomorphic to
$E_{z^\pm}/\nF_{z^\pm} = i \nF_{z^\pm}$, via projection onto the
second factor of the codomain. 

\begin{figure}[ht]
\includegraphics[width=4in,height=1in]{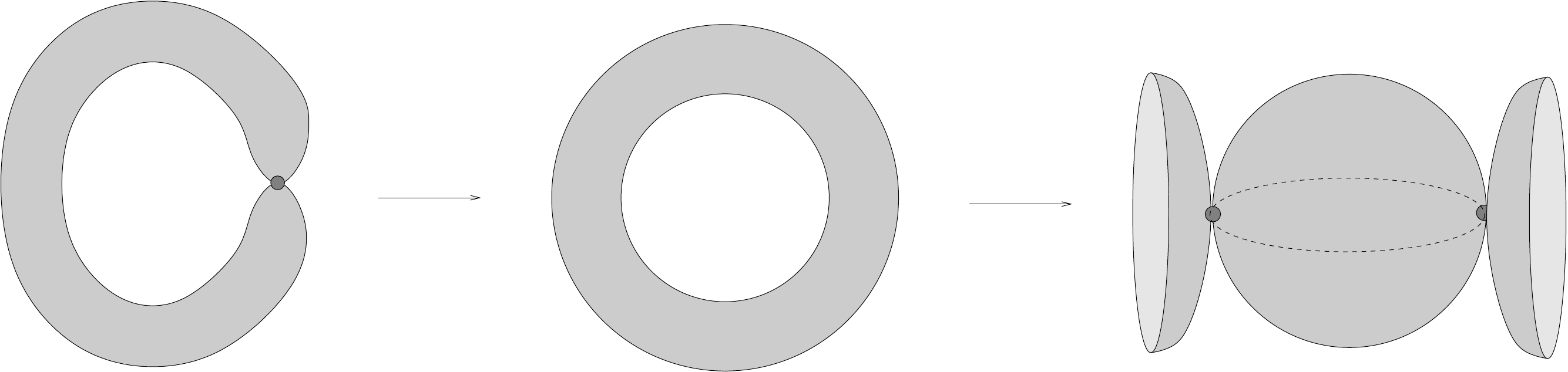}
\caption{Gluing a disk to itself}
\label{fig:selfglue}
\end{figure}

We compare the orientations coming from the two degenerations of the
annulus above.  Let $\nS_0,\nS_1$ be the nodal surfaces obtained
by stretching the two different directions, corresponding to the two boundary points in the moduli of annuli in which the annulus develops an boundary node (giving the surface $\nS_0$) or interior node (giving the surface $\nS_1$).    We compare the
identifications of the kernel and cokernels
\[ \ker(D_{E_0,F_0}) \cong \ker(D_{E_1,F_1}), \quad
\coker(D_{E_0,F_0}) \cong \coker(D_{E_1,F_1}) .\]
In the first case, the reduced operator in \eqref{eq:explicit2} is the
trivial operator on the space of sections with values in $\nF$.  The
kernel and cokernel are
\[ \ker(D_{E_0,F_0}) \cong \coker(D_{E_0,F_0}) \cong \nF_{w^\pm} .\]
via isomorphisms given by evaluation at a boundary point resp.
evaluation at a boundary point $w$ on the strip-like neck, see Figure
\ref{fig:twonecks}.  For the surface $\nS_1$ the reduced operator in
\eqref{eq:explicit3} has kernel and cokernel
\begin{equation} \label{eq:kercoker} 
\ker( D_{\nE_1,\nF_1}^{\red}) \cong \nF_{w^\pm} \quad \coker(
D_{\nE_1,\nF_1}^{\red}) \cong i\nF_{w^\pm} .\end{equation}
The isomorphisms in \eqref{eq:kercoker} are given by evaluation at a
point on one of the cylindrical necks for $\nS_1$.
\begin{figure}[ht]
\includegraphics[width=3in]{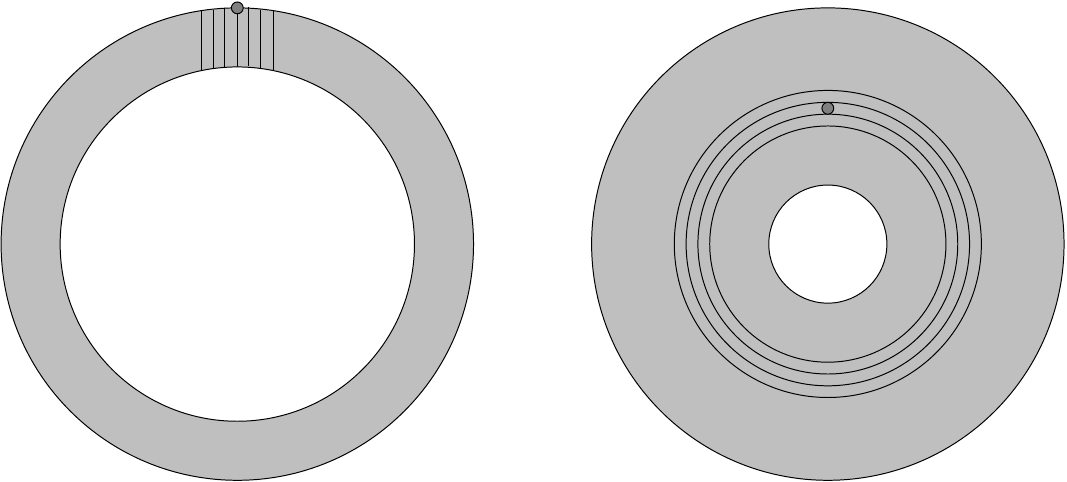}
\caption{Two kinds of neck}
\label{fig:twonecks}
\end{figure}
By construction, the bundle $\dE$ is trivial.  Choose a homotopy
between the two conformal structures on the annulus:
\[ j_t \in \J(S), t \in [0,1], \quad (S,j_\delta) \cong \nS_0^\delta,
\quad (S,j_{1-\delta}) \cong \nS_1^\delta \] 
where $\nS_0^\delta, \nS_1^\delta$ are the surfaces obtained
from $\nS_0,\nS_1$ by gluing with gluing parameter $\delta > 0 $
sufficiently small.  Let $(E_t,F_t)$ be the trivial bundle over the homotopy, and $D_{E_t,F_t}$ the trivial Cauchy-Riemann operators 
each with kernel and cokernel isomorphic to $F_{w^\pm}$.  We can also deform the evaluation maps to
all lie on the boundary of $S$, without changing the induced
orientations.  It remains to compare the orientations of 
\[ \nF_{w^\pm} \cong \coker(D_{E_0,F_0}) \cong \coker(D_{E_1,F_1})
\cong i \nF_{w^\pm} \]
given by \eqref{eq:fourterm}.  In each case, $\nF_{w^\pm}$ is identified
with the cokernel via wedge product with a one-form supported on the
neck.  For the surface $\nS_0$, the local coordinates on the neck
region depend on the ordering of the pair $w^\pm$.  On the strip-like
neck region on the left in Figure \ref{fig:twonecks}, choose horizontal
coordinate $s$ and vertical coordinate $t$.  The local complex
coordinate is $s + it$ if $w_-$ is numbered first, and $-s -it$ if
$w_+$ is numbered second.  On the other hand, the coordinates on the
cylindrical neck on the right in Figure \ref{fig:twonecks} are (on the
intersection of the two necks) $t - is$.  The identifications in
\eqref{eq:explicit2}, \eqref{eq:explicit3} are related by multiplication by
$\pm i$ if $w_\mp$ is ordered first.   The orientation on $F_{w_-} \oplus i F_{w_-} \cong E_{w} $ differs from the complex orientation defined in 
Remark \ref{rem:duals} \eqref{item:orient} by a factor $ (-1)^{\dim(F) (\dim(F) - 1 )/2}$.  It follows that the gluing map on determinant lines $\det(D_{E_0,F_0}) \to \det(D_{E_1,F_1})$ acts by the sign $(\pm 1)^{\rank(\nF)}  (-1)^{\dim(F) (\dim(F) - 1 )/2}$, with positive sign in $(\pm 1)$ if $w_-$ is ordered first, as claimed in the statement of the Theorem. 
\label{endsgluing}

\vskip .1in
\noindent {\em Case \eqref{case:gluingends}, Gluing of strip-like ends:} First, we
consider the case of a disconnected surface $\nS = S_- \cup S_+$ with
a single pair of strip-like ends, and $\nE \to \nS$ a complex vector
bundle over $\nS$ equipped with totally real boundary conditions
$\nF$. Let $e_-,e_+$ be ends and $(\nE_{e_+},\nF_{e_+}) \to
(\nE_{e_-},\nF_{e_-})$ an identification of the corresponding fibers.
Let $\tilde{S}$ denote the surface obtained by gluing $\nS$ together
along the ends, and $(\dE,\dF)$ the elliptic boundary value problem
obtained by gluing $\nE,\nF$.  We obtain surfaces without
strip-like ends by adding in the points at infinity:
\[ \ol{\nS} = \nS \cup \bigcup_{e} s_e, \ \ \ol{\tilde{S}} = \tilde{S}
\cup \bigcup_{ e \neq e^\pm} s_{e} .\]
Choose an ordering of the boundary components of $\ol{\dS}$.  The
strip-like ends of $S$ inherit an ordering from the ordering of the
ends of $\nS$.  We claim that the isomorphism of determinant lines
from $S$ to $\dS$ has the same sign as the isomorphisms of determinant
lines from $\ol{S}$ to $\ol{\dS}$.

To compute the gluing sign, we compare the deformation used in the
definition of the orientations on the surface with strip-like  ends to
the isomorphism of determinant lines induced by gluing strip-like
ends.  Consider the diagram of indices shown in Figure
\ref{fig:pinchglue}.  The top left picture represents $\det(D_{E,F})$.
The maps are defined as follows:
\begin{itemize}
\vskip .1in
\item[] ($\phi_1,\phi_6$) The map $\phi_1$ represents the isomorphism of determinant lines
  induced by deforming the boundary conditions to $F_\pm$ on the end as in the proof
  of Proposition \ref{prop:disktriv} so that $F_+ = F_-$ on the region indicated.  This deformation results in a
  boundary problem that is obtained from the nodal surface on the
  upper right of Figure \ref{fig:pinchglue} by gluing.  The map $\phi_6$
  represents a similar isomorphism of determinant lines induced by a
  deformation to bundles $F_\pm$ so that $F_ + = F_-$ on the neck region, and obtained by gluing. 
\vskip .1in
\item[] ($\phi_2,\phi_7$) The map $\phi_2$ represents the isomorphism
  of determinant lines induced by gluing at the two boundary nodes in
  the upper right surface.  Similarly, the map $\phi_7$ represents
  isomorphism induced by gluing along two nodes $w_1 = (w_{1,-},
  w_{1,+})$ and $w_2 = (w_{2,-}, w_{2,+})$ on the boundary.
\vskip .1in
\item[] ($\phi_3,\phi_4,\phi_5$) The maps $\phi_3,\phi_4,\phi_5$ are
  gluing isomorphisms for the gluing of strip-like ends $e_+,e_-$.

\end{itemize} 

In order to compute the gluing sign we establish commutativity of the
diagram.  The first square in the diagram commutes because deformation
commutes with gluing. The second commutes by associativity of gluing
for determinant lines in Proposition \ref{prop:assoc}:
\[ \phi_4 \phi_1 = \phi_6 \phi_3, \quad \phi_5 \phi_2 = \phi_7 \phi_4
.\]
\begin{figure}[ht]
\begin{picture}(0,0)%
\includegraphics{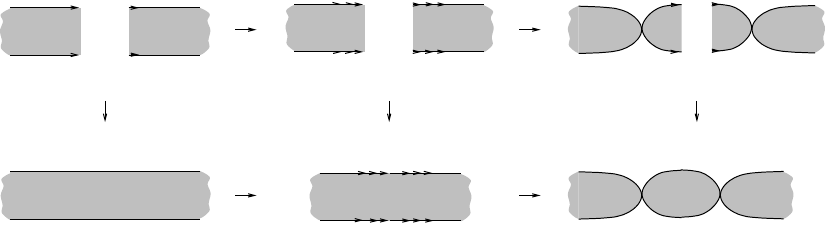}%
\end{picture}%
\setlength{\unitlength}{4144sp}%
\begingroup\makeatletter\ifx\SetFigFont\undefined%
\gdef\SetFigFont#1#2#3#4#5{%
  \reset@font\fontsize{#1}{#2pt}%
  \fontfamily{#3}\fontseries{#4}\fontshape{#5}%
  \selectfont}%
\fi\endgroup%
\begin{picture}(6288,1755)(669,-2856)
\put(2481,-1188){\makebox(0,0)[lb]{\smash{{\SetFigFont{5}{6.0}{\rmdefault}{\mddefault}{\updefault}{$\phi_1$}%
}}}}
\put(4628,-1201){\makebox(0,0)[lb]{\smash{{\SetFigFont{5}{6.0}{\rmdefault}{\mddefault}{\updefault}{$\phi_2$}%
}}}}
\put(1537,-1958){\makebox(0,0)[lb]{\smash{{\SetFigFont{5}{6.0}{\rmdefault}{\mddefault}{\updefault}{$\phi_3$}%
}}}}
\put(3708,-1952){\makebox(0,0)[lb]{\smash{{\SetFigFont{5}{6.0}{\rmdefault}{\mddefault}{\updefault}{$\phi_4$}%
}}}}
\put(6065,-1952){\makebox(0,0)[lb]{\smash{{\SetFigFont{5}{6.0}{\rmdefault}{\mddefault}{\updefault}{$\phi_5$}%
}}}}
\put(4580,-2818){\makebox(0,0)[lb]{\smash{{\SetFigFont{5}{6.0}{\rmdefault}{\mddefault}{\updefault}{$\phi_7$}%
}}}}
\put(2427,-2794){\makebox(0,0)[lb]{\smash{{\SetFigFont{5}{6.0}{\rmdefault}{\mddefault}{\updefault}{$\phi_6$}%
}}}}
\end{picture}%
\caption{Orientations for gluing strip-like ends}
\label{fig:pinchglue}
\end{figure}
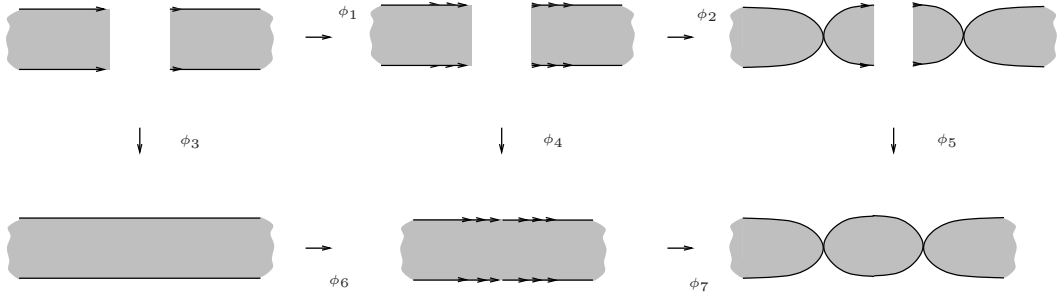
By definition, the composition of $\phi_1,\phi_2$ is orientation
preserving.  $\phi_5$ is orientation-preserving by construction, and
$\phi_6$ is orientation-preserving since it is induced by a
deformation.  Hence $\phi_3$ has the same sign as $\phi_7$.  By
definition $\phi_7$ is the composition of gluing isomorphisms for
resolution of the first node $w_1 = (w_{1,-}, w_{1,+})$, then second
boundary node $w_2 = (w_{2,-}, w_{2,+})$.  Choose the ordering of the
boundary components so that the disk boundary $\partial S_{\on{disk}}
\subset \partial S$ is ordered first and boundary nodes so that the
node on the disk $w_{k,+} \in \partial S_{\on{disk}}$ is ordered
first.  Then the first gluing isomorphism is orientation-preserving,
and the resulting surface is $\ol{S}$.  Hence $\phi_3$ has the same
sign as the isomorphism of determinants induced by the second gluing
operation.  By part \eqref{case:bounddiff}, this has the sign claimed in the statement of part \eqref{case:gluingends}.

The additional signs in the case of several ends are the Koszul signs
\eqref{eq:signs} of the permutation of the factors in the product of determinant lines in \eqref{eq:order}
that brings the factors for $S_-$ and $S_+$ into the order of
\eqref{eq:order} for the glued surface, while keeping together the
factors on which the argument above for a single pair of ends acts.
Let $\mE'_+ = \mE_+(S_-) \ssm \{ e_+ \}$ and $\mE'_- = \mE_-(S_+) \ssm \{ e_-
\}$.  Abbreviate 
\[ \Gamma_f = \Lambda^{\max}(\Gamma_f(0)^\dual), \quad \Gamma(\mE'_+) = \bigotimes_{f \in \mE'_+} \Gamma_f \] 
and 
\[ D^+(\mE'_+) =
\bigotimes_{f \in \mE'_+} \det(D^+_f) , \quad I(\mE'_-) = \bigotimes_{f \in \mE'_-}
(\det(D^-_f) \otimes \Gamma_f) \] 
and similarly $I(\mE_-(S_-))$,
$\Gamma(\mE_+(S_+))$, $D^+(\mE_+(S_+))$.  Let $C_\pm =
\det(D^c_\pm)$ denote the closed-off factors.  The tensor product of
\eqref{eq:order} for $S_-$ and $S_+$ is
\begin{multline*}
I(\mE_-(S_-)) \, C_- \, \Gamma(\mE'_+) \, \Gamma_{e_+} \, D^+(\mE'_+) \,
\det(D^+_{e_+}) \, \det(D^-_{e_-}) \, \Gamma_{e_-} \\ I(\mE'_-) \, C_+ \,
\Gamma(\mE_+(S_+)) \, D^+(\mE_+(S_+))
\end{multline*}
and the required order is
\begin{multline*}
I(\mE_-(S_-)) \, I(\mE'_-) \, C_- \, \Gamma_{e_+} \, \det(D^+_{e_+}) \,
\det(D^-_{e_-}) \, \Gamma_{e_-} \, C_+ \\ \Gamma(\mE'_+) \, \Gamma(\mE_+(S_+))
\, D^+(\mE'_+) \, D^+(\mE_+(S_+)) .
\end{multline*}
The parities are as follows:
\begin{enumerate} 
\item $C_\pm$ and $\Gamma_{e_\pm}$ have parity
$\rank(F)$, since $\Ind(D^c_\pm) = \rank(F)(1 - 2g_\pm) + \mu_\pm$ with
$\mu_\pm$ even;
\item $\det(D^+_{e_+}) \otimes \det(D^-_{e_-})$ has parity
$\rank(F)$ by compatibility; 
\item $I(\mE'_-)$ has parity $a = \sum_{f \in \mE'_-}
(\rank(F) - \Ind(D^-_f))$;
\item $\Gamma(\mE'_+)$ and $D^+(\mE'_+)$ have parities
$g = \sum_{f \in \mE'_+} \rank(F)$ and $d = \sum_{f \in \mE'_+} \Ind(D^+_f)$;
and 
\item $\Gamma(\mE_+(S_+))$ has parity $g_+ = \sum_{f \in \mE_+(S_+)}
\rank(F)$.  
\end{enumerate} 
The permutation is the composition of the following
exchanges:
\begin{enumerate}
\item[(i)] $I(\mE'_-)$ moves to the left past $C_- \, \Gamma(\mE'_+) \,
  \Gamma_{e_+} \, D^+(\mE'_+) \, \det(D^+_{e_+}) \, \det(D^-_{e_-}) \,
  \Gamma_{e_-}$, of total parity $4 \rank(F) + g + d \equiv g + d$,
  with sign $(-1)^{a(g+d)} = \dagger$;
\item[(ii)] $D^+(\mE'_+)$ moves to the right past $\det(D^+_{e_+}) \,
  \det(D^-_{e_-}) \, \Gamma_{e_-} \, C_+ \, \Gamma(\mE_+(S_+))$, of
  total parity $3 \rank(F) + g_+ \equiv \rank(F) + g_+$, with sign $(-1)^{d(\rank(F)
  + g_+)} = \diamondsuit$;
\item[(iii)] $\Gamma(\mE'_+)$ moves to the right past the factors $\Gamma_{e_+}$,
  $\det(D^+_{e_+})$, $\det(D^-_{e_-})$, $\Gamma_{e_-}$, $C_+$, of total
  parity $4 \rank(F) \equiv 0$, with sign $+1$.
\end{enumerate}
The incoming ends of $S_-$ and the outgoing ends of $S_+$ are not
moved.  This proves the formula in \eqref{case:gluingends}. 
\vskip .1in \noindent {\em Case \eqref{case:selfgluing}, Gluing two ends of the same surface:}  The argument for \eqref{case:gluingends} applies 
with $S_-  = S_+ = S$.  The commutativity of the diagram in Figure \ref{fig:pinchglue} uses only that the deformation commutes with gluing and Proposition \ref{prop:assoc}, neither of which requires the glued ends to lie on distinct components.  
After bubbling off the caps, the factors for $e_\pm$ in \eqref{eq:order}
are first brought next to each other at the cost of the sign $\kappa$.
The caps $S_{e_-}, S_{e_+}$ are then glued to the disk $\ti{S}_e$, which is orientation preserving by the compatibility \eqref{eq:glueDe}
 of the end orientations. The remaining map $\phi_7$ deforms the two boundary
 nodes $w_1,w_2$ joining $\ti{S}_e$ to $\ol{S}$. Deforming $w_1$ joins distinct components and has sign $\eps_1$ by \eqref{case:bounddiff}.  Deforming $w_2$ joins the resulting connected surface to itself, along two distinct boundary components if $e_-,e_+$ lie on distinct components of $\partial \ol{S}$, or along a single boundary component otherwise.  Its sign $\eps_2$ is given by 
 \eqref{case:bounddiff} resp. \eqref{case:boundsame}. 
\end{proof} 

\begin{corollary} \label{cor:nonadjacent} {\rm (Gluing without the adjacency assumption)}  Let $\ol{S}$ be obtained from $S$ by gluing ends or deformed nodes as in Theorem \ref{thm:gluingsigns} of different components, but without assuming Assumption \ref{ass:adj}. 
Give $\ol{S}$ the ordering of boundary components, patches and ends obtained from that of $S$ by replacing the two glued boundary components by the resulting boundary component in the position of the earlier one, and by deleting the glued ends. Let $S'$ be the surface $S$ with the ordering changed so that the two glued objects become adjacent, as in Assumption \ref{ass:adj}, by removing the {\em later} object next
to the earlier one without changing the relative order of the others. 
Then 
\[ \Or(\cG_{E,F})(o_{\ol{E},\ol{F}})= \Or(\cG'_{E,F})(o_{\ol{E},\ol{F}}) \ s, \]
where $\cG'_{E,F}$ is the gluing isomorphism for the reordered surface $S'$ whose sign is given by Theorem \ref{thm:gluingsigns} applied to $S'$ after reordering, and where the correction sign $s \in \{ \pm 1 \}$ is 
\begin{enumerate}
\item for gluing boundary nodes or boundary components $(\partial S)_a, (\partial S)_b$, $a < b$, a sign change given by $s = \det(\sigma)^{\rank(F)} = (-1)^{(b-a-1)\rank(F)}$ where $\sigma$ is the cyclic permutation moving $(\partial S)_b$ to position $a+1$ as in Remark 
\ref{rem:permuteremark} third item; and 
\item for gluing strip-like ends $e_-,e_+$ of distinct components, not consecutive in the ordering of the ends:  $s = (-1)^{\Ind(D_{e_+})\sum_i \Ind(D_{e_i})}$, the sum 
over the ends $e_i$ strictly between $e_-$ and $e_+$ in the original ordering as in 
Remark \ref{rem:permuteremark}, fourth item:
\item if the two components $S_\pm$ containing the glued ends are not adjacent in the ordering of the components, then in addition $(-1)^{\Ind(D_{S_+}) \sum_i \Ind(D_{S_i})}$, the sum over the components $S_i$ between $S_-$ and $S_+$ as in Remark \ref{rem:permuteremark}, second item. 
\end{enumerate}
\end{corollary}

\begin{proof} The orientation of $\det(D_{E,F})$ depends on the ordering of patches, boundary components, and ends.  By Remark \ref{rem:permuteremark} the identity on kernel and cokernel identifies $\det(D_{E,F})$ for the two orderings of $S$ and $S'$, 
and this identification changes orientations by the sign $s$
 above as the product of the signs for the individual transpositions. 
 The orderings on the glued surface obtained from $S$ and from $S'$ agree:
 The move $\sigma$ only changes the position of an object that is deleted or merged by the gluing, and does not change the relative order of the remaining ones.  So $\ol{S}$
 carries the same orientation in both cases, and the isomorphism $\cG_{E,F}$
 differs from  $\cG'_{E,F}$ by the identification above on $\det(D_{E,F})$
 and by the identity on $\det(D_{\ol{E},\ol{F}}).$  The statement of the Corollary follows.  \end{proof}

Finally, we check the signs for various special cases needed later.

\begin{remark} \label{rem:cases} 
\begin{enumerate} 
\item \label{case:annulus}
{\rm (Annulus)} Let $A = [0,1] \times S^1$ with trivial bundle $E$ and
two transverse, constant boundary conditions $F = (F_0,F_1)$.  By
definition, the orientation on the determinant line $\det(D_{E,F})$ in
induced from the isomorphism with the sphere with two bubbled-off
disks,
\[ \det(D_{E,F}) \to \Lambda^{\max}( (E^2)^\dual) \otimes
\Lambda^{\max}(F_0 \oplus E \oplus F_1) ,\]
see Figure \ref{fig:selfglue}.  The operator $D_{E,F}$ has trivial kernel
and cokernel, and so $\det(D_{E,F}) = \R$.  The induced orientation on
the determinant line $\det(D_{E,F}) = \R$ is the standard one if and
only if the isomorphism $F_0 \oplus F_1 \to E$ is orientation
preserving.  The quilted case is similar and left to the reader.
\item \label{case:strip} {\rm (Strip)} Let $S$ denote the strip $[0,1] \times \R$.  Let
  $e^0$ denote the incoming end, and $e^1$ the outgoing end of $S$.
  Let $E$ be the trivial bundle and $F^j, j= 0,1$ denote constant,
  transverse boundary conditions.  Choose a path 
\[\Gamma = (\Gamma^t)_{t \in [-\infty,\infty]}, \quad \Gamma^{-\infty}
  = F^0 , \quad \Gamma^{\infty} = F^1 \]
from $F^0$ to $F^1$, and compatible orientations on the resulting
operators on the once-punctured disks $D_{e^j}, j = 0,1$.
Compatibility means that the gluing map for strip-like ends
\[ \det(D_{e^1}) \otimes \det(D_{e^0}) \to \det(D_{\disk}) \]
to the operator $D_{\disk}$ with homotopically trivial boundary
conditions is orientation-preserving.  The orientation on $\det(D)$ is
defined so that the gluing isomorphism
\begin{equation} \label{eq:strip0}
 \det(D_{e^0}) \otimes
\Lambda^{\max}((\Gamma(0))^{\dual}) \otimes
\det(D_{\disk}) \otimes
\Lambda^{\max}((\Gamma(0))^{\dual}) \otimes
\det(D_{e^1}) \to \det(D_{E,F}) \end{equation}
is orientation-preserving. Permuting the factor $\det(D_{e^1})$ to the
beginning produces a factor of
\[(-1)^{\Ind(D_{e^1}) (\rank(F)+ \Ind(D_{e^0}) )} =
(-1)^{\Ind(D_{e^1})} .\]
Using compatibility of orientations and gluing to the annulus gives an
orientation-preserving isomorphism
\begin{multline} \label{eq:strip} \det(D_{e^1})  \otimes 
\det(D_{e^0}) \otimes \Lambda^{\max}((\Gamma(0))^{\dual}) \otimes
\det(D_{\disk}) \otimes \Lambda^{\max}((\Gamma(0))^{\dual}) \\ \to
\det(D_{\disk}) \otimes \Lambda^{\max}((\Gamma(0))^{\dual}) \otimes
\det(D_{\disk}) \otimes \Lambda^{\max}((\Gamma(0))^{\dual}).
\end{multline}
By gluing this tensor product is isomorphic to the determinant line
for the Cauchy-Riemann operator on the annulus.  By case \eqref{case:annulus}, the gluing isomorphism has orientation $(-1)^{\Ind(D_{e^1})}$.  The
arguments above give a total sign of $+1$ for the isomorphism of the
determinant line of the operator on the strip with the trivial line.
\item {\rm (Cup and Cap)} Suppose that $S_\cup,S_\cap$ are the disks
  with two outgoing resp. incoming ends of Example 4.1.6 of
  \cite{ww:co}.  Suppose these are equipped with constant vector
  bundles $E_\cup,E_\cap$ and constant real boundary conditions
  $(F^0_\cup,F^1_\cup) = (F^1_\cap,F^0_\cap)$.  Let $D_{\cup},
  D_{\cap}$ be the trivial Cauchy-Riemann operators.  For the
  two ends of $S_\cup$ we can choose the paths
  $\Gamma_\cup,\Gamma_\cap$ on the two ends $e_\cup^0,e_\cup^1$ to be
  related by time-reversal.  Choose the orientations on $D_{e_\cup^0},
  D_{e_\cup^1}$ so that the gluing map
\begin{equation} \label{eq:cup}
 \det(D_{\disk}) \otimes \Lambda^{\max}(\Gamma_\cup(0)^\dual)^2
 \otimes \det(D_{e_\cup^1}) \otimes \det(D_{e_\cup^0}) \to
 \det(D_\cup) \end{equation}
induces the standard orientation on $\det(D_\cup) = \R$.  Note that we
have isomorphisms
\[ \det(D_{e^1_\cap}) \cong \det(D_{e_\cup^0}), \ \ \det(D_{e^0_\cap}) \cong
\det(D_{e_\cup^1}) .\]
The compatibility condition for $D_{e_\cup^0},D_{e_\cap^0}$ differs
from that for $D_{e_\cup^1},D_{e_\cap^1}$ by a sign
\[(-1)^{\Ind(D_{e^1_\cup}) \Ind(D_{e_\cup^0})} = (-1)^{(\rank(F)
  - \Ind(D_{e_\cup^0})) \Ind(D_{e_\cup^0})} \]
given by the transposition of factors.  Changing the order of factors
in \eqref{eq:cup} to that in \eqref{eq:strip0} produces a sign $
(-1)^{\Ind(D_{e_\cup^0})^2} $ by \eqref{eq:signs}.  The choice of sign
orientation for $\det(D_{e^1_\cup})$ that makes the orientations
positive is this sign times the induced orientation from
$\det(D_{e^0_\cap})$.  The orientation on $D_\cap$ is defined by the
gluing isomorphism
\begin{equation} \label{eq:cap}
 \det(D_{e_\cap^0}) \otimes \Lambda^{\max}(\Gamma_\cup(0)^\dual)
\otimes \det(D_{e_\cap^1}) \otimes
\Lambda^{\max}(\Gamma_\cup(0)^\dual) \otimes \det(D_{\disk}) \to
\det(D_\cap) .\end{equation}  
Changing the order of factors to match that of \eqref{eq:strip} produces
by \eqref{eq:signs} a sign
\[(-1)^{\Ind(D_{e_\cap^0})^2 + \Ind(D_{e_\cap^1}) \rank(F) }.\]
Hence the orientation on $\det(D_{\cap}) = \R$ is related
to the standard one by a sign 
\[ (-1)^{\rank(F)  \Ind(D_{e_\cup^0}) + 2 \Ind(D_{e_\cup^0})^2
+ (\rank(F) - \Ind(D_{e_\cup^0})) \Ind(D_{e_\cup^0})} 
= (-1)^{\Ind(D_{e_\cup^0})} .\]
\label{disksigns}
\end{enumerate}
\end{remark} 

\subsection{Orientations for quilted surfaces}
\label{sec:oqs}

In this section, we define orientations for Cauchy-Riemann operators on
quilted surfaces.  Let $\ul{S} \to B$ be a family of quilted surfaces
possibly with strip-like ends and $D_{\ul{E},\ul{F}}$ be a family of
Cauchy-Riemann operators for $(\ul{E},\ul{F})$, and suppose that the
$\ul{F}$ is equipped with a collection of relative spin structures.
We define an orientation for $D_{\ul{E},\ul{F}}$ by deforming the seam
conditions to split form.

\begin{definition}
    \label{def:coloring}
    A {\em coloring} of  quilt $\ul{S}$ is a function on its patches  $\w: \cP \to \Z_2$
such that the two patches adjacent to any seam have different colors. 
\end{definition}

\begin{proposition} \label{prop:noends} {\rm (Orientations for quilted Cauchy-Riemann operators
via relative spin structures)}
 Let $\ul{S},\ul{E},\ul{F}$ be a family
  of quilted surfaces with bundles and boundary/seam conditions over a
  base $B$.
Suppose that the patches of $\ul{S}$ admit a coloring, and fix a coloring $\w$.   A relative spin structure on $\ul{F}$ and orientations
  for the ends of each component together induce an orientation on the
  determinant line bundle $\det(D_{\ul{E},\ul{F}}) \to B$.
\end{proposition} 

The proof uses a deformation of the seam conditions to split form for each element of the family,  after which we may apply the construction in the unquilted case; one must then check that the resulting orientation is independent of the choice of split form. 
(It is not possible, in general, to deform the whole family to split form as the higher homotopy groups of the relevant Grassmannian are non-vanishing.)  For
simplicity, we assume that the Hamiltonian perturbations on the
strip-like ends vanish.  We may assume that the ranks of the bundles
are at least three, after stabilizing by adding trivial bundles.  Note
that the map of Grassmannians of totally real subspaces
\[ \frac{U(n_1)}{SO(n_1)} \times \frac{U(n_2)}{SO(n_2)} \to \frac{U(n_1 + n_2)}{SO(n_1 +
n_2)} \] 
induces a surjection of first and second homotopy groups, by the long
exact sequence of homotopy groups and the isomorphisms 
\[\pi_1(SO(n_1)) \to \pi_1(SO(n_1+n_2)), n_1 \ge 3; \quad \pi_1(U(n_1)) \to
\pi_1(U(n_1+n_2)) .\] 
It follows that there exists a deformation of the seam conditions on
the strip-like ends to split form in the product of Grassmannians
$(U(\sum n_j)/SO(\sum n_j))^2 $ (where $n_1,\ldots,n_k$ are the
dimensions of the boundary and seam conditions) such that the path has
Maslov index zero:
\[ [\ul{F}_{e,\pm, \delta}] \in (U(\Sigma_j n_j)/SO(\Sigma_j n_j)), \quad
\delta \in [0,1], \quad [\ul{F}_{e,\pm, 1}] \in \prod_j
U(n_j)/SO(n_j) .\]
Any such path has a deformation with no crossing points, that is, so
that every set of conditions in the deformation are transversal:
\[ \ul{F}_{e,-,\delta} \cap \ul{F}_{e,+,\delta} = \{ 0 \}, \quad
\forall \delta \in [0,1] .\]
This deformation produces a family of Fredholm operators, and hence an
isomorphism of the determinant line of the original problem with the
problem with split form on each strip-like end.  The given path can be
deformed into a path in partially split form, that is, a path into
\[ U(n_1)/SO(n_1) \times U(n_1 + n_2)/SO(n_1 + n_2) \times \ldots
\times U(n_k)/SO(n_k) \]
uniquely up to homotopy of homotopies.  Finally, we homotope the seam
conditions $\ul{F}$ via a homotopy
\[ \ul{F}_\delta \subset \ul{E} |_{\partial \ul{S}}, 
\quad \ul{F}_0 = \ul{F}, \quad 
\ul{F}_1 = \ul{F}^{\on{split}} \]
to a set of boundary and seam conditions $\ul{F}^{\on{split}}$ of
split form over the entire surface, 
where each condition is given by 
\[ F_\sigma^{\on{split}} 
=  F^-_{p_-(\sigma)}
 \times F_{p_+(\sigma)}  \] 
 for some oriented sub-bundles $F_{p_\pm(\sigma)}$ of $E_{p_\pm(\sigma)}$.  
 On each boundary component, a relative spin structure on $\ul{F}^{\on{split}}$
 induces relative spin structures 
on the factors $F_{p_\pm(\sigma)}$
 up to the actions of $\Z_2$ 
simultaneously on two factors.   Since
each change in relative spin structure changes the orientation  by \cite[Proposition 8.1.16]{fooo} by a sign $-1$ (independent of the rank of the bundle), the simultaneous change in two relative spin structures leaves the sign unchanged. 

Having completed the deformation to split form, we now reduce to the
unquilted case.  The index problem on $(\ul{E},\ul{F}^{\on{split}})$
splits into a sum of problems on the various components:
\[ D_{\ul{E},\ul{F}^{\on{split}}}
\cong \bigoplus_{p \in \PP}  D_{E_p,F_p^{\on{split}}} . \]
  The relative spin structure on $\ul{F}$ induces a 
 relative spin structure on $\ul{F}^{\on{split}}$ by deformation; that is, there is a unique relative spin structure which extends over the deformation, and we may restrict this relative spin structure to $\ul{F}^{\on{split}}$.   
The
 constructions in the unquilted case give orientations on the various
 determinant lines
\[\det(D_{E_p,F_p^{\on{split}}}), \quad p \in \PP  .\]  
For a quilted surface with strip-like ends we combine these
orientations as follows.  The seams divide each quilted end $\ul{e}$ into strips
$e_1,\ldots,e_k$, ordered by their position across the end, the strip
$e_i$ lying in a patch $q(i)$; a patch may contain several of the
strips $e_i$.  We regard
$E_{\ul{e}} = \bigoplus_{i=1}^k E_{q(i)}^{\pm}$, with the summands for
patches of color $1$ conjugated and the summands in the order of the
strips, as a single end datum in the sense of Definition
\ref{def:enddatum}, with the direct sum path
$\Gamma_{\ul{e}} = \bigoplus_{i=1}^k \Gamma_{e,i}$ and the direct sum
orientation of Remark 
\ref{rem:enddatumsums} \eqref{item:sumend}.  An {\em orientation for the quilted end} $\ul{e}$ is an orientation for this end datum, 
that is, an orientation of the incoming cap $D^-_{\ul{e}} = \bigoplus_{i=1}^k D^-_{e,i}$; the outgoing cap $D^+_{\ul{e}}$ is oriented by compatibility as in Definition \ref{def:compat}, which differs from the direct sum orientation by the sign of Remark \ref{rem:compatops} \eqref{item:sumcompat}. 
We orient $\bigoplus_p D_{E_p, F_p^{\on{split}}}$ by the isomorphism \eqref{eq:order}, in which each quilted end
$\ul{e}$ is treated as a single end with end datum
$E_{\ul{e}}$:  the factors $\det(D^\pm_{\ul{e}})$, with the orientations just described, and  $\Lambda^{\max}(\Gamma_{\ul{e}}(0)^\dual)$ appear in positions given by 
\eqref{eq:order} for the ordering of the quilted ends, and the closed-off factors 
$\det(D^c_p)$ appear in the middle, in the ordering of the patches.  The identification with the product of the orientations of the patches is by the Koszul sign rule 
\eqref{eq:signs}.  
Here a patch $p$ with $\w(p) = 0$ is oriented by the construction in
the unquilted case.  A patch with $\w(p) = 1$ is regarded, as in the
folded picture, as a patch labelled by the conjugate $E_p^-$: its
determinant line is oriented by
\[ (-1)^{n_p(1 + g_p + m_p)} \, T^* o ,\]
where $T: D_{E_p, F_p^{\on{split}}} \to D_{E_p^-,F_p^{\on{split},-}}$
is the conjugation isomorphism, $o$ is the orientation of the conjugate
problem from the unquilted case for the given relative spin structure,
$n_p = \rank_\C(E_p)$, and $g_p, m_p$ are the genus and the number of
boundary components of the patch closed off at its strip-like ends by
the caps of Proposition \ref{prop:disktriv}; for a nodal patch the sign
is the product of the signs for the components of its normalization.
The normalizing sign depends only on the topology of the patch and the
rank of $E_p$.  It makes the gluing signs independent of the colors
(Lemma \ref{lem:quiltedgluing}), and changing the coloring amounts to
shifting background classes (Corollary \ref{cor:colorindep}).  For
definiteness we use the coloring with $\w(p_1) = 0$ for the first patch
$p_1$.
 
\begin{remark} \label{rem:dependence}
\begin{enumerate}
\item {\rm (Dependence on ordering of components)} Recall that for a
  disconnected unquilted surface $S$ with boundary value problem
  $E,F$, the orientation constructed on a Cauchy-Riemann operator
  $D_{E,F}$ depends on an ordering of the components of $S$.  In
  particular,  for a quilted surface the orientation depends on an
  ordering of the patches 
\[ \PP = \{ p_1,\ldots, p_l \}, \quad l = | \cP| .\]
However, if $\ul{S}$ is connected and $D_{\ul{E},\ul{F}}$ has index
zero resp. one then the orientation on $D_{\ul{E},\ul{F}}$ is
independent of the ordering of the connected components of the patches
$\ul{S}$.  Indeed, since the orientation constructed is independent of
the choice of deformation to split form, we may deform $\ul{F}$ to
boundary bundles of split form such that the index is zero on each
resp. all but one patch of $\ul{S}$. Then the determinant lines for
all connected components commute, see Remark \ref{rem:permuteremark}
\eqref{item:permute}.

\item {\rm (Translation-invariant quilted strips)}
 \label{trivbun}
 Suppose that $\ul{S}$ is a quilted strip and that 
 $D_{\ul{E},\ul{F}}$ is translation-invariant with 
 transverse limits; for example, the trivial Cauchy-Riemann operator for trivial bundles with constant and transverse seam and boundary conditions.  Then the kernel
 and cokernel of $D_{\ul{E},\ul{F}}$ are trivial,
 and so $\det(D_{\ul{E},\ul{F}})$ is canonically trivial. 
 We claim that the orientation of Proposition \ref{prop:noends}, for compatible orientations of the quilted ends, is the canonical one.  Indeed, 
 folding the patches of color $\w = 1$ identifies
 $D_{\ul{E},\ul{F}}$ with a translation-invariant operator 
 on an unquilted strip with end datum $E_{\ul{e}}$ at both ends, and by the convention in Proposition \ref{prop:noends}
 and the direct sum property Remark \ref{rem:crem},
 the orientation is that of the folded strip with the direct sum end orientations, which are compatible by the convention of Proposition \ref{prop:noends}. The claim follows from Remark \ref{rem:cases}
   \eqref{case:strip}.
\end{enumerate} 
\end{remark} 

\begin{lemma} \label{lem:indepsplit}
Let $\ul{F}'$ and $\ul{F}''$ be two deformations of 
$\ul{F}$ to split form with oriented factors. 
Then the orientations on $D_{\ul{E},\ul{F}'}$
and $D_{\ul{E},\ul{F}''}$ are identified
under the isomorphism 
$\det(D_{\ul{E},\ul{F}'}) \to \det(D_{\ul{E},\ul{F}''})$
induced by the homotopy of $\ul{F}'$ to $\ul{F}''$ through totally real sub-bundles obtained by composing the reverse of the deformation from $\ul{F}$ to $\ul{F}'$ with the deformation from $\ul{F}$ to $\ul{F}''$.  In particular the orientation
on $D_{\ul{E},\ul{F}^{\on{split}}}$ is independent
of the deformation to split form, and of the relative spin structure on the split form chosen.
\end{lemma}

\begin{proof}  
By the argument above we may arrange, after a homotopy within split form, that $\ul{F}'$ and $\ul{F}''$ agree
outside an arc $I_\sigma$ in each seam $\sigma$. 
Along $I_\sigma$ the two conditions form a loop 
\[ \gamma_\sigma = (\gamma^-_\sigma, \gamma^+_\sigma): S^1 \to \Lambda(n_-) \times \Lambda(n_+) \]
with  $n_\pm = \rank(F_{p_\pm(\sigma)})$.  Let
$a^\pm_\sigma$ be the Maslov indices of $\gamma^\pm_\sigma$.
Since $\ul{F}'$ and $\ul{F}''$ are both homotopic to $\ul{F}$, 
the loop $\gamma_\sigma$ is null-homotopic in $\Lambda(n_- + n_+)$ so that $a^-_\sigma + a^+_\sigma = 0$.

\vskip .1in \noindent {\em Step 1:  The Maslov indices $a^\pm_\sigma$ are even.}
  The Maslov index mod $2$ of a loop of totally real subspaces is the first Stiefel-Whitney class of the corresponding real bundle.   Hence 
  \[ w_1(F''_{p_\pm})
  = w_1(F'_{p_\pm}) + a^\pm_\sigma \ \text{mod} \ 2 \]
  and since both are oriented, $a^\pm_\sigma$ is even.

\vskip .1in \noindent {\em Step 2:  The effect of $\gamma_\sigma$ on the relative spin structures}   Since $a^\pm_\sigma$ is even, 
$\gamma^\pm_\sigma$ lifts to a loop in the orientation cover and defines a class in $\pi_1(SO(n_\pm)) = \Z_2$ equal to 
$a^\pm_\sigma/2$ mod $2$.  Thus the relative spin structure
on $F''_{p_-}$ differs from that induced by $F'_{p_-}$ by the class $a^-_\sigma/2$ mod $2$ in $H^1$ of the boundary circle containing $I_\sigma$, and the relative spin structure
on $F''_{p_+}$ differs from that on $F'_{p_+}$ by 
\[ a^+_\sigma/2= - a^-_\sigma/2 \equiv a^-_\sigma/2 \ \text{mod}  \ 2 .\]
So the pair changes by the same class on both sides $(\alpha,\alpha)$ with $\alpha = a^-_\sigma/2$ mod $2$.

\vskip .1in \noindent {\em Step 3:  Comparison of orientations.}  Each change of
relative spin structure on a boundary component changes the orientation by $-1$, independent of the rank, by 
\cite[Proposition 8.1.16]{fooo}.  The two simultaneous changes contribute $(-1)^{2 \alpha}  = +1$.
The isomorphism induced by the homotopy from $\ul{F}'$
to $\ul{F}''$ in $\Lambda(n_- + n_+)$ equals the gluing isomorphism with the pair of disks $D_\pm$ carrying the boundary conditions $\gamma^\pm_\sigma$, 
one in each patch.  Moving the disk $D_-$ (of index $a^-_\sigma$) past the determinant line of 
$D_{E_{p_+}, F'_{p_+}}$ gives the Koszul 
sign 
\[ (-1)^{a^-_\sigma \Ind(D_{E_{p_+}, F'_{p_+}})} = 
+1 \] 
since $a^-_\sigma$ is even.  By Remark 
\ref{rem:permuteremark} the remaining gluing signs 
depend only on ranks and orderings, which agree
for $\ul{F}'$ and $\ul{F}''$.  

It remains to orient the pair $(D_-,D_+)$.  Folding along $\sigma$
identifies the pair with a single disk $D_\sigma$ with boundary
condition the loop $\gamma^-_\sigma \times \ol{\gamma^+_\sigma}$ in
$\Lambda(n_- + n_+)$.  In the gluing description above,
$\det(D_\sigma)$ carries the orientation $o_h$ obtained by transport
along the null-homotopy of this loop given by the homotopy from
$\ul{F}'$ to $\ul{F}''$, starting from the constant loop, for which
$\coker(D_\sigma) = 0$ and $\ker(D_\sigma)$ is the oriented fiber of
$\ul{F}'$ at the endpoint of $I_\sigma$.  The simple connectivity of
the disk does not by itself determine this orientation: the two spin
structures on the boundary loop give opposite orientations by
\cite[Proposition 8.1.16]{fooo}, and null-homotopies differing by the
generator of $\pi_2(\Lambda(n_- + n_+)) \cong \Z_2$ induce different
spin structures.  Rather, the null-homotopy extends the boundary
condition over the disk and so induces a spin structure
$\lie{s}_h$ on the boundary loop, namely the transport of the
product spin structure on the constant loop.  Since the orientations
of the unquilted construction are continuous in families with the spin
structures carried along, and equal the fiber orientation for a
constant boundary condition with the product spin structure, $o_h$ is
the orientation of $D_\sigma$ for $\lie{s}_h$.

We claim that the relative spin structure on the boundary loops is a direct sum.  That is, we wish to show that $\lie{s}_h = \lie{s}_- \oplus \lie{s}_+$,
where $\lie{s}_\pm$ are the spin structures on
$\gamma^\pm_\sigma$ such that the relative spin structure on
$F''_{p_\pm}$ is obtained by gluing that on $F'_{p_\pm}$ with
$\lie{s}_\pm$, as in Step 2.  By the construction in
Proposition \ref{prop:noends}, the relative spin structures on the
factors of $\ul{F}'$ resp. $\ul{F}''$ are those whose direct sums are
the relative spin structures transported from $\ul{F}$ along the
deformations to split form.  Hence the direct sum on $\ul{F}''$ is
obtained from that on $\ul{F}'$ by transport along the homotopy from
$\ul{F}'$ to $\ul{F}''$, which along $I_\sigma$ is gluing with
$\lie{s}_h$.  The pair $(\lie{s}_-,\lie{s}_+)$ is
determined by this condition only up to the simultaneous change
$(\alpha,\alpha)$ of Step 2, which changes neither
$\lie{s}_- \oplus \lie{s}_+$ nor, by Step 3, the
orientation below.

Having reduced the relative spin structure to a direct sum, 
we compute the orientation for the linearized operator. 
By the direct sum property, Remark \ref{rem:crem}
\eqref{item:sums}, the orientation of $D_\sigma$ for
$\lie{s}_- \oplus \lie{s}_+$ is $o(D_-) \wedge T^* o(D_+)$,
where $o(D_\pm)$ are the orientations for $\lie{s}_\pm$, $D_+$
denotes the disk lying in the patch of color $1$, and $T$ is the
conjugation isomorphism; the normalizing sign
$(-1)^{n_p(1 + g_p + m_p)}$ of Proposition \ref{prop:noends} is $+1$
for a disk.  This is the orientation used for the pair in Proposition
\ref{prop:noends}.  Hence $o_h = o(D_-) \wedge T^* o(D_+)$, and the
isomorphism induced by the homotopy is orientation preserving, which
completes the proof.
\end{proof}

In the quilted case there are four types of gluing:
gluing at an interior node, gluing at a node on the true boundary, gluing at a node on a seam, and gluing along strip-like ends.  A {\em nodal quilted surface} 
is a quilted surface together with pairs of points $(w_-,w_+)$ in the interior of a patch, on the true boundary of a patch, or on a seam, or pairs of strip-like ends carrying the same seam or boundary conditions.
For a node on a seam $\sigma$ the two points $w_\pm$
lie on seams and a neighborhood of $w_\pm$ in $\ul{S}$
consists of a half-disk in each of the two patches
$p_-(\sigma), p_+(\sigma)$ adjacent to the seam. 
Deforming such a node means simultaneously 
deforming the node in the two adjacent patches, with the same gluing parameter.
Thus the neck is a strip (resp. cylinder) in each of the two patches glued along the seam.  
Thus gluing at a seam is, in each patch, gluing at a boundary node.  

\begin{lemma} {\rm (Signs for gluing of quilted surfaces)}  \label{lem:quiltedgluing}
Suppose that $D_{\ul{\ti{E}}, \ul{\ti{F}}}$ is obtained from
$D_{\ul{E},\ul{F}}$ by gluing one of the four types above, and that
patches joined by the gluing have the same color.  Then as in Corollary
\ref{cor:deformcor} we have a natural isomorphism
\[ \det(D_{\ul{E},\ul{F}}) \to \det(D_{\ul{\ti{E}},\ul{\ti{F}}}) , \]
whose sign with respect to the orientations of Proposition
\ref{prop:noends} is as follows.
\begin{enumerate}
\item \label{item:qgnode} For gluing at an interior node of a patch $p$ or at a node on the
  true boundary of $p$, the sign is that of Theorem
  \ref{thm:gluingsigns} and Corollary \ref{cor:nonadjacent} for the
  corresponding gluing of the unquilted surface.
\item \label{item:qgseam} For gluing at a node on a seam $\sigma$, the sign is the product of the
  boundary-node signs of Theorem \ref{thm:gluingsigns} in the patches
  $p_\pm(\sigma)$.
\item \label{item:qgends} For gluing a pair of quilted ends $\ul{e}_\pm$ of distinct
  components $\ul{S}_\pm$, the sign is $\kappa$ times the sign of
  Theorem \ref{thm:gluingsigns} \eqref{case:gluingends}, together with
  Corollary \ref{cor:nonadjacent}, for a single pair of ends with end
  datum $E_{\ul{e}}$ and the ordering $(\ul{e}_-,\ul{e}_+)$; that is, with $\rank(F)$ replaced by
  $R = \sum_{i=1}^k \rank(E_{q(i)})$ and $\Ind(D_e^\pm)$ by
  $\sum_{i=1}^k \Ind(D_{e,i}^\pm)$.  Here
  \[ \kappa = (-1)^{(N_- + R)\alpha + (N_+ + R)\beta} \,
  \kappa_{\on{merge}} , \]
  where $N_\pm = \sum_{p} \rank(E_p) m_p$, the sum over the patches $p$
  of $\ul{S}_\pm$, with $m_p$ the number of boundary components of the
  closed-off patch; 
  \[ \alpha = \sum_f (\rank(E_f) + \Ind(D^-_f)), \] 
  the
  sum over the incoming ends $f \neq \ul{e}_-$ of $\ul{S}_+$, with
  $E_f$ the end datum of $f$; 
  \[ \beta = \sum_f (\rank(E_f) +
  \Ind(D^+_f)) , \]
  the sum over the outgoing ends $f \neq \ul{e}_+$ of
  $\ul{S}_-$; and 
  \[ \kappa_{\on{merge}} = K_1 K_2 \] 
  is the sign
  \eqref{eq:kmerge} defined in the proof, which depends only on the
  ranks and orderings of the patches and on which patches contain
  which strips of $\ul{e}_\pm$.  In particular, if $\ul{S}_\pm$ are
  quilted strips with the patches ordered by position, then
  $\kappa = 1$ and the gluing is orientation preserving.
  For gluing two quilted ends of the same component, the gluing
  isomorphism has a sign depending only on ranks, orderings and the
  indices mod $2$ of the operators at the ends: it is the product of
  the Koszul sign $\kappa$ of Theorem \ref{thm:gluingsigns}
  \eqref{case:selfgluing}, computed with the factors of the quilted
  surface, the regrouping signs of the proof, and the signs for
  deforming the nodes strip by strip as in the proof, where a strip
  joining two pieces that are already joined contributes the signs
  $\eps_1 \eps_2$ of Theorem \ref{thm:gluingsigns}
  \eqref{case:selfgluing} with $\rank(F)$ replaced by the rank of that
  strip.  Its explicit value is not needed below: it enters only
  through the sign $\nu$ of Proposition \ref{prop:arc}, and in
  Proposition \ref{prop:delta} only the fact that $\nu$ does not depend
  on the map $u$ is used.
\end{enumerate}
These signs depend only on ranks, orderings and indices mod $2$ of the
operators at the ends; they do not depend on the colors, and are
independent of the choice of deformation to split form.
\end{lemma}

\begin{proof}
{\em Step 1: Compatible deformations.}  The deformation 
$\ul{F}_\delta$ of $\ul{F}$ to split form in the proof of
Proposition \ref{prop:noends} may be chosen to be of split form already in a neighborhood of each node and each pair of glued ends. For a pair of glued ends this holds 
because the deformation was constructed on the strip-like 
ends separately, and we use the same deformation on both
ends of a glued pair.   Note that the glued ends carry the same seam conditions.  For a node the seam condition in a small neighborhood of the node is a bundle over a contractible set, and the space of totally real subspaces $U(n)/O(n)$ is connected, so it can be deformed in a neighborhood of the node to a subbundle of split form, and the deformation extended to the rest of the seam. With these choices, $\ul{F}_\delta$ induces a deformation 
$\ul{\ti{F}}_\delta$ of the seam conditions of the glued surface $\ul{\ti{S}}$ to 
split form $\ul{\ti{F}}^{\on{split}}$, and the relative spin structure
on $\ul{F}^{\on{split}}$ induces one on 
$\ul{\ti{F}}^{\on{split}}$ as in the unquilted case. 

\vskip .1in \noindent {\em Step 2: Compatibility with the deformation isomorphisms.}
By Corollary \ref{cor:deformcor} and Remark \ref{rem:families}, for each $\delta \in [0,1]$
one has a gluing isomorphism 
$\det(D_{\ul{E},\ul{F}_\delta}) \to 
\det(D_{\ul{\ti{E}}, \ul{\ti{F}}_\delta})$, depending continuously on $\delta$ by Proposition \ref{prop:doublecover}.  The deformation 
isomorphisms 
\[ \det(D_{\ul{E},\ul{F}_0}) \to 
\det(D_{\ul{E},\ul{F}_1}), \quad \det(D_{\ul{\ti{E}},\ul{\ti{F}}_0}) \to 
\det(D_{\ul{\ti{E}},\ul{\ti{F}}_1}) \] 
are given by 
parallel transport in the families of Fredholm operators. 
The resulting square commutes up to a sign which is continuous in $\delta$, hence constant, and equal to 
$+1$ at $\delta = 1$ where the two gluing isomorphisms agree by definition.  By the definition of orientations in Proposition \ref{prop:noends}, the gluing sign for $(\ul{E},\ul{F})$ therefore equals the gluing sign for the split problem $(\ul{E}, \ul{F}^{\on{split}})$.

\vskip .1in \noindent {\em Step 3:  The split problem.}  For $\ul{F}^{\on{split}}$
the operator is the direct sum $\bigoplus_p D_{E_p,F_p^{\on{split}}}$, 
oriented as an unquilted surface whose components are the patches.  
Gluing at an interior node or at a node on the true boundary of a patch
$p$ preserves $\Ind(D_{E_p, F_p^{\on{split}}})$.  It acts
by the unquilted gluing isomorphism on the factor for $p$ and by the
identity on the other factors, so its sign is that of Theorem
\ref{thm:gluingsigns}.  If $p$ has color $1$, then by Lemma
\ref{lem:colorshift}, applied to each component, the orientation of
its factor agrees with the orientation for color $0$ for the relative
spin structure shifted by $w_2(E_p)$.  This shift changes the
orientation by $(-1)^{\mu/2}$, where $\mu$ is the Maslov index of the
closed-off problem, and $\mu$ is additive under gluing; so the sign is
again that of Theorem \ref{thm:gluingsigns}. 
A node on a seam $\sigma$ is, in split form, a pair of boundary nodes
in the patches $p_-(\sigma)$ and $p_+(\sigma)$.  The gluing isomorphism
is the tensor product of the two boundary-node gluing isomorphisms on
the corresponding factors.  Each preserves the index of its factor, so
no factors are permuted and the sign is the product of the two
boundary-node signs. 

Consider next the gluing of a pair of quilted ends $\ul{e}_\pm$ of
distinct components $\ul{S}_\pm$.  By the convention in Proposition
\ref{prop:noends}, the factors of \eqref{eq:order} associated to
$\ul{e}_\pm$ are the determinant lines $\det(D^\pm_{\ul{e}})$ and
$\Lambda^{\max}(\Gamma_{\ul{e}}(0)^\dual)$ of the direct sum end datum
$E_{\ul{e}}$.  The gluing joins each strip $e_{+,i}$ of $\ul{e}_+$ to
the corresponding strip $e_{-,i}$ of $\ul{e}_-$; the patches
$q_+(i)$ of $\ul{S}_-$ and $q_-(i)$ of $\ul{S}_+$ containing them have
the same color.  For a patch of color $1$, by Lemma
\ref{lem:colorshift} the orientation of its factor agrees with the
orientation for color $0$ for the relative spin structure shifted by
$w_2(E_q)$, for the corresponding end orientations of Remark
\ref{rem:compatops} \eqref{item:conjcompat}; as for nodes, no additional sign arises from
the colors.

We follow the exchanges (i)--(iii) in the proof of Theorem
\ref{thm:gluingsigns} \eqref{case:gluingends}, with the closed-off
factor of $S_\pm$ replaced by
$C_\pm = \bigotimes_p \det(D^c_p)$, the product over the patches of
$\ul{S}_\pm$ in their order.  Since the split boundary and seam
conditions are oriented, the Maslov indices $\mu_p$ of the closed-off
problems are even.  Thus 
\[ \Ind(D^c_p) = \rank(E_p)(2 - 2g_p - m_p) + \mu_p \equiv \rank(E_p)
m_p \pmod 2 \]
and $C_\pm$ has parity $N_\pm$, while the fibers
$\Lambda^{\max}(\Gamma_{\ul{e}}(0)^\dual)$ have parity $R$.  The factor
$C_-$ is passed in (i) by the determinant lines associated to the remaining incoming ends of $\ul{S}_+$, and $C_+$ is passed in (ii) and (iii) by the determinant lines associated to the remaining outgoing ends
of $\ul{S}_-$, of total parities $\alpha$ resp. $\beta$; all other
exchanges are as in the unquilted case with $\rank(F)$ replaced by
$R$.  Replacing the parity $R$ of the closed-off factors by $N_\pm$
therefore gives the factor $(-1)^{(N_- + R)\alpha + (N_+ + R)\beta}$.

We record the signs produced by permuting factors.  After these exchanges, and gluing the caps
$\det(D^+_{\ul{e}}) \otimes \det(D^-_{\ul{e}}) \to
\det(\ti{D}_{\ul{e}})$, which is orientation preserving by
compatibility, the factors at the glued end form the block
\[ C_- \otimes \Lambda^{\max}(\Gamma_{\ul{e}}(0)^\dual) \otimes
\det(\ti{D}_{\ul{e}}) \otimes \Lambda^{\max}(\Gamma_{\ul{e}}(0)^\dual)
\otimes C_+ .\]
In split form the linearized operator $\ti{D}_{\ul{e}} = \bigoplus_i \delta_i$ is the direct
sum of the trivial operators $\delta_i$ on the disks obtained by gluing
the caps of the strips $e_{\pm,i}$, $\Gamma_{\ul{e}}(0) = \bigoplus_i
\Gamma_{e,i}(0)$, Furthermore, each of the two nodes splits into nodes joining
the disk of $\delta_i$ to the closed-off patches $q_+(i)$
resp. $q_-(i)$.  Let 
\[ n_i = \rank(E_{q(i)}), \quad T_i = \Lambda^{\max}(\Gamma_{e,i}(0)^\dual) \otimes \det(\delta_i)
\otimes \Lambda^{\max}(\Gamma_{e,i}(0)^\dual) \] 
which is a determinant line of parity $n_i$.
Using the direct sum identifications, regrouping the three middle
factors as $T_1 \otimes \ldots \otimes T_k$ has Koszul sign
$K_1 = (-1)^{\sum_{i<j} n_i n_j}$.

The nodes are now deformed one strip at a time, in the order of the
strips.  Suppose first that no strip joins two patches which have
already been joined through earlier strips.  Define words in the factors inductively:
initially each patch $p$ of $\ul{S}_\pm$ has the word $c_p =
\det(D^c_p)$, of parity $\rank(E_p) m_p$; at the $i$-th step the words
$w$ and $w'$ of the pieces containing $q_+(i)$ resp. $q_-(i)$ are
replaced by the single word $w \, T_i \, w'$.  Let $w(\ti{p})$ denote the
final word of a patch $\ti{p}$ of the glued surface, and order the patches
of the glued surface as those of $\ul{S}_-$, with a merged patch in
the position of its first constituent in $\ul{S}_-$, followed by the
remaining patches of $\ul{S}_+$.  Let $K_2$ be the Koszul sign of the
permutation
\begin{equation} \label{eq:kmerge}
C_- \, T_1 \ldots T_k \, C_+ \ \mapsto \ w(\ti{p}_1) \ldots w(\ti{p}_r) ,
\qquad \kappa_{\on{merge}} = K_1 K_2 .
\end{equation}
Each step $w \, T_i \, w' \mapsto \det(D^c_{w \cup w'})$ deforms the
two nodes of a closed-off factor, a disk of rank $n_i$ and a second
closed-off factor, in the configuration of the single pair of ends in
the proof of Theorem \ref{thm:gluingsigns} \eqref{case:gluingends},
and so has sign $(\pm 1)^{n_i}$ with the sign there, for the ordering
$(e_{-,i},e_{+,i})$ of the strips induced by the ordering
$(\ul{e}_-,\ul{e}_+)$.  If the closed-off pieces joined at a step have
several boundary components, Corollary \ref{cor:nonadjacent} contributes
in addition the sign for making the joined boundary components adjacent,
which depends only on ranks and orderings; for the quilts of Section
\ref{sec:hfsec} all closed-off patches meeting a glued end are disks
and no such sign arises.  The product of
these signs is $(\pm 1)^R$, the corresponding factor for the end
datum $E_{\ul{e}}$.  If a strip joins two patches already joined, the
corresponding step is a gluing of a piece to itself: the factor $T_i$
is moved to the right of the word $w$ of that piece (this exchange is
part of the permutation defining $K_2$), the word becomes $w \, T_i$,
and the step has the signs $\eps_1 \eps_2$ of Theorem
\ref{thm:gluingsigns} \eqref{case:selfgluing} with $\rank(F)$
replaced by $n_i$.  If the glued surfaces are quilted strips with patches ordered by position, then $\alpha = \beta = 0$, each patch is a disk after
closing off, and $w(\ti{p}_i) = c_{q_+(i)} T_i c_{q_-(i)}$, all of parity
$n_i$; hence $K_2 = (-1)^{\sum_{i<j} n_i n_j} = K_1$ and $\kappa = 1$
in this case.  For two quilted ends of the same component the same
steps apply, with the single closed-off factor $C = \bigotimes_p
\det(D^c_p)$ in place of $C_\pm$ and the exchanges (i)--(iii)
replaced by the Koszul sign $\kappa$ of Theorem \ref{thm:gluingsigns}
\eqref{case:selfgluing}; this gives the description in the statement.

All of these signs depend only on ranks, orderings, and parities of the indices of the operators at the ends.
By Step 1 in the proof of Lemma \ref{lem:indepsplit}, two split forms
differ by loops of even Maslov index, so these parities are
independent of the choice of split form.  The orientations themselves
are independent of that choice by Lemma \ref{lem:indepsplit}.
\end{proof}

For the quilted surfaces defining the compositions in the generalized
Fukaya category the signs $\kappa$ are in general non-trivial; they are
absorbed by the normalization of the composition maps \eqref{eq:musharp}, which gives the associativity
relation there.

\begin{lemma} \label{lem:colorshift} {\rm (Change of color as a shift
of background class)} Let $p$ be a patch of a quilted surface with
oriented split boundary and seam conditions $F_p^{\on{split}}$,
equipped with a relative spin structure, and let $n_p =
\rank_\C(E_p)$.  Let $g_p$ and $m_p$ be the genus and the number of
boundary components of the patch closed off at its strip-like ends by
the caps of Proposition \ref{prop:disktriv}.  Consider the following
two orientations of the single determinant line
$\det(D_{E_p,F_p^{\on{split}}})$, for corresponding end orientations as
in Remark \ref{rem:compatops} \eqref{item:conjcompat}:
\begin{enumerate}
\item \label{item:firsto} the orientation $T^* o$, where $o$ is the orientation of the
  conjugate problem $D_{E_p^-, F_p^{\on{split},-}}$ given by the
  construction in the unquilted case for the given relative spin
  structure; and
\item \label{item:secondo} the orientation given by the construction in the unquilted case
  for the relative spin structure shifted by $w_2(E_p)$ as in
  Proposition \ref{prop:double} \eqref{th}, using the isomorphism
  $E_p|_{\partial} \cong F_p^{\on{split}} \otimes \C$.
\end{enumerate}
These orientations differ by $(-1)^{n_p(1 + g_p + m_p)}$.  Thus the
orientation of a patch of color $1$ in Proposition \ref{prop:noends}
is the orientation \eqref{item:secondo}.
\end{lemma}

\begin{proof}  In the isomorphism \eqref{eq:order}, 
the conjugation $T$ preserves the orientations of the caps, by the definition of the conjugate end orientations in 
Remarks \ref{rem:enddatumsums} \eqref{item:conjend} and
\ref{rem:compatops} \eqref{item:conjcompat}, and of the totally real fibers $\Gamma_e(0)$.
 By Remark 
\ref{rem:permuteremark} \eqref{item:conjlin}, $T^* o$ differs from the orientation given by the unquilted construction for the same relative spin structure by $(-1)^{(\Ind(D^c) - m_p n_p)/2}$ where $D^c$ is the capped-off  operator. By Riemann-Roch, 
$\Ind(D^c) = n_p(2 - 2g_p - m_p) + \mu$, where 
$\mu$ is the Maslov index of the capped-off operator.
The index is even since the boundary conditions are oriented. 
Hence 
\[ (\Ind(D^c) - m_p n_p)/2 \equiv n_p( 1 + g_p + m_p) + \mu/2 \mod 2 .\] 
On the other hand, shifting the relative spin structure
by $w_2(E_p)$ changes the stable trivializations of the boundary conditions by $(-1)^{\mu/2}$, by Lemma \ref{lem:togglesame} with $U = E_p$, 
and hence the orientation by $(-1)^{\mu/2}$, by 
\cite[Proposition 8.1.16]{fooo}.  The two factors
$(-1)^{\mu/2}$ cancel. \end{proof}

\begin{remark} {\rm (Comparison with Solomon \cite{so:lb})}
The sign $(-1)^{(\Ind(D^c) - m_p n_p)/2}$ of the conjugation agrees 
with \cite[Proposition 2.12]{so:lb} for oriented boundary conditions, 
namely $\frac{1}{2} \mu(\mu+1) + n_p(1 - g_p) + m_p n_p$. Solomon's conventions \cite{so:lb} for determinant lines, duals and the ordering of the factors in the pinching isomorphism differ from ours but this does not affect the sign:  The operator and its conjugate are oriented by the same construction, and the conjugation acts factor by factor, by $+1$ on the disks with trivial boundary conditions,  by the complex index on the closed component, and by $\rank_\C(E)$ on each node fiber.
\end{remark}

\begin{corollary} \label{cor:colorindep} {\rm (Change of coloring)}
Let $\w$ be a coloring of a quilted surface $\ul{S}$, and let $\w'$ be
the coloring obtained by replacing $\w$ with $1 - \w$ on a union $\cQ$
of connected components of $\ul{S}$; for connected $\ul{S}$ these are
the only two colorings, and $\cQ$ consists of all patches.  Then the
orientation of Proposition \ref{prop:noends} for $\w'$ and a given
relative spin structure on $\ul{F}$ equals the orientation for $\w$ and
the relative spin structure obtained by shifting the background
classes of all patches in $\cQ$ by $w_2$, as in Proposition
\ref{prop:double} \eqref{th} applied to each boundary and seam
condition $F$ of $\cQ$, using that the restriction of
$\bigoplus_q E_q$ to $F$ is $F \otimes \C$.
\end{corollary}

\begin{proof}
Since $\w$ and $\w'$ are both colorings, the orientations for both are
independent of the deformation to split form by Lemma
\ref{lem:indepsplit}, and we may use the same deformation for both.
In split form the operator is the direct sum of the patch operators,
and the shift of a seam condition by the sum of the classes $w_2$ of
the two adjacent patches is the sum of the shifts of the two factors.
The two orientations differ only in the factors
$\det(D_{E_q,F_q^{\on{split}}})$, $q \in \cQ$.  If $\w(q) = 0$, the
factor is oriented for $\w'$ by the orientation \eqref{item:secondo} of Lemma
\ref{lem:colorshift}. This orientation is the orientation for color $0$ and the
shifted relative spin structure.  If $\w(q) = 1$, the factor is
oriented for $\w'$ by the construction in the unquilted case, which by
Lemma \ref{lem:colorshift}, applied to the shifted relative spin
structure, is the orientation for color $1$ and the shifted relative
spin structure, since shifting twice gives back the original relative
spin structure.  The ordering of the factors in \eqref{eq:order} is
unchanged.
\end{proof}

\begin{remark}
For Lagrangian boundary and seam conditions the shift in Corollary
\ref{cor:colorindep} is always defined: for a Lagrangian correspondence
$L \subset M_-^- \times M_+$ one has
\[ (w_2(M_-) + w_2(M_+))|_L = w_2(TL \oplus TL) = 0 . \]
By contrast, a
shift of the background class of only one of the two patches adjacent
to a seam need not be possible. \end{remark}

\subsection{Inserting a diagonal for Cauchy-Riemann operators}

In this section, we explore the effect of adding an additional seam
with seam condition given by a diagonal. 

\begin{definition} 
\begin{enumerate} 
\item {\rm (Adding a seam to a quilted surface)} Let $\ul{S}$ be a
  quilted surface, $S_p$ a patch, and $I \subset S_p$ an embedded
  one-manifold.  The {\em surface obtained by adding a seam} is the
  surface $\ul{S}^{\sDel}$ dividing the patch $S_p= S_p' \cup_I S_p''
  $ into two patches $S_p',S_p''$ joined by the seam $\sigma$ with
  image $I$:
\[\ul{S}^{\sDel} = \ul{S}/ (S_p \mapsto S_p', S_p'') .\]
The ordering of the connected components of the patches
$\ul{S}^{\sDel}$ is such that $S_p''$ follows $S_p'$ immediately (or
vice-versa), and the new boundary component of $S_p'$ (resp. $S_p''$)
is last (resp. first) in the ordering of boundary components.  See
Figure \ref{fig:insert} below.
\item {\rm (Adding a diagonal seam condition)} Let $(\ul{E},\ul{F})$
  be a collection of bundles with totally real seam and boundary
  conditions on $\ul{S}$, and suppose that $\ul{S}^{\sDel}$ is
  obtained by adding a seam $\sigma$.  The pair
\[ \ul{E}^{\sDel} := \ul{E}/(E_p \mapsto (E_p,E_p)), \quad F^{\sDel}
  = (\ul{F},\Delta) \]
obtained by {\em adding a diagonal seam condition} is the pair
obtained from $\ul{E},\ul{F}$ by replacing $E_p$ with two copies and
assigning to the new seam $\sigma$ the diagonal sub-bundle $\Delta$ of $E_p
\oplus E_p$.
\end{enumerate}
\end{definition}  

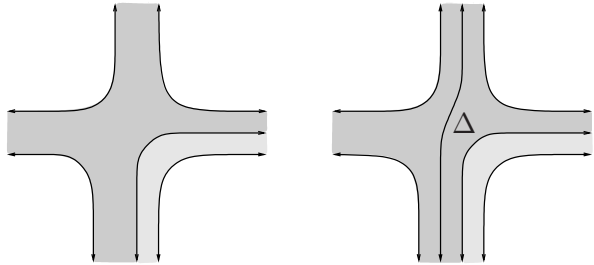
\begin{figure}[ht]
\begin{picture}(0,0)%
\includegraphics{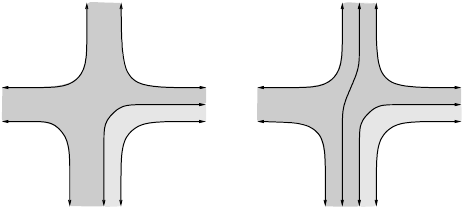}%
\end{picture}%
\setlength{\unitlength}{4144sp}%
\begingroup\makeatletter\ifx\SetFigFont\undefined%
\gdef\SetFigFont#1#2#3#4#5{%
  \reset@font\fontsize{#1}{#2pt}%
  \fontfamily{#3}\fontseries{#4}\fontshape{#5}%
  \selectfont}%
\fi\endgroup%
\begin{picture}(3528,1581)(889,-1369)
\put(3566,-579){\makebox(0,0)[lb]{\smash{{$\Delta$}%
}}}
\end{picture}%
\caption{Inserting a seam in a quilted surface}
\label{fig:insert}
\end{figure}

\begin{remark} \label{rem:DDp}
 {\rm (Identification of determinant lines obtained by adding a seam)}
 Suppose that $\ul{F}$ is equipped with a relative spin structure.
 Let $\ul{D}$ be a Cauchy-Riemann operator for $(\ul{E},\ul{F})$,
 $\ul{S}^{\sDel},\ul{E}^{\sDel}, \ul{F}^{\sDel}$ are obtained by
 adding a seam $\sigma$ with diagonal seam condition $\Delta$, and $\ul{D}^{\sDel}$ is
 the Cauchy-Riemann operator for $(\ul{E}^{\sDel},\ul{F}^{\sDel})$
 obtained from $\ul{D}$.  There is a canonical identification of
 kernels and cokernels
\[ \ker(\ul{D}) \to \ker(\ul{D}^{\sDel}), \quad 
\coker(\ul{D}) \to \coker(\ul{D}^{\sDel}) \] 
given by patching together the restrictions to the two components
obtained by the division.  Hence we have an isomorphism of determinant
lines 
\[ \det(\ul{D}) \to \det(\ul{D}^{\sDel}) .\]
\end{remark} 

\begin{definition} {\rm (Relative spin structures for bundles obtained
by inserting seams)} \label{def:relspinstrs} Let $\ul{S}$ be a quilted
  surface and $\ul{S}^{\sDel}$ the surface obtained by inserting a seam
  into a patch $S_p$.  Let $\ul{E},\ul{F}$ be bundles with
  boundary/seam conditions on $\ul{S}$.
\begin{enumerate} 
\item {\rm (The inserted seam is separating)}.  Suppose that the
  inserted seam $\sigma$ divides $\ul{S}^{\sDel}$ into quilted
  surfaces $\ul{S}^{\sDel} = \ul{S}_+ \cup \ul{S}_-$.  Let $\ul{S}$
  be connected.  The collection $\ul{F}^{\sDel}$ has a canonical
  relative spin structure, given a choice of component of the
  complement of $\sigma$.  Indeed, the diagonal $\Delta_p$ is
  isomorphic to $E_p$, via projection on either factor (the two
  projections agree as maps to $E_p$), hence has
  a canonical relative spin structure as in Proposition \ref{prop:double}.  The
  background classes
\[ b(F_\sigma^{\sDel}) = b(F_\sigma) + (w_2(E_{p_-(\sigma)}) ,
  w_2(E_{p_+(\sigma)}) ), \quad F_\sigma \subset \ul{S}_\pm \]
for the relative spin structure on the components $F_\sigma$ of
$\ul{F}^{\sDel}$ labelling seams in $\ul{S}_\pm$ differ from
those of $\ul{F}$ by adding $w_2(E_p)$ to all the background classes
for components on one side:
\begin{equation} \label{eq:bshift}
b_p \mapsto b_p + w_2(E_p),  S_p \subset \ul{S}_\pm \end{equation}
where $\ul{S}_\pm$ is either $\ul{S}_+$ or $\ul{S}_-$, one side of the
new seam $\sigma$.
\item {\rm (The inserted seam is not separating)} The same
  construction assigns to $\ul{F}^{\sDel}$ a canonical relative spin
  structure after adding {\em two} new seams $\sigma_-, \sigma_+$ separating $\ul{S}$ into
  two components.  The patches in one of the components have shifted
  background classes in the sense of \eqref{eq:bshift}.
\end{enumerate} 
\end{definition}  

We wish to examine in each of these cases the effect of the insertion on  orientations.  We begin with the following simple
case:

\begin{proposition} \label{prop:circle} {\rm (Insertion of a separating circle)} 
Suppose that $\ul{S}^{\sDel}$ is obtained from a quilted surface
$\ul{S}$ by inserting a seam $\sigma$ into a patch $S_p$, that $\sigma$
is diffeomorphic to a circle disjoint from the strip-like ends, and
that $\sigma$ separates $\ul{S}^{\sDel}$ into $\ul{S}_- \cup \ul{S}_+$,
where $\ul{S}_-$ contains the first patch of $\ul{S}^{\sDel}$, and
$S_p' \subset \ul{S}_-$ and $S_p'' \subset \ul{S}_+$.  Let $\w$ be a
coloring of the patches of $\ul{S}$, and equip $\ul{S}^{\sDel}$ with
the coloring that agrees with $\w$ on $\ul{S}_-$ and with $1 - \w$ on
$\ul{S}_+$.  Let $(\ul{E},\ul{F})$ be equipped with a relative spin
structure, so that the components of $\ul{F}$ are oriented and the
split forms of Proposition \ref{prop:noends} have oriented factors,
and let $(\ul{E}^{\sDel}, \ul{F}^{\sDel})$ be obtained by inserting a
diagonal $\Delta$, oriented by projection onto either factor $E_p$
with its complex orientation, with the
relative spin structure of Definition \ref{def:relspinstrs} with the
shift on $\ul{S}_+$.  Then the isomorphism of Remark \ref{rem:DDp} is
orientation preserving.
\end{proposition}

\begin{proof}  We deal first with the case of a single unquilted two-sphere. 
That is, suppose that $\ul{S}$ has a single component $S$ isomorphic
to the two-sphere with standard complex structure, $E \to S$ is
trivial, and $D$ is the standard Cauchy-Riemann operator.  The
orientation for $\ul{D}^{\sDel}$ is defined by deforming $\Delta$ to a
condition of split form as in Proposition \ref{prop:noends}.  Let
$F_t^{\sDel} \to \sigma, t \in [0,1]$ denote the family of seam
conditions in the deformation.  If $F_t^{\sDel}$ is constant along
the seam (that is, a trivial bundle for each $t \in [0,1]$) then the
induced family $\ul{D}_t^{\sDel}$ of Cauchy-Riemann operators is
surjective, with kernel isomorphic to any fiber of $F_t^{\sDel}$ by
evaluation at a point $z \in S$ on the seam:
\[  \ker(\ul{D}_t^{\sDel}) \cong (F_t^{\sDel})_z, \quad \xi \mapsto \xi(z) .\]
Hence the orientation on $\ul{D}_0^{\sDel}$ is induced by evaluation
at a point on the seam and the orientation on the fibers of $\Delta$.
On the other hand, the orientation on $\det(D)$ is induced by the
complex structure on $E$.  The statement of the Proposition follows since the projection of $\Delta$ on either factor is orientation-preserving.
This holds for either coloring of the hemispheres, since complex
conjugation acts trivially on constant sections and the normalizing
sign of Proposition \ref{prop:noends} is trivial for a disk.

We reduce the general case to case of a single two-sphere by deforming
the surface to a surface with nodes, so that an unquilted two-sphere
is created by the deformation.  Suppose that $S_p$ is the component of
$\ul{S}$ containing $\sigma$.  Choose a trivialization $\tau''$ of $E_p$ over $S_p''$, which exists since $S_p''$ has non-empty boundary. By assumption $\ul{F}^{\sDel}$ carries the relative spin structure of Definition \ref{def:relspinstrs} with the shift on $\ul{S}_+$.  Under $\tau''$ its restriction to $\Delta$ is the spin structure given by the
trivialization $\tau'' |_\sigma$.
Let $\sigma_\pm$ be small translates of the
seam $\sigma$ to either side.  Contracting the lines $\sigma_\pm$ to
nodes one obtains a nodal surface
\[\ul{S}_{\delta} = (\ul{S}_{-,\delta} \sqcup \ul{S}_{0,\delta} \sqcup
\ul{S}_{+,\delta})/\sim \]
consisting of quilted surfaces $\ul{S}_{-,\delta}, \ul{S}_{0,\delta},
\ul{S}_{+,\delta}$, with $\ul{S}_{0,\delta}$ a sphere.  Applying the
same construction to $\ul{S}^{\sDel}$ yields a surface
\[\ul{S}^{\sDel}_{\delta}= (\ul{S}_{-,\delta} \sqcup
\ul{S}^{\sDel}_{0,\delta} \sqcup \ul{S}_{+,\delta})/ \sim \]
with $\ul{S}^{\sDel}_{0,\delta}$ a quilted sphere.  The deformation is
illustrated in Figure \ref{fig:pinchseam}.
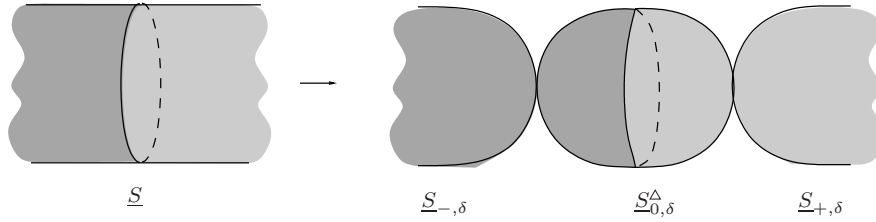
\begin{figure}[ht]  
\begin{picture}(0,0)%
\includegraphics{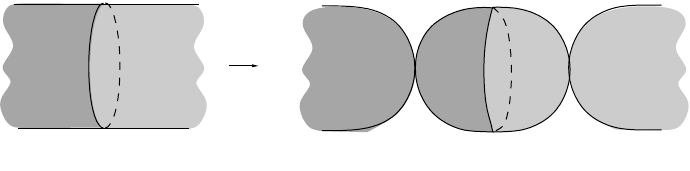}%
\end{picture}%
\setlength{\unitlength}{4144sp}%
\begingroup\makeatletter\ifx\SetFigFont\undefined%
\gdef\SetFigFont#1#2#3#4#5{%
  \reset@font\fontsize{#1}{#2pt}%
  \fontfamily{#3}\fontseries{#4}\fontshape{#5}%
  \selectfont}%
\fi\endgroup%
\begin{picture}(5248,1290)(1057,-1455)
\put(1600,-1367){\makebox(0,0)[lb]{\smash{{\SetFigFont{8}{9.6}{\rmdefault}{\mddefault}{\updefault}{$\ul{S}$}%
}}}}
\put(3359,-1398){\makebox(0,0)[lb]{\smash{{\SetFigFont{8}{9.6}{\rmdefault}{\mddefault}{\updefault}{$\ul{S}_{-,\delta}$}%
}}}}
\put(4635,-1398){\makebox(0,0)[lb]{\smash{{\SetFigFont{8}{9.6}{\rmdefault}{\mddefault}{\updefault}{$\ul{S}^{\sDel}_{0,\delta}$}%
}}}}
\put(5612,-1398){\makebox(0,0)[lb]{\smash{{\SetFigFont{8}{9.6}{\rmdefault}{\mddefault}{\updefault}{$\ul{S}_{+,\delta}$}%
}}}}
\end{picture}%
\caption{Pinching off a seam}
\label{fig:pinchseam}
\end{figure}
Let $\cG$ and $\cG^{\sDel}$ denote the signs of the gluing isomorphisms
$\det(\ul{D}) \to \det(\ul{D}_\delta)$ and $\det(\ul{D}^{\sDel}) \to
\det(\ul{D}^{\sDel}_\delta)$ with respect to the orientations of
Proposition \ref{prop:noends}, and let $P$ denote the sign comparing
the orientations of $\ul{D}_\delta$ and $\ul{D}^{\sDel}_\delta$, each
oriented as a product over the pieces $\ul{S}_{-,\delta}$, the middle
sphere and $\ul{S}_{+,\delta}$ and the fibers at the nodes, ordered in
the same way for both.  Since the isomorphisms of Remark \ref{rem:DDp}
commute with the gluing isomorphisms, the sign of the isomorphism in
the statement is $\cG \, \cG^{\sDel} P$.

We compare the signs on the pieces as follows.  
By Lemma \ref{lem:quiltedgluing} the gluing signs at interior nodes
do not depend on the colors, so $\cG = \cG^{\sDel}$.   The pieces
$\ul{S}_{-,\delta}$ carry the same data and colors for $\ul{S}$ and
$\ul{S}^{\sDel}$.  By the first paragraph, the quilted sphere for
$\ul{S}^{\sDel}$ is oriented by evaluation, compatibly with the
complex orientation $o_\C$; the sphere for $\ul{S}$ is oriented by
\[ \begin{cases} o_\C & \text{ if}  \  \w(p) = 0 \\
 (-1)^{n_p(1 + 0 + 0)} T^* o_\C = (-1)^{n_p} (-1)^{n_p} o_\C = o_\C  & \text{if} \  \w(p) = 1 \end{cases} . \]
 So the difference in orientations on the middle spheres contributes $+1$.  Finally,
$\ul{S}_{+,\delta}$ is a union of connected components of the nodal
surface, all of whose patches change color between $\ul{S}_\delta$ and
$\ul{S}^{\sDel}_\delta$, and whose background classes are shifted by
$w_2$ by the choice of relative spin structure on $\ul{F}^{\sDel}$.  By
Corollary \ref{cor:colorindep} these pieces contribute $+1$.  Hence
$\cG \cG^{\sDel} P = 1$.
\end{proof}

We prove a similar result when the added seam is labelled with the
diagonal and meets the strip-like ends.  

\begin{definition} \label{def:newordering} {\rm (Ordering of components and boundary components
of a surface with an inserted seam)} Suppose that $\ul{S}^{\sDel}$ is
  obtained from $\ul{S}$ by inserting a new seam  $\sigma$ connecting two ends
  in a patch $S_j$, as in Figure \ref{fig:insert}.  
\begin{enumerate} 
\item The ordering of the components of $\ul{S}$ induces an ordering
  of the components of $\ul{S}^{\sDel}$, by replacing the index of the
  old components with those of the new component and ordering the
  component $S_{j,-}$ before $S_{j,+}$.  
\item An ordering of the ends of the components of $\ul{S}$ induces
  an ordering of the ends of $\ul{S}^{\sDel}$, since these are in bijection.
\item The ordering of boundary components of $\ul{S}$ induces an
  ordering of the boundary components for each component of
  $\ul{S}^{\sDel}$: For each old component, the ordering is the same,
  while for the new components $S_{j,\pm}^{\sDel}$ one puts the new
  seam last (resp. first) for $S_{j,-}$ resp. $S_{j,+}$, and the other
  components ordered as before.
\end{enumerate} 
\end{definition} 

\begin{proposition} \label{prop:arc} {\rm (Insertion of a separating arc)}  Suppose that
$\ul{S}^{\sDel}$ is obtained from $\ul{S}$ by inserting a seam $\sigma \cong \R$
into a patch $S_p$, joining an incoming end $e_-$ and an outgoing end $e_+$
of $S_p$, and that $\sigma$ is separating.  Suppose that $\ul{E}^{\sDel}, \ul{F}^{\sDel}$ are obtained from $(\ul{E},\ul{F})$ by labelling $\sigma$
by the diagonal, that the colorings, the relative spin structure and the orientation of $\Delta$ are as in Proposition \ref{prop:circle}, and that the orientations for the ends of $\ul{S}^{\sDel}$ and the orderings of the components and boundary components are induced from those of $\ul{S}$ as in Definition \ref{def:newordering}.
Let $\ul{S}^\sharp$, $\ul{S}^{\sDel,\sharp}$ be the surfaces obtained by 
gluing $e_-$ to $e_+$, so that $\sigma$ becomes a separating circle 
$\sigma^\sharp \subset \ul{S}^{\sDel, \sharp}$.  Then the isomorphism 
$\det(\ul{D}) \to \det(\ul{D}^{\sDel})$ maps $o_{\ul{E},\ul{F}}$
to 
\[ \nu \, o_{\ul{E}^{\sDel}, \ul{F}^{\sDel}} \] 
where $\nu = \nu(\cG) \nu(\ul{\cG})$ is the product of the signs, given
by Lemma \ref{lem:quiltedgluing}, of the gluing isomorphisms $\cG$ for
$\ul{S}$ and $\ul{\cG}$ for $\ul{S}^{\sDel}$ along $e_\pm$.
\end{proposition}

\begin{proof}  
The proof is by a reduction to the case that the new seam $\sigma$ is a circle
in Proposition \ref{prop:circle}.  The surface $\ul{S}$ may be quilted;
the ends $e_\pm$ are then quilted ends, and $\sigma$ divides the strip of
$e_-$ and the strip of $e_+$ lying in $S_p$ into two strips each, lying in
$S_p'$ and $S_p''$.  Gluing the ends $e_\pm$
produces a surface $\ul{S}^{\sharp}$ with two fewer strip-like ends
and a Cauchy-Riemann operator $\ul{D}^{\sharp}$ obtained from $\ul{D}$, and
similarly $\ul{S}^{\sDel,\sharp}$ and $\ul{D}^{\sDel,\sharp}$; see Figure
\ref{fig:splitglue}.

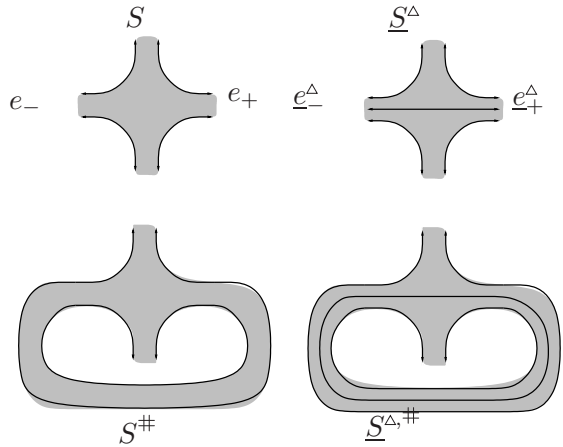
\begin{figure}[ht]
\begin{picture}(0,0)%
\includegraphics{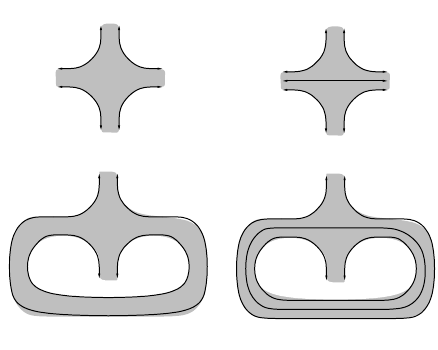}%
\end{picture}%
\setlength{\unitlength}{4144sp}%
\begingroup\makeatletter\ifx\SetFigFont\undefined%
\gdef\SetFigFont#1#2#3#4#5{%
  \reset@font\fontsize{#1}{#2pt}%
  \fontfamily{#3}\fontseries{#4}\fontshape{#5}%
  \selectfont}%
\fi\endgroup%
\begin{picture}(3324,2663)(364,-1995)
\put(1685, 64){\makebox(0,0)[lb]{{{$e_+$}%
}}}
\put(3390, 26){\makebox(0,0)[lb]{{{$\ul{e}_+^{\sDel}$}%
}}}
\put(379, 21){\makebox(0,0)[lb]{{{$e_-$}%
}}}
\put(2078, 32){\makebox(0,0)[lb]{{{$\ul{e}_-^{\sDel}$}%
}}}
\put(2649,514){\makebox(0,0)[lb]{{{$\ul{S}^{\sDel}$}%
}}}
\put(1075,550){\makebox(0,0)[lb]{{{$S$}%
}}}
\put(1032,-1941){\makebox(0,0)[lb]{{{$S^{\sharp}$}%
}}}
\put(2511,-1936){\makebox(0,0)[lb]{{{$\ul{S}^{\sDel,{\sharp}}$}%
}}}
\end{picture}%

\caption{Inserting a diagonal and gluing the ends together}
\label{fig:splitglue}
\end{figure}

We compute the effect of adding a seam by studying the gluing signs
for gluing the two quilted ends.  Suppose that the ordering of the
boundary components of the glued surface is such that the first
boundary component of the new surface corresponds to the boundary of
$S_p$ between $e_-$ and $e_+$.  Consider the gluing isomorphisms from
\eqref{eq:linearglue}
\[ \cG: \det(\ul{D}) \to \det(\ul{D}^{\sharp}), \quad \ul{\cG}: \det(\ul{D}^{\sDel}) \to \det(\ul{D}^{\sDel,\sharp}) .\]%
By Lemma \ref{lem:quiltedgluing}, $\cG$ and $\ul{\cG}$ have signs
$\nu(\cG)$ and $\nu(\ul{\cG})$.  Here $\cG$ glues each strip of $e_-$ to
the corresponding strip of $e_+$, in particular the strips in $S_p$ to
each other, and $\ul{\cG}$ does the same for $\ul{S}^{\sDel}$, gluing
the strip of $S_p'$ at $e_-$ to that at $e_+$, and similarly for
$S_p''$.  Since $\sigma$ joins $e_-$ to $e_+$, strips on the same side
of $\sigma$ are glued to each other, so the glued patches have the same
color, and $\ul{S}_-$ and $\ul{S}_+$ are each glued to themselves.
Hence the circle $\sigma^\sharp$ separates $\ul{S}^{\sDel, \sharp}$
into the glued surfaces $\ul{S}_-^\sharp \ni S_p'$ and
$\ul{S}_+^\sharp \ni S_p''$, with the colorings, relative spin
structure and orderings of Proposition \ref{prop:circle}, so by
Proposition \ref{prop:circle} the natural isomorphism
$\det(\ul{D}^\sharp) \to \det(\ul{D}^{\sDel,\sharp})$ is orientation
preserving.  Since gluing along $e_\pm$ commutes with the isomorphisms of Remark \ref{rem:DDp}, the isomorphism $\det(\ul{D}) \to \det(\ul{D}^{\sDel})$
has sign $\nu(\cG) \nu(\ul{\cG}).$
\end{proof}

\subsection{Orientations for compositions of totally real correspondences} 

\begin{definition}  {\rm (Smooth composition of linear seam conditions)} 
Let $\ul{S}$ be a quilted surface with two adjacent patches $S_0,S_2$,
equipped with complex vector bundles $\ul{E}$ and boundary and seam
conditions $\ul{F}$.  Suppose that $E_1$ is a complex vector bundle
over the seam joining $S_0,S_2$.  Let 
\[F_{01} \subset \ol{E}_0 \times
E_1, \quad F_{12} \subset \ol{E}_1 \times E_2 \]
be totally real seam conditions.  We say that
\[ F_{02} := \pi_{02} ( F_{01} \times_{\Delta_1} F_{12}) \subset \ol{E}_0 \times E_2  \] 
is a {\em smooth composition} of totally real subbundles if the
intersection 
\[(F_{01} \times F_{12}) \pitchfork (E_0 \times \Delta_1 \times
E_2) \subset \ol{E}_0 \times E_1 \times \ol{E}_1 \times E_2 \] 
is transverse.
\end{definition}

\begin{remark} \label{rem:composition}  Let $F_{02}$ be a smooth composition of seam conditions $F_{01}$ and $F_{12}$. 
\begin{enumerate} 
\item {\rm (Relative spin structure for the composition)} Relative
  spin structures for $F_{01}, F_{12}$, and the diagonal induce a
  relative spin structure for $F_{02}$, because of the isomorphism
\[ \pi_{02}^* F_{02} \oplus \Delta_1^\perp \to F_{01} \oplus F_{12}, 
\quad \Delta_1^\perp := \{ (e, - e) \in E_1 \oplus E_1, e \in E_1 \}\]
and the discussion in Proposition \ref{prop:double}.
\item {\rm (Quilted surface obtained by composition)} Let $\ul{S}'$
  denote the quilted surface with two additional seams separating
  $S_0,S_2$.  The surface $\ul{S}'$ contains, in comparison with
  $\ul{S}$, two additional patches $S_1^-,S_1^+$ each isomorphic to
  strips.  Let $\ul{E}'$ be the collection of complex vector bundles,
  equal to $\ul{E}$ on all but the new components where they are given by $E_1$
  (pulled back by projection onto the seam).  Let
\[\ul{F}' = \ul{F}/ (F_{02} \mapsto (F_{01}, \Delta_1, F_{12})) \]
be the collection of boundary and seam conditions obtained by
replacing $F_{02}$ with $F_{01}, \Delta_1, F_{12}$.
\item {\rm (Identification of determinant lines)} \label{item:compdet} 
 Let $\ul{D}'$ be a Cauchy-Riemann operator for the pair $(\ul{E}',\ul{F}')$ on $\ul{S}'$, and let $\ul{D}$ denote the 
  Cauchy-Riemann operator on $\ul{S}$ obtained by omitting the operator for $M_1$.  In the setting of \cite[Theorem 2.2.2]{ww:isom}, the identifications 
\[ \ker(\ul{D}')\to \ker(\ul{D}), \quad \coker(\ul{D}') \to
  \coker(\ul{D}) \]
induce an identification $\det(\ul{D}')
  \to \det(\ul{D})$.
\item {\rm (Orientations for the ends)} \label{item:compends} Orientations for the ends of
  $\ul{S}'$ induce orientations for the ends of $\ul{S}$ as follows.
  Choose a path of subspaces connecting
  $(F_{01} \oplus F_{12})_{\ul{e}'}$ to
  $(F_{02} \oplus \Delta_1^\perp)_{\ul{e}'}$, and a deformation of
  $(\Delta_1)_{\ul{e}'}, (\Delta_1^\perp)_{\ul{e}'}$ to split form so
  that every space in the family has transverse intersection.  After
  this deformation the end datum $E_{\ul{e}'}$ of Proposition
  \ref{prop:noends} is the direct sum of the end datum $E_{\ul{e}}$ for
  $\ul{S}$ and of the summands $E_{1,\ul{e}}$ for the strips of
  $\ul{e}'$ lying in patches labelled $M_1$, which carry constant split
  boundary conditions.  Fix orientations of the incoming caps
  $D^-_{1,\ul{e}}$ of these summands, depending only on the ranks and
  the split boundary conditions; we choose them, and the paths for
  these summands, so that for each pair of such strips at an end the
  cap of the first strip, in the order across the end, has index
  $\equiv \rank(E_1)$ and that of the second has even index, and so that the operator with split
  conditions on a patch labelled $M_1$ with two incoming ends, which
  has index zero and is invertible, is oriented by \eqref{eq:order}
  with its canonical orientation.  The incoming cap $D^-_{\ul{e}}$ for
  $\ul{S}$ is then oriented so that the isomorphism
  $\det(D^-_{\ul{e}'}) \to \det(D^-_{\ul{e}}) \otimes
  \det(D^-_{1,\ul{e}})$ induced by permuting the summands is
  orientation preserving, and the outgoing caps are oriented by
  compatibility as in Definition \ref{def:compat}.
\end{enumerate}
\end{remark} 

\begin{proposition}\label{prop:quiltcyl} {\rm (Orientations for
composition of linear seam conditions)} Suppose that the orientations
for the ends of $\ul{S}$ are induced from those of $\ul{S}'$ as in
Remark \ref{rem:composition}
\eqref{item:compends}, and that the patches of $\ul{S}$ are ordered as
those of $\ul{S}'$ with the patches labelled $M_1$ omitted.  Let $D_1$
denote the operator with the split boundary conditions on the patches
labelled $M_1$; it has index zero and is invertible.  Then the
isomorphism $\det(\ul{D}') \to \det(\ul{D})$ of Remark \ref{rem:composition} \eqref{item:compdet} has sign
$\eps_K \, \eps_\Gamma \, \eps_+ \, \eps_1$, where
\begin{enumerate}
\item the sign $\eps_K$ is the Koszul sign \eqref{eq:signs} of the permutation
  of the factors of \eqref{eq:order} for $\ul{S}'$ which moves the
  factors belonging to the patches labelled $M_1$ to the end, keeping
  the order of the remaining factors: the closed-off factors of these
  patches, in the order of the patches, followed by the factors
  $\det(D^\pm_{1,\ul{e}})$ split off from the caps as in Remark \ref{rem:composition} \eqref{item:compends},
  in the order of the ends (after the permutation of the fibers in
  the definition of $\eps_\Gamma$, the summands of the fibers for the
  patches labelled $M_1$ at each end form a block of even rank);
\item the sign $\eps_\Gamma = \prod_{\ul{e}} (-1)^{\sigma_{\ul{e}}}$, the
  product over the incoming ends, where $\sigma_{\ul{e}}$ is the Koszul
  sign of the permutation of the summands of $\Gamma_{\ul{e}'}(0)$, one
  for each strip of $\ul{e}'$ and of parity the rank of that strip,
  which moves the summands for the strips lying in patches labelled
  $M_1$ to the end and keeps the order of the others.  At an outgoing
  end the same permutation sign occurs both in the factor
  $\Lambda^{\max}(\Gamma_{\ul{e}'}(0)^\dual)$ and in the compatibility of
  $D^+_{\ul{e}'}$, and cancels.  For instance, if an incoming end has two
  strips labelled $M_1$ separated by a single strip labelled $M_2$, then
  $\sigma_{\ul{e}} \equiv \rank(E_1) \rank(E_2)$;
\item the sign $\eps_+ = \prod_{\ul{e}} (-1)^{\Ind(D^+_{1,\ul{e}})
  \Ind(D^-_{\ul{e}})}$, the product over the outgoing ends, is the
  sign by which the compatible orientation of $D^+_{\ul{e}'}$ differs
  from the orientation induced by $\det(D^+_{\ul{e}}) \otimes
  \det(D^+_{1,\ul{e}})$, as in Remark \ref{rem:compatops} \eqref{item:sumcompat}; and
\item the sign $\eps_1$ is the sign of the orientation of $\det(D_1)$ given by
  the factors moved in $\eps_K$, glued as in \eqref{eq:order},
  relative to its canonical orientation.
\end{enumerate}
The sign $\eps_1$ depends only on the ranks and on the incidence of
the patches labelled $M_1$ with the ends.
\end{proposition}

\begin{proof}
To compare orientations we deform the seam conditions to the composed
seam conditions plus a trivial factor.  Namely, there is a canonical
deformation of $F_{01} \oplus F_{12} \oplus \Delta_1$ to
$\sigma_{1423}(F_{02} \oplus \Delta_1^\perp) \oplus \Delta_1$ within
the space of totally real sub-bundles, where $\sigma_{1423}$ is the
isomorphism
\[E_0 \oplus E_1 \oplus E_1 \oplus
  E_2 \to E_0 \oplus E_2 \oplus E_1 \oplus E_1\] 
given by permutation of factors.  Indeed, any complex vector bundle
admits a Hermitian, hence a symplectic structure.  The fiber bundle
of totally real subspaces of maximal dimension is canonically
isomorphic to the Lagrangian Grassmannian.  Hence the existence of the deformation follows
from the symplectic case, considered in \cite[Lemma 3.1.9]{ww:quiltfloer}.

We compare the orientations induced by deformation to split form. 
By Lemma \ref{lem:indepsplit} the orientation of
Proposition \ref{prop:noends} for $\ul{S}'$ is independent of the
deformation to split form, so we may compose this deformation with
deformations of $F_{02}$, $\Delta_1^\perp$ and $\Delta_1$ to split
form.  The resulting split form for $\ul{S}'$ is the split form for
$\ul{S}$ together with the split conditions on the patches labelled
$M_1$, and the operator is the direct sum of the split operator for
$\ul{S}$ and $D_1$.  In \eqref{eq:order} for $\ul{S}'$ the factors of
the patches labelled $M_1$ occur among the closed-off factors and as
summands of the caps and fibers.  By Remark \ref{rem:composition} \eqref{item:compends}, the determinant lines 
for the incoming caps split as $\det(D^-_{\ul{e}}) \otimes \det(D^-_{1,\ul{e}})$ without
sign, while the determinant lines for the outgoing caps, oriented by compatibility, split with
the sign $\eps_+$.  Splitting the fibers
$\Lambda^{\max}(\Gamma_{\ul{e}'}(0)^\dual)$ into the factor for $\ul{S}$ and the
summands for the patches labelled $M_1$ is a permutation of summands whose
ranks need not be even; it gives the sign $\eps_\Gamma$ at the incoming
ends, while at the outgoing ends it is cancelled by the same permutation
in the canonical orientation of the glued disks $\ti{D}_{\ul{e}'}$, which
enters the compatibility of the outgoing caps as in Remark
\ref{rem:compatops} \eqref{item:sumcompat}.  Moving the factors of the patches labelled $M_1$
to the end gives the sign $\eps_K$; the remaining factors are those of
\eqref{eq:order} for $\ul{S}$, and the moved factors orient $D_1$
with the sign $\eps_1$ relative to its canonical orientation.  Since $D_1$
has index zero, the direct sum isomorphism involves no further sign.
\end{proof} 

\subsection{Orientability of families of quilted Cauchy-Riemann operators} 

\begin{proposition}  \label{prop:trivial}
{\rm (Trivializability of the orientation double cover of a family
  with nodal degeneration)} Let $B$ be a stratified space in the sense of Definition \ref{def:stratified}
  of finite dimension, and let $(\ul{S}_b, \ul{E}_b, \ul{F}_b)_{ b \in
B}$ be a family of complex vector bundles with totally real boundary/seam 
conditions over $B$, with relative spin structures on $\ul{F}$ and orientation of the ends.
For each stratum $B_\Gamma$ let $o_\Gamma$ be the orientation of $\det(D_{\ul{E},\ul{F}}) |_{B_\Gamma}$ from 
Proposition \ref{prop:noends}.  Suppose the link of every stratum of positive codimension is connected.  Then 
$\det(D_{\ul{E},\ul{F}})$ is trivializable, by a trivialization whose restriction to each $B_\Gamma$ is $\pm o_\Gamma$.
\end{proposition} 

\begin{proof} We construct the trivialization by induction on the codimension of the strata.   Let $B_{\ge k}$ be the union of strata of dimension $\ge k$.  We construct trivializations $\tau_k$ of $\det(D_{\ul{E},\ul{F}})$  as follows.  For $k = \dim(B)$, define $\tau_k = o_\Gamma$ on each top-dimensional stratum. 
Suppose $\tau_{k+1}$ is given, and let $B_\Gamma$ have dimension $k$. 
By Definition \ref{def:stratified}, there is a neighborhood $U \cong V \times CL$ with $V$ a ball in $B_\Gamma$
and $L$ the connected link.  Then 
\[ U \setminus B_\Gamma \cong V \times L \times (0,\infty) \] 
lies in $ B_{\ge k+1}$.
By Corollary ~\ref{cor:deformcor}, the gluing isomorphisms define a trivialization $\tau_\Gamma$ of $\det|_U $
extending $o_\Gamma |_V$.  The gluing isomorphisms are continuous across strata by associativity in Proposition \ref{prop:assoc}. The ratio of $\tau_\Gamma$ with $\tau_{k+1}$ is locally constant in $ \{ \pm 1 \}$.
Since $V \times L \times (0,\infty)$ is connected, the ratio is a constant $\eps_\Gamma$.  Set $\tau_k = \eps_\Gamma \tau_\Gamma$ on $U$.  Distinct strata of the same dimension have disjoint neighborhoods after shrinking, so these definitions agree on overlaps, and this completes the argument.
\end{proof} 

\begin{remark}\label{rem:nonconnected link}
{\rm (Example of a family with disconnected link)}
The connectedness hypothesis in Proposition~\ref{prop:trivial} cannot be dropped.
Let $S_0$ be a nodal disk with three components $D_1,D_2,D_3$ with no ends and two boundary
nodes: 
\[ S_0 = D_1 \cup_{w_1} D_2 \cup_{w_2} D_3, \quad w_1\in D_1\cap D_2, \quad w_2\in D_2\cap D_3 .\]
Consider the family
$(S_\delta)_{\delta\in[-1/2,1/2]}$, in which $S_\delta$ for $\delta\in(0,1/2]$
is obtained from $S_0$ by smoothing $w_1$ with gluing parameter $\delta$, and
$S_\delta$ for $\delta\in[-1/2,0)$ is obtained by smoothing $w_2$ with gluing
parameter $-\delta$.  Identifying $S_{1/2}\cong S_{-1/2}$ (both are the
nodal disk obtained by smoothing a node) by sending the first circle to the first circle gives a family over the circle
\[ B=[-1/2,1/2]/(-1/2\sim 1/2). \] 
It has a point stratum $\{0\}$
whose link is the two-point set $\{\delta>0\}\sqcup\{\delta<0\}$.
Order the two boundary circles of the fibers over 
$B \smallsetminus \{ 0 \} $ continuously, so that 
near the $w_1$-end they are $(\partial (D_1 \# D_2), \partial D_3)$ and near the $w_2$-end they are 
$(\partial (D_2 \# D_3), \partial D_1)$.

By Theorem \ref{thm:gluingsigns} \eqref{case:bounddiff} the gluing isomorphisms differ from the orientations 
$o_\Gamma$ of Proposition \ref{prop:noends} by signs
$\eps_+ = 1$ and 
$\eps_- = (-1)^{\rank(F)}$ respectively. 
If $\rank(F)$ is odd, there is no trivialization restricting to $\pm o_\Gamma$ on the stratum that extends continuously over $0$, so the determinant line bundle over $B$ is not orientable.  Hence the connectedness of the links
in Proposition \ref{prop:trivial} cannot be dropped.
\end{remark}

\section{Orientations for moduli of pseudoholomorphic quilts} 
\label{sec:qgluing}

In this section, we apply the orientations for Cauchy-Riemann operators
developed in Section \ref{sec:families} to the case of quilted
pseudoholomorphic maps.

\subsection{Construction of orientations for linearized operators
with Lagrangian boundary conditions}

First, we describe the Cauchy-Riemann operators we would like to
orient.  Let $M$ be a symplectic manifold equipped with a compatible
almost complex structure 
\[J: TM \to TM, \quad J^2 = - \on{Id}_{TM}.\]
Let $S$ be a Riemann surface with complex structure 
\[ J_S: TS \to TS, \quad J_S^2 = - \on{Id}_{TS}.\]

\begin{definition} {\rm (Pseudoholomorphic maps with Lagrangian boundary conditions)}  
\begin{enumerate} 
\item A smooth map $u: S \to M$ is {\em pseudoholomorphic} if 
\[\d u \circ j = J \circ \d u \in \Omega^1( S,u^* TM) .\]
\item A {\em Lagrangian submanifold} of $M$ is an embedded submanifold $L
  \subset M$ of half the dimension of $M$, such that the restriction
  of the symplectic form to $L$ vanishes:
\[ \dim(L) = \dim(M)/2, \quad \omega |_L = 0 .\]
\item Suppose that $\partial S$ has components $I_1,\ldots, I_k$.
  Given Lagrangian submanifolds 
\[L_1,\ldots,L_k \subset M \]
a {\em pseudoholomorphic map with Lagrangian boundary conditions} in
$L_1,\ldots,L_k$ is a pseudoholomorphic map
\[u: S \to M, \quad u(I_j) \subset L_j, j = 1,\ldots, k .\]
\end{enumerate} 
\end{definition}  

Recall the linearized Cauchy-Riemann operators associated to
pseudoholomorphic maps and their Fredholm properties:

\begin{remark} {\rm (Fredholm nature of linearized operators)} 
\label{rem:qpres}
\begin{enumerate} 
\item {\rm (Maps with boundary)} Associated to any pseudoholomorphic
  map $u: S \to M$ with strip-like ends and Lagrangian boundary
  conditions in $L$ is a {\em linearized real Cauchy-Riemann operator}
\begin{equation} D_u: \Omega^0(S, u^* TM, (u |_{\partial S})^* TL) \to
  \Omega^{0,1}(S, u^* TM) ,\end{equation}
as in Definition \ref{def:surfstrip}.
\item {\rm (Quilted maps)} Similarly, if $\ul{S}$ is a quilted surface
  with strip-like ends, $\ul{M}$ is a collection of symplectic
  manifolds associated to the patches of $\ul{S}$, and $\ul{L}$ is a
  collection of Lagrangian submanifolds and correspondences associated
  to the boundary and seam components of $\ul{S}$, and $u: \ul{S} \to
  \ul{M}$ is a pseudoholomorphic map with Lagrangian boundary
  conditions transverse at infinity on each strip-like end, we denote
  by $D_u$ the associated linearized Cauchy-Riemann operator.
\item {\rm (Quilted sections)} Let $S$ be a surface with boundary and
  strip-like ends.  A {\em symplectic Lefschetz-Bott fibration}
  \cite{per:lag}, \cite{ww:ex} is a space $E$ equipped with a closed
  two-form $\omega_E$ non-degenerate near the fibers and a compatible
  almost complex structure near the singularities of $\pi$ and a
  projection $E \to S$ with singularities of Morse-Bott type that is
  locally holomorphic near the singularities of $\pi$.  Given a
  Lagrangian boundary condition $Q \subset \partial E$ (that is, a
  sub-fiber-bundle of $E | \partial S$ that is Lagrangian in each
  fiber) and a pseudoholomorphic section $u: S \to E$ let
\[D_u : \Omega^0(S,\partial S; u^* T^\myvert E, u^* T^\myvert Q) \to
  \Omega^{0,1}(S, u^* T^\myvert E) \]
denote the corresponding linearized Cauchy-Riemann operator.
\end{enumerate} 
\end{remark} 

\begin{definition} 
\begin{enumerate} 
\item {\rm (Relative spin structures for collections of Lagrangians)}
  Let $M$ be a symplectic manifold and $\ul{L} = (L_0,\ldots,L_d)$ be
  a sequence of oriented Lagrangian submanifolds in $M$.  A {\em
    relative spin structure} for $(L_0,\ldots,L_d)$ is a stable
  relative spin structure for the immersion $ L_0 \cup \ldots \cup L_d
  \to M .$ In particular, this means that each $L_j$ has a stable relative
  spin structure with the same background class $b \in H^2(M,\Z_2)$.
\item {\rm (Relative spin structures for collections of Lagrangian
  correspondences)} Let $\ul{S}$ be a (possibly quilted) surface
  (possibly) with boundary and strip-like ends, $\ul{M}$ a collection
  of symplectic manifolds, and $\ul{L}$ a collection of boundary and
  seam conditions for $\ul{S}$.  A {\em relative spin structure} for
  $\ul{L}$ with background classes $b_p, p \in \cP$ is a relative spin
  structure for the immersion
\[ \bigcup_{\sigma \in \mathcal{S}} L_\sigma \to \bigcup_{p_1,p_2 \in \cP} M_{p_1}
\times M_{p_2} \]
with respect to the background classes $\pi_1^* b_{p_1} + \pi_2^*
b_{p_2}$.  
\end{enumerate}
\end{definition}  

\begin{remark}
\begin{enumerate} 
\item {\rm (Moduli spaces of quilted trajectories)} Let $\ul{L}$ be a
  periodic sequence of Lagrangian correspondences equipped with
  admissible brane structures with symplectic manifolds
  $M_0,\ldots,M_m$,  and 
\[H_j \in C^\infty([0,1] \times M_j), \quad Y_j \in \Map([0,1],
  \Vect(M_j)), \quad j = 0,\dots, m\] 
a time-dependent Hamiltonian resp.  their Hamiltonian vector fields
for each patch.  Define
\begin{equation*} 
\cI(\ul{L}) := \left\{ \ul{x}=\bigl(x_j: [0,\delta_j] \to
M_j\bigr)_{j=0,\ldots,m} \, \left|
\begin{aligned}
\dot x_j(t) = Y_j(x_j(t)), \\
(x_{j}(\delta_j),x_{j+1}(0)) \in L_{j(j+1)} 
\end{aligned} \right.\right\} 
\end{equation*}
the set of generalized intersection points.  As in standard Floer
theory, the moduli spaces of ``quilted pseudoholomorphic strips''
$\M(\ul{x}^-,\ul{x}^+)$ arise from quotienting out by simultaneous
$\R$-shift in all components $u_j$.  The moduli spaces are regular for
generic domain-dependent almost complex structures and Hamiltonian
perturbations, as explained in \cite{ww:quiltconst}.
\item {\rm (Moduli spaces of quilts)} Given a quilted surface $\ul{S}$
  with patch labels $\ul{M}$ and seam and boundary conditions $\ul{L}$ and
  a collection of limits $\ul{x}_e \in \cI(\ul{L}_e)$ for each
  strip-like end $e$, let 
\[\M(\ul{M},\ul{L},\ul{x})
 = \left\{ \ul{u}: \ul{S} \to \ul{M} \ \left|
  \ \ul{u}(\partial \ul{S}) \subset \ul{L}, \ \eqref{eq:Jhol}, \ \lim_{s
    \to \pm \infty} \ul{u}( \eps_e(s,t)) = \ul{x}_e, \forall e
  \in \cE \right. \right\} \]
denote the moduli space of quilted pseudoholomorphic maps with limits
$\ul{x}_e$ along each strip-like end $e \in \cE$.  The moduli spaces
$\M(\ul{M},\ul{L},\ul{x})$ are regular for generic almost complex
structures and Hamiltonian perturbations, see \cite{ww:quiltfloer}, in
the sense that $\M(\ul{M},\ul{L},\ul{x})$ is cut out of a Banach space
of maps by Fredholm equation with surjective linearized operator $D_u$.
Suppose that $\M(\ul{M},\ul{L},\ul{x})$ is regular.  By definition, its
tangent space at a pseudoholomorphic map $u: \ul{S} \to \ul{M}$ is the
kernel of the linearized operator $D_u$:
\[ T_u \M(\ul{M},\ul{L},\ul{x}) = \ker(D_u) .\] 
If $D_u$ is oriented, then so is $\M(\ul{M},\ul{L},\ul{x})$ at $u$.
\item {\rm (Orientations for quilted trajectories)} Let
  $\M(\ul{x}_+,\ul{x}_-)$ denote the moduli space of quilted
  pseudoholomorphic trajectories from $\ul{x}_-$ to $\ul{x}_+$.  If
  $\M(\ul{x}_+,\ul{x}_-)$ is regular at a trajectory $u$ then the
  tangent space
\[T_{[u]} \M(\ul{x}_+,\ul{x}_-) = \ker(D_u)/\R, \quad 0 =
  \coker(D_u) \]
where the first is the quotient of the linearized operator $D_u$ by
the $\R$-action given by translation.  Thus any orientation for $D_u$
induces an orientation on the tangent space $T_{[u]}
\M(\ul{x}_+,\ul{x}_-)$.
\item {\rm (Orientations for quilted sections)} Let $E \to S$ be a
  symplectic Lefschetz fibration and $Q$ a Lagrangian boundary
  condition.  Denote by $\M(E,Q;\ul{x})$ the moduli space of
  pseudoholomorphic sections $u: S \to E$ with boundary values in $Q$
  and limits $\ul{x}$.  If $\M(E,Q;\ul{x})$ is regular at $u$ then
\[ T_u \M(E,Q;\ul{x}) \cong \ker(D_u), \quad 0 = \coker(D_u) .\]
So any orientation on $D_u$ induces an orientation on
$\M_u(E,Q;\ul{x})$ at $u$.
\end{enumerate} 
\end{remark}

We generalize the discussion in Definition \ref{def:stratified} to pseudoholomorphic quilts:

\begin{definition} \label{def:detfamily}
\begin{enumerate} 
\item {\rm (Families of pseudoholomorphic quilts)} Let $B$ be a stratified
  space as in Definition \ref{def:stratified}, and $\ul{S}$ a family of
  quilts over $B$.  A {\em family of pseudoholomorphic quilts} with domain
  $\ul{S}$ is a triple $(C,f,u)$ consisting of a space $C$, a
  continuous map $f: C \to B$, and a continuous map $u : \ul{S} \times_B C \to
  \ul{M}$ such that
\begin{enumerate} 
\item the restriction of $u$ to each fiber $\ul{S} \times_B \{ c \}$
  is a pseudoholomorphic quilt, and
\item $u$ is continuous with respect to the Gromov topology on maps.
  That is, if $c_{\nu} \to c$ then $u_{f(c_\nu)}$ Gromov converges to
  $u_{f(c)}$.
\end{enumerate} 
\item {\rm (Determinant line bundle for families of quilts)} Given a
  family of pseudoholomorphic quilts, for any $c_\nu \to c$ the linearized
  operator $D_{u_{c_\nu}}$ is canonically deformable to the operator
  obtained from $D_{u_{c}}$ by the gluing construction, by exponential
  decay on the necks.  One obtains a determinant line bundle
  $\det(D_{u_{c}}) \to C$ by Proposition \ref{prop:doublecover}.
\end{enumerate} 
\end{definition} 
\begin{definition} \label{def:regular family}
Let $(C,f,u)$ be a family of pseudoholomorphic quilts over a stratified
space $B$ as in Definition \ref{def:detfamily}. We say the family is {\em regular with exponential decay} if the linearized operators $D_{u_c}$ satisfy 
the following condition:    For every convergent sequence $c_\nu \to c$ in $C$, with $f(c_\nu)$ obtained from $f(c)$ by gluing with parameters $\gamma_\nu \to \infty$,
the operators $D_{u_{c_\nu}}$ converge, uniformly in weighted norms and after
the identifications of Definition \ref{def:gluing}, to the glued operator of $D_{u_c}$.
\end{definition}

\begin{theorem} \label{thm:mainres}
{\rm(Orientations via relative spin structures for families of quilts)} 
Let $\ul{S} \to B$ be a family of quilted surfaces with strip-like ends
over a stratified space $B$, and let $\ul{M},\ul{L}$ be symplectic manifolds for the patches and Lagrangian boundary and seam conditions.  Suppose $\ul{L}$ is equipped 
with a relative spin structure and orientations on the ends. 
Let $u: \ul{S} \times_B C \to \ul{M}$ be a regular family of pseudoholomorphic quilts with exponential decay over $C$, and suppose that the link of every stratum of $B$ of positive codimension is connected. Then the determinant lines $\det(D_{u_c})$ form a topological line bundle over $C$, and this bundle has a trivialization whose restriction to the preimage of each stratum $B_\Gamma$ is, up to a locally constant 
sign, the orientation of Proposition \ref{prop:noends}.
\end{theorem}

\begin{proof}  By definition \ref{def:regular family}, pullback of the gluing isomorphisms of Corollary \ref{cor:deformcor} defines local trivializations of $\det(D_{u_c})$ near every point of $C$ over a stratum of positive codimension. 
By Proposition \ref{prop:assoc} these agree on overlaps.  So $\det(D_u) \to C$ is a topological line bundle, as in Proposition \ref{prop:doublecover}.  The claim follows
from Proposition \ref{prop:trivial} applied to the pullback of the family of operators to $C$. \end{proof}

\subsection{Behavior under gluing and background shifts}

\begin{remark}  {\rm (Gluing)} \label{rem:laggluing} 
 Theorem \ref{thm:mainres} is a family version of Theorem \ref{thm:main}, and
 includes that Theorem as a special case except for the universal
 signs.  The signs for gluing in the interior, gluing at the boundary,
 and gluing strip-like ends are given in Theorem \ref{thm:gluingsigns}.  The
 signs for gluing Floer trajectories with surfaces with strip-like
 ends are determined as follows.  The orientation on the moduli space
 $\M(\ul{x}_+,\ul{x}_-)$ of Floer trajectories, induced from the
 isomorphism
\[ T_u \M(\ul{x}_+,\ul{x}_-) \oplus \R \to T_{\ti{u}} \ti{\M}
(\ul{x}_+,\ul{x}_-) \]
where second factor is the tangent space to the translational
$\R$-action and the codomain is the tangent space to the moduli space
of parametrized trajectories.  There exists a gluing map
\[ \M(x^-_e,y)_0 \times \M_S(\ul{x}^-|_{x^-_e\to y},\ul{x}^+)_0 \times
[0,\eps)
\to \overline{\M_S(\ul{x}^-,\ul{x}^+)_1} \]
%
that factors through the product
\[\ti{\M}(x^-_e,y)_1 \times \M_S(\ul{x}^-|_{x^-_e\to y},\ul{x}^+)_0 \]
preserving the orientation on the $\R$ orbits.  Taking the conventions
of Remark \ref{rem:cases} \eqref{disksigns} shows that the sign of the
gluing map is positive.  A similar description for the outgoing Floer
trajectories in Theorem \ref{thm:gluingsigns} \eqref{case:gluingends} shows that the sign is negative.
\end{remark} 

\begin{remark} \label{rem:bshift}
{\rm (Shift of Background Class)} Define an involution $\Upsilon$ on
the set of relative spin structures on a Lagrangian submanifold $L$
that shifts the background class as follows.  The bundle $TM$ has a
canonical splitting (up to homotopy) after restriction to any
Lagrangian submanifold $L$:
\[ TM |L \cong T(T^*L)|_L  \cong TL \oplus T^\dual L \cong TL \oplus TL
\]
where the last is canonical up to homotopy.  It follows from
Proposition \ref{prop:double} that $TM |L $ has a canonical spin structure, up to
isomorphism.  We say that two relative spin structures for $\ul{L}$
corresponding to bundles $R_1,R_2 \to
TM$ and spin structures on $TL \oplus R_1 |L , TL \oplus R_2 |L $
are {\em equivalent mod $TM$}  if and only if
the following hold:
\begin{enumerate} 
\item There exists a stable isomorphism $R_2 \cong R_1 \oplus TM $
 and
 \item the spin structure on $TL \oplus R_2 | L $ is
that induced by the stable isomorphism 
\[ TL \oplus R_2|L \cong TL \oplus R_1|L \oplus TM |_L .\]
\end{enumerate}
Thus the background class of the second spin structure is 
\[ w_2(R_2) = w_2(R_1) +  w_2(TM) .\]
Given a sequence $\ul{L}$ of oriented Lagrangian submanifolds in $M$,
we denote by
\begin{equation} \label{eq:Ups}
 \Upsilon (\ul{L}) = (\Upsilon(L_0),\ldots,
 \Upsilon(L_d))\end{equation}
the same sequence with shifted relative spin structures.  The relative
spin structures shifted by $\Upsilon$ induce a new relative spin
structure on $(u | \partial S_j)^* TL$.  By Lemma \ref{lem:togglesame},
the orientations on the moduli spaces of pseudoholomorphic disks are
reversed by the identity map exactly if the Maslov index is equal to
$2$ mod $4$.
\end{remark}

\section{Lagrangian Floer invariants over the integers} 
\label{sec:hfsec}

We apply the construction of orientations on pseudoholomorphic 
quilts in the previous section to construct quilted Floer cohomology and relative invariants associated to quilted surfaces over the integers, in the monotone case.  We expect that the same results should hold over the rationals, without monotonicity, given a suitable perturbation scheme.

\subsection{Floer cohomology}

Let $M$ be a compact symplectic manifold equipped with an $N$-fold
Maslov cover $\Lag^N(M) \to \Lag(M)$ for some even integer $N$.  The
cover $\Lag^N(M)$ is by definition an $N$-fold cover of the bundle
$\Lag(M)$ of Lagrangian subspaces of $TM$ that restricts to the
standard $N$-fold cover on any fiber.  We assume that the mod $2$
reduction
\[\Lag^2(M) := \Lag^N(M) \times_{\Z_N} \Z_2 \]
of $\Lag^N(M)$ is the oriented double cover of $\Lag(M)$.

\begin{definition}    
\begin{enumerate} 
\item {\rm (Lagrangian branes)} Let $L \subset M$ be a Lagrangian
  submanifold.  A {\em brane structure} on $L$ is an orientation,
  relative spin structure and {\em grading}
\[\sigma_L: L \to \Lag^N(M)|L \]
lifting the canonical section 
\[L \to \Lag(M), \quad l \mapsto T_l L .\]
A {\em Lagrangian brane} is a compact Lagrangian submanifold equipped
with brane structure.
\item {\rm (Admissibility)} A Lagrangian brane is {\em admissible} if
  it has minimal Maslov number at least three and torsion fundamental
  group; these conditions, together with monotonicity of the ambient symplectic manifold, imply monotonicity for pairs, triples etc. in the sense
of \cite[Remark 3.7]{ww:quilts}. 
\item {\rm (Periodically-graded Floer complex)} Let $L_0,L_1 \subset
  M$ be admissible Lagrangian branes.  The grading induces a degree
  map
\[  \cI(L_0,L_1) \to \Z_N ,
\qquad x\mapsto |x|=d(\sigma_{L_0}(x),\sigma_{L_1}(x))  \] 
and $|x|$ is related to the indices of the operators $D_{e}^\pm$
of \eqref{eq:order} by 
\[ |x| = \Ind(D_e^+) = \dim(TL) - \Ind(D_e^-) \ \pmod N \]
by the definition of grading.
The {\em Floer cochain group} is the $\Z_N$-graded group
\[ CF(L_0,L_1) = \bigoplus_{d \in \Z_N} CF^d(L_0,L_1), \qquad
CF^d(L_0,L_1) =\bigoplus_{x\in\cI(L_0,L_1), |x| = d} \Z \bra{x} . \]
The {\em Floer coboundary operator} is the map of degree $1$,
\[ \partial^d : \ CF^d(L_0,L_1) \to CF^{d+1}(L_0,L_1) ,\]
defined for $|x_-| = d$ by 
\[
\partial^d\bra{x_-} := \sum_{x_+\in\cI(L_0,L_1)}
\Bigl( \sum_{u \in\M(x_+,x_-)_0} \eps(u)\Bigr) \bra{x_+} .\]
where
\[ \eps: \M_{\ul{S}}(x_+,x_-)_0 \to \{-1,+1\} \]
is defined by comparing the constructed orientation to the canonical
orientation of a point.
\item {\rm (Graded Floer complex using a formal variable)} There is
  also a $\Z$-graded version of Floer homology whose differential is
  defined over 
  \[ \Lambda = \Z[q,q^{-1}], \deg(q)= N, \]
  a ring of Laurent polynomials in a formal
  variable $q$ of degree $N$ that keeps track of the difference in gradings.  Let
\[\ti{\cI}(L_0,L_1) = \Z \times_{\Z_N} \cI(L_0,L_1) \]
and $\ti{d}$ the extended degree map 
\[ \ti{d}: 
\ti{\cI}(L_0,L_1) \to \Z, \quad [n,x] \mapsto n .\]
Choose a lift $\ti{d}(x) \in \Z$ of the degree $|x| \in \Z_N$
for each $x \in \cI(L_0,L_1)$.  Let 
\[ \widetilde{CF}(L_0,L_1) = \bigoplus_{x \in \cI(L_0,L_1)} \Lambda \bra{x} , 
\qquad \deg \bra{x} = \ti{d}(x) .\]
Define 
\[ \ti{\partial} \bra{x_-} := \sum_{x_+ \in\cI(L_0,L_1)}
\Bigl( \sum_{u \in \M(x_+,x_-)_0} \eps(u) \Bigr) 
q^{(\ti{d}(x_-) + 1 - \ti{d}(x_+))/N}  \bra{x_+} . \] 
Here $\ti{d}(x_-) + 1 - \ti{d}(x_+)$ is divisible by $N$, since 
$\M(x_+,x_-)_0$ is non-empty only if $|x_+| = |x_-| + 1$ in $\Z_N$.
The operator $\ti{\partial}$ has degree $1$ and, since the exponent of $q$ 
in $\ti{\partial}^2$ depends only on the lifts of the endpoints, 
$\ti{\partial}^2 = 0$ follows from $\partial^2 = 0$.  A different choice of lifts $\ti{d}'(x) = \ti{d}(x) + N m_x$ gives an isomorphic complex via $\bar{x} \mapsto q^{m_x} \bar{x}$, so the cohomology $\ti{HF}(L_0,L_1)$ is independent of this choice up to isomorphism. 
\end{enumerate}
\end{definition} 

\begin{theorem}  \label{thm:floertheorem} 
Suppose that $M$ is monotone and a pair of Lagrangian branes
$(L_0,L_1)$ in $M$ is monotone and satisfies the monotonicity
conditions L1-3 of \cite{ww:quiltfloer}: each is compact, oriented,
monotone, and has minimal Maslov number at least three.  Then the
Floer differentials $\partial, \ti{\partial}$ satisfy $\partial^2 = 0,
\ti{\partial}^2 = 0$.
\end{theorem} 

\begin{proof}  That $\partial^2 = 0 $ follows from argument in
Oh \cite{oh:fl1} and that gluing along strip-like ends is orientation
preserving for the corresponding Cauchy-Riemann operators.  The proof
for $\ti{\partial}$ is similar.
\end{proof} 

If the assumptions of Theorem \ref{thm:floertheorem} are satisfied then
the {\em Floer cohomology} is defined by
\[ HF(L_0,L_1) := \bigoplus_{d\in\Z_N} HF^d(L_0,L_1), \qquad
  HF^d(L_0,L_1) := {\ker(\partial^d)}/{\on{im}}(\partial^{d - 1}) .\]
Similarly, $\widetilde{HF}(L_0,L_1)$ is the cohomology of
$\ti{\partial}$.

\begin{proposition} \label{prop:indepor}
$HF(L_0,L_1)$ resp. $\widetilde{HF}(L_0,L_1)$ is a well-defined
  $\Z_N$-graded resp. $\Z$-graded group, independent, up to
  isomorphism, of the choices made in the construction of the
  orientations.
\end{proposition} 

\begin{proof} Suppose that $(\Gamma_\pm,D_\pm,\eps_\pm,\delta_\pm),
(\Gamma_\pm',D_\pm',\eps_\pm',\delta_\pm')$ are two {orientation}s for
  the ends $e \in \cE(S)$.  Define maps
\[ \sigma: CF(L_0,L_1) \to CF(L_0,L_1) , \ \ \ \bra{x} \mapsto
\sigma(x) \bra{x} \]
as follows.  Let $(E_\pm,F_\pm)$ denote the corresponding elliptic
boundary value problems on the once-punctured disk $S_1$.  Let
\[ (\ol{E}_\pm,\ol{F}_\pm,\ol{D}_\pm) = (E_\pm,F_\pm,D_\pm) \#
(E_\pm',F_\pm',D_\pm') \]
denote the bundles and Cauchy-Riemann operator obtained by gluing
together the problems $E_\pm,F_\pm,D_\pm$ and $E_\pm',F_\pm',D_\pm'$
along the strip-like ends.  By the gluing formula there exists an
isomorphism
\[ \det(D_\pm) \otimes \det(D_\pm') \to \det(\ol{D}_\pm) .\]
We define $\sigma(x) = \pm 1$ depending on whether the orientation
induced by $\eps_\pm,\eps_\pm'$ and the gluing isomorphism agrees with
the orientation induced by the trivialization of $\ol{F}_\pm$.  The
gluing law for indices implies that the map $\sigma$ intertwines the
relative invariants $\Phi_S$ for $S$ associated to the two different
choices of orientation.
\end{proof} 

\begin{remark} \label{rem:rconjs} {\rm (Conjugates)} Suppose that $M^-$ is
  the symplectic manifold $M$ with symplectic form reversed. Given
  Lagrangian branes $L_k$ we denote by $L^-$ the corresponding
  branes in $M^-$, with background class   
shifted as in Remark \ref{rem:bshift}:
\[b(L^-) = b(L) + w_2(M) .\]
For any intersection point $x_\pm$ we denote by $x^-_\pm$
the corresponding intersection point of $L_j^-$ and
$\M(x^-_+,x^-_-)$ the moduli space of Floer trajectories.
Each trajectory $u(s,t)$ for $(L_0,L_1)$ defines a trajectory
$u(s,1-t)$ for $(L_1^-,L_0^-)$ giving a bijection
\[\M(x_+,x_-) \to \M(x^-_+,x^-_-) \]
from trajectories for $(L_0,L_1)$ to trajectories for $(L_1^-,L_0^-)$.
The bijection acts on orientations at a trajectory $u$ by a sign given
by
\[(-1)^{(\Ind(D_{x_+}) - \Ind(D_{x_-}) - 1 - \dim T_{\ul{u}}
  \M(x_+,x_-)) /2} .\]
By Lemma \ref{lem:togglesame}, this shift cancels out the shift in the
background class.  Hence, as oriented manifolds
\[ \M(x_+,x_-) \cong \M(x_+^-,x_-^-) .\]
Equality of the counts of elements in the zero-dimensional component implies that the Floer operators are equal, so that we have an isomorphism
\begin{equation} \label{eq:hfduals} 
HF(L_1^-,L_0^-) \cong HF(L_0,L_1) .\end{equation} 

In the $\Z$-graded version one can also define an involution in the
power series ring
\[\Z[q] \to \Z[q], \quad q \mapsto (-1)^{N/2} q .\]  
This involution extends to an involution of $\widetilde{HF}(L_0,L_1)$.
The natural identification 
\[\widetilde{CF}(L_0,L_1) \to \widetilde{CF}(L_1^-,L_0^-)\]
composed with the involution intertwines with the Floer differentials
$\ti{\partial}$, $\ti{\partial}^-$.

The same considerations apply for the Floer homologies $HF(\ul{L}),
\widetilde{HF}(\ul{L})$ of a periodic sequence of Lagrangian
correspondences equipped with admissible brane structures.
In the case that the sequence is of odd length, we insert 
a diagonal in order to ensure
that the set of patches admits a coloring in the sense of Definition \ref{def:coloring}.
In
particular, suppose that $\ul{L} = (\Delta,\Delta)$  is two copies of the diagonal.  The Floer
cohomology satisfies 
\[HF(\ul{L}) = \widetilde{HF}(\ul{L}) / (q-1) = QH(M)/(q-1) \]
where $QH(M)$ is the quantum cohomology.  Then $QH(M)/(q-1)$ is {\em
  not} isomorphic via the identity map to $QH(M^-)/(q-1)$.

For example let $M$ be complex projective $n$-space.  Then
\[QH(M)/(q-1) = \Z[x]/(x^{n+1} - 1) .\]  
On the other hand, $M^-$ is also isomorphic to the projective space,
but now the hyperplane class is $-x$, and 
\[QH(M^-)/(q-1) = \Z[-x]/((-x)^{n+1} -1).\] 
The latter is isomorphic to $QH(M)/(q-1)$ via the map $x \mapsto -x$,
but not via the identity if $n$ is even.  On the other hand,
\[QH(M) = \Z[x,q]/ (x^{n+1} - q) \cong QH(M^-) = \Z[-x,q]/((-x)^{n+1} -q)\]
via the map $x \mapsto x, q \mapsto (-1)^{n+1} q$.
\end{remark}

\begin{remark} {\rm (Disk potentials)} If the minimal Maslov number of 
  either Lagrangian is only two, then the Floer operator may not 
  square to zero due to the presence of pseudoholomorphic disks.  We 
  denote by 
\[\M_1(L) = \{ ( u : (D,\partial D) \to (X,L), z) | \partial_J u = 0 \}
/ \Aut(D) \cong SL(2,\R) .\]
the moduli space of pseudoholomorphic disk with a single marking on the
boundary.  Consider the evaluation map $\ev_1:\M_1(L) \to L$.  A
generic element $\ell \in L$ is a regular value and we define
following Oh \cite{oh:fl1}
\begin{equation} \label{eq:wL} w(L) = 
 \sum_{u \in\ev_1^{-1}(l) \subset \M(L)_{\dim(L)}} \eps(u) \in 
\Z\end{equation}
  the {\em disk potential} of $L$.  One sees by a parametrized version 
  of the moduli space that the number $w(L)$ is independent of the 
  choice of regular value.  The boundary of 
  $\M(L_0,L_1)_1$ has two extra components corresponding to bubbling 
  off disks at the left and right boundaries of the strip:
\begin{multline} 
 \partial \M(L_0,L_1)_1 =  \left( \M(L_0,L_1)_0 \times_{L_0 \cap 
    L_1} \M_1(L_0,L_1)_0 \right) \\ \cup  ( \M_1(L_0) \times_M L_1 ) \cup
(\M_1(L_1) \times_M L_0 ).\end{multline}
The second and third configurations consist of a disk attached via a
node to a constant strip.  Recall that automorphisms of the disk
fixing a point form a two-dimensional subgroup of $\Aut(D)$ generated
by translations and dilations, once the complement of the marking is
identified with the half-space.  Under the gluing map the
one-parameter subgroup consisting of translations on the disk becomes
identified with the opposite translation resp. translation on the
strip.  As a result the square of the Floer operator is the difference
in disk potentials:
\[ \partial^2 = w(L_1) - w(L_0) .\]
\end{remark}

We investigate the effect of reversing the symplectic form on disk
potentials.  If we consider $L^- \subset M^-$ the same Lagrangian in
the symplectic manifold with reversed symplectic form, then we have a
diffeomorphism of moduli space of pseudoholomorphic disks
\[ \M_1(L) \to \M_1(L^-), \quad u \mapsto u^-, \ u^-(z) = u(\ol{z}) \]

\begin{lemma} For the shifted relative spin structure on $L^-$ with
  background class $b(L) + w_2(M)$, the diffeomorphism
  $\M_1(L) \to \M_1(L^-)$ is orientation reversing. Furthermore, the disk
  potentials for $L, L^-$ are related by $ w(L^-) = - w(L) .$
\end{lemma}

\begin{proof} The complex conjugation on the domain reverses the
  orientation on the group of automorphisms of the disk fixing a
  point.  Indeed after identifying the complement of a point with the
  upper half plane, this group of automorphisms is generated by
  dilations and translations, while conjugation becomes identified
  with a reflection.  While dilations are invariant under conjugation
  by the reflection, translations are reversed producing a sign
  change.  On the other hand, complex conjugation on the index of the
  linearized operator produces a sign change exactly if the Maslov
  index is equal to $2$ mod $4$, by Remark \ref{rem:crem}.   The additional signs produced by the shift of relative spin structure in Lemma \ref{lem:togglesame}
  imply that the orientation on the moduli spaces are reversed.  The disk potential $w(L)$ is the signed count of holomorphic disks, and the claim on the $w(L^-)$ follows.
\end{proof}

These results extend to the quilted case as follows.  

\begin{definition} 
\begin{enumerate}
\item {\rm (Quilted Floer cohomology)} The quilted Floer coboundary operator 
\[\partial^d : \ CF^d(\ul{L}) \to CF^{d+1}(\ul{L}) \] 
is defined by
\[
\partial^d\bra{\ul{x}_-} := \sum_{\ul{x}_+\in\cI(\ul{L})}
\Bigl( \sum_{\ul{u} \in\M(\ul{x}_+,\ul{x}_-)_0} \eps(\ul{u})\Bigr) \bra{\ul{x}_+} ,
\] 
where the signs 
\[ \eps: \M(\ul{x}_+,\ul{x}_-)_0 \to \{ \pm 1 \} \]
are
defined by comparing the given orientation to the canonical
orientation of a point.  By studying the ends of the one-dimensional
moduli spaces as in the unquilted case one obtains $\partial^2 = 0$.
The {\em quilted Floer cohomology} defined in \cite{ww:quiltfloer} is
\[ HF(\ul{L}) := \bigoplus_{d\in\Z_N} HF^d(\ul{L}), \qquad
HF^d(\ul{L}) := {\ker(\partial^d)}/{\on{im}}(\partial^{d - 1}) \]
and is a $\Z_N$-graded group.  In case that the Lagrangians are
$N$-graded, the datum associated to each intersection point $\ul{x}$ is
equipped with an orientation for the operator $D_e$ given by the path
induced by the grading.  Similarly, $\widetilde{HF}(\ul{L})$ is defined
as the $\Z$-graded group over the formal power series ring in a formal
variable $q$.
\item {\rm (Relative invariants)} Suppose that $\ul{S}$ is a quilted
  surface with strip-like ends.  Let $\ul{L}$ be a collection of
  Lagrangian boundary and seam conditions for a collection $\ul{M}$ of
  compact monotone symplectic manifolds attached to the patches of
  $\ul{S}$.  A {\em brane structure} for $\ul{L}$ is a collection of
  gradings and a relative spin structure.  Let $\ul{x}$ be a
  collection of perturbed intersection points at infinity, and
  $\M(\ul{x})$ the moduli space of perturbed pseudoholomorphic maps
  $\ul{S} \to \ul{M}$ with boundary values in $\ul{L}$ and limits
  $\ul{x}$.  A choice of relative spin structure for $\ul{L}$, if it
  exists, together with {orientation}s on the ends, induces an
  orientation on $\M(\ul{x})$, by Remark \ref{rem:qpres}.  Assuming
  suitable monotonicity conditions on the tuple $\ul{L}$ that rule out
  sphere and disk bubbling in zero and one-dimensional moduli space,
  there is a chain level relative invariant constructed in
  \cite{ww:co} defined by
\[ C\Phi_{\ul{S}} : \bigotimes_{e \in \mE_-(\ul{S})} CF( \ul{L}_e) \to
  \bigotimes_{e \in \mE_+(\ul{S})} CF( \ul{L}_e) \]
\[ \otimes_{e \in \mE_-(\ul{S})} \bra{\ul{x}_e} \to \sum_{
    (\ul{x}_e)_{e \in \mE_+(\ul{S})}, u \in \M(\ul{x}) } \eps(u)
  \otimes_{e \in \mE_+(\ul{S})} \bra{\ul{x}_e} .\]
These descend to maps between the cohomologies of the tensor products
of the Floer cochain complexes.  With coefficients in a field, for
example rational coefficients, the K\"unneth isomorphism gives a
cohomology level invariant
\[ \Phi_{\ul{S}} : \bigotimes_{e \in \mE_-(\ul{S})} HF(\ul{L}_e) \to
\bigotimes_{e \in \mE_+(\ul{S})} HF(\ul{L}_e) .\]
\end{enumerate} 
\end{definition} 

\begin{proposition} \label{prop:indeporinv}
The invariants $C\Phi_{\ul{S}}$ descend to cohomology and the
resulting cohomological invariants $\Phi_{\ul{S}}$ are independent up
to isomorphism of the choice of perturbation data and orientations.
\end{proposition} 

\begin{proof}  Without signs, this is the main result of \cite{ww:quilts}.  
The cochain level invariants descend to cohomology by the
determination of the gluing signs in Section \ref{sec:qgluing}; especially 
Theorem \ref{thm:gluingsigns} part \eqref{case:gluingends} which shows that the gluing maps are orientation-preserving if all of the ends are glued, for our choice of orderings.  It
follows again from gluing that the maps for two choices of
orientations intertwine with the maps in the proof of Proposition
\ref{prop:indepor}.
\end{proof}

Before comparing orientations, we note the following:

\begin{lemma} \label{lem:compare} {\rm (Comparison of orientations on strips)}  Let $\cX$ be a set of end data,
and suppose that two systems of orientations $o_u$, $o'_u$ are given on the determinant lines $\det(D_u)$, $\det(D'_u)$
of quilted strip operators $D_u,D'_u$ with ends in $\cX$,
together with isomorphisms 
\[ \Phi_u: \det(D_u) \to \det(D'_u) \] 
such that 
\begin{enumerate}
\item \label{itema} $\Phi_u$ is continuous in families of linear data 
with fixed ends, and the space of such data with fixed ends, index, and relative spin class is connected; 
\item \label{itemb} $\Phi$ commutes with gluing of strips at a common end, and gluing of strips is orientation preserving for both systems;
\item \label{itemc} for the translation-invariant strip at each $x \in \cX$,
$\Phi$ is the identity of the canonically trivial line, 
and $o_u, o'_u$ are the canonical orientations. 
\end{enumerate}
Let $s(u) \in \{ \pm 1 \}$ be defined by $\Phi_u(o_u) = s(u)
o'_u$.  Then there exist signs $\eta : \cX \to \{ \pm 1 \}$
and a homomorphism $\chi$ from the group of index shifts
to $\{ \pm 1 \}$ such that $s(u) = \eta(x_-) \eta(x_+) \chi(k(u))$ where $k(u)$ is the difference between 
$\Ind(D_u)$ and a fixed function of the ends. 
\end{lemma}

\begin{proof} The proof is straight-forward. 
By \eqref{itema}, $s(u) = s(x_-,x_+;k)$ depends only on the ends and the index.  By \eqref{itemb}, 
\begin{equation} \label{eq:sac} 
s(a,c;k+1) = s(a,b;k) s(b,c;1) \end{equation}
and by \eqref{itemc}, $s(a,a;0) = 1$.  Fix
$x_0$ and put $\eta(x) = s(x_0,x;0)$ and
$\chi(k) = s(x_0,x_0;k)$.  The multiplicativity relation 
\eqref{eq:sac} gives the desired formula.
\end{proof}

\begin{proposition}\label{prop:annulus} {\rm (Floer invariant of the torus
as the graded dimension)} Let $\ul{L}$ be a periodic sequence of
  Lagrangian correspondences equipped with admissible brane
  structures.  Let $\ul{T}$ denote the quilted torus, with one seam
  for each element of $\ul{L}$. Then the relative invariant
  $\Phi_{\ul{T}}$ is the graded dimension of $HF(\ul{L})$,
\[ \Phi_{\ul{T}} = \rank HF^{\on{even}}(\ul{L}) - \rank HF^{\on{odd}}(\ul{L}) .\]
\end{proposition}

\begin{proof}  The equality in the statement of the Proposition follows
from the discussion of the annulus signs in Remark \ref{rem:cases}. Indeed
that Remark shows that the contribution to the invariant
$C\Phi_{\ul{T}}$ from $x \in \cI(\ul{L})$ is $(-1)^{|x|}$.
\end{proof}

We investigate the behavior of the quilted Lagrangian Floer invariants
under the following basic operations:

\begin{remark} \label{rem:ops}
\begin{enumerate}
\item {\rm (Products)} Let $\ul{L}_j, j= 0,1$ denote two periodic
  sequences of Lagrangian correspondences of the same length equipped
  with admissible brane structures so that the Floer homology groups
  $HF(\ul{L}_j),j =0,1$ are well-defined.  Let $\ul{L}$ denote the
  sequence obtained from $\ul{L}_j,j =0,1$ by direct sum.  Let $\ul{x}
  = (\ul{x}_0,\ul{x}_1)$ be a generalized intersection point for
  $\ul{L}$.  Consider the natural map
\begin{equation} \label{eq:product} \M(\ul{x}_0) \times \M(\ul{x}_1) \to \M(\ul{x}), \quad (u_0,u_1)
\to u_0 \times u_1 .\end{equation}

The linearized operator $D_{u_0 \times u_1}$ is naturally isomorphic to the direct sum $D_{u_0} \oplus D_{u_1}$,
and the  orientations are related by the isomorphism 
\eqref{eq:caniso} up to signs depending only on the ranks
and on the indices of the operators at the ends, 
by Remark \ref{rem:crem} \eqref{item:sums} and 
Remark \ref{rem:compatops} \eqref{item:sumcompat}.  In particular, 
with $\Z_2$ coefficients the Floer complex 
$CF(\ul{L})$ is the tensor product of $CF(\ul{L}_0)$
and $CF(\ul{L}_1)$, and similarly for the relative
invariants; with integer coefficients the two complexes have the same generators and the same rigid trajectories, with signs related as above. 
\item {\rm (Disjoint Unions of domains)} Let $S_j ,j = 1,2$ be
  surfaces with strip-like ends.  Suppose that $S = S_1 \cup S_2$, and
  $S_j$ has $d_j^\pm$ incoming resp. outgoing ends for $j=1,2$.  A
  pair $u = (u_1,u_2)$ of pseudoholomorphic maps of index zero has
  determinant line with orientation related to the orientations of the
  determinant lines
\begin{eqnarray*}
 \eps(u) &=& \eps(u_1) \eps(u_2) (-1)^{\rank(F) ( \# \pi_0(\partial
S_2) + d_2^+)( \sum_{e \in \mE_{-,1}} (\dim(M)/2 + \Ind(D_e)) )}  \\
&=& \eps(u_1) \eps(u_2) (-1)^{|\Phi_{S_2}| \sum_{e \in \mE_{-,1}} 
|x_e|}.
\end{eqnarray*}
This formula implies that the relative invariant $C\Phi_S$ is the
graded tensor product
\begin{eqnarray*} \label{eq:discon} 
C\Phi_S( \otimes_{e \in \mE_-(S)} \bra{x_e} ) &=& (-1)^{|\Phi_{S_2}| \sum_{e \in
    \mE_{-,1}} |x_e|} C\Phi_{S_1}( \otimes_{e \in \mE_-(S_1)} \bra{x_e}
) \\ && \otimes C\Phi_{S_2}( \otimes_{e \in \mE_-(S_2)} \bra{x_e}) \\
&=&
  (C\Phi_{S_1} \otimes C\Phi_{S_2})( \otimes_{e \in \mE_-(S_1) \cup
    \mE_-(S_2)} \bra{x_e}).
\end{eqnarray*} 
With coefficients in a field, it follows that $\Phi_S$ is the graded
tensor product of $\Phi_{S_1}$ and $\Phi_{S_2}$.
\item {\rm (Folding of quilts)} \label{qfold} A ``quilt folding'' isomorphism was
  considered in \cite{ww:quiltfloer}.  Let
\[ \ul{L} = (L_{01}, L_{12}, \ldots, L_{k0}) \]
be a cyclic Lagrangian correspondence with $k$ odd.  In \cite[Section
  5]{ww:quiltfloer} we identified
\begin{equation} \label{eq:folded}
 HF(\ul{L}) \cong 
HF(L_{01} \times L_{23} \times \ldots \times L_{(k-1)k},
 \sigma(L_{12}^- \times L_{34}^- \times \ldots \times L_{k0}^-)) \end{equation}
where
\[\sigma: M_1^- \times M_2 \times \ldots \times M_0 \to M_0
\times M_1^- \times \ldots \times M_k^- \]
is the cyclic shift.   

To see that \eqref{eq:folded} holds with integer coefficients, note that 
the folding map is a bijection between generalized intersection points of the two sides, and between the rigid quilts defining the two Floer differentials, and that it identifies the linearized operators and the operators at the ends compatibly with gluing.  Since $k+1$ is even, the patches admit the alternating coloring, whose patches of 
color $1$ are those labelled $M_j^-$. By the convention in Proposition \ref{prop:noends}, the orientations of the quilted operators are those of the folded operators with the direct sum end data.  The two Floer problems therefore differ only in the end orientations and in the order of the summands, and the hypotheses of Lemma \ref{lem:compare} hold with $\Phi$
equal to the identity.  Indeed, gluing of strips is orientation-preserving for both problems by Theorem \ref{thm:gluingsigns} \eqref{case:gluingends} and Lemma 
\ref{lem:quiltedgluing}, translation-invariant strips carry their canonical orientations by Remarks \ref{rem:cases}
\eqref{case:strip} and \ref{rem:dependence} \eqref{trivbun},
and reordering the summands changes the orientation of a disk with
boundary in the seam conditions by a sign depending only on the ranks of
the summands; this sign is the same for the quilted and the folded
problems, which have the same summands, so that $\chi = 1$.  Hence 
there are signs $\eps(x)$ such that 
$\bra{x} \mapsto \eps(x) \bra{x}$ is an isomorphism of cochain
complexes, and \eqref{eq:folded} follows. 
  \end{enumerate}
\end{remark} 

\begin{remark} {\rm (Quilted sphere bubbling)} In the case that some
  correspondences $L_{(j-1)j}$ have minimal Maslov number equal to
  two, the quilted Floer operator may not square to zero.   Let $\M_1(L_{(j-1)j}, M_{j-1},M_j)$ denote the moduli space of
  pseudoholomorphic quilted spheres with a single marking on the seam.  Let $l \in L_{(j-1)j}$ be a regular value
  of the evaluation map
\[ \ev_1:  \M_1(L_{(j-1)j}, M_{j-1},M_j) \to L_{(j-1)j} .\]
Define the {\em quilted sphere potential}
\[ w(L_{(j-1)j},M_{j-1}, M_j) = 
\sum_{u \in\ev_1^{-1}(l) \subset \M(L_{(j-1)j},M_{j-1},M_j)_{\dim(L)}}
\eps(u) \in \Z.\]
Then the square of the Floer coboundary operator satisfies
\[ \partial^2 = w(L_{01},M_0,M_1) + w(L_{12},M_1,M_2) + \ldots + 
w(L_{k0}, M_k, M_0) .\]
  In the case of the folding identification, folding produces an 
identification of moduli spaces 
\[ \M_1(L_{(j-1)j},M_{j-1},M_j) \cong \M_1(L_{(j-1)j}) \cong
\M_1(L_{(j-1)j}^-)^{\on{op}}\]
where the superscript $\on{op}$ denotes the opposite orientation.
Hence
\[ w(L_{01},M_0,M_1) = w(L_{01},M_0^- \times M_1) = -w(L_{01},M_0
\times M_1^-) .\]
This equality is compatible with the folding identification of spaces
with Floer operators
 \[  CF(\ul{L}) \cong CF(L_{01} \times L_{23} \times \ldots \times L_{(k-1)k},
 \sigma(L_{12}^- \times L_{34}^- \times \ldots \times L_{k0}^-)) \]
from Remark \ref{rem:ops} \eqref{qfold} since 
\[ \sum w(L_{(j-1)j}, M_{j-1}, M_j) =  
w(L_{01} \times L_{23} \times \ldots \times L_{(k-1)k}) - 
 w(L_{12}^- \times L_{34}^- \times \ldots \times L_{k0}^-) .\]
\end{remark}

\begin{remark} {\rm (Extension to Lefschetz fibrations)} In the case of a
  Lefschetz fibration $E$ over $S$ with Lagrangian boundary condition
  $Q$, the gluing signs are the same as for pseudoholomorphic
  surfaces.  In the case $Q$ is oriented and has minimal Maslov number
  at least two, similar constructions define a
  relative invariant
\[\Phi_{E,Q}: \otimes_{e \in \mE_-(S)} HF(\ul{L}_e) \to
  \otimes_{e \in \mE_+(S)} HF(\ul{L}_e) \]
mapping the tensor product of Floer homologies for the incoming ends
to the product for the outgoing ends, as needed in the constructions
in \cite{ww:ex}.
\end{remark}

\subsection{Fukaya category} 
\label{sec:dfsec}

The Fukaya category is defined by counting pseudoholomorphic disks
with boundary lying in Lagrangian submanifolds.  In this section, we
explain how the results on orientations naturally define the Fukaya
category in the case of minimal Maslov number at least three, and its quilted version over the integers.  The sign computation would be similar in any definition of the 
Fukaya category in general using a coherent transversality scheme. 

First, we recall the definition of Stasheff's associahedron via nodal
disks.  A {\em nodal disk} $D$ is a contractible space obtained from a
union of disks $D_i,i =1,\ldots,l$ (called the components of $D$) by
identifying pairs of points $w_j^+,w_j^-, j =1,\ldots, k$ on the
boundary (the {\em nodes} in the resulting space)
\[ D = \sqcup_{i=1}^l D_i / (w_j^+ \sim w_j^-, j = 1,\ldots,k ) \]
so that each node $w_j\in D$ belongs to exactly two disk components
$D_{i_-(j)}, D_{i_+(j)}$.  A {\em set of markings} is a set
$\{ z_0,\ldots,z_d \}$ of the boundary $\partial D$ in
counterclockwise order, distinct from the singularities.  A {\em
  marked} nodal disk is a nodal disk with markings.  A {\em morphism
  of marked nodal disks} from $(D,\ul{z})$ to $(D',\ul{z}')$ is a
homeomorphism $\varphi:D \to D'$ restricting to a pseudoholomorphic
isomorphism $\varphi|_{D_i}$ on each component $i=1,\ldots,l$ and
mapping the marking $z_j$ to $z_j'$.  A marked nodal disk $(D,\ul{z})$
is {\em stable} if it has no automorphisms or equivalently if each
disk component $D_i \subset D$ contains at least three nodes or
markings.  The {\em combinatorial type} of a nodal disk with markings
is the tree
\[\Gamma = (\Ver(\Gamma),\Edge(\Gamma)), \quad \Edge(\Gamma) = \Edge_{<\infty}(\Gamma)
\sqcup \Edge_\infty(\Gamma) \] 
obtained by replacing each disk with a vertex $v \in \Ver(\Gamma)$,
each node with a finite edge $e \in \Edge_{<\infty}(\Gamma)$, and each
marking with a semi-infinite edge $e \in \Edge_{\infty}(\Gamma)$.  The
semi-infinite edges $\Edge_\infty(\Gamma)$ are labelled by
$0,\ldots,d$ corresponding to which marking they
represent. 

We introduce the following notation for moduli spaces of disks.  For
each combinatorial type $\Gamma$ let $\RR_\Gamma^d$ denote the set of
isomorphism classes of semistable nodal $d+1$-marked disks of
combinatorial type $\Gamma$, and
\[ \ol{\RR}^d = \bigcup_{\Gamma} \RR^d_{\Gamma} \]
the moduli space of stable disks.  Each stratum is naturally oriented
by identifying each open stratum $\RR^d \cong (0,1)^{d-2} $ by
identifying the unit disk with upper half space, so that the first and
last markings map to $0,1$ while the zeroth maps to infinity.

Associated to a tuple of Lagrangians is a moduli space of
pseudoholomorphic polygons.  Namely, given Lagrangian branes
$L_0,\ldots L_d$ intersecting pairwise transversally and intersection
points $x_j \in L_{(j-1)} \cap L_j$ for $j = 0,\ldots, d$ let
$\M(x_0,\ldots,x_d)$ denote the moduli space of stable disks $S$
equipped with pseudoholomorphic maps $u: S \to M$ mapping the part of
the boundary between $z_j$ and $z_{j+1}$ to $L_{j}$.  For a comeager
set of almost complex structures the moduli space is smooth, the
zero-dimensional component is finite, and the one-dimensional
component is compact up to bubbling off Floer trajectories at the
markings.  More precisely, the one-dimensional component
$\ol{\M}^d(\ul{x})_1$ has a compactification as a one-manifold with
boundary the union
 \begin{equation} \label{eq:bound} \partial \ol{\M}^d_1(\ul{x}) = \bigcup_{\Gamma}
  {\M}^{d}_{\Gamma,1}(\ul{x}) \end{equation}
  of strata $\M^d_{\Gamma,1}(\ul{x})$ of $\ol{\M}^d_{1}(\ul{x})$
  corresponding to trees with two vertices (where either (1) $\Gamma$
  is stable with two vertices, or (2) $\Gamma$ is unstable and 
  corresponds to bubbling of a Floer trajectory).   

  In the stable range each moduli space of polygons is oriented by
  identifying its determinant line with the product of determinant
  lines for the associahedron and the linearized operator.  By
  deforming the parametrized linear operator to the linearized
  operator plus a trivial operator and bubbling off marked disks on
  each strip-like  end we may identify the determinant line with the
  product of determinant lines for the associahedron and the
  linearized operator for $d \ge 3$:
\[  \det( T_{[C,u]} {\M}^d (x_0,\ldots,x_d)_0) 
    \cong\det(T_{[C]}\RR^d) \otimes 
\det(D_u) .\]
Using \eqref{eq:order}, we re-write the determinant line of the
linearized operator as follows.  For any intersection point $x_j$ let
\[\DD^-_{x_j} =  \det(D_{e_j}^-) \otimes 
\Lambda^{\max}(\Gamma_{e_j}(0)^\dual), \qquad 
\DD^+_{x_j} = \Lambda^{\max}(\Gamma_{e_j}(0)^\dual) \otimes \det(D_{e_j}^+) \]
denote the determinant line associated to $x_j$, given as the tensor
product of the determinant line of the Lagrangian with the determinant
of a Cauchy-Riemann operator on the disk with a single end depending
on a choice of path from $T_{x_j} {L}_{j-1}$ to $T_{x_j} {L}_{j}$ in
$T_{x_j} {M}$, as in \eqref{eq:order}.  The determinant line for the
surface with boundary but without ends is the product of determinant
lines for the trivial operator, isomorphic to $\det(TL)$, and the
determinant line on a complex space which may be ignored for the
purposes of the sign computation.  Thus (omitting tensor products to
save space) we identify
  \begin{equation} \label{eq:convent}
 \det( T {\M}^d (x_0,\ldots,x_d)_0) 
    \to \det(T\RR^d) 
\DD^-_{x_1} \ldots 
    \DD^-_{x_d} 
\det(T {L}) 
 \DD^+_{x_0} 
.\end{equation}
  The chosen orientations for the ends, the associahedron, and the
  Lagrangian induce an orientation on the moduli space
  $\M^d(x_0,\ldots,x_d)$. In the unstable range $d \leq 2$ a similar
  construction, after replacing $T \RR^d$ with the Lie algebra of the
  automorphism group of a disk with $1$ or $2$ markings on the
  boundary, gives an orientation on the moduli space of pseudoholomorphic
  polygons.

\begin{definition} \label{def:dfn DonFuk}
The {\em Fukaya category}
$\Fuk(M) := \Fuk(M,\Lag^N(M),\omega,b)$
is defined as follows:

\begin{enumerate}
\item The objects of $\Fuk(M)$ are admissible Lagrangian branes in $M$
  with background class $b$.
\item The morphism spaces of $\Fuk(M)$ are the $\Z_N$-graded Floer
  cochain groups 
\[ \Hom(L,L') := CF(L,L') .\]
\item
The composition law in the category $\Fuk(M)$ is defined by counting
holomorphic polygons:
\[ \mu^d:  \Hom(L^0,L^1) \times \ldots \times \Hom(L^{d-1},L^d) 
\to \Hom(L^0,L^d) \]
by 
\begin{equation} \label{eq:highercomp} \mu^d(\bra{x_1},\ldots,\bra{x_d})
  = (-1)^\heartsuit \sum_{u \in \M^d(x_0,\ldots,x_d)_0} \eps(u)
  \bra{x_0}
\end{equation}
where\footnote{Our previous paper \cite{Ainfty} used a different constant, corresponding to a different ordering of determinant lines in equation (15).   The computations using the formula here are slightly cleaner. }
\begin{equation} \label{eq:heartsuit} \heartsuit = {\sum_{i=1}^d (d-i) |x_i|}
  .\end{equation}
\end{enumerate}
\end{definition}

The \ainfty associativity, Seidel's convention \cite{se:bo}, reads (with inputs ordered from left-to-right, in line with our conventions on correspondences)
\begin{equation}\label{ainftyassoc} 
0 = \sum_{\substack{i, j \geq 0\\ i + j \leq d}} (-1)^{\maltese_{1}^j} \mu^{d-i+1} ( a_1, \ldots, a_{j},   \mu^i (a_{j+1},\ldots,a_{j+ i}), a_{i+j+1},\ldots,a_{d})
\end{equation}
where 
\begin{equation}\label{maltese}
\maltese_{l}^{k}:= \sum_{l \leq i \leq k} \Vert a_i \Vert,\ \quad \Vert a_i \Vert:= |a_i| + 1.
\end{equation}

\begin{theorem} \label{thm:fukaya}  The higher compositions in $\Fuk(M)$ satisfy the
  \ainfty axiom. The resulting category $\Fuk(M)$ is independent, up
  to homotopy equivalence of \ainfty categories, of choices of
  perturbation data and orientations.
\end{theorem} 

\begin{proof}   A coherent system of transversality data on the associahedra, together
with compactness and gluing results and the definition of the hom spaces of
non-transverse pairs is given in Ma'u-Wehrheim-Woodward \cite{Ainfty}; the 
remaining problem addressed here is to prove the claimed signs.
We consider the difference of orientations in the identification 
  \[ \partial \ol{\M}^d_1(\ul{x}) = \bigcup_{\Gamma}
  {\M}^{d}_{\Gamma,1}(\ul{x}) \] 
  of \eqref{eq:bound}.  Let $x_j \in \cI({L}^{j-1},{L}^{j})$ for
  $j = 0,\ldots, d$ indexed mod $d+1$.  Consider the gluing map
  \begin{equation} \label{eq:gmap} \M^m({y},{x}_{n+1},\ldots,
    {x}_{n+m})_0 \times \M^{d-m
      +1}({x}_0,{x}_1,\ldots,{y},\ldots,{x}_d)_0 \to
    \M^d({x}_0,\ldots,{x}_d)_1 .\end{equation}
The gluing map \eqref{eq:gmap} takes the form
(omitting tensor products from the notation to save space)
\begin{multline} 
 \det(T\RR^m) 
\DD^-_{x_{n+1}} \ldots
 \DD^-_{x_{n+m}} \det(T {L})   \DD^+_y 
 \\ \det(T\RR^{d-m+1}) 
 \DD^-_{x_1} \ldots \DD^-_y \ldots \DD^-_{x_d}  \det(T {L}) 
 \DD^+_{x_0}
\\ \cong \det(T\RR^{d-m+1}) 
 \DD^-_{x_1} \ldots  \DD^-_{x_n} \det(T\RR^m) \DD^-_{x_{n+1}} \ldots
 \DD^-_{x_{n+m}} \det(T {L})   \DD^+_y  \DD^-_y  \ldots \DD^-_{x_d}  \det(T {L}) 
 \DD^+_{x_0} \\ 
 \to \det(T\RR^d)
 \DD^-_{x_1} \ldots \DD^-_{x_d}  \det(TL) 
 \DD^+_{x_0}
.\end{multline}
Here the first isomorphism is orientation-preserving, using that the moduli space 
with $m$ inputs and one output $y$ is rigid, and so the product of determinant
lines on the first line is even degree. 
To determine the sign of the gluing map, first note that the gluing map
on source moduli spaces
\[ (0,\eps) \times \RR^m \times \RR^{d-m+1} \to \RR^d \] 
describing a neighborhood of the facet in the
associahedron is given in coordinates (using the automorphisms to fix
the location of the first and last marked point for $\RR^m$ and
$\RR^{d- m + 1}$) as in \cite[(32)]{Ainfty} by
\begin{multline} \label{eq:signs2}
 (\delta,(z_2,\ldots,z_{m-1}), (w_2,\ldots,w_{d-m})) \\ \to (w_2,
  w_3,\ldots, w_{n+1}, w_{n+1} + \delta z_2, \ldots, w_{n+1} + \delta
  z_{m-1} , w_{n+1} + \delta, w_{n+2},\ldots, w_{d-m}) .\end{multline}
This map acts on orientations by 
\[ (S1) = mn + m + n + 1  \ \text{mod} \ 2 . \] 
These signs combine with the contributions from $\heartsuit$
for both $\mu_m$ and $\mu_{d-m+1}$ which compare to a (independent of $n,m$) sign of $(-1)^{\sum_{k=1}^d (d-k) |x_k|}$ of 
\begin{eqnarray*}
    (S2) &=&   \sum_{i=1}^m (m-i)|x_{n+i}| + 
   \sum_{i=1}^n (d - m + 1 -i) |x_i| + (d - m + 1 - n + 1) |y| \\&& + \sum_{i=n+m+1}^d (d- i )|x_i| \\
    &=& 
 (m-1) \left(\sum_{i=1}^n  |x_{i}|  \right)
+ (d  - m - n)  \left( \sum_{i = n+1}^{n+m} |x_i| \right)
+ (d - m - n) |y|  \\ &=&   
(m-1) \sum_{i=1}^n |x_{i}|  
+   (d- m - n) m   \end{eqnarray*}
again using rigidity of the moduli space with output $y$ and inputs $x_{n+1},\ldots, x_{n+m}$.  We have a contribution from permuting $\det(T\RR^m)$ (which is the determinant line of a vector space of dimension $m -2$)  with 
$\det(T\RR^{d-m+1}) 
 \DD^-_{x_1} \ldots  \DD^-_{x_n} $ of 
\[ (S3) = m  \left( d- m + 1 + \sum_{i \leq n} |x_i| \right).\] 
On the other hand, the sign in the \ainfty axiom contributes
\[ (S4) = \sum_{i=1}^n (|x_i| - 1) .  \] 
Combining the signs one obtains in total
\begin{eqnarray*} (S) := (S1) + (S2) + (S3) + (S4)  &=&
 mn + m + n + 1  +  
 (m-1) \sum_{i \leq n} |x_{i}|  
+   (d- m - n) m \\ && 
m \left(  d- m + 1 +  \sum_{i \leq n} |x_i| \right)
+ \sum_{k=1}^n (|x_k| - 1) \\ 
&=& 1
. \end{eqnarray*}
The gluing computation in the unstable
  range $d-m+1 \leq 2$ or $m \leq 2$ is similar and the
  \ainfty-associativity relation follows.
\end{proof} 

\begin{remark}
\begin{enumerate} 
\item {\rm (Duals)} Let $M^-$ denote $M$ with the symplectic form reversed.  By Remark \ref{rem:rconjs}, each brane $L$ for $(M,b)$ defines a brane $L^-$ for $(M^-, b + w_2(M))$, and $L \mapsto L^-$ is a bijection on objects.  Recall that the {\em opposite} \ainfty category $\cC^{\on{op}}$ of an \ainfty category $\cC$ has the same objects, morphism spaces $\Hom_{\cC^{\on{op}}}(L_0,L_1) = \Hom_\cC(L_1,L_0)$,
and compositions 
\[ \mu^d_{\cC^{\on{op}}}(a_1,\ldots, a_d) = (-1)^{\sum_{i=1}^d 
( |a_i| - 1)} \mu^d_\cC(a_d,\ldots a_1) ; \] 
see for example \cite[Definition 2.3, (2.21)]{Ganatra}.  
Reversal of the complex structures on the domain and 
the almost complex structure on the target identifies
pseudoholomorphic polygons for $(M,L^0,\ldots, L^d)$
with pseudoholomorphic polygons for $(M^-, (L^d)^-,\ldots, 
(L^0)^-)$.   It reverses the cyclic order of the marked points and so gives bijections 
\[ \M^d(x_0,\ldots, x_d)_0 \to \M^d(x_0^-,x_d^-,\ldots, x_1^-)_0 .\]
To compute the change in orientations, note that 
reversing the order of the factors $\DD^-_{x_1},\ldots, 
\DD^-_{x_d}$ contributes a sign 
$\sum_{i < j} |x_i| |x_j|$.  Reflecting 
$\det(T\RR^d)$ by the map 
$(z_2,\ldots, z_{d-1}) \mapsto (- z_{d-1}, \ldots, -z_2)$
contributes $d + \binom{d}{2} + 1$.  Reversing the order
of the inputs changes the sign $(-1)^\heartsuit$ in \eqref{eq:highercomp} by $(d-1) \sum_{i=1}^d |x_i| $. 
Each end contributes as in Remark \ref{rem:rconjs}, and 
these contributions cancel against the shift in the background class by Lemma \ref{lem:togglesame}.
The total sign change is then $d+ 1 + \sum_{i < j}(|x_i| - 1)(|x_j| - 1)$.  That is, the composition maps are related by 
\begin{multline} \nonumber \mu^d_{\Fuk(M^-, b+ w_2(M))}(x_d^-,\ldots, x_1^-)
= (-1)^{d + 1 + \sum_{i < j} ( |x_i| - 1)(|x_j| - 1) } 
\mu^d_{\Fuk(M,b)}(x_1,\ldots, x_d)^- .\end{multline}
Define on each Hom space $\Hom(L_i,L_j)$
\[ \iota: CF(L_i,L_j) \to CF(L_i,L_j), \quad 
 \bra{x} \mapsto -(-1)^{(|x|-1)(|x|-2)/2} \bra{x} .\] 
The map $\iota$ intertwines with the compositions in $\Fuk(M^-, b+ w_2(M))^{\on{op}}$, 
and induces an isomorphism of \ainfty categories 
\[ \Fuk(M,b) \to \Fuk(M^-, b+ w_2(M))^{\on{op}} .\]
In particular, if the shift $w_2(M)$ in the background
class is trivial (for example, if all minimal Maslov numbers are divisible by $4$) then $\Fuk(M,b)$ is isomorphic to the opposite category of $\Fuk(M^-,b)$. 

\item {\rm (Disjoint Unions)} Let $\cC_j, j= 0,1$ be categories
  enhanced in groups.  Define a {\em disjoint union} category $\cC$ by
  taking an object to be an object in $\cC_0$ or $\cC_1$ and morphism
  groups to be trivial unless the two objects are objects of the same
  category $\cC_j, j = 0,1$. Suppose that $M_j, j = 0,1$ are compact
  monotone symplectic manifolds, and $M = M_0 \sqcup M_1$.  Then
\[\Fuk(M_0 \sqcup M_1) = \Fuk(M_0) \sqcup \Fuk(M_1) .\]
\item {\rm (Products)} Let $\cC_j, j = 0,1$ be categories enhanced in
  $\Z_N$ graded cochain complexes.  The {\em product category} $\cC$
  is the category whose objects are pairs of objects of $\cC_0$ and
  $\cC_1$, and whose morphism spaces are graded tensor product of
  morphism spaces of $\cC_j, j = 0,1$.  Let $M_j , j =0,1$ be compact
  monotone symplectic manifolds and $M = M_0 \times M_1$. Then
  $H(\Fuk(M))$ is the category obtained by taking the cohomology of
  the cochain-level categories underlying
  $H(\Fuk(M_0)), H(\Fuk(M_1))$.  In particular, 
  in the torsion-free case  (for example, over a field) 
  we have a full and faithful embedding given by the obvious product operation on objects:
  \[ H(\Fuk(M_0)) \otimes H(\Fuk(M_1)) \to 
  H(\Fuk(M_0 \times M_1)) .\]
The A-infinity version of this result is addressed in Amorim
\cite{amor:tens}. 
\end{enumerate} 
\end{remark} 

\subsection{Quilted Floer categories} 
\label{sec:fukops}

 The definition of the Fukaya category extends to generalized Lagrangian objects using quilts.  
 
 \begin{definition}
     \begin{enumerate}
         \item For a generalized Lagrangian brane
  $\ul{L} = (L_0, L_{01}, \ldots, L_{(k-1)k})$ (a sequence of
  correspondences from a point to $M$) through manifolds
  $M_0,\ldots,M_{k-1}$ to $M_k = M$ let $n_l = \tfrac{1}{2} \dim(M_l)$
  and denote its (first-order) {\em parity} and {\em second-order parity}
  \[ s(\ul{L}) = \sum_{l < k} n_l, \quad e(\ul{L}) = \sum_{l < l' < k}
  n_l n_{l'} .\]
  \item 
  The quilted Fukaya category
  $\Fuk^{\sharp}(M)$ has as objects the generalized Lagrangian branes
  $\ul{L}$ in $M$, equipped with relative brane structures, and as
  morphism spaces the shifted quilted Floer cochain groups
  \[ \Hom(\ul{L},\ul{L}') = CF(\ul{L},\ul{L}')[d(\ul{L},\ul{L}')]  \] 
  with the degree shifts 
  \[  d(\ul{L},\ul{L}') = \tfrac{1}{2} \Bigl( \sum_{N} \dim(N) + \sum_{N'}
  \dim(N') \Bigr) , \]
  defined over the integers, with the grading and shift of
  \cite[Section 4.1]{Ainfty}.  Here $CF(\ul{L},\ul{L}')$ is the quilted
  Floer cochain group of the cyclic correspondence obtained by
  concatenating $\ul{L}$ with the transpose of $\ul{L}'$, graded by the
  degrees in $\Z_N$ of the generalized intersection points induced by
  the gradings of the correspondences, and the sums are over the
  backgrounds $N$ of $\ul{L}$ and $N'$ of $\ul{L}'$ as in
  \cite[Section 4.1]{Ainfty}.  The composition maps are the maps
  $\mu^d_\sharp$ of \eqref{eq:musharp} below, defined by counting the
  rigid quilted polygons of \cite{Ainfty} for the objects
  $\ul{L}^0,\ldots,\ul{L}^d$, with the signs described below.
  \end{enumerate}
  \end{definition}

  Below $|x|$ denotes the degree in $\Hom(\ul{L},\ul{L}')$.  The objects
  may pass through different manifolds and have different lengths.  We
  order the patches of each quilted surface used below by their distance
  in the tree of patches from the patch labelled by the target manifold,
  the target patch first, and patches at the same distance in the order
  of the boundary arcs.  For the quilted surfaces defining the
  compositions in $\Fuk^\sharp(M)$, and the functors of Section
  \ref{sec:comprelspin}, the graph of patches is a tree, and we color
  each patch by the parity of its distance in this tree from the patch
  labelled by the target manifold.  When pieces whose colorings do not
  match are glued, we recolor a whole piece.  By Corollary
  \ref{cor:colorindep} this amounts to shifting all of its background
  classes by $w_2$.

  The signs require an additional normalization, which we first describe
  for quilted polygons whose objects $\ul{K}^0,\ldots,\ul{K}^d$ all
  satisfy $s(\ul{K}^j) \equiv s$ for some $s \in \Z_2$.  Let $\ul{K}^j$
  label the boundary arc between the $j$-th and $(j+1)$-st end.  The
  polygon has one patch labelled $M$ and, along the $j$-th boundary arc,
  one patch for each manifold through which $\ul{K}^j$ passes, so the
  closed-off factors do not have the parity of the end data and Lemma
  \ref{lem:quiltedgluing} \eqref{item:qgends} contributes the sign
  $\kappa$.  Moreover, since $\mu^d$ has degree $2-d$, the degree of a
  generator $x$ satisfies 
  \[ |x| \equiv \Ind(D^-_x) + \rank(E_x) + s
  \pmod 2 . \]
  Thus the factors $\DD^\pm_x$ have parity $|x| + s$. 

  \begin{definition} {\rm (Quilted composition maps for objects of fixed parity)}
  For objects $\ul{K}^0,\ldots,\ul{K}^d$ with $s(\ul{K}^j) \equiv s$ let
  \begin{equation} \label{eq:musharp} 
  \mu^d_\sharp(\bra{x_1},\ldots,\bra{x_d}) = \sum_{x_0}
  (-1)^{\heartsuit + \heartsuit_\sharp} \, \# \M^d(x_0,\ldots,x_d)_0 \,
  \bra{x_0} \end{equation}
  where 
  \[ \heartsuit_\sharp = s \sum_{j=1}^d (j-1)(|x_j| - 1) +
  \rho(\ul{K}^0,\ldots,\ul{K}^d) + e(\ul{K}^d) + s d ,\]
  with $\heartsuit$ as in \eqref{eq:heartsuit}, and $\rho$ the Koszul
  sign \eqref{eq:signs} of the permutation of the closed-off factors of
  the patches along the boundary arcs, of parities $n_l$, from the
  ordering of patches above to the ordering by boundary arcs, the arcs
  $1,\ldots,d$ first and the arc $0$ last, and along each arc from the
  patch farthest from the target patch to the nearest.  If all objects
  pass through the same manifolds, then
  \[ \heartsuit_\sharp = s \sum_{j}(j-1)(|x_j| - 1) + e \, d(d-1)/2 \] 
  with $e = e(\ul{K}^j)$.  
  \end{definition}

 For objects of
  different parities we describe the signs using a stabilization procedure.  This stabilization changes neither the objects, nor
  the generators of the morphism spaces, nor the moduli spaces of quilted
  polygons, but only the determinant lines used to orient them, and reduces to the even parity case.

  \begin{definition} {\rm (Stabilization)}
  Let $T^2 = S^1 \times S^1$ with its standard symplectic form, and let
  $\alpha = S^1 \times \{ \pt \}$ and $\beta = \{ \pt \} \times S^1$, two
  circles meeting transversally in a single point, equipped with
  gradings and spin structures; we use the background class $0$ on
  $T^2$.  The {\em stabilization} of a generalized Lagrangian brane
  $\ul{L}$ is the generalized Lagrangian brane
  \[ \on{St}(\ul{L}) = (\alpha, \beta \times L_0, L_{01}, \ldots,
  L_{(k-1)k}) \]
  through $T^2, M_0, \ldots, M_{k-1}$, with
  $s(\on{St}(\ul{L})) = s(\ul{L}) + 1$.  For $k \in \Z$ let
  $\on{St}^{\langle k \rangle}$ denote $\on{St}$ if $k$ is odd and the
  identity if $k$ is even, and let
  $\ul{L}^{\on{st}} = \on{St}^{\langle s(\ul{L}) \rangle}(\ul{L})$, so that
  $s(\ul{L}^{\on{st}})$ is even. This ends the Definition.
  \end{definition}

\begin{remark}  The quilts with some patches labelled by the two-torus fit naturally 
into the previous framework.    Since $\pi_2(T^2) = 0$, the torus $T^2$ is monotone for any monotonicity  constant.  Since $\pi_2(T^2,\alpha) = \pi_2(T^2,\beta) = 0$ the
  circles bound no pseudoholomorphic disks.  The composition
  $\alpha \circ (\beta \times L_0) = L_0$ is transverse and embedded, and
  $w_2(T^2) = 0$, so recoloring does not affect patches labelled $T^2$.
  In any quilt below, a patch labelled $T^2$ has boundary and seam
  conditions in $\alpha$ and $\beta$ and strip-like ends at
  $\alpha \cap \beta$; lifting to the universal cover shows that it is
  mapped to the point $\alpha \cap \beta$, and the operator on it is
  invertible.  The circles $\alpha$ and $\beta$ are not admissible, since
  their fundamental groups are not torsion, but this causes no
  difficulty: admissibility is used only to guarantee monotonicity of
  the quilted Floer problems, that is, the relation between energy and
  index and the absence of disk bubbling.  Since every patch labelled
  $T^2$ is constant and carries an invertible operator, it contributes
  neither energy nor index, and since
  $\pi_2(T^2) = \pi_2(T^2,\alpha) = \pi_2(T^2,\beta) = 0$ no sphere or
  disk bubbles form in it.  The same holds for limits of sequences, so
  the moduli spaces involving stabilized objects, together with their
  energies, indices and compactifications, are those of the
  corresponding quilts with admissible labels, and the Floer cochain
  groups and composition maps for stabilized objects are well-defined.
  Hence replacing objects by their stabilizations does not
  change the moduli spaces of quilts, and each strip labelled $T^2$ contributes the
  single point $\alpha \cap \beta$, so that $CF(\ul{L},\ul{L}')$ and
  $CF(\ul{L}^{\on{st}},\ul{L}'^{\on{st}})$ have the same generators.  We
  orient the caps of the strips labelled $T^2$ at each end as in Remark
  \ref{rem:composition} \eqref{item:compends}, so that the operators on
  the patches labelled $T^2$ are canonically oriented.  The
  determinant lines of the original and of the stabilized problems are
  then identified, and we use this identification to orient the moduli
  spaces $\M^d(x_0,\ldots,x_d)$ of quilted polygons for objects
  $\ul{L}^0,\ldots,\ul{L}^d$ of $\Fuk^\sharp(M)$, with the normalization
  $\heartsuit_\sharp$ of the objects $(\ul{L}^j)^{\on{st}}$, that is,
  \[ \heartsuit_\sharp = \rho((\ul{L}^0)^{\on{st}},\ldots,(\ul{L}^d)^{\on{st}})
  + e((\ul{L}^d)^{\on{st}}) . \]
  Thus the composition maps of $\Fuk^\sharp(M)$ count the quilted polygons
  for the objects $\ul{L}^j$, with the signs of the stabilized objects.
\end{remark}

  \begin{proposition} \label{prop:quiltext} 
  $\Fuk^\sharp(M)$ is an \ainfty category over the integers.
  \end{proposition}

\begin{proof}  By construction the composition maps are the maps
  $\mu^d_\sharp$ of \eqref{eq:musharp} for the objects
  $(\ul{L}^j)^{\on{st}}$, which all have even parity, and for objects of
  fixed parity the signs $\kappa$, the shifted parities of $\DD^\pm_x$,
  and the sign for reordering the patches of the glued polygon combine
  with (S1)--(S4) to give the signs of the \ainfty axiom.
\end{proof}

\begin{remark}    For later use we record the effect of adding a layer of patches labelled by the two-torus to the boundary of a quilt.  Replacing every
  object $\ul{K}^j$ of a quilted polygon, of fixed parity, by
  $\on{St}(\ul{K}^j)$ adds a patch labelled $T^2$ along each boundary
  arc and identifies the moduli spaces.  With the caps oriented as
  above, the counts normalized by \eqref{eq:musharp} for the objects
  $\ul{K}^j$ and for the objects $\on{St}(\ul{K}^j)$ differ by the factor
  \[ (-1)^{(n+1)d + 1} = (-1)^{n(|x_0| + |x_1| + \ldots + |x_d|) + d + 1}, \quad n = \tfrac{1}{2} \dim(M) , \]
  where the equality holds since $|x_0| \equiv |x_1| + \ldots + |x_d| + d$
  mod $2$ for rigid polygons.  That is, adding a layer is the strict
  isomorphism 
  \begin{equation} \label{eq:layer}
  \lambda_n : x \mapsto - (-1)^{n |x|} x ,
  \end{equation}
  since 
  \begin{eqnarray*}
  \lambda_n(\mu^d(x_1,\ldots,x_d)) &=& 
  (-1)^{n|x_0| + 1} \mu^d(x_1,\ldots,x_d), \\
  \mu^d(\lambda_n(x_1),\ldots,\lambda_n(x_d)) &=& (-1)^{n(|x_1| + \ldots + |x_d|) + d} \mu^d(x_1,\ldots,x_d),
  \end{eqnarray*}
  and the ratio of these two signs is the factor above.  
  Since $\lambda_n^2$ is the identity, adding two layers does not change the
  signs.  In particular, for objects of fixed even parity the signs of
  $\Fuk^\sharp(M)$ are those of \eqref{eq:musharp} without
  stabilization, and for objects of fixed odd parity they differ from
  those of \eqref{eq:musharp} by the strict isomorphism $\lambda_n$.

  To compute the factor, note that since the patches labelled $T^2$
  are constant with invertible operators, and their caps are oriented
  as in Remark \ref{rem:composition} \eqref{item:compends}, they play
  the role of the patches labelled $M_1$ in Proposition
  \ref{prop:quiltcyl}: the identification of the determinant lines of
  the quilts for $\on{St}(\ul{K}^j)$ and for $\ul{K}^j$ has sign
  $\eps_K \eps_\Gamma \eps_+ \eps_1$, and the factor is this sign times
  the change in the normalization $\heartsuit_\sharp$, in which $s$ is
  replaced by $s + 1$ and $\rho, e$ by their values for the objects
  $\on{St}(\ul{K}^j)$.  At each incoming end the two strips labelled $T^2$ are the
  first and the last strip, and the strips between them have total rank
  $n + s(\ul{K}^{j-1}) + s(\ul{K}^j) \equiv n$, so 
  $\eps_\Gamma = (-1)^{nd}$.  The remaining signs depend only on the
  positions of the factors of the patches labelled $T^2$ in
  \eqref{eq:order}: the closed-off factors of these $d+1$ patches and
  the caps split off at the $d$ incoming ends have odd index, the caps
  split off at the outgoing end have total index one, and the factors
  they are moved past are the factors $\det(D^\pm_x)$, of parity
  $|x| + n + s$ resp. $|x| + s$, the fibers
  $\Lambda^{\max}(\Gamma_{\ul{e}}(0)^\dual)$, of parity $n$, and the
  closed-off factors of the other patches, of parities $n$ and $n_l$.
  The sign $\eps_1$ includes the sign $-1$ of Remark \ref{rem:compatops}
  \eqref{item:sumcompat} comparing the compatible orientation of the
  pair of outgoing caps labelled $T^2$ with the product of the
  compatible orientations of the two caps, whose indices are $0$ and $1$.  A direct count shows that the dependence on $n$
  cancels in $\eps_K \eps_+$, that the dependence on the degrees $|x_j|$
  and on the patches of the objects cancels against the change in
  $\heartsuit_\sharp$, and that $\eps_K \eps_+ \eps_1$ combines with
  this change to $(-1)^{d+1}$.  Together with $\eps_\Gamma$ this gives
  the factor above.
\end{remark}

Admissible Lagrangian correspondences define \ainfty functors between
these quilted Fukaya categories.  Let $L_{01} \subset M_0^- \times M_1$ be an admissible correspondence with a brane structure, with $n_0, n_1$ the
  half-dimensions of $M_0, M_1$.  The functor $\Phi(L_{01})$ of
  \cite{Ainfty} maps an object $\ul{L}$ to $(\ul{L}, L_{01})$, and its
  coefficients are counts of quilts with inputs in $\Fuk^\sharp(M_0)$.
  For the signs we use again the stabilized objects. 
  
  \begin{definition} 
  For objects
  $\ul{L}^j$ of $\Fuk^\sharp(M_0)$ and an integer $a \ge 0$ consider the
  quilts defining $\Phi(L_{01})$ for the inputs
  $\on{St}^a((\ul{L}^j)^{\on{st}})$, that is, with $a$ additional patches
  labelled $T^2$ along each boundary arc, the {\em $a$-sign} of the quilt is
  $(-1)^{\heartsuit + \nu_a}$, with 
  \begin{equation} \label{eq:nua}
  \nu_a = \heartsuit_\sharp + (\langle a \rangle + n_0) n_1 |x_0| +
  \langle a \rangle (n_0 + n_1 + 1) d ,\end{equation}
  where $\langle a \rangle \in \{ 0, 1 \}$ is the parity of $a$,
  $\heartsuit_\sharp$ is the normalization \eqref{eq:musharp} for the
  objects $\on{St}^a((\ul{L}^j)^{\on{st}})$ and the inputs
  $x_1,\ldots,x_d$, and $x_0$ is the output. 
  \end{definition}

  \begin{definition}  The \ainfty functor associated to an admissible Lagrangian correspondence $L_{01}$ is the functor
  \[ \Phi(L_{01}) : \Fuk^\sharp(M_0) \to \Fuk^\sharp(M_1), \quad \ul{L}
  \mapsto (\ul{L}, L_{01}) \]
whose coefficients are the counts of the quilts of \cite{Ainfty}, oriented
  by the stabilized problem with $a = \langle n_0 \rangle$ (that is, with
  a patch labelled $T^2$ along each boundary arc if $n_0$ is odd), with
  the normalization
  \[ \nu = \nu_{\langle n_0 \rangle} = \heartsuit_\sharp + n_0 n_1 d . \]
  The inputs $x_1,\ldots,x_d$ are generators of the morphism spaces of
  $\Fuk^\sharp(M_0)$, which are also the generators for the objects
  $\on{St}^a((\ul{L}^j)^{\on{st}})$, $a = \langle n_0 \rangle$; the term
  $\langle a \rangle d$ in the coefficient
  $\langle a \rangle(n_0+n_1+1)d$ of $\nu_a$ accounts for the
  identification of the polygons for these objects with those of
  $\Fuk^\sharp(M_0)$ by the strict isomorphism $\lambda_{n_0}^a$ of
  \eqref{eq:layer}, see the proof of Proposition \ref{prop:functor}.
  A similar definition holds for a sequence of correspondences
  $(L_{01}, L_{12}, \ldots, L_{(r-1)r})$ from $M_0$ to $M_r$, with
  $n_1 |x_0|$ replaced by $n_r |x_0|$, $\langle a \rangle + n_0$ by
  $\langle a \rangle + n_0 + \ldots + n_{r-1}$, and $n_0 + n_1 + 1$ by
  $n_0 + \ldots + n_r + 1$.  \end{definition}

\begin{proposition} \label{prop:functor} $\Phi(L_{01})$ is an \ainfty functor.
\end{proposition}

\begin{proof} The polygons that break off at the inputs have the
  objects $\on{St}^a((\ul{L}^j)^{\on{st}})$, and are identified with
  those of $\Fuk^\sharp(M_0)$ by adding layers as above, that is, by
  the strict isomorphism $\lambda_{n_0}^a$ of \eqref{eq:layer}. 
  The outputs are the objects
  $(\on{St}^a((\ul{L}^j)^{\on{st}}), L_{01})$, of parity $a + n_0$; for
  $a = \langle n_0 \rangle$ they are even and differ from the
  stabilizations $(\ul{L}^j, L_{01})^{\on{st}}$ by an even number of
  layers, which does not change the signs.
  With the
  normalization $\nu_a$ of \eqref{eq:nua} the signs $\kappa$, the shifted parities of
  $\DD^\pm_x$, and the signs for reordering patches combine to give the
  signs of the \ainfty functor relations.  Replacing $a$ by $a + 1$
  multiplies the normalized counts by $(-1)^{n_1 |x_0| + 1}$; that is, adding
  a layer to the quilts for a functor composes it with the strict
  isomorphism $\lambda_{n_1}$ of \eqref{eq:layer} on the target.  Indeed, as in the computation
  for polygons above, adding a layer changes the orientations of the
  quilts, together with $\heartsuit_\sharp$, by
  $(-1)^{(n_0 + n_1 + 1) d + 1}$, where now $\eps_\Gamma = (-1)^{n_0 d}$,
  and the last two terms of $\nu_a$ change by $n_1 |x_0| + (n_0 + n_1 + 1) d$.
\end{proof}

\subsection{Inserting a diagonal for pseudoholomorphic quilts} 
\label{sec:inserting}

In this and the following subsection we investigate the effect of
composition of seam conditions on pseudoholomorphic quilt invariants.  The
first step is to investigate the effect of inserting a diagonal
seam.

\begin{definition}  {\rm (Inserting a diagonal Lagrangian seam condition)} 
A triple $(\ul{S}^{\sDel}, \ul{M}^{\sDel},\ul{L}^{\sDel})$ is obtained
from a labelled quilted surface $(\ul{S}$, $\ul{M}$, $\ul{L})$ by {\em
  inserting a diagonal} if and only if
\begin{enumerate} 
\item the quilted surface $\ul{S}^{\sDel}$ is obtained from $\ul{S}$
  by inserting a new seam $\sigma$ into a patch $S_p$ of $\ul{S}$;
\item the labels $\ul{M}^{\sDel},\ul{L}^{\sDel}$ are obtained by
  inserting a diagonal seam condition.
  That is, if $M_p$ is the symplectic manifold labelling $S_p$ then
  the patches $S_p', S_p''$ are labelled $M_p$, and the new seam $\sigma$ is
  labelled $\Delta_{M_p}$.
\end{enumerate} 
\end{definition} 

\begin{proposition} \label{prop:delta}  {\rm (Isomorphism of Floer homologies
and relative invariants for insertion of separating diagonals)}
  Suppose that the new seam $\sigma$ is inserted into the component $S_p$, and
  that the new seam $\sigma$ is separating.  Then there exists a collection of
  isomorphisms 
\[HF(\ul{L}_{\ul{e}}) \to HF(\ul{L}_{\ul{e}^{\sDel}}^{\sDel}), \quad e
  \in \mE(\ul{S}) .\]
These are induced by isomorphisms of cochain complexes which
intertwine the cochain level relative invariants
$C\Phi_{\ul{S}}, C\Phi_{\ul{S}^{\sDel}}$ defined by
$\ul{S},\ul{S}^{\sDel}$, up to the sign $\nu$ of Proposition
\ref{prop:arc}; in particular, with coefficients in a field, they
intertwine $\Phi_{\ul{S}}$ and $\Phi_{\ul{S}^{\sDel}}$ up to $\nu$.
\end{proposition}

\begin{proof}
We take the perturbation data for $\ul{L}^{\sDel}$ to be induced by
perturbation data for $\ul{L}$.  Then 
$\cI(\ul{L}_{\ul{e}})$ and $\cI(\ul{L}^{\sDel}_{\ul{e}^{\sDel}})$ are
canonically in bijection,
as are the moduli spaces of quilted strips.  The isomorphisms $\Phi_u: \det(D_u) \to \det(D_u^{\sDel})$
of Remark \ref{rem:DDp} satisfy the hypotheses of Lemma 
\ref{lem:compare}: They are continuous in families and 
commute with gluing, since both are defined by patching 
kernels and cokernels; gluing of quilted strips is orientation preserving for both $\ul{L}_{\ul{e}}$
and $\ul{L}^{\sDel}_{\ul{e}}$ by Lemma 
\ref{lem:quiltedgluing}; and on translation-invariant strips both orientations are canonical by Remark \ref{rem:dependence} \eqref{trivbun}.

Strip data with fixed ends but different indices are related, up to
homotopy, by gluing at a boundary node a disk $D$ whose boundary
condition is a loop of Maslov index $k$ in the boundary or seam
condition of some patch $q$.  The boundary-node gluing
signs are the same for $\ul{L}_{\ul{e}}$
and $\ul{L}^{\sDel}_{\ul{e}}$ by Lemma \ref{lem:quiltedgluing}.  If $q \subset \ul{S}_-$
the disk is unchanged by the insertion, 
while if $q \subset \ul{S}_+$ then its color changes and its background
class is shifted by $w_2$, by the choice of relative spin structure in
Proposition \ref{prop:circle}, which by Corollary \ref{cor:colorindep}
does not change its orientation.  Hence $\Phi$ is orientation
preserving on $D$, and $\chi = 1$.

We show that the isomorphisms intertwine with insertion maps.
By Lemma \ref{lem:compare}, 
$\eps^{\sDel}(u) = \eta(x_-) \eta(x_+) \eps(u)$ for every
rigid quilted strip $u$, so that 
\[ CF(\ul{L}_{\ul{e}}) \to CF(\ul{L}^{\sDel}_{\ul{e}^{\sDel}}),
\quad \bra{x} \mapsto \eta(x) \bra{x}  \]
is an isomorphism of cochain complexes.  For a quilted
surface $\ul{S}$, the same argument, gluing strips onto each end, shows that the sign of $\Phi_u$ for $u \in \M_{\ul{S}}(\ul{x})$ is $\nu \prod_{e} \eta(x_e)$, with $\nu$
the constant of Proposition \ref{prop:arc}; hence the isomorphisms intertwine $C\Phi_{\ul{S}}$ and 
$C\Phi_{\ul{S}^{\sDel}}$ up to the overall sign $\nu$.
\end{proof}

\begin{remark}  {\rm (The spin case)}  In the case that $M_p$ is spin, the diagonal 
$\Delta_p$ is also spin.  So $\Delta_p$ has a relative spin structure
  with background classes $(0,0)$.  Thus the periodic Floer cohomology
  $HF(\Delta_p)$ is well-defined (after identifying with $HF(\Delta_p,\Delta_p)$, to ensure that the quilted surface admits a coloring.)  In general without the spin
  assumption, the quilted Floer cohomology $HF(\Delta_p)$ may be
  defined as the periodic Floer cohomology $HF(\Id_{M_p})$ of the
  identity on $M_p$, that is, treating $\Delta_p$ as the generalized
  Lagrangian correspondence of length $0$.  The isomorphism of Floer
  homology groups can be defined as follows from the isomorphisms
  $HF(\Delta_p,\Delta_p) \to QH(M_p)$.  Let
\[\phi_e: HF(\ul{L}_{\ul{e}}) \to HF(\ul{L}^{\sDel}_{\ul{e}(\sDel)}),
  \quad \psi_e : HF(\ul{L}^{\sDel}_{\ul{e}(\sDel)}) \to
  HF(\ul{L}_{\ul{e}}) \]
denote the morphism associated to the quilted surface shown in Figure
\ref{fig:insertid} resp. the reversed surface.  
\begin{figure}[ht]
\begin{picture}(0,0)%
\includegraphics{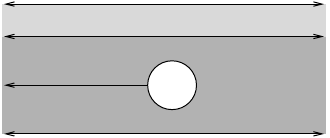}%
\end{picture}%
\setlength{\unitlength}{4144sp}%
\begingroup\makeatletter\ifx\SetFigFont\undefined%
\gdef\SetFigFont#1#2#3#4#5{%
  \reset@font\fontsize{#1}{#2pt}%
  \fontfamily{#3}\fontseries{#4}\fontshape{#5}%
  \selectfont}%
\fi\endgroup%
\begin{picture}(2499,1046)(2689,-1427)
\put(3939,-1028){\makebox(0,0)[lb]{$1_{M_p}$}%
}
\put(3072,-1028){\makebox(0,0)[lb]{$\Delta_{M_p}$}%
}
\end{picture}%
\caption{Isomorphism of Floer homologies after inserting a seam}
\label{fig:insertid}
\end{figure}
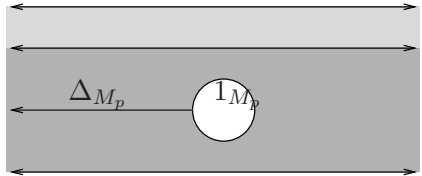
In other words, to the infinite strip we add a cylindrical end in the
component separated by the seam $\sigma$, and insert at that
cylindrical end the identity in $1_{M_p} \in HF(\Delta_{M_p})$.  The
identities
\[ \psi_e \phi_e = 1_{HF(\ul{L}_{\ul{e}})}, \ \ \phi_e \psi_e = 1_{HF(\ul{L}^{\sDel}_e)} \]
follow from Proposition \ref{prop:indeporinv}
applied to the surface
on the inner circle in Figure \ref{fig:removeseam}.  Compatibility with
the relative invariants is proved in the same way.   This ends the Remark. 
\end{remark}

\begin{figure}[ht]
\begin{picture}(0,0)%
\includegraphics{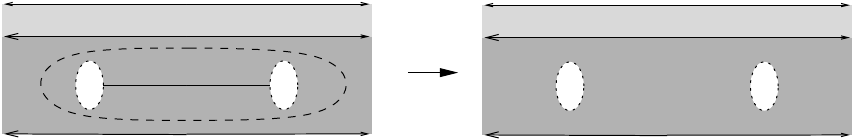}%
\end{picture}%
\setlength{\unitlength}{4144sp}%
\begingroup\makeatletter\ifx\SetFigFont\undefined%
\gdef\SetFigFont#1#2#3#4#5{%
  \reset@font\fontsize{#1}{#2pt}%
  \fontfamily{#3}\fontseries{#4}\fontshape{#5}%
  \selectfont}%
\fi\endgroup%
\begin{picture}(6506,1057)(214,-1442)
\put(2342,-1028){\makebox(0,0)[lb]{\smash{{\SetFigFont{5}{6.0}{\rmdefault}{\mddefault}{\updefault}{$1_{M_p}$}%
}}}}
\put(1051,-1042){\rotatebox{360.0}{\makebox(0,0)[rb]{\smash{{\SetFigFont{5}{6.0}{\rmdefault}{\mddefault}{\updefault}{$1_{M_p}$}%
}}}}}
\put(1502,-997){\makebox(0,0)[lb]{\smash{{\SetFigFont{5}{6.0}{\rmdefault}{\mddefault}{\updefault}{$\Delta_{M_p}$}%
}}}}
\put(6003,-1036){\makebox(0,0)[lb]{\smash{{\SetFigFont{5}{6.0}{\rmdefault}{\mddefault}{\updefault}{$1_{M_p}$}%
}}}}
\put(4712,-1050){\rotatebox{360.0}{\makebox(0,0)[rb]{\smash{{\SetFigFont{5}{6.0}{\rmdefault}{\mddefault}{\updefault}{$1_{M_p}$}%
}}}}}
\end{picture}%
\caption{Removing a seam}
\label{fig:removeseam}
\end{figure}
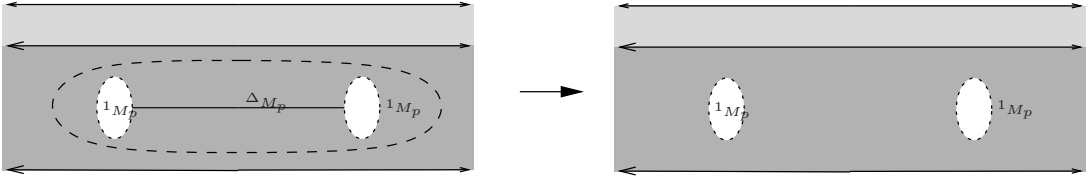

\subsection{Orientations for compositions of Lagrangian 
correspondences}
\label{sec:comprelspin}

In this final step we investigate the effect of replacing a triple of
adjacent seam conditions, the middle of which is a diagonal, with the
composed condition.

\begin{definition} \label{def:composedseams} {\rm (Composed Lagrangian seam conditions)}  
Let $\ul{S}$ denote a quilted surface, $\ul{M}$ a set of symplectic
manifolds for the components of $\ul{S}$, and $\ul{L}$ a collection of
Lagrangian and seam conditions.  Suppose that $\ul{S}$ contains a pair
of adjacent components $M_1,M_1$ diffeomorphic to infinite strips,
with boundary conditions $L_{01},\Delta_1,L_{12}$, or a single
component labelled $M_1$ with boundary conditions $L_{01}, L_{12}$; the
latter is obtained from the former by removing the seam $\Delta_1$, and
the arguments below apply to both, since Proposition \ref{prop:quiltcyl}
only involves the operator on the patches labelled $M_1$ as a whole.  Let $\ul{S}^\circ$
denote the surface obtained by removing the $M_1$-labelled components.  The
{\em composed Lagrangian seam conditions}
\[\ul{L}^\circ = \ul{L}/ 
((L_{01},\Delta_1,L_{12}) \mapsto   L_{01} \circ L_{12} ) \] 
assuming that the composition is smooth and embedded by the projection
onto $M_0^- \times M_2$.  If $\ul{L}$ are equipped with
relative spin structures, then $\ul{L}^\circ$ inherits a relative spin structure from
the composition, with background classes on the patches equal to those inherited
from $\ul{L}$.
\end{definition}  

\begin{proposition}  \label{prop:collapse} {\rm (Geometric composition theorem)} 
Suppose that $\ul{S}^\circ$ is obtained from $\ul{S}$ and the quilt
data for $\ul{S}^\circ$ is obtained from quilt data $\ul{M},\ul{L}$
for $\ul{S}$ by replacing a triple of seams $L_{01}, \Delta_1, L_{12}$
with the geometric composition $L_{02}$.  Suppose that $\ul{M}$ are
compact monotone with the same monotonicity constants and
$\ul{L},\ul{L}^\circ$ are admissible correspondences so that the
quilted Floer cohomologies and relative invariants are well-defined.
For each quilted end $\ul{e}$ changed by the replacement to a quilted
end $\ul{e}^\circ$ there exists an isomorphism (with a degree shift explained
in \cite[Section 5]{ww:quilts})
\[HF(\ul{L}_{\ul{e}})
\to HF(\ul{L}_{\ul{e}^\circ}^\circ)\] 
such that the tensor products over the negative and positive ends of
$\ul{S},\ul{S}^\circ$ intertwine the relative invariants
$\Phi_{\ul{S}}, \Phi_{\ul{S}^{\circ}}$ for $\ul{S},\ul{S}^\circ$, up to
the signs of Proposition \ref{prop:quiltcyl}.
\end{proposition}  

\begin{proof}
For $\Z_2$ coefficients this is Theorem 5.4.1 of 
\cite{ww:quiltfloer} whose proof uses the adiabatic limit
analysis of \cite{ww:isom}.  We compare the orientations of the moduli spaces induced by the relative spin structures. 
The linearization of the map constructed in \cite{ww:isom}
at a solution of the quilted problem for $\ul{S}$ is, up to a small correction, the projection onto the components other than those labelled $M_1$.  Under the identification 
of the tangent spaces of the moduli spaces with kernels of Cauchy-Riemann operators with totally real boundary and seam conditions, it therefore induces the isomorphisms 
$\ker(D_{\ul{S}}) \to \ker(D_{\ul{S}^\circ})$ 
and $\coker(D_{\ul{S}}) \to \coker(D_{\ul{S}^\circ})$, 
hence an identification 
$\det(D_{\ul{S}}) \to \det(D_{\ul{S}^\circ})$, which is the
identification $\det(\ul{D}') \to \det(\ul{D})$ of Proposition
\ref{prop:quiltcyl} for $\ul{S}' = \ul{S}$ and $\ul{S}$ replaced by $\ul{S}^\circ$.  By Proposition \ref{prop:quiltcyl} this identification agrees with the one obtained from the deformation 
of the seam conditions $F_{01} \oplus \Delta_1 
\oplus F_{12}$  to the split form 
$\sigma_{1423}(F_{02} \oplus \Delta_1^\perp) \oplus \Delta_1$. 
This deformation exists and is canonical up to homotopy, because by Lemma 3.1.9 of \cite{ww:quiltfloer} 
the space of Lagrangian subspaces
transverse to $V_0 \times \Delta_{V_1} \times V_2$ with fixed composition is contractible. 
Hence the identification $\det(D_{\ul{S}}) \to \det(D_{\ul{S}^\circ})$ has the
sign of Proposition \ref{prop:quiltcyl}, which proves the statement of
the Proposition.
\end{proof}

\begin{corollary}  Given Lagrangian correspondences $L_{01}$, $L_{12}$, $L_{02}$, $L_{20}$
with admissible brane structures such that $L_{02} := L_{01} \circ
L_{12}$ is smooth and embedded, the canonical bijection 
\[ \cI(L_{01},\Delta_1,L_{12},L_{20}) \to \cI(L_{02},L_{20}) \] 
induces an isomorphism
\[ HF(L_{01}, \Delta_1, L_{12}, L_{20}) \to
HF(L_{02},L_{20}) \] 
of quilted Floer cohomology groups with integer coefficients.
\end{corollary}

Functors for Lagrangian correspondences equipped with brane structures
were constructed in \cite{Ainfty} by the authors together with
S. Ma'u.  For each admissible Lagrangian correspondence
$L_{01} \subset M_0^-\times M_1$ equipped with a brane structure
counting quilts with seam in $L_{01}$ defines an \ainfty functor
\[\Phi(L_{01}) : \GFuk(M_0) \to
\GFuk(M_1) \] 
acting in the expected way on Floer cohomology: for Lagrangian branes
$L_0 \subset M_0, L_1 \subset M_1$ there is an isomorphism with
$\Z_2$-coefficients
\[ H \Hom( \Phi(L_{01})L_0, L_1) \cong HF(L_0 \times L_1, L_{01}) . \]
With integer coefficients, by Section \ref{sec:fukops}, $\Phi(L_{01})$ is a functor $\Fuk^\sharp(M_0) \to \Fuk^\sharp(M_1)$, mapping $\ul{L}$ to
$(\ul{L}, L_{01})$; if $n_0$ is odd, its quilts are oriented using a patch
labelled $T^2$ along each boundary arc.  The normalizations of the
functors in the proof of Theorem \ref{thm:maincompose} are the
corresponding cases of Section \ref{sec:fukops}.
With the coloring convention of Section \ref{sec:fukops}, the patch labelled $M_0$ in the quilts defining
$\Phi(L_{01})$ has color $1$, and the polygons for $\Fuk^\sharp(M_0)$ glued
to it at the inputs are recolored.  Let $b_0$ denote the background
class of this patch, that is, the $M_0$-component of the background
class of $L_{01}$.  By Corollary \ref{cor:colorindep} the orientations
of the recolored polygons are those of $\Fuk^\sharp(M_0, b_0 + w_2(M_0))$, with
all background classes of the patches of the objects shifted.
Thus
\[ \Phi(L_{01}) : \Fuk^\sharp(M_0, b_0 + w_2(M_0)) \to \Fuk^\sharp(M_1,b_1),
\quad \ul{L} \mapsto (\Upsilon(\ul{L}), L_{01}) ,\]
with $\Upsilon$ as in \eqref{eq:Ups} applied to every patch of
$\ul{L}$; it does not change patches labelled $T^2$ used for the signs,
since $w_2(T^2) = 0$.  In the same way, for $L_{12}$
with background classes $(b_1', b_2)$, the functor $\Phi(L_{12})$ is
defined on the objects $(\ul{L}, L_{01})$ of $\Fuk^\sharp(M_1, b_1' + w_2(M_1))$, and applying $\Upsilon$ to
every patch of a generalized Lagrangian brane, so that the composition
$\Phi(L_{12}) \circ \Phi(L_{01})$ is defined if
$b_1' = b_1 + w_2(M_1)$.  In the quilts obtained by gluing quilts for
$\Phi(L_{01})$ to the inputs of quilts for $\Phi(L_{12})$, the colors of
the pieces for $\Phi(L_{01})$ do not match, and these pieces are
recolored as in Section \ref{sec:fukops}; by
Corollary \ref{cor:colorindep} this shifts the background class of
their patch labelled $M_1$ from $b_1$ to $b_1 + w_2(M_1) = b_1'$, that
of the patch labelled $M_1$ of the quilts for $\Phi(L_{12})$.  So the
patches labelled $M_1$ merge into a single patch, with seam
conditions $L_{01}$ and $L_{12}$ on its two sides, and no diagonal seam
is needed.

\begin{remark} \label{rem:collapsesigns} {\rm (Collapse signs for the quilts of Theorem
\ref{thm:maincompose})}
\begin{enumerate}
\item The patches labelled $M_1$ of the quilts in Theorem
  \ref{thm:maincompose} are of three kinds: strips from an incoming to
  an outgoing end, whose operators with split conditions are
  canonically oriented by Remark \ref{rem:cases}
  \eqref{case:strip}; disks with two incoming ends, which are
  canonically oriented by the choice in Remark \ref{rem:composition} \eqref{item:compends}; and, in the
  quilts for $\cH$, a band meeting the output end twice, a disk with
  two outgoing ends.  The latter is also canonically oriented.  Indeed,
  let $S_\cup$ be such a disk, with outgoing caps $D^+_1, D^+_2$ in
  the order across the end, and let $S_\cap$ be a disk with two
  incoming ends with caps $D^-_1, D^-_2$, where $D^-_1$ has the split
  conditions of $D^+_2$ and $D^-_2$ those of $D^+_1$.  Gluing the
  second outgoing end of $S_\cup$ to the first incoming end of
  $S_\cap$ gives a strip with constant split conditions and the
  compatible caps $D^-_2$, $D^+_1$, which is canonically oriented by
  Remark \ref{rem:cases} \eqref{case:strip}.  The cap $D^-_2$
  is the cap of the first strip of a pair, so
  \[ k = \Ind(D^-_2) \equiv \rank(E_1),\quad \Ind(D^+_1) = \rank(E_1) - k  \] 
  the first by the choice in Remark \ref{rem:composition} \eqref{item:compends}. 
  By Theorem
  \ref{thm:gluingsigns} \eqref{case:gluingends}, with the ordering
  $(e_-,e_+)$ of the glued ends used throughout,
  the gluing has sign 
  \[ \dagger \diamondsuit = (-1)^{k(\rank(E_1) + 1)
  + \rank(E_1)(\rank(E_1) + k)} = (-1)^{k + \rank(E_1)} = 1 .\]
  Since
  gluing of invertible operators identifies their canonical
  orientations, and $S_\cap$ is canonically oriented, so is
  $S_\cup$.
\item For the polygons defining the composition maps in
  $\Fuk^\sharp(M_2)$ with objects of the form $(\ul{L}, L_{01}, L_{12})$,
  collapsed to $(\ul{L}, L_{02})$, where $\ul{L}$ are generalized
  Lagrangian branes in $M_0$ with $s(\ul{L}) \equiv \epsilon$, the sign of
  Proposition \ref{prop:quiltcyl}, which includes
  $\eps_\Gamma = (-1)^{n_1 n_2 d}$ from the $d$ inputs, combines with the
  normalizations $\heartsuit + \heartsuit_\sharp$ of Section
  \ref{sec:fukops}  for the two kinds of objects to
  $(-1)^{\psi(x_0) + \psi(x_1) + \ldots + \psi(x_d)}$, where
  $n_l = \tfrac{1}{2} \dim(M_l)$ and
  \[ \psi(x) = n_1 (n_2 |x| + \epsilon + n_0) ; \]
  note that $\epsilon + n_0 = s(\ul{L}, L_{02})$.  Hence the bijection of
  generators of Proposition \ref{prop:collapse}, multiplied by
  $(-1)^{\psi(x)}$, is a strict isomorphism $\Psi_0$ of the
  corresponding \ainfty structures.
\item Consider the quilts used in the proof of Theorem
  \ref{thm:maincompose}, with seams $L_{01}, L_{12}$, the strip
  labelled $M_1$ of small width $\delta$, and
  $a = \langle n_0 \rangle + \langle n_1 \rangle$ patches labelled $T^2$
  along each boundary arc, and their collapse to the quilts for
  $\Phi(L_{01} \circ L_{12})$ with the same $a$ patches labelled $T^2$. 
  The sign of Proposition \ref{prop:quiltcyl} is
  $(-1)^{n_1 (\epsilon d + \epsilon + n_0)}$, where
  $\epsilon = \langle a \rangle$ is the parity of the input objects;
  here $\eps_\Gamma = 1$, since the inputs carry no strips labelled
  $M_1$.  If the quilts of the first type are counted with the normalization
  $(-1)^{\heartsuit + \nu_a}$ of \eqref{eq:nua} for the sequence
  $(L_{01}, L_{12})$, denoted $\nu(\cH)$ in the
  proof of Theorem \ref{thm:maincompose}, and the quilts of the second
  type with the normalization
  $(-1)^{\heartsuit + \nu_a}$ of \eqref{eq:nua}
  for $L_{01} \circ L_{12}$, the two counts
  differ by $(-1)^{n_1 (n_2 |x_0| + \epsilon + n_0)}$.  The quilts for
  $\Phi(L_{01} \circ L_{12})$ carry $\langle n_0 \rangle$ patches labelled
  $T^2$; removing the additional $\langle n_1 \rangle$ layers multiplies
  the count by $(-1)^{n_1 (n_2 |x_0| + 1)}$, the effect of
  $\lambda_{n_2}^{\langle n_1 \rangle}$ in \eqref{eq:layer}.  Since $\epsilon \equiv n_0 + n_1$,
  \[n_1(\epsilon + n_0 + 1) \equiv n_1(n_1 + 1) \equiv 0 .\] 
  The count of the quilts of the first type is therefore equal to the coefficient of
  $\Phi(L_{01} \circ L_{12})$, which is the coefficient of
  $\Psi \circ \Phi(L_{01} \circ L_{12})$ for the isomorphism $\Psi$ of
  Theorem \ref{thm:maincompose}.
\end{enumerate}
\end{remark}

Parametrized
versions of the arguments for invariance under geometric composition
provide the signs necessary for the following theorem:

\begin{theorem} \label{thm:maincompose}
{\rm (Geometric composition theorem)} Suppose that $M_0,M_1,M_2$ are
monotone symplectic manifolds with the same monotonicity constant.
Let $L_{01} \subset M_0^- \times M_1, L_{12} \subset M_1^- \times M_2$
be admissible Lagrangian correspondences with gradings and relative
spin structures with background classes $(b_0,b_1)$ resp.
$(b_1 + w_2(M_1), b_2)$, such that $L_{01} \circ L_{12}$ is smooth, embedded in $M_0^-
\times M_2$, and admissible, and equip $L_{01} \circ L_{12}$ with the
induced grading and relative spin structure, with background classes
$(b_0,b_2)$.  Then there exists a homotopy of \ainfty functors
$\Fuk^\sharp(M_0, b_0 + w_2(M_0)) \to \Fuk^\sharp(M_2, b_2)$
\[ 
\Phi(L_{12}) \circ \Phi(L_{01}) \cong \Psi \circ \Phi(L_{01} \circ L_{12}) ,
\]
where $n_l = \tfrac{1}{2} \dim(M_l)$ and $\Psi$ is the strict \ainfty
isomorphism from the full subcategory of $\Fuk^\sharp(M_2, b_2)$ on the
objects $(\Upsilon(\ul{L}), L_{01} \circ L_{12})$ to the full
subcategory on the objects $(\ul{L}, \Upsilon(L_{01}), L_{12})$, given on
generators by the bijection of Proposition \ref{prop:collapse}.  In terms of the stabilized objects of
Section \ref{sec:fukops}, $\Psi$ is the composition of adding
$\langle n_1 \rangle$ layers, which is the strict isomorphism
$\lambda_{n_2}^{\langle n_1 \rangle}: x \mapsto (-1)^{n_1 (n_2 |x| + 1)} x$ of \eqref{eq:layer}, with the inverse of the strict
isomorphism $\Psi_0$ of Remark \ref{rem:collapsesigns}, for which
$\epsilon + n_0 \equiv n_1$, up to an even number of layers.
\end{theorem}  

\noindent This ``composition commutes with categorification'' theorem is, in
some sense, the main result of \cite{Ainfty}. 

\begin{proof}
The homotopy is constructed in \cite[Theorem 7.20]{Ainfty} by counting rigid elements of the parametrized moduli spaces
$\M^d_{(0,\infty)}$ of quilts over the families $\ul{S}_d$
of quilted disks with seam conditions $L_{01}, L_{12}$ in which the strip labelled by $M_1$ 
has width $\delta > 0$, and the ends of the one-dimensional 
moduli space are identified there for $\Z_2$ coefficients. 
Under the given monotonicity assumptions the compactness and transversality arguments \cite[Theorems 1.4,3.3]{Ainfty} are independent of the coefficient ring, and the strip-shrinking analysis of \cite[Section 4]{ww:isom} applies unchanged. 
It remains to check that these moduli spaces carry orientations compatible with the relative spin structures, and that the resulting signs are those of an \ainfty homotopy in the sense of  \cite[Definition 8.1(h)]{Ainfty}. 
By  \cite[Definition 8.1(h)]{Ainfty}, a homotopy 
$\cH$ has degree $|\cH| = -1$, so the sign $\dagger$
in Definition 8.1 of \cite{Ainfty} is even and the sign in the second sum 
is $(-1)^{i + \sum_{j \le i} |a_j|}$. 
The relation 
\[ \Phi(L_{12}) \circ \Phi(L_{01}) - 
\Psi \circ \Phi(L_{01} \circ L_{12}) = \mu^1(\cH) \] 
is therefore equivalent for homogeneous $a_1,\ldots, a_d$ to the following.
With 
\[ \cF_1 = \Phi(L_{12}) \circ \Phi(L_{01}), \quad \cF_2 = \Psi \circ \Phi(L_{01}\circ L_{12}) \] 
we claim, with $\cC_0 = \Fuk^\sharp(M_0, b_0 + w_2(M_0))$ and
$\cC_2 = \Fuk^\sharp(M_2, b_2)$,
\begin{multline} \label{eq:homotopysigns}
\F_1^d(a_1,\ldots, a_d) - \F_2^d(a_1,\ldots,a_d)  \\ = 
\sum  \mu^m_{\cC_2} \bigl(\F_1^{i_1}(\ldots), 
\ldots,  \F_1^{i_{k-1}}(\ldots), 
\cH^{i_k}(\ldots), \F_2^{i_{k+1}}(\ldots), \ldots, 
\F_2^{i_m}(\ldots)  \bigr) \\
- \sum_{i,e} (-1)^{i + \sum_{j \le i} |a_j|} \cH^{d- e + 1} 
(a_1,\ldots, a_i,\mu^e_{\cC_0}(a_{i+1},\ldots, a_{i+e}), a_{i+e+1},\ldots, a_d). \end{multline}
We orient $\M^d_{(0,\infty)}$ by the fiber followed by the interval $(0,\infty)$, and normalize $\Phi(L_{01})^d$, $\Phi(L_{12})^d$, $\Phi(L_{01} \circ L_{12})^d$
and $\cH^d$ by the factors $(-1)^{\heartsuit + \nu}$, with $\heartsuit$
as in \eqref{eq:heartsuit}, with the ordering of patches of Section
\ref{sec:fukops}.  We write
$n_l = \tfrac{1}{2}\dim(M_l)$, $x_0$ for the output, and
$\heartsuit_\sharp^{(a)}$ for the normalization of Section
\ref{sec:fukops}  of the quilted polygons for the
objects $\on{St}^a((\ul{L}^j)^{\on{st}})$ and the inputs
$x_1,\ldots,x_d$, and $\heartsuit_\sharp'$ for that of the objects
$\on{St}^{\langle n_1 \rangle}(\on{St}^{\langle n_0 \rangle}((\ul{L}^j)^{\on{st}}), L_{01})$
and the inputs of $\Phi(L_{12})$.  Then 
\begin{eqnarray*}
\nu(\Phi(L_{01})) &=& \heartsuit_\sharp^{(\langle n_0 \rangle)} + n_0 n_1 d , \\
\nu(\Phi(L_{12})) &=& \heartsuit_\sharp' + n_1 n_2 d , \\
\nu(\cH) &=& \heartsuit_\sharp^{(a)} + (n_0 + n_1) n_2 d , \quad a = \langle n_0 \rangle + \langle n_1 \rangle , \\
\nu(\Phi(L_{01} \circ L_{12})) &=& \heartsuit_\sharp^{(\langle n_0 \rangle)} + n_0 n_2 d .
\end{eqnarray*}
The first, second and fourth are the normalizations of Section
\ref{sec:fukops}.  The quilts for $\cH$ carry
$a$ patches labelled $T^2$ along each boundary arc, so that their
outputs are the objects
$(\on{St}^{\langle n_1 \rangle}(\on{St}^{\langle n_0 \rangle}(\ul{L}^{\on{st}})), L_{01}, L_{12})$,
which differ from the stabilizations of the objects
$(\ul{L}, L_{01}, L_{12})$ of $\Phi(L_{12}) \circ \Phi(L_{01})$ by an even
number of layers.  For each fixed
$\delta$ they define an \ainfty functor $\cC_0 \to \cC_2$, and
$\nu(\cH)$ is the normalization $\nu_a$ of \eqref{eq:nua} for the sequence $(L_{01}, L_{12})$, for which the
signs of these functor relations are standard.
With these normalizations, and the normalization of
$\mu_\sharp$ in \eqref{eq:musharp}, the
signs $\kappa$ of Lemma \ref{lem:quiltedgluing} \eqref{item:qgends}, the shifted
parities of the factors $\DD^\pm_x$, and the signs for reordering
patches cancel in the functor relations for $\Phi(L_{01})$ and
$\Phi(L_{12})$, and in the ends of types \eqref{item:stripbreak} and \eqref{item:domainbreak} below at which
quilts for $\Phi(L_{01})$ are glued to quilts for $\Phi(L_{12})$ or
polygons for $\Fuk^\sharp(M_0)$ are glued to quilts for $\cH$.  The sign $(-1)^{\maltese}$
of Definition 8.1 of \cite{Ainfty} equals the sign $(-1)^{n+ \sum_{i \leq n} |a_i|}$ since 
$|a| + 1 \equiv |a| - 1$ mod $2$.  The ends of a compact
one-dimensional 
component of $\M^d_{(0,\infty)}$ 
are of the following types:
\begin{enumerate}
\item \label{item:stripshrink} {\rm (Strip shrinking, $\delta \to 0$)} By 
\cite[Theorem 2.2.2]{ww:isom} and \cite{Ainfty} these ends
are in bijection with rigid quilts for the 
composed seam condition $L_{01} \circ L_{12}$.  We orient these ends
as ends of the family, that is, by the orientations of Proposition
\ref{prop:noends} for the quilts with patches labelled
$M_0, M_1, M_2$ at small $\delta$, and count them with the
normalization $(-1)^{\heartsuit + \nu_2}$, where
$\nu_2 = \nu(\cH)$.  By Proposition \ref{prop:collapse} and
Remark \ref{rem:collapsesigns}, this count gives the coefficients of
$\F_2^d = (\Psi \circ \Phi(L_{01} \circ L_{12}))^d$, so these ends give
the term $\F_2^d$ in \eqref{eq:homotopysigns}. 
\item \label{item:stripbreak} {\rm (Strip breaking, $\delta \to \infty$)}  These ends
are in bijection with pairs of rigid quilts defining 
$\Phi(L_{01})$, with $a$ patches labelled $T^2$ along each boundary arc, and
$\Phi(L_{12})$, the output of each quilt for $\Phi(L_{01})$ matching the corresponding input of the quilt for $\Phi(L_{12})$.  By Section \ref{sec:fukops}, the counts of the quilts for $\Phi(L_{01})$ with $a$ patches labelled $T^2$, normalized by $\nu_a$, are those of $\Phi(L_{01})$ composed with $\lambda_{n_1}^{\langle n_1 \rangle}$, that is, multiplied by $(-1)^{|x_0| + 1}$ if $n_1$ is odd.  By Theorem \ref{thm:gluingsigns}
and Lemma \ref{lem:quiltedgluing}, the orientation of the glued
family agrees with the product orientation up to the signs $\kappa$
and the signs for reordering patches.  These combine with the
normalizations $\nu(\Phi(L_{01}))$, $\nu(\Phi(L_{12}))$ and
$\nu(\cH)$ to a sign that is independent of how the inputs are
distributed among the quilts for $\Phi(L_{01})$, so that these ends
give $\F_1^d = (\Phi(L_{12}) \circ \Phi(L_{01}))^d$; the composition
formula in Definition 8.1 of \cite{Ainfty} has no additional sign. 
\item \label{item:domainbreak} {\rm (Breaking in the domain)}  The remaining ends
are breakings at boundary nodes of the domain. 
Those at which a polygon with $e$ inputs breaks off at an input of a quilt contributing to $\cH$, a polygon for the objects $\on{St}^a((\ul{L}^j)^{\on{st}})$ which is identified with a polygon for $\cC_0$ by adding layers as in Section \ref{sec:fukops}, are oriented as in the proof of Theorem \ref{thm:fukaya}:  The interval direction 
comes last and the disk is rigid, so the computation of 
(S1)--(S4) applies, together with the signs $\kappa$ of Lemma
\ref{lem:quiltedgluing} \eqref{item:qgends}, which are cancelled by
$\nu(\cH)$ and the normalization $\heartsuit_\sharp$ of $\mu_{\cC_0}$.  They give the second sum in \eqref{eq:homotopysigns}.  The other ends, at which 
quilts contributing to $\F_1^{i_1},\ldots, \F_1^{i_{k-1}},
\cH^{i_k}, \F_2^{i_{k+1}},\ldots, \F_2^{i_m}$ are glued to the inputs of a polygon for $\cC_2$, give the first sum.  For these ends the signs $\kappa$, the parities
$|x| + s$ of the factors $\DD^\pm_x$, and the signs for reordering
patches cancel against the normalizations $\heartsuit_\sharp$ of
$\mu_{\cC_2}$ and $\nu(\Phi(L_{01}))$, $\nu(\Phi(L_{12}))$,
$\nu(\cH)$, and $\nu_2$ for the quilts at small $\delta$ in item \eqref{item:stripshrink}.  Hence the signs of these ends relative to those of
the second sum are the relative signs of the ends
in the functor relation in Definition 8.1 of \cite{Ainfty}, namely that of $\mu^m (\F^{i_1},\ldots, \F^{i_m})$ versus 
$(-1)^{i + \sum_{j \le i} |a_j|} \F(\ldots, \mu, \ldots )$.
\end{enumerate}
The ends in items \eqref{item:stripshrink} and \eqref{item:stripbreak} have opposite signs,
since $\partial_\delta$ points into resp. out of the 
domain moduli space.   Counting boundary points of the one-dimensional components, and using that the signed number of boundary points of a compact oriented one-manifold is zero, we obtain \eqref{eq:homotopysigns} up to an overall sign, 
which is absorbed by replacing $\cH$ with $-\cH$.
\end{proof}

\bibliographystyle{amsalpha} 
\bibliography{Reference}

\end{document}